\documentclass[english]{article}

\usepackage{cite}             % nicer citations

\usepackage{geometry}
\usepackage{graphics}
\usepackage{amsmath}          % dito
\usepackage{amssymb}          % various ams stuff
\usepackage{amsxtra}          % dito
\usepackage[latin1]{inputenc} % smart input of special characters
\usepackage{eucal}            % nicer caligraphic fonts
\usepackage{longtable}        % tables longer than one page
\usepackage{exscale}          % larger summation signs
\usepackage{bbm}              % nicer bbm fonts
\usepackage{theorem}          % nettere Theoreme
\usepackage{paralist}            % kleinere Listen COOL!
\usepackage[all,2cell]{xy}  \UseAllTwocells \SilentMatrices
\usepackage{stmaryrd}         % more symbols & brackets
\usepackage{url}              %websiten in Referenzen

\usepackage{tikz}             %the real real diagrams       
\usepackage{tikz-cd}          %add on 
\usetikzlibrary{matrix, arrows}

\usepackage{leftidx}          %left dual star closer to the object  
\usepackage{a4wide}
\usepackage{xr}

\numberwithin{equation}{section}

\newcommand{\C}        {\mathbb{C}}
\newcommand{\N}        {\mathbb{N}}
\newcommand{\cc}[1]      {\overline{{#1}}}
\newcommand{\mac}[1]{{\mathcal {#1}}}

\newcommand{\tensor}[1]   {\mathbin{\otimes_{\scriptscriptstyle{#1}}}}
\newcommand{\id}         {\operatorname{\mathsf{id}}}
\newcommand{\unit}         {\operatorname{\mathbf{1}}}
\newcommand{\rr}         {\rightarrow}

\newcommand{\act}        {\operatorname{\triangleright}}
\newcommand{\actrd}        {\operatorname{\triangleright}^{\#}}
\newcommand{\ractrd}      {\operatorname{\triangleleft}^{\#}}

\newcommand{\actld}        {\leftidx{^{\#}\!\!}{\operatorname{\triangleright}}{}}
\newcommand{\ractld}       {\leftidx{^{\#}\!\!}{\operatorname{\triangleleft}}{}}
\newcommand{\ract}       {\operatorname{\triangleleft}}
\newcommand{\tr}         {\operatorname{\mathsf{tr}}}
\newcommand{\opp}             {{\mathrm{op}}} 

\newcommand{\Cat}[1]         {\operatorname{\mathcal{#1}}}
\newcommand{\Cath}[1]         {\operatorname{\mathcal{\text{h}#1}}}

\newcommand{\op}[1]           {\Cat{#1}^{\operatorname{\mathrm{op}}}}
\newcommand{\rev}[1]           {\Cat{#1}^{\operatorname{\mathrm{rev}}}}

\newcommand{\Fun}         {\operatorname{\mathrm{Fun}}}
\newcommand{\Funbal}         {\operatorname{\mathrm{Fun}}^{\mathrm{bal}}}
\newcommand{\Hom}        {\operatorname{\mathsf{Hom}}}
\newcommand{\IHom}        {\operatorname{\underline{\mathsf{Hom}}}}

\newcommand{\End}        {\operatorname{\mathsf{End}}}

\newcommand{\Obj}        {\operatorname{\mathsf{Obj}}}
\newcommand{\Bimod}[5]{\sideset{^{\scriptscriptstyle{#1}}_{\scriptscriptstyle{#2}}}{^{\scriptscriptstyle{#4}}_{\scriptscriptstyle{#5}}}{\operatorname{#3}}}
\newcommand{\LCC}  {\Bimod{}{\Cat{C}}{\Cat{C}}{}{}}

\newcommand{\LDD}  {\Bimod{}{\Cat{D}}{\Cat{D}}{}{}}
\newcommand{\RDD}  {\Bimod{}{}{\Cat{D}}{}{\Cat{D}}}

\newcommand{\CM}  {\Bimod{}{\Cat{C}}{\Cat{M}}{}{}}

\newcommand{\MC}  {\Bimod{}{}{\Cat{M}}{}{\Cat{C}}}
\newcommand{\MpC}  {\Bimod{}{}{\Cat{M}}{\prime}{\Cat{C}}}
\newcommand{\MppC}  {\Bimod{}{}{\Cat{M}}{\prime \prime}{\Cat{C}}}
\newcommand{\MpppC}  {\Bimod{}{}{\Cat{M}}{\prime \prime \prime}{\Cat{C}}}

\newcommand{\DM}  {\Bimod{}{\Cat{D}}{\Cat{M}}{}{}}

\newcommand{\CN}  {\Bimod{}{\Cat{C}}{\Cat{N}}{}{}}
\newcommand{\CpN}  {\Bimod{}{\Cat{C}}{\Cat{N}}{\prime}{}}
\newcommand{\CppN}  {\Bimod{}{\Cat{C}}{\Cat{N}}{\prime \prime}{}}
\newcommand{\CpppN}  {\Bimod{}{\Cat{C}}{\Cat{N}}{\prime \prime \prime}{}}

\newcommand{\NC}  {\Bimod{}{}{\Cat{N}}{}{\Cat{C}}}
\newcommand{\NE}  {\Bimod{}{}{\Cat{N}}{}{\Cat{E}}}

\newcommand{\DN}  {\Bimod{}{\Cat{D}}{\Cat{N}}{}{}}
\newcommand{\ND}  {\Bimod{}{}{\Cat{N}}{}{\Cat{D}}}

\newcommand{\CY}  {\Bimod{}{\Cat{C}}{\Cat{Y}}{}{}}
\newcommand{\DY}  {\Bimod{}{\Cat{D}}{\Cat{Y}}{}{}}

\newcommand{\CMD}  {\Bimod{}{\Cat{C}}{\Cat{M}}{}{\Cat{D}}}

\newcommand{\DMpC}  {\Bimod{}{\Cat{D}}{\Cat{M}}{\prime}{\Cat{C}}}

\newcommand{\DMC}  {\Bimod{}{\Cat{D}}{\Cat{M}}{}{\Cat{C}}} 
\newcommand{\DMD}  {\Bimod{}{\Cat{D}}{\Cat{M}}{}{\Cat{D}}}

\newcommand{\CND}  {\Bimod{}{\Cat{C}}{\Cat{N}}{}{\Cat{D}}}

\newcommand{\DNC}  {\Bimod{}{\Cat{D}}{\Cat{N}}{}{\Cat{C}}}
\newcommand{\DNE}  {\Bimod{}{\Cat{D}}{\Cat{N}}{}{\Cat{E}}}

\newcommand{\CNE}  {\Bimod{}{\Cat{C}}{\Cat{N}}{}{\Cat{E}}}
\newcommand{\CNpE}  {\Bimod{}{\Cat{C}}{\Cat{N}}{\prime}{\Cat{E}}}

\newcommand{\END}  {\Bimod{}{\Cat{E}}{\Cat{N}}{}{\Cat{D}}}
\newcommand{\ENF}  {\Bimod{}{\Cat{E}}{\Cat{N}}{}{\Cat{F}}}
\newcommand{\ENpD}  {\Bimod{}{\Cat{E}}{\Cat{N}}{\prime}{\Cat{D}}}

\newcommand{\DKD}  {\Bimod{}{\Cat{D}}{\Cat{K}}{}{\Cat{D}}}

\newcommand{\EKD}  {\Bimod{}{\Cat{E}}{\Cat{K}}{}{\Cat{D}}}

\newcommand{\HKE}  {\Bimod{}{\Cat{H}}{\Cat{K}}{}{\Cat{E}}}

\newcommand{\KE}  {\Bimod{}{}{\Cat{K}}{}{\Cat{E}}}

\newcommand{\DYC}  {\Bimod{}{\Cat{D}}{\Cat{Y}}{}{\Cat{C}}}
\newcommand{\DYE}  {\Bimod{}{\Cat{D}}{\Cat{Y}}{}{\Cat{E}}}

\newcommand{\DAE}  {\Bimod{}{\Cat{D}}{\Cat{A}}{}{\Cat{E}}}
\newcommand{\DAC}  {\Bimod{}{\Cat{D}}{\Cat{A}}{}{\Cat{C}}}

\newcommand{\DBE}  {\Bimod{}{\Cat{D}}{\Cat{B}}{}{\Cat{E}}}

\newcommand{\DDD}  {\Bimod{}{\Cat{D}}{\Cat{D}}{}{\Cat{D}}}
\newcommand{\CCC}  {\Bimod{}{\Cat{C}}{\Cat{C}}{}{\Cat{C}}}

\newcommand{\CaCCb}  {\Bimod{}{(\Cat{C},a)}{\Cat{C}}{}{(\Cat{C},b)}}

\newcommand{\CCCc}  {\Bimod{}{\Cat{C}}{\Cat{C}}{}{\overline{\Cat{C}}}}
\newcommand{\CCCcld}  {\Bimod{\#}{\Cat{C}}{\Cat{C}}{}{\overline{\Cat{C}}}}
\newcommand{\CcCC}  {\Bimod{}{\overline{\Cat{C}}}{\Cat{C}}{}{\Cat{C}}}
\newcommand{\CcCCld}  {\Bimod{\#}{\overline{\Cat{C}}}{\Cat{C}}{}{\Cat{C}}}
\newcommand{\DDDc}  {\Bimod{}{\Cat{D}}{\Cat{D}}{}{\cc{\Cat{D}}}}
\newcommand{\DcDDld}  {\Bimod{\#}{\cc{\Cat{D}}}{\Cat{D}}{}{\Cat{D}}}

\newcommand{\CMDrd}  {\Bimod{}{\Cat{C}}{\Cat{M}}{\#}{\Cat{D}}} 
\newcommand{\CMDld}  {\Bimod{\#}{\Cat{C}}{\Cat{M}}{}{\Cat{D}}} 
\newcommand{\CNDrd}  {\Bimod{}{\Cat{C}}{\Cat{N}}{\#}{\Cat{D}}} 
\newcommand{\CNDld}  {\Bimod{\#}{\Cat{C}}{\Cat{N}}{}{\Cat{D}}} 

\newcommand{\DDDld}  {\Bimod{\#}{\Cat{D}}{\Cat{D}}{}{\Cat{D}}} 
\newcommand{\DDDrd}  {\Bimod{}{\Cat{D}}{\Cat{D}}{\#}{\Cat{D}}}

\newcommand{\DMDrd}  {\Bimod{}{\Cat{D}}{\Cat{M}}{\#}{\Cat{D}}} 
\newcommand{\DMDld}  {\Bimod{\#}{\Cat{D}}{\Cat{M}}{}{\Cat{D}}} 

\newcommand{\DKDrd}  {\Bimod{}{\Cat{D}}{\Cat{K}}{\#}{\Cat{D}}} 
\newcommand{\DKDld}  {\Bimod{\#}{\Cat{D}}{\Cat{K}}{}{\Cat{D}}} 

\newcommand{\DMCrrd}  {\Bimod{}{\Cat{D}}{\Cat{M}}{\#\#}{\Cat{C}}} 
\newcommand{\DMDrrd}  {\Bimod{}{\Cat{D}}{\Cat{M}}{\#\#}{\Cat{D}}} 

\newcommand{\DMClld}   {^{ \scriptscriptstyle{\#\#}\!\! \!\! \!\!}\DMC} % {\Bimod{}{\Cat{D} }{^{\# \# \!}\Cat{M}}{}{\Cat{C}}} % {_{\Cat{D} \!\!} ^{\# \# \!\!} \Cat{M}_{\Cat{D}}} 
\newcommand{\DcMClld}   {^{ \scriptscriptstyle{\#\#}\!\! \!\! \!\!}\DcMC} 

\newcommand{\DMCc}  {\Bimod{}{\Cat{D}}{\Cat{M}}{}{\cc{\Cat{C}}}} 
\newcommand{\CcMDld}  {\Bimod{\#}{\cc{\Cat{C}}}{\Cat{M}}{}{\Cat{D}}} 
\newcommand{\CMDcld}  {\Bimod{\#}{\Cat{C}}{\Cat{M}}{}{\cc{\Cat{D}}}} 

\newcommand{\DcMC}  {\Bimod{}{\cc{\Cat{D}}}{\Cat{M}}{}{\Cat{C}}}

\newcommand{\MDrd}  {\Bimod{}{}{\Cat{M}}{\#}{\Cat{D}}} 
\newcommand{\MDld}  {\Bimod{\#}{}{\Cat{M}}{}{\Cat{D}}} 
\newcommand{\CMrd}  {\Bimod{}{\Cat{C}}{\Cat{M}}{\#}{}} 
\newcommand{\CMld}  {\Bimod{\#}{\Cat{C}}{\Cat{M}}{}{}} 
\newcommand{\DMrd}  {\Bimod{}{\Cat{D}}{\Cat{M}}{\#}{}} 
\newcommand{\DMld}  {\Bimod{\#}{\Cat{D}}{\Cat{M}}{}{}} 
 
\newcommand{\NCld}  {\Bimod{\#}{}{\Cat{N}}{}{\Cat{C}}} 

\newcommand{\Nrd}  {\Bimod{}{}{\Cat{N}}{\#}{}} 
\newcommand{\Nld}  {\Bimod{\#}{}{\Cat{N}}{}{}} 

\newcommand{\Mrd}  {\Bimod{}{}{\Cat{M}}{\#}{}} 
\newcommand{\Mld}  {\Bimod{\#}{}{\Cat{M}}{}{}}

\newcommand{\ENCld}  {\Bimod{\#}{\Cat{E}}{\Cat{N}}{}{\Cat{C}}} 
\newcommand{\ENCrd}  {\Bimod{}{\Cat{E}}{\Cat{N}}{\#}{\Cat{C}}}

\newcommand{\CMstar}  {\Bimod{}{}{\Cat{C}}{*}{\Cat{M}}}

\newcommand{\DMstar}  {\Bimod{}{}{\Cat{D}}{*}{\Cat{M}}}

\newcommand{\Mod}      {\operatorname{\mathsf{Mod}}}

\newcommand{\ModbC}      {{\mathsf{Mod}}^{\mathrm{bal}}_{\scriptscriptstyle{\Cat{C}}}}
\newcommand{\BimodbC}      {{\mathsf{Bimod}}^{\mathrm{bal}}_{\scriptscriptstyle{\Cat{C}}}}

\newcommand{\AModC}    {\Bimod{}{A}{\mathsf{Mod}(\Cat{C})}{}{}}
\newcommand{\ModCB}    {\Bimod{}{}{\mathsf{Mod}(\Cat{C})}{}{B}}
\newcommand{\ModCA}    {\Bimod{}{}{\mathsf{Mod}(\Cat{C})}{}{A}}
\newcommand{\ModDA}    {\Bimod{}{}{\mathsf{Mod}(\Cat{D})}{}{A}}
\newcommand{\AModD}    {\Bimod{}{A}{\mathsf{Mod}(\Cat{D})}{}{}}
\newcommand{\AddModD}    {\Bimod{}{A^{**}}{\mathsf{Mod}(\Cat{D})}{}{}}

\newcommand{\AModCA}    {\Bimod{}{A}{\mathsf{Mod}(\Cat{C})}{}{A}}
\newcommand{\AModCB}    {\Bimod{}{A}{\mathsf{Mod}(\Cat{C})}{}{B}}

\newcommand{\IP}[4]{{\,}_{\scriptscriptstyle{#2}\!\!}\left\langle{{#1}}\right\rangle^{\scriptscriptstyle{#3}}_{\scriptscriptstyle{#4}}}
\newcommand{\IPl}[4]{{\,}_{\scriptscriptstyle{#2}\!\!}^{*}\left\langle{{#1}}\right\rangle^{\scriptscriptstyle{#3}}_{\scriptscriptstyle{#4}}}
\newcommand{\IPs}[4]{{\,}_{\scriptscriptstyle{#2}\!\!}\left\langle{{#1}}\right\rangle^{{#3}}_{\scriptscriptstyle{#4}}}

\newcommand{\igen}[1]    {\IP{{#1}}{}{}{}}
\newcommand{\igenstar}[1]    {\IPs{{#1}}{}{*}{}}

\newcommand{\idgen}[1]    {\IP{{#1}}{\Cat{D}}{}{}}

\newcommand{\icgen}[1]    {\IP{{#1}}{\Cat{C}}{}{}}

\newcommand{\imd}[1]    {\IP{{#1}}{}{\Cat{M}}{\Cat{D}}}
\newcommand{\inc}[1]    {\IP{{#1}}{}{\Cat{N}}{\Cat{C}}}

\newcommand{\igenc}[1]    {\IP{{#1}}{}{{}}{\Cat{C}}}

\newcommand{\icm}[1]    {\IP{{#1}}{\mathcal{C}}{\mathcal{M}}{}}
\newcommand{\icmstar}[1]    {\IPs{{#1}}{\Cat{C}}{*}{}}
\newcommand{\idstar}[1]    {\IPs{{#1}}{\Cat{D}}{*}{}}

\newcommand{\idm}[1]    {\IP{{#1}}{\Cat{D}}{\Cat{M}}{}}

\newcommand{\idn}[1]    {\IP{{#1}}{\Cat{D}}{\Cat{N}}{}}
\newcommand{\idmp}[1]    {\IP{{#1}}{\Cat{D}}{\Cat{M}'}{}}

\newcommand{\idnp}[1]    {\IP{{#1}}{\Cat{D}}{\Cat{N}'}{}}

\newcommand{\idmstar}[1]    {\IP{{#1}}{\Cat{D}}{*}{}}
\newcommand{\idnstar}[1]    {\IP{{#1}}{\Cat{D}}{*}{}}

\newcommand{\imc}[1]    {\IP{{#1}}{}{\mathcal{M}}{\Cat{C}}}
\newcommand{\imcstar}[1]    {\IPl{{#1}}{}{}{\Cat{C}}}
\newcommand{\icstarl}[1]    {\IPl{{#1}}{}{}{\Cat{C}}}

\newcommand{\icnstarl}[1]    {\IPl{{#1}}{}{}{\Cat{C}}}

\newcommand{\imldd}[1]    {\langle {#1} \rangle^{^{\#\!}\Cat{M}_{\Cat{D}}}_{\scriptscriptstyle{\Cat{D}}}}
\newcommand{\icmrd}[1]    {{\,}_{\scriptscriptstyle{\Cat{C}}\!}\langle {#1}\rangle^{_{\Cat{C}\!}\Cat{M}^{\#}}}

\newcommand{\icdualm}[1]     { _{\scriptscriptstyle{\Cat{C}_{\Cat{M}}^{*}}\!}\langle {#1}\rangle^{\Cat{M}}}
\newcommand{\icdualmn}[1]    { _{\scriptscriptstyle{\Cat{C}_{\Cat{M}}^{*}}\!}\langle {#1}\rangle^{\Cat{M} \Box \Cat{N}}}
\newcommand{\icn}[1]    {\IP{{#1}}{\Cat{C}}{\Cat{N}}{}}

\newcommand{\icgenstar}[1]    {\IP{{#1}}{\Cat{C}}{*}{}}

\newcommand{\ine}[1]    {\IP{{#1}}{}{\mathcal{N}}{\Cat{E}}}

\newcommand{\idmn}[1]    {\IP{{#1}}{\Cat{D}}{\Cat{M} \Box \Cat{N}}{}}
\newcommand{\idmpnp}[1]    {\IP{{#1}}{\Cat{D}}{\Cat{M}' \Box \Cat{N}'}{}}

\newcommand{\idmmld}[1]    {\IP{{#1}}{\Cat{D}}{\Cat{M} \Box ^{\scriptscriptstyle{\#}\!\!}\Cat{M}}{}}
\newcommand{\idmmrd}[1]    {\IP{{#1}}{}{\Cat{M} \Box ^{\scriptscriptstyle{\#}\!\!}\Cat{M}}{\Cat{D}}}
\newcommand{\icmrdm}[1]    {\IP{{#1}}{\Cat{C}}{\Cat{M} ^{\scriptscriptstyle{\#}} \Box \Cat{M}}{}}
\newcommand{\imrdmc}[1]    {\IP{{#1}}{}{\Cat{M} ^{\scriptscriptstyle{\#}} \Box \Cat{M}}{\Cat{C}}}
\newcommand{\imrdd}[1]    {\IP{{#1}}{}{\Cat{M} ^{\scriptscriptstyle{\#}}}{\Cat{D}}}
\newcommand{\icmld}[1]    {\IP{{#1}}{\Cat{C}}{ ^{\scriptscriptstyle{\#}\!\!}{\,}_{\sss{\Cat{C}}}\Cat{M}}{}}

\newcommand{\imne}[1]    {\IP{{#1}}{}{\Cat{M} \Box \Cat{N}}{\Cat{E}}}

\newcommand{\Funl}[2]     {\operatorname{\mathrm{Fun}}_{\scriptscriptstyle{#1}}{(#2)}}
\newcommand{\Funball}[2]  {\operatorname{\mathrm{Fun}}^{\mathrm{bal}}_{\scriptscriptstyle{#1}}{(#2)}}

\newcommand{\Funlmul}[3]     {\operatorname{\mathrm{Fun}}_{\scriptscriptstyle{#1}}^{\scriptscriptstyle{#2}}{(#3)}}

\newcommand{\Funrd} [1] {\mathsf{#1}^{\#}}
\newcommand{\Funld} [1] {^{\# \!}\mathsf{#1}}

\newcommand{\Funrtd} [1] {{#1}^{\widetilde{\#}}}
\newcommand{\Funrrtd} [1] {{#1}^{\widetilde{\#}\widetilde{\#}}}

\newcommand{\Funltd} [1] {^{\widetilde{ \#}\!}{#1}}
\newcommand{\Funlltd} [1] {^{\widetilde{ \#}\widetilde{ \#}\!}{#1}}

\newcommand{\BimCat}       {\operatorname{\mathsf{Bimod}}}   %{\mathfrak{B}}%{\underline{\underline{\mathsf{B}}}}  
\newcommand{\BimCattet}       {\operatorname{\mathsf{Bimod}}^{\theta}}   %{\mathfrak{B}}%{\underline{\underline{\mathsf{B}}}}  

\newcommand{\Bm}       {\operatorname{\mathsf{Bimod}^{multi}}}   %{\mathfrak{B}}%{\underline{\underline{\mathsf{B}}}}  
\newcommand{\Bms}       {\operatorname{\mathsf{B}^{m}}}  
\newcommand{\tm}        {\times^{multi}}

\newcommand{\Vect}   {\mathsf{Vect}}

\newcommand{\Catlin}   {\mathsf{Cat}^{\mathrm{lin}}}

\newcommand{\ev}[1]   {\operatorname{\mathsf{ev}}_{#1}}

\newcommand{\coev}[1]   {\operatorname{\mathsf{coev}}_{#1}}
\newcommand{\evp}[1]   {\operatorname{\mathsf{ev}}_{#1}^{\prime}}

\newcommand{\coevp}[1]   {\operatorname{\mathsf{coev}}_{#1}^{\prime}}

\newcommand{\evvect}[1]   {\operatorname{\mathsf{ev}}_{#1}^{\Vect}}
\newcommand{\coevvect}[1]   {\operatorname{\mathsf{coev}}_{#1}^{\Vect}}  
\makeatletter
\newcommand{\Rmnum}[1]{\expandafter\@slowromancap\romannumeral #1@}
\makeatother

\newcommand{\sss}[1] {\scriptscriptstyle{#1}}

\theoremheaderfont{\normalfont\bfseries}
\theorembodyfont{\itshape}
\newtheorem{lemma}{Lemma}[section]
\newtheorem{proposition}[lemma]{Proposition}
\newtheorem{theorem}[lemma]{Theorem} 
\newtheorem{corollary}[lemma]{Corollary}
\newtheorem{definition}[lemma]{Definition}

\theorembodyfont{\rmfamily}

\newtheorem{example}[lemma]{Example}
\newtheorem{remark}[lemma]{Remark}

\newenvironment{theorem-1}[1][Theorem 1]{\begin{trivlist}
  \item[\hskip \labelsep {\bfseries #1}]}{\end{trivlist}}
\newenvironment{theorem-2}[1][Theorem 2]{\begin{trivlist}
  \item[\hskip \labelsep {\bfseries #1}]}{\end{trivlist}}
\newenvironment{theorem-3}[1][Theorem 3]{\begin{trivlist}
  \item[\hskip \labelsep {\bfseries #1}]}{\end{trivlist}}
\newenvironment{cor-n}[1][Corollary]{\begin{trivlist}
  \item[\hskip \labelsep {\bfseries #1}]}{\end{trivlist}}



\newcommand{\refitem}[1] {~\textit{\ref{#1})}}

\makeatletter
\newcommand\qedsymbol{\hbox{$\boxempty$}}
\newcommand\qed{\relax\ifmmode\boxempty\else
  {\unskip\nobreak\hfil\penalty50\hskip1em\null\nobreak\hfil\qedsymbol
    \parfillskip=\z@\finalhyphendemerits=0\endgraf}\fi}
\makeatother
\newenvironment{proof}[1][{}]{\par\noindent Proof{#1}. }{\qed}
\newenvironment{theoremlist}{\begin{enumerate}}{\end{enumerate}}
\newenvironment{remarklist}{\begin{enumerate}}{\end{enumerate}}
\newenvironment{examplelist}{\begin{enumerate}}{\end{enumerate}}
\newenvironment{lemmalist}{\begin{enumerate}}{\end{enumerate}}
\newenvironment{propositionlist}{\begin{enumerate}}{\end{enumerate}}
\newenvironment{definitionlist}{\begin{enumerate}}{\end{enumerate}}
\newenvironment{corollarylist}{\begin{enumerate}}{\end{enumerate}}

\usepackage{color}

\tikzset{
  doublearrow/.style={draw, thin, double distance=3pt, ->, >=implies},
  ldoublearrow/.style={draw, thin, double distance=3pt, <-, >=implies},
  thirdline/.style={draw, thin, ->, >=implies}, 
  lthirdline/.style={draw, thin, <-, >=implies},
  equality/.style={draw, thin, double distance=3pt, -} 
}

\makeatletter
\def\latearrow#1#2#3#4{%
  \toks@\expandafter{\tikzcd@savedpaths\path[/tikz/commutative diagrams/every arrow,#1]}%
  \global\edef\tikzcd@savedpaths{%
    \the\toks@%
    (\tikzmatrixname-#2)% \noexpand\tikzcd@sourceanchor)%
    to%
    node[/tikz/commutative diagrams/every label] {$#4$}
    (\tikzmatrixname-#3)% \noexpand\tikzcd@targetanchor)
;}}
\makeatother

\title{Pivotal tricategories and a categorification of inner-product modules}
 \author{
\textbf{Gregor Schaumann}
\thanks{email:schaumann@mpim-bonn.mpg.de}
\\[0.1cm]
Max Planck Institute for Mathematics, \\
Vivatsgasse 7\\
53111 Bonn\\
Germany
}
\date{}
 
\begin{document}
 \maketitle

 \abstract{This article investigates duals for bimodule categories over finite tensor categories. We show that finite bimodule categories  form a tricategory and 
discuss the dualities in this tricategory using  inner homs. 
 We consider inner-product bimodule categories over pivotal tensor categories  with additional structure on the inner homs. Inner-product module categories  are 
related to Frobenius algebras and lead to the notion of $*$-Morita equivalence for pivotal tensor categories. 
We  show that inner-product bimodule categories form a tricategory with two  duality operations 
 and  an additional pivotal structure. This is work is motivated by defects in topological field theories.
}

 \section{Introduction}

Tensor categories are intimately linked with low-dimensional topology via 3-dimensional topological field 
theory and  2-dimensional conformal field theories (TFT and CFT). 
There is ongoing work by many researchers to generalize  these theories by constructing on the one hand   non-semisimple theories \cite{Lyu, Runkel} and on the other to include defects in existing semisimple theories 
\cite{DualDef, Kong}.
It is expected by physics and topological reasons 
that defects of all codimension in these theories  form a tricategory with certain dualities. 
In an oriented theory, the dualities should be equipped with additional structures. 
We propose the notion of a pivotal tricategory with duals as  structure of oriented defects in 3-d TFT. 
It is expected that an important class of defects can be obtained from bimodule categories over tensor categories \cite{FuSchwVal, Kong}. 
We define the class of inner-product bimodule categories over pivotal 
finite tensor categories and show that these define a pivotal 
tricategory with duals.
This naturally leads to a notion of $*$-Morita equivalence for pivotal finite tensor categories, in particular it allows to transport pivotal structures to the categories 
of endofunctors. We show that inner-product module categories are  linked with  Frobenius algebras in tensor categories.

\paragraph{Defects in TFT} 

Defects arose first in the physics literature as lower-dimensional regions seperating 
different phases of a statistical mechanical theory. The defects themselves can have defects of lower dimensions. 
 In a 3-dimensional theory, a general defect is located on a surface, 
that itself could contain  line defects which in turn  might exhibit
various point defects.  If the theory is topological and the defects are topological, there should be a notion of a fusion of two defects 
of a given codimension to a defect of the same codimension. Furthermore one requires some sort of associativity and 
unital properties of the fusion. In the case of 2-dimensional theories \cite{DavyRun} considers the bicategory of bimodules over algebras as defects. 
 In three dimensions the relevant notion is that  of a tricategory.  If additionally the 
defects are oriented, orientation reversal corresponds to certain duality operations in the tricategory including a relation between the left and right duals of the defects. These structures are formalized in the notion 
of a pivotal tricategory with duals. 
In particular, in a pivotal tricategory with duals describing defects,  the monoidal category of line defects from a surface defect to itself has duals and  a pivotal structure. 

The dualities in a pivotal tricategory with duals are weak duals in the sense that they are not equipped with fixed coherence structures. In \cite{Schaum} it is shown that a pivotal tricategory with duals 
can be strictified to a Gray category with duals \cite{GrayDuals}, where all coherence structures are fixed such that they suit a 3-dimensional diagrammatic calculus. 

It is a fundamental problem to identify for a given TFT and CFT algebraic input data for defects and to describe the 
corresponding tricategory. In the case of a semisimple TFT, it is argued in \cite{FuSchwVal}, that bimodule 
categories describe a large class of surface defects.  In \cite{KapSau, FuSchwVal}, the close relationship between defects and Frobenius algebras is discussed.
The oriented semisimple TFT of \cite{TurVir, BarWes} uses a spherical fusion category as algebraic input datum. It is natural to expect that the bimodule categories have to be compatible with the spherical structure in order to 
define surface defects. A possible compatibility that is intimately related to Frobenius algebras is provided by the notion of a (bi)module trace \cite{Moduletr}. 
For non-semisimple tensor categories we propose a categorification of inner-product modules as algebraic datum for surface defects.

\paragraph{Inner-product modules}

Inner-product module were first defined by Kaplansky \cite{Kap} in the setting of $C^{*}$-algebras and used by Rieffel \cite{Rief} to define the notion of strong Morita equivalence for $C^{*}$- and $W^{*}$-algebras. 
The theory was developed in an algebraic setting  in \cite{BurWald} for ${*}$-algebras $C$ over $\C$ and more general ordered rings, where 
$*$ is an antilinear involutive antihomomorphism. 
The following is a slight modification of the definition in \cite{BurWald}.  
\begin{definition}
\label{definition:inner-prod-starald}
Let  $C$ be a   ${*}$-algebra over $\C$.
  An inner-product module over $C$  
  is a $C$-module $M$ with a $C$-valued sesquilinear inner product 
$_{C}\langle .,. \rangle^{M}: M \times M \rr C$ that is non-degenerate, $C$-linear 
in the first argument, i.e. $_{C} \langle c \act m, \widetilde{m} \rangle = c \cdot \langle m, \widetilde{m} \rangle$ and  satisfies $_{C}\langle m, \widetilde{m} \rangle^{M} = (_{C}\langle \widetilde{m},m \rangle^{M})^{*}$.

An inner product bimodule $_{D\!\!}M_{C}$ for two $*$-algebras is a bimodule that is both a  left $D$ and right $C$ inner product module, 
such that $_{C}\langle m \ract c, \widetilde{m} \rangle^{M}= _{ \; C \! \! \!}\langle m, \widetilde{m} \ract c^{*} \rangle^{M}$ and 
$\langle d \act m, \widetilde{m} \rangle_{D}^{M}=  \langle m, d^{*} \act \widetilde{m} \rangle _{D}^{M}$.
\end{definition}

One motivation for this definition is that inner-product bimodules 
over $*$-algebras form a bicategory with objects $*$-algebras, 1-morphisms inner-product bimodules and 2-morphisms intertwiners. 
The $D$-valued inner-product of the relative tensor product $_{D}M_{C} \tensor{C} {}_{C}N_{E}$ of two bimodule 
categories is thereby defined by the so-called Rieffel-induction: 
For $m \otimes n, \widetilde{m} \otimes \widetilde{n} \in M \otimes N$,
\begin{equation}
  \label{eq:Rieffel-ind-inn}
  _{D\!}\langle m \ract \; _{C\!}\langle n, \widetilde{n} \rangle^{N}, \widetilde{m} \rangle^{M} 
\end{equation}
is $C$-balanced and induces a $D$-valued inner product on  $_{D}M_{C} \tensor{C} {}_{C}N_{E}$. 
Restricted to invertible inner-product bimodules, this bicategory leads to the $*$-Picard groupoid for 
$*$-algebras and to the notion of
$*$-Morita equivalence \cite{Rief}. 

 We apply these notions to tensor categories using the following observation. 
If $\Cat{C}$ is a pivotal tensor category over $\C$, its complexified Grothendieck ring $Gr(\Cat{C})$ is naturally 
a $*$-algebra, where the $*$-structure is given by the (say left) duality operation on $\Cat{C}$ extended antilinearly to the complexification. The pivotal structure serves to guarantee the identity $**=1$ on $Gr(\Cat{C})$.
Furthermore, for bimodule categories $\DMC$ over finite tensor categories, there exist inner hom functors, that 
provide the Grothendieck group $Gr(\Cat{M})$ with the structure of an $(Gr(\Cat{D}), Gr(\Cat{C}))$-bimodule 
 with two algebra valued inner-products, apart from the condition  $\igenstar{m,n}=\igen{n,m}$.
Our definition of an inner-product bimodule category is such that its Grothendieck group 
satisfies also this relation. 

While in Definition \ref{definition:inner-prod-starald} the inner-products are additional structure on the  bimodules, 
in our categorified version, the inner products are canonically given by the inner-hom for a bimodule category. It is just a coherent isomorphism 
   $I_{m,n}: \igenstar{m,n}\simeq \igen{n,m}$ for objects $m,n \in \Cat{M}$ that appears as additional structure. It might be interesting to consider 
also bimodule categories, where a different inner product is part of the structure. 

\newpage
\paragraph{Results}

Our first result considers bimodule categories over finite tensor categories. 
\begin{theorem}
  Finite tensor categories, finite bimodule categories, right exact bimodule functors and bimodule natural transformations form an algebraic tricategory $\BimCat$ in the sense of \cite{Gurski}. 
\end{theorem}
 This result is obtained using the 2-functorial properties of the tensor product of bimodule categories in a systematic way. It extends the work of \cite{Green} and is built heavily on the results of \cite{FinTen, ENOfuhom}. 
For the existence of the tensor product of bimodule categories  we rely on \cite{DSS} and the subsequent work \cite{DSSbal}\footnote{We are grateful to the authors for communicating  early drafts of \cite{DSSbal, DSStri}.}.

 Using the inner hom functors we describe  duals for  the bimodule categories in this tricategory.  These dualities were  obtained by different methods in \cite{DSS} and independently in  the  semisimple case  in  \cite{Schaum}. 
The methods in this article are  generalizations of the methods used in \cite{Schaum}. 
 We show that for separable bimodule categories  \cite{DSS},  the bimodule functors have both adjoints and thus the tricategory of 
separable bimodule categories is a tricategory with (two types of) duals  \cite{FinTen, DSS}.  
The calculus of the inner homs allows furthermore to characterize the Serre equivalences 
between the left and right duals of separable bimodule categories. 

 For a pivotal tricategory with duals we furthermore ask for a pivotal structure for the bimodule functors. 
To this end we define inner-product bimodule categories over pivotal finite tensor category and show the following. 
\begin{theorem}
Inner-product bimodule categories over  pivotal finite tensor categories form a pivotal tricategory with duals. In particular, the category of endofunctors for an inner-product module category is again a pivotal finite tensor category
in a canonical way. 
\end{theorem}
It is shown that a version of the Rieffel-induction on inner homs holds  for finite bimodule categories and this is used in a crucial way to induce the structure of an inner-product bimodule category on 
the tensor product of two inner-product bimodule categories. 
The invertible morphisms in this tricategory lead to notion of $*$-Morita equivalence for pivotal finite tensor categories. 
Examples of inner-product module categories are obtained from  Frobenius algebras and  from  module traces  \cite{Moduletr} in  the semisimple case:
\begin{theorem}
  \begin{theoremlist}
    \item If $A \in \Cat{C}$ is a special symmetric Frobenius algebra in a pivotal finite tensor category $\Cat{C}$, then the category of modules $\ModCA$ is a $\Cat{C}$-inner-product module category. 
\item A semisimple bimodule category $\DMC$ over spherical fusion categories $\Cat{C}$, $\Cat{D}$ with bimodule trace is an inner-product bimodule category. 
  \end{theoremlist}
\end{theorem}

TFTs with defects of all  codimension are expected to define also fully-extended TFTs \cite{Lurie}. It is not surprising 
that some of the structures in this work appear also in the work \cite{DSS} on fully dualizable 3-d TFTs. It is shown in \cite{DSStri}  using different methods that finite bimodule categories over finite tensor categories form even a symmetric monoidal tricategory and  in \cite{DSS}  that   finite tensor categories  are 2-dualizable 
in this symmetric monoidal tricategory. Furthermore it is argued that they should lead to a non-compact framed 3-d TFT. The tricategory $\BimCat$ should then constitute defects for this TFT.
It is natural to expect that inner-product bimodule categories are related to a homotopy fixed point for passing from this framed to an oriented theory.   

Our description of the duals in $\BimCat$ uses mainly tools from enriched category theory and might be interesting for other higher categories as well.

\paragraph{Structure of the article}

In Section \ref{sec:Prelim} we recall the basic notions of bimodule categories, bimodule functors, bimodule natural transformations and balanced functors. 
 We then investigate the tensor product of module categories in Section \ref{sec:tricat-bimod}. 
In the remainder of this section we develop the theory of multi-module categories that serves as an important tool in the proof that bimodule categories over finite tensor categories form a tricategory. 
In Section \ref{sec:bimod-categ-as} we define tricategories with duals and  pivotal tricategories. Next we  enhance the existing calculus of the inner hom and use it to define 
the duals for 
 bimodule categories.  Furthermore we discuss the Serre bimodule functors between the left and right duals of a separable bimodule category. 
In the last section we define inner-product bimodule categories and show that they form a pivotal tricategory. 
The example of Frobenius algebras and the relation with bimodule traces for semisimple module categories is discussed. 
The appendix contains definitions and conventions for duals in monoidal categories, bicategories and the definition of an algebraic tricategory. 

Part of this work, especially  some results of Section 3 appeared in the authors PhD thesis \cite{Schaum}.

 \section{Preliminaries on (bi)module categories}
 \label{sec:Prelim}
We  summarize definitions and known results about module categories over finite tensor categories. Let $\Bbbk$ be a  field. %todo check need perfect?  
Throughout this work, all categories are 
assumed to be $\Bbbk$-linear and abelian  and all functors are requested to be linear unless stated otherwise.

\paragraph{Module categories, functors and natural transformations}

The definition of a finite tensor category is recalled in Definition \ref{definition:Tensor-cat}.
The systematic investigation of module categories over finite tensor categories, was initiated in \cite{FinTen}. 

 \begin{definition}[\cite{Ostrik}, \cite{Benabou}]
   \label{definition:mod-cat}
 Let $\Cat{C}$ be a finite tensor category. 
   A (left) $\Cat{C}$-module category is a finite $\Bbbk$-linear abelian category $\Cat{M}$,
   together  with a bilinear  exact  functor  
  $\act : \Cat{C} \times \Cat{M} \rr \Cat{M}$,  called the action of $\Cat{C}$ on $\Cat{M}$, and 
   natural isomorphisms
   \begin{equation}
     \label{eq:structures-module-category}
     \mu_{x,y,m}^{\Cat{M}}: (x\otimes y) \act m \rr x\act (y \act m), \quad
     \lambda_m^{\Cat{M}}: \unit_{\Cat{C}}\act m \rr m, 
   \end{equation}
   for all $x, y \in \Cat{C}$, $m \in \Cat{M}$, called  the module constraints,  such that  the diagrams
   \begin{equation}
     \label{eq:diagramm}
     \begin{tikzcd}
       {}        & ((x \otimes y) \otimes z) \act m \ar{ld}{\omega_{x, y, z}\act \id_m}
       \ar{rd}{\mu^{\Cat{M}}_{x\otimes y,z,m}}& \\
       (x \otimes (y \otimes z)) \act m \ar{d}{\mu^{\Cat{M}}_{x,y\otimes z,m}} & & (x
       \otimes y ) \act (z \act m) \ar{d}{\mu^{\Cat{M}}_{x,y, z\act m}}\\
       x \act (( y \otimes z) \act m)\ar{rr}{\id_x\act \mu^{\Cat{M}}_{y,z,m}}& & x \act (y \act (z \act m)),
     \end{tikzcd}   
   \end{equation}
   and
   \begin{equation} 
     \label{eq:triangle}
     \begin{tikzcd}
       (x\otimes 1) \act m \ar{rr}{ \mu^{\Cat{M}}_{x,1,m}}\ar{dr}{\rho_x \act m} & & x \act (1 \act  m) \ar{dl}{1_x \act \lambda_m^{\Cat{M}}}\\
       & x \act m & 
     \end{tikzcd}
   \end{equation}
   commute for all objects $x,y,z \in \Cat{C}$ and $m \in \Cat{M}$, where the isomorphisms  $\omega_{x,y,z}:(x \otimes y )\otimes z \rightarrow x \otimes( y \otimes z)$ and 
   $\rho_{x}:x \otimes 1 \rightarrow x$ are the constraint  morphisms of $\Cat{C}$ as a monoidal category. To emphasize that $\Cat{M}$ is a left $\Cat{C}$-module category, we  denote it $\CM$. 
   Whenever this is unambiguous, we denote the constraints of $\Cat{M}$ just by $\mu$ and $\lambda$.

\end{definition}

 The definition of a  right $\Cat{C}$-module category $\MC$ is analogously given in terms of  a bilinear  exact functor $\ract: \MC \times 
 \Cat{C} \rr \MC$. We denote the constraint for the unit  of a right module category  by $\rho^{\Cat{M}}_{m}: m \ract \unit_{\Cat{C}} \rightarrow m$ and where it is otherwise ambiguous, we 
 denote a left module action on a category $\Cat{M}$ by $\mu^{\Cat{M},l}$ or just $\mu^{l}$ and the right module action by $\mu^{\Cat{M},r}$ or just $\mu^{l}$.

 It is clear, that a $\Cat{D}$-module category structure on $\Cat{M}$ is the same as a tensor functor 
 $L_{-}: \Cat{D} \rightarrow \Fun(\Cat{M},\Cat{M})$, where $L_{d}(m)= d \act m$ for $d \in \Cat{D}$ and $m \in \Cat{M}$. 

The following is an important subclass of module categories, that is investigated in detail in \cite{FinTen}. 
\begin{definition}[{{\cite[Def. 3.1]{FinTen}}}]  
  A module category $\CM$ is called exact, if for any projective object $P \in \Cat{C}$ and any object
$m \in \Cat{M}$, the object $P \act m$ is projective in $\Cat{M}$. 
\end{definition}
If $\Cat{C}$ is semisimple, a module category $\CM$ is exact if and only if it is semisimple \cite{FinTen}.

 We denote by $\rev{C}$  the category $\Cat{C}$ with the
 reversed monoidal product, but the same source and target map for the
 morphisms. This has to be  distinguished from  $\op{C}$ which is  the category $\Cat{C}$ with reversed order of
 the arrows but with the same monoidal product as $\Cat{C}$.
 It follows directly from the definitions, that a $\Cat{C}$-right module category is the same as a $\rev{C}$-left module
 category.
 \begin{example}
\label{example:unit-module-cat-alg}
We consider some examples of module categories over $\Cat{C}$.   
\begin{examplelist}
\item Let $\Vect$ denote the category of finite dimensional $\Bbbk$-vector spaces regarded as semisimple tensor category. Every finite linear category  $\Cat{M}$ is a $\Vect$-module category 
with action determined by $\Hom_{\Cat{M}}(\widetilde{m}, V \otimes m)= \Hom_{\Cat{M}}(\widetilde{m},m) \tensor{\Bbbk} V$ for $V \in \Vect$, $m,\widetilde{m} \in \Cat{M}$. 
  \item  \label{item:unit-module}  The category $\Cat{C}$ itself is a left $\Cat{C}$- and right $\Cat{C}$-module category  with actions given by the tensor product. It is exact as left $\Cat{C}$- and right $\Cat{C}$-module category.
\item Let $A \in \Cat{C}$ be an algebra object, then the category $\ModCA$ of $A$-right modules in $\Cat{C}$ is naturally a left $\Cat{C}$-module category with module action given by the tensor product. 
   \end{examplelist}
 \end{example}

 \begin{remark}
   \label{remark:clash-biadditve}
 Let $\CM$ be a $\Cat{C}$-module category and $\Cat{N}$ any finite  category. It is clear that the 
 functor $\act \times 1_{\Cat{N}}: \Cat{C} \times \Cat{M} \times \Cat{N} \rightarrow \Cat{M} \times \Cat{N}$ satisfies the properties (\ref{eq:diagramm}) and (\ref{eq:triangle}) of a $\Cat{C}$-module  action 
on $\Cat{M}  \times \Cat{N}$.  We will thus abuse notation and  call the category $\CM \times \Cat{N}$ also a $\Cat{C}$-module category, although the functor $\act \times 1_{\Cat{N}}$ is of course 
only bilinear with respect to  the first argument  but not bilinear as a functor $\Cat{C} \times (\Cat{M} \times \Cat{N}) \rightarrow \Cat{M} \times \Cat{N}$. 
\end{remark}

Module functors between left $\Cat{C}$-module categories $(\CM, \mu^{\Cat{M}}, l^{\Cat{M}})$ and
$(\CN, \mu^{\Cat{N}}, l^{\Cat{N}})$ are  functors  with additional constraint  isomorphisms that relate the two module actions.

\begin{definition}[\cite{Ostrik}]
  A  $\Cat{C}$-module functor ${\mathsf{F}}: \CM \rr\CN$ is a linear
  functor ${\mathsf{F}}$ together with natural isomorphisms $\phi^{\mathsf{F}}_{x,m}: \mathsf{F}(x \act
  m ) \rr x  \act \mathsf{F}(m)$, such that the diagrams 
  \begin{equation}
    \label{eq:Module-functor}
    \begin{tikzcd} 
      {} & \mathsf{F}((x \otimes y) \act m) \ar{dl}{\mathsf{F}(\mu_{x,y,m}^{\Cat{M}})}
      \ar{dr}{\phi^\mathsf{F}_{x\otimes y,m}} & \\
      \mathsf{F}(x \act (y\act  m)) \ar{d}{\phi^\mathsf{F}_{x,y \act m}} & & (x \otimes y) \act \mathsf{F}(m)
      \ar{d}{\mu^{\Cat{N}}_{x,y,\mathsf{F}(m)}} \\
      x \act \mathsf{F}(y \act m) \ar{rr}{\id_x \act \phi^\mathsf{F}_{y,m}}  && x \act (y\act  \mathsf{F}(m))
    \end{tikzcd}
  \end{equation}
  and 
  \begin{equation}
    \label{eq:mod-functor-triangle}
    \begin{xy}
      \xymatrix{
        {} & \mathsf{F}(\unit_{\Cat{C}} \act m) \ar[dl]_{\phi^\mathsf{F}_{\unit_{\Cat{C}},m}} \ar[dr]^{\mathsf{F}(\lambda^{\Cat{M}}_m)} & \\
        \unit_{\Cat{C}} \act \mathsf{F}(m) \ar[rr]^{\lambda^{\Cat{N}}_{\mathsf{F}(m)}}& &\mathsf{F}(m),
      }
    \end{xy}
  \end{equation}
  commute for all $x,y \in \Cat{C}$ and $m \in \Cat{M}$. We sometimes write $(\mathsf{F},\phi^\mathsf{F})$ for a
  module functor and call $\phi^\mathsf{F}$ a left module constraint for $\mathsf{F}$. Whenever this is unambiguous, we denote the constraint just  by $\phi$. There is the analogous definition 
  for module functors between right $\Cat{C}$-module categories.
\end{definition}
Natural transformations between module functors are required to be compatible in the following way. 
\begin{definition} [\cite{Ostrik}]
  Let  $(\mathsf{F},\phi^{\mathsf{F}}) :\CM\rr
  \CN$ and $(\mathsf{G},\phi^{\mathsf{G}}): \CM \rr \CN$ be module functors.  A module natural transformation $ \eta: \mathsf{F} \rr \mathsf{G}$ is a natural
  transformation such that the diagram
  \begin{equation}
    \label{eq:module-nat-transf}
    \begin{tikzcd}
      \mathsf{F}(x \act m) \ar{r}{\eta_{x \act m}} \ar{d}{\phi^{\mathsf{F}}_{x,m}} & \mathsf{G}(x \act m) \ar{d}{\phi^{\mathsf{G}}_{x,m}} \\
      x \act \mathsf{F}(m) \ar{r}{\id_x \act \eta_{m}} & x \act \mathsf{G}(m),
    \end{tikzcd}
  \end{equation}
  commutes for all $ x \in \Cat{C}$ and $m \in \Cat{M}$.
\end{definition}
It is easy to see that the composite of module natural transformations is again a module natural transformation. Hence, for module categories $\CM$ and $\CN$, 
the module functors and module natural transformations from $\CM$ to $\CN$ form a category 
that is denoted  $\Funl{\Cat{C}}{\CM,\CN}$.

\paragraph{Bimodule categories}

When combined, the notions of left and right module categories lead to the notion of bimodule categories. 
First we present a compact definition of a bimodule category that uses the notion of the Deligne product \cite{DelCat}  of abelian 
categories. 
 For finite tensor categories $\Cat{C}$ and $\Cat{D}$, the 
category $\Cat{C} \boxtimes \Cat{D}$ is again a finite tensor category, see also \cite[Section 1.46]{ENO:Notes} 
for more details.

\begin{definition}
  \label{definition:bimodule-cat}
  A $(\Cat{D},\Cat{C})$-bimodule category $\DMC$ is a left $\Cat{D} \boxtimes
  \rev{C}$-module category. A bimodule category $\DMC$ is called exact if it is exact as a  $\Cat{D} \boxtimes
  \rev{C}$-module category.
\end{definition}
If one unpacks this definition, one sees that
 a $(\Cat{D},\Cat{C})$-bimodule category $\DMC$ is the same as  a
 left  $\Cat{D}$-  and right $\Cat{C}$-module category $\DMC$ together  with a family of
  natural isomorphisms $\gamma_{d,m,c}: (d \act m) \ract c \rr d \act
  (m \ract c)$, for $ d \in \Cat{D} $, $c \in \Cat{C}$ and $m \in \Cat{M}$, 
that satisfies pentagon diagrams with respect to the action of $\Cat{C}$, $\Cat{D}$ and the triangle diagram with respect to the units, see e.g. \cite[Proof of Prop. 2.12]{Green}.
This second view on bimodule categories allows for the  notion of biexact bimodule categories.
\begin{definition}
  \label{definition:exact-bimodule-cat}
Let $\Cat{C}$, $\Cat{D}$ be finite tensor categories. A biexact $(\Cat{D},\Cat{C})$-bimodule category $\DMC$ is a finite bimodule category that is exact both as a left $\Cat{C}$-module and as right $\Cat{D}$-module category. 
\end{definition}
This definition differs from the weaker notion of an exact bimodule category in Definition \ref{definition:bimodule-cat}. 
To see the difference consider $\LCC$ as $(\Cat{C},\Vect)$-bimodule category. 
It is an exact left $\Cat{C} \boxtimes \Vect \simeq \Cat{C}$- module category, hence an exact bimodule category in the sense of  Definition \ref{definition:bimodule-cat}. 
However  it is only an exact right $\Vect$-module category if $\Cat{C}$ is semisimple, hence in general not a 
biexact bimodule category. However the converse statement holds in general. 
\begin{lemma}
\label{lemma:biex-then-ex}
  A biexact $(\Cat{D},\Cat{C})$ bimodule category $\DMC$ is an exact $\Cat{D} \boxtimes \rev{C}$-module category. 
\end{lemma}
\begin{proof}
  According to \cite[Lemma 3.3.6]{DSS}, exactness of a module category $\MC$ is equivalent to the property that for a set of  generating projective objects $P= \{p_{\alpha}\}$, $p_{\alpha} \act m$ is projective for all $m \in \Cat{M}$ and 
all $\alpha$. If $\{p_{\alpha} \}$ and $\{q_{\beta} \}$ are sets of  generating projective objects for $\Cat{C}$ and $\Cat{D}$, then $\{p_{\alpha} \boxtimes q_{\beta} \}$ is a set of generating projective objects of $\Cat{C} \boxtimes \Cat{D}$. This can be seen most easily if we choose $\Vect$-algebras $A$ and $B$, such that $\Cat{C} \simeq \Mod_{A}$, $\Cat{D} \simeq \Mod_{B}$ as linear categories. Then $\{A\}$ and $\{B\}$ are generating projective objects of $\Cat{C}$ and $\Cat{D}$, while $\{A \otimes B\}$ is generating projective for $\Cat{C} \boxtimes \Cat{D} \simeq \Mod_{A \otimes B}$. Assume now that $\DMC$ is a biexact bimodule category and  $\{p_{\alpha} \}$ and $\{q_{\beta} \}$ are sets of  generating projective objects for $\Cat{C}$ and $\Cat{D}$. Then $(p_{\alpha} \boxtimes q_{\beta}) \act m = p_{\alpha} \act m \ract q_{\beta}$ is projective, hence $\DMC$ is an exact bimodule category. 
\end{proof} 
 Clearly, $\Cat{C}$ is a $(\Cat{C},\Cat{C})$-bimodule category with actions given by the tensor product, see Example \ref{example:unit-module-cat-alg}\refitem{item:unit-module}. 
This bimodule category is denoted $\CCC$ and called the unit bimodule category.

The compact definition of a bimodule category  also directly defines bimodule functors and bimodule natural transformations between $(\Cat{D},\Cat{C})$-bimodule categories 
 as $\Cat{D} \boxtimes \rev{C}$-module functors  and $\Cat{D} \boxtimes \rev{C}$-module natural transformation. The following gives a more explicit characterization of bimodule functors,  see \cite{Green}.
\newpage  
\begin{lemma}
  A bimodule functor  $\mathsf{F}: \DMC \rightarrow \DNC$ is the same as a left $\Cat{D}$- and right $\Cat{C}$-module 
functor with module constraints $\phi^{l}$ and 
$\phi^{r}$, respectively, such that 
  \begin{equation}
    \label{eq:bimodule-functor-charact}
    \begin{tikzcd}
      {} & \mathsf{F}((x \act m) \ract y) \ar{dl}{\mathsf{F}(\gamma_{x,m,y})}
      \ar{dr}{\phi^{r}_{x \act m, y} } & \\
      \mathsf{F}(x \act (m \ract y)) \ar{d}{\phi^l_{x, m \ract y}} & & \mathsf{F}(x \act m)
      \ract y \ar{d}{\phi^l_{x,m} \ract 1_y}\\
      x \act \mathsf{F}(m \ract y) \ar{dr}{1_x \act \phi^r_{m,y}} & & (x \act \mathsf{F}(m))
      \ract y \ar{dl}{\gamma_{x,\mathsf{F}(m),y}} \\
      & x \act (\mathsf{F}(m) \ract y) & 
    \end{tikzcd}
  \end{equation}
  commutes for all possible objects.
\end{lemma}
A bimodule natural transformations in turn, is  the same as a left and right module natural transformation, see \cite[Lemma 2.3.6]{Schaum} for details.

The category of bimodule functors and bimodule natural transformations between two bimodule categories $\DMC$ and $\DNC$ is denoted $\Funl{\Cat{D},\Cat{C}}{\DMC, \DNC}$. 
It is straightforward to see that  for two finite tensor categories $\Cat{C}$ and $\Cat{D}$, 
  the $(\Cat{D},\Cat{C})$-bimodule categories $\DMC$, $\DNC$, $(\Cat{D},\Cat{C})$-bimodule functors $\mathsf{F},\mathsf{G}: \DMC \rightarrow \DNC$ and 
  $(\Cat{D},\Cat{C})$-bimodule natural transformations $\eta: \mathsf{F} \rightarrow \mathsf{G}$ form a 2-category called $\BimCat(\Cat{D},\Cat{C})$
  with the composition of functors as horizontal composition and the composition of natural transformations as vertical composition.

\begin{proposition}
\label{proposition:adj-ex}
Let $\CMD$ and $\CND$ be bimodule categories. 
\begin{propositionlist}
  \item A right (left) exact bimodule functor  $\mathsf{F}: \CMD \rightarrow \CND$ has a right (left) adjoint that is  naturally bimodule functors from $\CND$  to $\CMD$ such that
the  adjunctions consist of bimodule natural isomorphisms. 
\item If  $\CMD$ and $\CND$ are exact bimodule categories, every bimodule functor $\mathsf{F}:\CMD \rightarrow \CND$ is exact, in particular the monoidal category  $\Funl{\Cat{D},\Cat{C}}{\DMC, \DMC}$  has 
left and right duals. 
\end{propositionlist}
\end{proposition}
\begin{proof}
It is a well known fact that a functor between finite linear categories has a right (left) adjoint if and only it is right (left) exact, see \cite{DSSbal} for a detailed discussion. 
There is a unique  way to equip the adjoint of a bimodule functor with the structure of a module functor, such that the adjunctions consist of bimodule natural isomorphisms, see e.g. \cite[Sec. 3.3]{FinTen}. 
The second statement is shown in  \cite[Lemma 3.21]{FinTen}.
  The duals in  $\Funl{\Cat{D},\Cat{C}}{\DMC, \DMC}$ are thus given by the 
the left and right adjoint functors $\mathsf{F}^{r}$, and $\mathsf{F}^{l}$.
\end{proof}
 Moreover, for each 
bimodule natural transformation $\eta: \mathsf{F} \rightarrow \mathsf{G}$ between exact bimodule functors, there are canonical bimodule natural transformations $\eta^{l}: \mathsf{G}^{l} \rightarrow \mathsf{F}^{l}$ and 
$\eta^{r}: \mathsf{G}^{r} \rightarrow  \mathsf{F}^{r}$.  It is shown in \cite{FinTen}, that for exact module categories 
$\CM$ and $\CN$, all module functors in  $\Funl{\Cat{C}}{\CM,\CN}$ are exact and thus 
$\Funl{\Cat{C}}{\CM,\CM}$ has left and right duals, moreover it is a again a finite tensor category that is also denoted $\CMstar$.

\paragraph{The dual bimodule categories}

Let $\DMC$ be a $(\Cat{D},\Cat{C})$-bimodule category. We then  define   two $(\Cat{C},\Cat{D})$-bimodule categories $\CMDrd$ and $\CMDld$ as follows:
As categories, they are both $\op{M}$, with actions
\begin{equation}
  \label{eq:rd-action}
  c \actrd m \ractrd d = \leftidx{^*}{d}{} \act m\ract \leftidx{^*}{c}{},
\end{equation}
for $m \in \CMDrd$, and 
\begin{equation}
  \label{eq:ld-action} 
  c \actld m \ractld d = d^{*} \act m \ract c^{*}
\end{equation}
for $m \in \CMDld$.  These conventions  agree   with \cite[Def. 3.4.4]{DSS}.

\paragraph{Balanced functors}

In the sequel we require, in addition to module functors, another notion of compatibility of functors with module structures. The following definition is taken from 
\cite[Def 3.1]{ENOfuhom} with the minor change of adding the obvious compatibility axiom with the units.

\begin{definition} 
  \label{definition:balanced-functors} Let $\Cat{A}$ be a linear category.
  \begin{definitionlist}
  \item   A bilinear  functor $\mathsf{F}: \MC \times \CN \rr \Cat{A}$ is called $\Cat{C}$-balanced with balancing constraint $\beta^{\mathsf{F}}$, if it is equipped with a  family of natural isomorphisms 
    \begin{equation}
      \beta^{\mathsf{F}}_{m,c,n}: \mathsf{F}(m \ract c \times n) \rr \mathsf{F}(m \times c \act n),
    \end{equation}
    such that  the pentagon diagram 
    \begin{equation}
      \label{eq:balanced-functor}
      \begin{tikzcd} 
        {} & \mathsf{F}(m \ract (x \otimes y) \times n) \ar{dl}{\mathsf{F}(\mu_{m,x,y}^{\Cat{M}}\times 1_{n} )}
        \ar{dr}{\beta^{\mathsf{F}}_{m,x\otimes y,n}} & \\
        \mathsf{F}((m \ract x) \ract y \times n) \ar{d}{\beta^{\mathsf{F}}_{ m \ract x, y, n}} & &  \mathsf{F}(m \times (x \otimes y ) \act n)
        \ar{d}{\mathsf{F}(1_{m} \times \mu_{x,y,n}^{\Cat{N}})} \\
        \mathsf{F}(m \ract x \times  y \act n) \ar{rr}{ \beta^{\mathsf{F}}_{m,x, y \act n }}  && \mathsf{F}( m \times x \act (y \act n))
      \end{tikzcd}
    \end{equation}
    and the triangle diagram
    \begin{equation}
      \label{eq:balanced-unit}
      \begin{tikzcd}
        \mathsf{F}(m \ract \unit_{\Cat{C}} \times n) \ar{r}{\beta^{\mathsf{F}}_{m,1,n}} \ar{d}[xshift=-45pt]{\mathsf{F}(\rho^{\Cat{M}}_{m} \times 1_{n})}  & \mathsf{F}( m \times \unit_{\Cat{C}} \act n) \ar{dl}{\mathsf{F}(1_{m} \times\lambda^{\Cat{N}}_{n})} \\
        \mathsf{F}(m \times n) & 
      \end{tikzcd}
    \end{equation}
commute for all possible objects. 
    We  denote the balancing constraint $\beta^{\mathsf{F}}$  simply by $\beta $      if this is unambiguous.
  \item Let  $\mathsf{F},\mathsf{G} \colon \MC \times \CN \rr \Cat{A}$  be balanced functors. 
 A balanced natural transformation $\eta: \mathsf{F} \rr \mathsf{G}$ is a natural transformation $\eta: \mathsf{F} \rr \mathsf{G}$, such that the diagrams 
    \begin{equation}
      \label{eq:balanced-nat}
      \begin{xy}
        \xymatrix{
          \mathsf{F}(m \ract c \times n) \ar[r]^{\eta_{m \ract c \times n}} \ar[d]_{\beta^{\mathsf{F}}_{m,c,n}}  & \mathsf{G}(m \ract c \times n) \ar[d]^{\beta^{\mathsf{G}}_{m,c,n}} \\
          \mathsf{F}(m \times c \act n) \ar[r]_{\eta_{m \times c \act n}} & \mathsf{G}(m \times c \act n)
        }
      \end{xy}
    \end{equation}
    commute for all possible objects. 
  \end{definitionlist}
\end{definition}
It is clear that the identity natural transformation $1_{\mathsf{F}}:\mathsf{F} \rightarrow \mathsf{F}$ for a balanced functor $\mathsf{F}$ is balanced and that the  composition of balanced natural transformations yields a 
balanced natural transformation. Hence the  balanced functors and balanced natural transformations from $\MC \times \CN$ to $\Cat{A}$ form a category that 
is denoted $\Funbal(\MC \times \CN, \Cat{A})$.

As example for a balanced functor, consider a module category $\CM$. It is  straightforward to see that the action $\act: \Cat{C} \times \Cat{M} \rightarrow \Cat{M}$ is a balanced functor. 
The following statement follows directly from the definitions. 
\begin{lemma}
  \label{lemma:balancing-and-composition}
  The  composition of  functors and natural transformations defines functors
  \begin{lemmalist}
  \item \label{item:bal-bimod-ist-bal}
    $  \Funbal(\MpC \times \CpN,\Cat{A}) \times \Funl{\Cat{C},\Cat{C}}{\MC \times \CN, \MpC \times \CpN} \rr \Funbal(\MC \times \CN, \Cat{A}),$
  \item \label{item:beliebig-bal-ist-bal} 
    $     \Fun(\Cat{A}, \Cat{B}) \times \Funbal(\MC \times \CN, \Cat{A}) \rr \Funbal(\MC \times \CN, \Cat{B}).$
  \end{lemmalist}
\end{lemma}

\paragraph{Balanced module functors}

The following combines the notion of module functor with the notion of a balanced functor.
\begin{definition}
  \label{definition:balanced-module}
  A balanced (left) $\Cat{D}$-module functor $\mathsf{F}: \DMC \times \CN \rightarrow \DY$ is a balanced functor with balancing structure $b^{\mathsf{F}}$ that is also a module functor with module structure $\phi^{\mathsf{F}}$ such that the diagram
  \begin{equation}
    \label{eq:balanced-module-diag}
    \begin{tikzcd}[column sep=large]
      \mathsf{F}((d\act m) \ract c \times n ) \ar{r}{\beta^{\mathsf{F}}_{d \act m,c,n}} \ar{d}{\mathsf{F}(\gamma_{d,m,c} \times n)} &  \mathsf{F}(d \act m \times c \act n)  \ar{d}{\phi^{\mathsf{F}}_{d,m \times c \act m}} \\
      \mathsf{F}(d\act(m\ract c) \times n) \ar{d}[yshift=3pt]{\phi^{\mathsf{F}}_{d,m \ract c \times n}} &   d \act \mathsf{F}(m \times c \act n)\\
      d\act \mathsf{F}( m \ract c \times n) \ar{ur}[below, right, xshift=3pt, yshift=-2pt]{d \act \beta^{\mathsf{F}}_{m,c,n}} & 
    \end{tikzcd}
  \end{equation}
  commutes for all objects $c \in \Cat{C}$, $d \in \Cat{D}$, $m \in \Cat{M}$ and $n \in \Cat{N}$. 
  Balanced right module functor are defined analogously. 

  A balanced bimodule functor is a bimodule functor that is a balanced left- and a balanced right module functor.
  A balanced module natural transformation between balanced module functors is a natural transformation that is balanced and a module natural transformation.
\end{definition}
It is clear that the balanced bimodule functors from $ \DMC \times \CNE$ to $ \DYE$ together with balanced bimodule natural transformations form a category  $\Funball{\Cat{D},\Cat{E}}{\DMC \times \CNE, \DYE}$.

For example, if $\DMC$ is a bimodule category, the actions $\act: \Cat{D} \times \DMC \rightarrow \DMC$ and $\ract: \DMC \times \Cat{C} \rightarrow \DMC$ are balanced bimodule functors, when we consider $\Cat{C}$ and $\Cat{D}$ as bimodule 
categories. 
The following statements follow directly from the definitions. 
\begin{lemma}
  \label{lemma:characterize-bal-module}
  \begin{lemmalist}
    \item The left action of $\Cat{C}$ on $\CMD \times \DN$  is given by $\Cat{D}$-balanced module functors $L_{c}: \CMD \times \DN \rightarrow \CMD \times \DN$ for all $c \in \Cat{C}$. 
\item 
 A left $\Cat{C}$-module functor $\mathsf{F}: \CMD \times \DN \rightarrow \CY$ that is also $\Cat{D}$-balanced is a balanced 
  module functor if and only if the left module constraints 
  $\phi^{\mathsf{F}}_{c}:\mathsf{F} \circ L_{c}^{\Cat{M} \times \Cat{N}} \rightarrow L^{\Cat{Y}}_{c} \circ \mathsf{G}$ are balanced natural isomorphisms for all $c \in \Cat{C}$. 
  \end{lemmalist}
\end{lemma}

\paragraph{Inner hom objects} 

An important tool in the theory of module categories is the inner hom.  In this work, the inner hom will play a dominant role in the 
construction of duals for bimodule categories. Let $\DM$ be a left $\Cat{D}$-module category. For $m, \widetilde{m} \in \Cat{M}$, the functor $\Hom_{\Cat{M}}( (-) \act m, \widetilde{m}): \op{C} \rightarrow \Vect$ is
left exact and thus representable \cite[Sec. 3.2]{FinTen}.
The following is a  change in convention with regard to the definition used in \cite{Ostrik} as will be explained in detail in Lemma \ref{lemma:convent-inner-hom}.

\begin{definition}
  \label{definition:inner-hom}
  Let $\DM$ be a left $\Cat{D}$-module category. An inner hom  for $\Cat{M}$ is an object 
  $\idm{m,\widetilde{m}} \in \Cat{D}$ for all $m,\widetilde{m} \in \Cat{M}$  together with  natural isomorphisms
  \begin{equation}
    \label{eq:inner-hom}
    \alpha_{d,m,\widetilde{m}}^{\Cat{M}}:  \Hom_{\Cat{M}}( m, d \act \widetilde{m}) \simeq \Hom_{\Cat{D}}( \idm{m,\widetilde{m}},d), 
  \end{equation}
  for all $d \in \Cat{D}$ and $m,\widetilde{m} \in \Cat{M}$.
\end{definition}

We write $\idgen{m, \widetilde{m}}$ for the inner hom objects and omit the labels of $\alpha$, when the relevant module category $\Cat{M}$ is clear from the context.
If they exist inner homs  are unique up to a unique isomorphism and determine a bilinear functor
\begin{equation}
  \label{eq:inner-hom-functor-DM}
  \Cat{M} \times \op{M} \ni(m \times \widetilde{m})  \mapsto \idm{m,\widetilde{m}} \in \Cat{D},
\end{equation}
called the inner hom functor.
Analogously, a right $\Cat{C}$-module category $\NC$ gives rise  to an inner hom 
with natural isomorphisms 
\begin{equation}
  \label{eq:nat-iso-inner-right}
  \alpha_{\widetilde{n},n,c}^{\Cat{N}}: \Hom_{\Cat{C}}(\inc{\widetilde{n},n},c) \simeq \Hom_{\Cat{N}}(n, \widetilde{n} \ract c),
\end{equation}
that yield a functor
\begin{equation}
  \label{eq:inner-hom-functor-MC}
  \op{N} \times \Cat{N} \ni (\widetilde{n} \times n) \mapsto \inc{\widetilde{n},n} \in \Cat{C}.
\end{equation}
Next we show that inner hom objects exist in our setting and compare their definition with the usual definition of inner hom in the literature. 
\begin{lemma}
  \label{lemma:convent-inner-hom}
Let $\DMC$ be a module category.
\begin{lemmalist}
  \item  There exists a left exact functor 
 $$\IHom_{\Cat{D}}: \Cat{M} \times \op{M} \ni (m,\widetilde{m}) \mapsto \IHom_{\Cat{D}}(m,\widetilde{m})  \in \Cat{D}$$ together with natural isomorphisms 
    \begin{equation}
      \label{eq:conv-ihom}
      \Hom_{\Cat{M}}(d \act \widetilde{m}, m) \simeq \Hom_{\Cat{D}}(d, \IHom(\widetilde{m},m)).
    \end{equation}
\item Similarly, there exists a left exact functor 
 $$\IHom_{\Cat{C}}: \op{M} \times \Cat{M} \ni (\widetilde{m},m) \mapsto \IHom_{\Cat{C}}(\widetilde{m},m) \in \Cat{C}$$ 
with natural isomorphisms 
\begin{equation}
  \label{eq:conv-ihomc}
  \Hom_{\Cat{M}}(m \ract c, \widetilde{m}) \simeq \Hom_{\Cat{C}}(c, \IHom_{\Cat{C}}(m,\widetilde{m})).
\end{equation}
\item There are natural isomorphisms 
$$\idm{m,\widetilde{m}} \simeq \leftidx{^*}{ \IHom_{\Cat{D}}(m,\widetilde{m})}{} \quad \text{and} \quad \imc{\widetilde{m},m} \simeq \IHom(m,\widetilde{m})^{*}.$$
In particular, inner hom objects exist for finite tensor categories. Moreover, the functors $\idm{-,-}$ and $\imc{-,-}$ are right exact in each argument. 
\end{lemmalist}
\end{lemma}
\begin{proof}
  The objects $\IHom(m,\widetilde{m})$ in the first two parts  are what are  more commonly  called inner hom objects, see \cite[Sec. 3.2]{FinTen}. The existence and the  left exactness of $\IHom(-,-)$ 
follows directly from the left exactness of the 
$\Hom_{}$-functor. For the last part we compute using the duality in $\Cat{D}$
\begin{equation}
  \label{eq:comp-inn-std}
  \begin{split}
     \Hom_{}(\idm{m,\widetilde{m}}, d) &\simeq \Hom_{}(m, d \act \widetilde{m}) \simeq \Hom_{}(d^{*} \act m , \widetilde{m}) \\
&\simeq \Hom_{}(d^{*}, \IHom_{\Cat{D}}(m,\widetilde{m})) \simeq \Hom_{}(\leftidx{^*}{\IHom(m,\widetilde{m})}{} ,d).
  \end{split}
 \end{equation}
All isomorphisms are natural in all arguments and induce a natural isomorphism  $\idm{m,\widetilde{m}} \simeq \leftidx{^*}{ \IHom_{\Cat{D}}(m,\widetilde{m})}{}$ by the Yoneda-Lemma. The second isomorphism 
is obtained similarly. Since the duality functor of a finite tensor category is an exact functor $(-)^{*}: \Cat{D} \rightarrow \op{D}$, it follows that the inner hom functors from Definition \ref{definition:inner-hom}
 and Equation (\ref{eq:nat-iso-inner-right})  are right exact. 
\end{proof}

The inner hom functors are compatible with the module structures in the following way. 
\begin{proposition}
  \label{proposition:properties-inner-hom}
  Let $\Cat{C}$ and $\Cat{D}$ be finite tensor categories and $\DMC$ a  bimodule category.
  \begin{propositionlist}
  \item The $\Cat{D}$-valued inner hom is a $\Cat{C}$-balanced bimodule functor $$\idm{-,-}: \DMC \times \CMDld \rightarrow \DDD,$$ i.e. there are coherent natural  isomorphisms 
    \begin{equation}
      \label{eq:idm-D-module}
      \idm{ d \act m,\widetilde{m}} \simeq d \otimes \idm{m, \widetilde{m}}, \quad \idm{m, \widetilde{m}\; \ractld \; d } \simeq \idm{m,\widetilde{m}} \otimes d
    \end{equation}
and 
\begin{equation}
  \label{eq:idm-C-bal}
  \idm{m\ract c, \widetilde{m}} \simeq \idm{m, c \; \actld \; \widetilde{m}}.
\end{equation}
    
  \item The $\Cat{C}$-valued inner hom is a $\Cat{D}$-balanced bimodule functor $$\imc{-,-}: \CMDrd \times \DMC \rightarrow \CCC,$$ i.e. there are coherent natural isomorphisms
    \begin{equation}
      \label{eq:imc-C-module}
      \imc{\widetilde{m},m \ract c} \simeq \imc{\widetilde{m},m} \otimes c, \quad \imc{c \; \actrd \; \widetilde{m},m} \simeq c\otimes \imc{\widetilde{m},m}
    \end{equation}
and
    \begin{equation}
      \label{eq:imc-D-balanced}
      \imc{\widetilde{m}, d\act  m} \simeq \imc{\widetilde{m} \; \ractrd \: d, m}.
    \end{equation}
  \end{propositionlist}
\end{proposition}
\begin{proof}
Most of these natural isomorphisms are defined in \cite[Lemma 5]{Ostrik}.  All  natural isomorphisms are obtained directly  from the definitions of the inner hom and the isomorphisms induced by the dualities in the tensor categories.
\end{proof}
For the unit bimodule category $\CCC$, the inner homs are for example given by $\icgen{x, \widetilde{x}}= x \otimes \leftidx{^*}{\widetilde{x}}{}$ and $\igenc{\widetilde{x},x}= \widetilde{x}^{*} \otimes x$.

If we pass from a module category $\DM$ to the Grothendieck ring $Gr(\Cat{M})$, Proposition \ref{proposition:properties-inner-hom}
shows that the inner hom satisfies all requirements of a $Gr(\Cat{D})$-valued inner product except the compatibility with the $*$-involution. This compatibility will be considered in Section \ref{sec:inner-prod}.

The inner hom allows to prove the following theorem.  
\begin{theorem}[\cite{FinTen}]
\label{theorem:all-alg}
Let $\CM$ be an  module category over $\Cat{C}$. Then there exists an algebra object $A \in \Cat{C}$, such 
that $\CM$ is equivalent to $\ModCA$, the category of $A$-right modules in $\Cat{D}$ with $\Cat{D}$-left action 
given by the tensor product.  
 \end{theorem}
Next we recall the inner hom for the module categories obtained from algebras. 
\begin{example}
  \label{example:inner-homs-algebra}
We summarize the computation of  the inner hom objects in the case of $\CN= \ModCB$ and the $\CMstar$-valued inner hom for $\MC= \AModC$ in \cite[Example 3.19]{FinTen}. 
Note that for a right $B$-module in $\ModCB$, $\leftidx{^*}{n}{}$ is naturally a left $B$-module, while for $m \in \AModC$, $m^{*}$ is a right $A$-module.
Using the tensor product over the algebra $B$ in $\Cat{C}$, see \cite[Def. 2.9.22]{ENO: Notes}, we compute
\begin{equation*}
  \Hom_{\Cat{C}}(\icn{n,\widetilde{n}},c)\simeq \Hom_{\Cat{N}}(n, c \otimes \widetilde{n}) = \Hom_{\Cat{C}}(n \tensor{B} \leftidx{^*}{n}{}, c).
\end{equation*}
It follows that  $\icn{n,\widetilde{n}}= n \tensor{B} \leftidx{^*}{\widetilde{n}}{}$. 
It is shown in \cite[Prop. 2.12.2]{ENO:Notes}, that  $$ \CMstar = \Funl{\Cat{C}}{\AModC, \AModC}= \AModCA,$$ where $\AModCA$ is 
 the category of $(A,A)$-bimodules in $\Cat{C}$ with the tensor product $\tensor{A}$ over $A$ as monoidal structure. 
We compute the inner hom of $\AModC$ regarded as left $\AModCA$-module category. Let $x \in \AModCA$, then
\begin{equation}
  \label{eq:A-A-valued-inner-hom}
  \begin{split}
      \Hom_{_{A\!} \mathsf{Mod}(\Cat{C})_{A}}(\icdualm{m,\widetilde{m}},x) &=\Hom_{\Cat{M}}(m, x \tensor{A} \widetilde{m}) \simeq \Hom_{_{A\!} \mathsf{Mod}(\Cat{C})_{A}}(A, x \tensor{A} \widetilde{m} \otimes m^{*}) \\
 &= \Hom_{_{A\!} \mathsf{Mod}(\Cat{C})_{A}}( \leftidx{^*}{(\widetilde{m} \otimes m^{*})}{},x),
  \end{split}
\end{equation}
where the left dual in the last expression is taken in $\AModCA$. It follows that $\icdualm{m,\widetilde{m}}= \leftidx{^*}{(\widetilde{m} \otimes m^{*})}{}$.
\end{example}

\section{The tricategory of bimodule categories}
\label{sec:tricat-bimod}
The goal of this section is to show that finite tensor categories,  bimodule categories, right exact bimodule functors and bimodule natural transformation form an algebraic tricategory \cite{Gurski}, see also Definition
 \ref{definition:tricategory}. The idea of the proof is a direct generalization of the analogous statement in \cite{Schaum} in the semisimple case. 
To simplify notation, we make the following assumptions. 
\emph{In this section we assume that all tensor categories are finite and all module categories are finite unless 
specified otherwise. Furthermore, all functors are assumed to be right exact}.

The results in this section are a generalization of the results in \cite[Sec. 3]{Schaum} to the non-semisimple case. 
\subsection{The tensor product of module categories}
\label{sec:tens-prod-module}

In this subsection we first recall the definition of the tensor product of module categories. The tensor product is defined by an universal property with respect to right exact balanced functors. 
Then we show that the tensor product naturally defines a 2-functor from a suitable 2-category into the 2-category of 
abelian categories.

Let $\MC$ and $\CN$ be left and right  $\Cat{C}$-module categories, respectively. A tensor product $\MC \Box \CN$ of $\MC$ and $\CN$ is a linear abelian category that is defined -up to equivalence 
of categories- by a universal property that can be regarded as the analogue of the universal property of the tensor product of modules over a ring.
The definition uses the  category  $\Funbal(\MC \times \CN, \Cat{A})$  of (right exact) balanced functors and balanced natural transformations from $\MC \times \CN$ to a linear category $\Cat{A}$. 
The following definition is an extension of  \cite[Definition 3.3]{ENOfuhom} in the sense that we require a fixed adjoint equivalence as part of the data of a tensor product.
\begin{definition}
  \label{definition:tensor-prod}
  A tensor product $(\MC \Box \CN, \mathsf{B}_{\Cat{M},\Cat{N}}, \Psi_{\Cat{M},\Cat{N}}, \varphi_{\Cat{M},\Cat{N}}, \kappa_{\Cat{M},\Cat{N}})$ of a right $\Cat{C}$-module category $\MC$ with a left $\Cat{C}$-module category $\CN$ is a finite linear
 abelian category  $\MC \Box  \CN$ together with
  \begin{definitionlist}
  \item 
    a $\Cat{C}$-balanced functor $\mathsf{B}_{\Cat{M},\Cat{N}}:\MC \times \CN \rr \MC \Box \CN$,
    such that the functor 
    \begin{equation}
      \label{eq:univer-prop-Box}
      \begin{split}
        \Phi_{\Cat{M},\Cat{N}}: \Fun(\MC \Box \CN, \Cat{A})  &\rightarrow \Funbal(\Cat{M} \times \Cat{N}, \Cat{A}) \\
        \mathsf{G} &\mapsto \mathsf{G} \circ \mathsf{B}_{\Cat{M},\Cat{N}} 
      \end{split}
    \end{equation}
    is an  equivalence of categories. Here $\Fun(\MC \Box \CN, \Cat{A})$ denotes the category of (right exact) functors $\MC \Box \CN \rightarrow \Cat{A}$ to some linear abelian category $\Cat{A}$.
  \item a choice of a functor 
    \begin{equation}
      \label{eq:Psi}
      \Psi_{\Cat{M},\Cat{N}}: \Funbal(\MC \times \CN, \Cat{A}) \rightarrow \Fun(\MC \Box \CN, \Cat{A})  
    \end{equation}
    together with a specified adjoint equivalence $\varphi_{\Cat{M},\Cat{N}}:1 \rightarrow  \Phi_{\Cat{M},\Cat{N}} \Psi_{\Cat{M},\Cat{N}}$ and $\kappa_{\Cat{M},\Cat{N}}:1 \rightarrow  \Psi_{\Cat{M},\Cat{N}} \Phi_{\Cat{M},\Cat{N}}$
    between $\Phi_{\Cat{M},\Cat{N}}$ and $\Psi_{\Cat{M},\Cat{N}}$.
  \end{definitionlist}
\end{definition}
For simplicity we sometimes write $m \Box n$ instead of $\mathsf{B}(m \times n)$ for $m \times n \in \MC \times \CN$. We record the  existence of the tensor product from the literature. 
\begin{theorem}
\label{theorem:existence-tensor}
The tensor product $\MC \Box \CN$ of finite module categories $\MC$, $\CN$ exists. In particular it has the following descriptions.
  \begin{theoremlist}
    \item  The tensor product is equivalent to the following functor categories $$ \MC \Box \CN \simeq \Funl{\Cat{C}}{\CMrd,\CN} \simeq \Funl{\Cat{C}}{\NCld, \MC}.$$
\item If we choose algebras  $A,B \in \Cat{C}$ such that   $\MC \simeq \AModC$, $\CN \simeq \ModCB$, then 
$\MC \Box \CN \simeq \AModCB$, the category of $(A,B)$-bimodules in $\Cat{C}$. In this case, the universal 
balancing functor is given by the tensor product $\otimes: \AModC \times \ModCB \rightarrow \AModCB$.
  \end{theoremlist}
\end{theorem}                    
\begin{proof}
The first description of the tensor product  is shown in \cite[Cor. 3.4.11]{DSS} in general and in  \cite{ENOfuhom} in the semisimple case. The second statement is shown in 
\cite{DSSbal}, see \cite[Thm. 3.2.17]{DSS}.
\end{proof}
Concretely, for every balanced functor 
 $\mathsf{F}: \MC \times \CN \rightarrow \Cat{A}$, the tensor product yields a functor  $\widehat{\mathsf{F}}= \Psi_{\Cat{M},\Cat{N}}(\mathsf{F}): \Cat{M} \Box \Cat{N} \rightarrow \Cat{A}$, that is unique up to unique natural isomorphism. 
The following lemma is a  direct consequence of the properties of the  adjoint equivalence in the definition of the tensor product.
 \begin{lemma}
  \label{lemma:action-Psi-bal}
  Let $\mathsf{F},\mathsf{G}: \MC  \times \CN \rightarrow \Cat{A}$ be balanced functors. For every balanced natural transformation $\rho: \mathsf{F} \rightarrow \mathsf{G}$ there exists a unique natural transformation 
  $\widehat{\rho}: \widehat{\mathsf{F}} \rightarrow \widehat{\mathsf{G}}$, such that 
  \begin{equation}
    \label{eq:26}
    (\widehat{\rho} \circ \mathsf{B}) \cdot \varphi(\mathsf{F}) = \varphi(\mathsf{G}) \cdot \rho. 
  \end{equation}
\end{lemma}

Next we consider the 2-functorial properties of the tensor product. Denote by $\Catlin$ the 2-category of linear abelian categories, (right exact) linear functors and linear natural transformations. 
\begin{proposition}
  \label{proposition:tensor-functors}
 The tensor product defines  a 2-functor 
  \begin{equation}
    \label{eq:explicit-2-fun-Box}
    \Box: \Mod^{r}(\Cat{C}) \times \Mod^{l}(\Cat{C}) \rightarrow \Catlin,
  \end{equation}
  where $\Mod^{r}(\Cat{C})$, $\Mod^{l}(\Cat{C})$ denote the 2-categories of right and left $\Cat{C}$-module categories, respectively. 
This amounts for the following structures: Let $\MC$, $\MpC$ and $\CN$, $\CpN$ be $\Cat{C}$-left- and right module categories, respectively. 
  \begin{propositionlist}
  \item \label{item:first}  For every bimodule functor $\mathsf{F}: \MC \times \CN \rightarrow \MpC \times \CpN$, 
    the tensor product of module categories defines a functor $\Psi( \mathsf{B}_{\Cat{M}',\Cat{N}'} \mathsf{F}): \MC \Box \CN \rightarrow \MpC \Box \CpN$,  called $\widehat{\mathsf{F}}$ in the sequel, 
    and a balanced natural isomorphism 
    \begin{equation}
      \label{eq:Box-2-func}
      \begin{tikzcd}
        \Cat{M} \times \Cat{N} \ar{r}[name=A]{\mathsf{F}} \ar{d}{\mathsf{B}_{\Cat{M},\Cat{N}}} &    \Cat{M}' \times \Cat{N}' \ar{d}{\mathsf{B}_{\Cat{M}',\Cat{N}'}}
 \ar[shorten <= 13pt, shorten >=13pt, Rightarrow]{dl}{\varphi_{\mathsf{F}}} \\  
        \Cat{M} \Box \Cat{N} \ar{r}[name=B, below]{\widehat{\mathsf{F}}} & \Cat{M}' \Box \Cat{N}'.
      \end{tikzcd}
    \end{equation}
  \item \label{item:sec} For every pair of  bimodule  functors  $\mathsf{F},\mathsf{G}: \MC \times \CN \rightarrow \MpC \times \CpN$ and every bimodule natural transformation $\rho: \mathsf{F}\rightarrow  \mathsf{G}$, 
    there is a unique natural transformation $\widehat{\rho} : \widehat{\mathsf{F}} \rightarrow \widehat{\mathsf{G}}$, such that 
    \begin{equation}
      \label{eq:deterimes-Box-of-natural}
      (  \widehat{\rho} \circ \mathsf{B}_{\Cat{M},\Cat{N}}) \cdot \mathsf{B}_{\mathsf{F}} =  \mathsf{B}_{\mathsf{G}} \cdot \rho.
    \end{equation}
    This is equivalent to imposing   the following condition on the associated diagrams
    \begin{equation}
      \label{eq:condition-BMN}
      \begin{tikzcd}[baseline=-0.65ex,scale=0.5, column sep=large, row sep=large]
        \Cat{M} \times \Cat{N} \ar{d}{\mathsf{B}_{\Cat{M},\Cat{N}}} \ar[bend left]{r}[name=A]{ \mathsf{F}} \ar[bend right]{r}[ name=B]{ \mathsf{G}}   & \Cat{M}'\times \Cat{N}' \ar{d}{\mathsf{B}_{\Cat{M}',\Cat{N}'}} 
 \ar[shorten <= 20pt, shorten >=20pt, Rightarrow]{dl}{\varphi_{\mathsf{G}}}\\
        \Cat{M}\Box \Cat{N} \ar{r}[name=C, below]{\widehat{\mathsf{G}}} & \Cat{M}' \Box \Cat{N}' ,
        \arrow[shorten <= 2pt, shorten >=2pt, Rightarrow,to path=(A) -- (B)\tikztonodes]{}{\rho}
      \end{tikzcd}
      =
      \begin{tikzcd}[baseline=-0.65ex,scale=0.5,  column sep=large, row sep=large]
        \Cat{M} \times \Cat{N} \ar{d}{\mathsf{B}_{\Cat{M},\Cat{N}}}  \ar{r}[name=C]{  \mathsf{F}} & \Cat{M}' \times \Cat{N}' \ar{d}{\mathsf{B}_{\Cat{M}',\Cat{N}'}} 
 \ar[shorten <= 25pt, shorten >=25pt, Rightarrow]{dl}[above, yshift=2pt]{\varphi_{\mathsf{F}}} \\
        \Cat{M} \Box \Cat{N}  \ar[bend left]{r}[name=A, below]{ \widehat{\mathsf{F}}} \ar[bend right]{r}[ name=B, below]{ \widehat{\mathsf{G}}}  & \Cat{M}' \Box \Cat{N}' 
        \arrow[shorten <= 4pt, shorten >=2pt, Rightarrow,to path=(A) -- (B)\tikztonodes]{}{\widehat{\rho}}.
      \end{tikzcd}
    \end{equation}
  \item \label{item:third} For any two composable   bimodule functors $\mathsf{F}:\MC \times \CN \rightarrow \MpC \times \CpN $,  $\mathsf{G}:\MpC \times \CpN \rightarrow \MppC \times \CppN$, there is a unique 
    natural isomorphism $\phi_{\mathsf{G},\mathsf{F}}: \widehat{\mathsf{G}}\widehat{\mathsf{F}} \rightarrow \widehat{\mathsf{G}\mathsf{F}}$ such that the following diagram of natural isomorphisms commutes:
    \begin{equation}
      \label{eq:diag-varphiGF-uniqueness}
      \begin{tikzcd}[column sep=large]
        \mathsf{B}_{\Cat{M''},\Cat{N''}}\mathsf{G}\mathsf{F} \ar{r}{\varphi_{\mathsf{G}}\mathsf{F}} \ar{d}{\varphi_{\mathsf{G}\mathsf{F}}} & \widehat{\mathsf{G}}\mathsf{B}_{\Cat{M}',\Cat{N}'}\mathsf{F} \ar{d}{\widehat{\mathsf{G}}\varphi_{\mathsf{F}}} \\
        \widehat{\mathsf{G}\mathsf{F}} \mathsf{B}_{\Cat{M},\Cat{N}} & \widehat{\mathsf{G}}\widehat{\mathsf{F}} \mathsf{B}_{\Cat{M},\Cat{N}} \ar{l}{\phi_{\mathsf{G},\mathsf{F}}\mathsf{B}_{\Cat{M},\Cat{N}}}.
      \end{tikzcd}
    \end{equation}
  \item \label{item:forth}For  three composable bimodule functors $\mathsf{F}:\MC \times \CN \rightarrow \MpC \times \CpN $, $\mathsf{G}:\MpC \times \CpN \rightarrow \MppC \times \CppN $, 
$\mathsf{H}: \MppC \times \CppN \rightarrow \MpppC \times \CpppN$, 
    the following diagram of natural isomorphisms commutes
    \begin{equation}
      \label{eq:49}
      \begin{tikzcd}
        \widehat{\mathsf{H}}\widehat{\mathsf{G}}\widehat{\mathsf{F}} \ar{r}{\phi_{\mathsf{H},\mathsf{G}}\widehat{\mathsf{F}}} \ar{d}{\widehat{\mathsf{H}}\phi_{\mathsf{G},\mathsf{F}}} & \widehat{\mathsf{H}\mathsf{G}} \widehat{\mathsf{F}} \ar{d}{\phi_{\mathsf{H}\mathsf{G},\mathsf{F}}} \\
        \widehat{\mathsf{H}} \widehat{\mathsf{G}\mathsf{F}} \ar{r}{\phi_{\mathsf{H},\mathsf{G}\mathsf{F}}} & \widehat{\mathsf{H}\mathsf{G}\mathsf{F}}.
      \end{tikzcd}
    \end{equation}
  \item \label{item:fith}The natural transformation $\kappa_{\Cat{M},\Cat{N}}$ from Definition \ref{definition:tensor-prod} defines a natural isomorphism 
    \begin{equation}
      \label{eq:51}
      \kappa_{\Cat{M},\Cat{N}}(1_{\Cat{M}\Box \Cat{N}}): \widehat{1_{\Cat{M} \times \Cat{N}}} \rightarrow 1_{\Cat{M}\Box \Cat{N}}, 
    \end{equation}
    such that for all bimodule functors
    $\mathsf{F}: \MC \times \CN \rightarrow \MpC \times \CpN$
    the following diagrams  commute
    \begin{equation}
      \label{eq:52}
      \begin{tikzcd}
        \widehat{\mathsf{F}}\widehat{1_{\Cat{M}\times \Cat{N}}} \ar{d}[xshift=-50pt]{\widehat{\mathsf{F}}\kappa(1_{\Cat{M} \Box \Cat{N}})}  \ar{r}{\phi_{\mathsf{F},1}}  & \widehat{\mathsf{F}}   &  
        \widehat{1_{\Cat{M}\times \Cat{N}}} \widehat{\mathsf{F}} \ar{r}{\phi_{1,\mathsf{F}}} \ar{d}[xshift=-53pt]{\kappa(1_{\Cat{M}' \Box \Cat{N}'} )\widehat{\mathsf{F}}} & \widehat{\mathsf{F}} \\
        \widehat{\mathsf{F}} 1_{\Cat{M} \Box \Cat{N}}, \ar{ur}{\id} &{} & 
        1_{\Cat{M}' \Box\Cat{N}'}\widehat{\mathsf{F}}. \ar{ur}{\id} &
      \end{tikzcd}
    \end{equation}
  \end{propositionlist}
\end{proposition}
\begin{proof}
  The functor $\mathsf{B}\mathsf{F}$ in the first part is balanced  and hence the functor $\Psi(\mathsf{B}\mathsf{F})=\widehat{\mathsf{F}}$ is well defined.
  From  the natural isomorphism $\varphi:1 \rightarrow \Phi \Psi$   in Definition \ref{definition:tensor-prod}, we obtain the  balanced  natural isomorphism $\varphi_{\mathsf{F}}: \mathsf{B}\mathsf{F} \rightarrow \widehat{ \mathsf{F}}\mathsf{B}$. This shows the first part. 
  Part \refitem{item:sec} follows directly by applying Lemma \ref{lemma:action-Psi-bal}  to the natural transformation $\Psi(\mathsf{B}\rho)$, which is denoted  $\widehat{\rho}$ in the sequel. 
  To show statement \refitem{item:third}, note that from the first part we obtain  balanced natural isomorphisms 
  \begin{equation}
    \label{eq:50}
    \begin{tikzcd}
      \widehat{\mathsf{G}}\widehat{\mathsf{F}}\mathsf{B} \ar{r}{\widehat{\mathsf{G}}\varphi_{\mathsf{F}}^{-1}} & \widehat{\mathsf{G}}\mathsf{B} \mathsf{F} \ar{r}{\varphi_{\mathsf{G}}^{-1} \mathsf{F}} & \mathsf{B}\mathsf{G}\mathsf{F} \ar{r}{\varphi_{\mathsf{G}\mathsf{F}}} & \widehat{\mathsf{G}\mathsf{F}}\mathsf{B},
    \end{tikzcd}
  \end{equation}
  which compose to a balanced natural isomorphism from $ \widehat{\mathsf{G}}\widehat{\mathsf{F}}\mathsf{B}$ to $\widehat{\mathsf{G}\mathsf{F}}\mathsf{B}$. 
  The natural isomorphism $\phi_{\mathsf{G},\mathsf{F}}: \widehat{ \mathsf{G}}\widehat{\mathsf{F}} \rightarrow \widehat{ \mathsf{G}\mathsf{F}}$ is then defined as 
  \begin{equation*}
    \phi_{\mathsf{G},\mathsf{F}}= \kappa^{-1}(\widehat{\mathsf{G}\mathsf{F}}) \cdot \Psi\left(  \varphi_{\mathsf{G}\mathsf{F}} \cdot \varphi_{\mathsf{G}}^{-1}\mathsf{F} \cdot \widehat{\mathsf{G}}\varphi_{\mathsf{F}}^{-1}\right) \cdot \kappa(\widehat{\mathsf{G}}\widehat{\mathsf{F}}).
  \end{equation*}
  This proves  the existence and uniqueness of the natural isomorphism $\phi_{\mathsf{G},\mathsf{F}}$, such that 
  (\ref{eq:diag-varphiGF-uniqueness}) commutes. Hence the third part follows.
  To show the forth part, note that by  definition of $\phi_{\mathsf{F},\mathsf{G}}$  and by the interchange law for 2-categories, the following  diagram commutes
  \begin{equation}
    \label{eq:53}
    \begin{tikzcd}
      {} & \widehat{\mathsf{H}}\mathsf{B}\mathsf{G}\mathsf{F} \ar{r}{\widehat{\mathsf{H}} \varphi_{\mathsf{G}}\mathsf{F}}  & \widehat{\mathsf{H}} \widehat{\mathsf{G}}\mathsf{B}\mathsf{F} \ar{dr}{\widehat{\mathsf{H}}\widehat{\mathsf{G}}\varphi_{\mathsf{F}}} \ar{dl}{\phi_{\mathsf{H},\mathsf{G}}\mathsf{B}\mathsf{F}} &{} \\
      \mathsf{B}\mathsf{H}\mathsf{G}\mathsf{F} \ar{dr}[below,xshift=-6pt]{\varphi_{\mathsf{H}\mathsf{G}\mathsf{F}}} \ar{ur}{\varphi_{\mathsf{H}}\mathsf{G}\mathsf{F}} \ar{r}{\varphi_{\mathsf{H}\mathsf{G}}\mathsf{F}}& \widehat{\mathsf{H}\mathsf{G}}\mathsf{B}\mathsf{F} \ar{dr}{\widehat{\mathsf{H}\mathsf{G}}\varphi_{\mathsf{F}}}& & \widehat{\mathsf{H}}\widehat{\mathsf{G}}\widehat{\mathsf{F}}\mathsf{B} \ar{dl}{\phi_{\mathsf{H},\mathsf{G}}\widehat{\mathsf{F}}\mathsf{B}} \\
      & \widehat{\mathsf{H}\mathsf{G}\mathsf{F}}\mathsf{B}  & \widehat{\mathsf{H}\mathsf{G}} \widehat{\mathsf{F}}\mathsf{B}. \ar{l}{\phi_{\mathsf{H}\mathsf{G},\mathsf{F}}\mathsf{B}}  &
    \end{tikzcd}
  \end{equation}
  Here  the interchange law is used to establish the commutativity of the parallelogram on the right, and part \refitem{item:third} shows the commutativity of the two parallelograms on the left in (\ref{eq:53}). 
  It also follows from \refitem{item:third} that the following diagram commutes
  \begin{equation}
    \label{eq:54}
    \begin{tikzcd}
      {} & \widehat{\mathsf{H}}\mathsf{B}\mathsf{G}\mathsf{F} \ar{r}{\widehat{\mathsf{H}} \varphi_{\mathsf{G}}\mathsf{F}} \ar{ddr}{\widehat{\mathsf{H}}\varphi_{\mathsf{G}\mathsf{F}}}  & \widehat{\mathsf{H}} \widehat{\mathsf{G}}\mathsf{B}\mathsf{F} \ar{dr}{\widehat{\mathsf{H}}\widehat{\mathsf{G}}\varphi_{\mathsf{F}}} &{} \\
      \mathsf{B}\mathsf{H}\mathsf{G}\mathsf{F} \ar{dr}[below,xshift=-6pt]{\varphi_{\mathsf{H}\mathsf{G}\mathsf{F}}} \ar{ur}{\varphi_{\mathsf{H}}\mathsf{G}\mathsf{F}} & & & \widehat{\mathsf{H}}\widehat{\mathsf{G}}\widehat{\mathsf{F}}\mathsf{B} \ar{dl}{\widehat{\mathsf{H}}\phi_{\mathsf{G},\mathsf{F}}\mathsf{B}} \\
      & \widehat{\mathsf{H}\mathsf{G}\mathsf{F}}\mathsf{B}  & \widehat{\mathsf{H}} \widehat{\mathsf{G}\mathsf{F}}\mathsf{B}. \ar{l}{\phi_{\mathsf{H},\mathsf{G}\mathsf{F}}\mathsf{B}}  &
    \end{tikzcd}
  \end{equation}
  Since all outer arrows in the diagrams (\ref{eq:53}) and (\ref{eq:54}) that do not contain  $\phi$ agree and all arrows are labeled by natural isomorphisms,  it follows   that the diagram 
  \begin{equation*}
    \begin{tikzcd}
      \widehat{\mathsf{H}}\widehat{\mathsf{G}}\widehat{\mathsf{F}} \mathsf{B}\ar{r}{\phi_{\mathsf{H},\mathsf{G}}\widehat{\mathsf{F}}\mathsf{B}} \ar{d}{\widehat{\mathsf{H}}\phi_{\mathsf{G},\mathsf{F}}\mathsf{B}} & \widehat{\mathsf{H}\mathsf{G}} \widehat{\mathsf{F}}\mathsf{B} \ar{d}{\phi_{\mathsf{H}\mathsf{G},\mathsf{F}}\mathsf{B}} \\
      \widehat{\mathsf{H}} \widehat{\mathsf{G}\mathsf{F}}\mathsf{B} \ar{r}{\phi_{\mathsf{H},\mathsf{G}\mathsf{F}}\mathsf{B}} & \widehat{\mathsf{H}\mathsf{G}\mathsf{F}}\mathsf{B}
    \end{tikzcd}
  \end{equation*}
  commutes. As the functor $\Phi$ is fully faithful, this shows that (\ref{eq:49}) commutes.
  For the last statement, note that the natural isomorphism $\kappa_{\Cat{M},\Cat{N}}: 1 \rightarrow \Psi_{\Cat{M},\Cat{N}}\Phi_{\Cat{M},\Cat{N}}$ from Definition \ref{definition:tensor-prod} provides a natural isomorphism
  \begin{equation}
    \label{eq:unit-comp-tensor}
    \begin{tikzcd}[column sep=large]
      \widehat{1_{\Cat{M} \times \Cat{N}}}= \Psi(\mathsf{B} \circ 1_{\Cat{M} \times \Cat{N}})=\Psi (1_{\Cat{M} \Box \Cat{N}} \circ \mathsf{B}) \ar{rr}{\kappa_{\Cat{M} \Box \Cat{N}}(1_{\Cat{M} \Box \Cat{N}})} && 1_{\Cat{M} \Box \Cat{N}}.    
    \end{tikzcd}
  \end{equation}
The snake identity (\ref{eq:duality-snake-left}) for the adjoint equivalence then implies that the diagram
  \begin{equation}
    \label{eq:snake-it-unit}
    \begin{tikzcd}
      \Phi(1_{\Cat{M}\times \Cat{N}})=\mathsf{B}  \ar{r}{\Phi \varphi} \ar{dr}[below]{1}&  \Phi \Psi \Phi (1_{\Cat{M}\times \Cat{N}} )= \widehat{1_{\Cat{M} \times \Cat{N}}}\mathsf{B}  \ar{d}{\kappa \Phi} \\
      & \Phi(1_{\Cat{M}\times \Cat{N}})
    \end{tikzcd}
  \end{equation}
  commutes. 
  Hence the diagram 
  \begin{equation}
    \label{eq:needed}
    \begin{tikzcd}
      \widehat{\mathsf{F}} \mathsf{B} \ar{r}{\varphi_{\mathsf{F}}} \ar{d}[xshift=-30pt]{\widehat{\mathsf{F}} \varphi(\mathsf{B})}  \ar{dr}{\id} &\mathsf{B}\mathsf{F} \ar{d}{\varphi_{\mathsf{F}}} \\
      \widehat{\mathsf{F}} \widehat{1}\mathsf{B} \ar{r}{\widehat{\mathsf{F}} \kappa \mathsf{B}} & \widehat{\mathsf{F}}\mathsf{B}
    \end{tikzcd}
  \end{equation}
  commutes. By the unique characterization of the natural isomorphism $\phi_{\widehat{\mathsf{F}},\widehat{1}}: \widehat{\mathsf{F}}\widehat{1} \rightarrow \widehat{\mathsf{F}}$ from part \refitem{item:third}, we deduce that 
  $\phi_{\widehat{\mathsf{F}},\widehat{1}}=\widehat{\mathsf{F}}\kappa(1_{\Cat{M} \Box \Cat{N}})$. The remaining identity is proven analogously using the unique characterization of $\phi_{\widehat{1},\widehat{\mathsf{F}}}$ from part \refitem{item:third}. 
\end{proof}
Note that the notation $\widehat{\mathsf{F}}$ was already used for image of $\Psi$ on balanced functors. It should be clear from the context, whether a functor is balanced or a module functor, 
hence the notation is unambiguous. Moreover, we next unify the notions of balanced functors and module functors, and regard them 
as 1-morphisms in a certain 2-category.
This provides  further justification for   the notation $\widehat{\mathsf{F}}$.

The map $\mathsf{F} \mapsto \widehat{\mathsf{F}}$ from the previous proposition for balanced functors $\mathsf{F}$ 
is compatible with the composition of bimodule functors and balanced functors. To make this statement precise, we define the following 2-category, that 
 combines balanced functors and bimodule functors into a single 2-category. 
\begin{proposition}
  \label{proposition:2-cat-extend-bimod}
  The following data defines a 2-category $\ModbC$ for a finite tensor category $\Cat{C}$. 
  \begin{propositionlist}
  \item  The objects of $\ModbC$ are $(\Cat{C},\Cat{C})$-bimodule categories $\MC \times \CN$ and linear categories $\Cat{A}$.
  \item The  categories of 1- and 2-morphisms between the objects are given by:
    \begin{propositionlist}
    \item For bimodule categories $\MC \times \CN$ and $\MpC \times \CpN$, $\ModbC(\MC \times \CN,\MpC \times \CpN)$ 
      is the category $\BimCat(\MC \times \CN,\MpC \times \CpN)$ of  bimodule functors and bimodule natural transformation between them. 
    \item For a bimodule category $\MC \times \CN$ and a category $\Cat{A}$,  $ \ModbC(\MC \times \CN,\Cat{A})$ is the
       category $\Funbal(\MC \times \CN,\Cat{A})$ of balanced functors and balanced natural transformations between them.
    \item For two categories $\Cat{A}$, and $\Cat{B}$, $\ModbC(\Cat{A},\Cat{B})$ is the 
      category $\Fun(\Cat{A},\Cat{B})$ of  functors and natural transformations between them 
    \item There is just the zero morphism  from a category $\Cat{A}$ to a bimodule category  $\MC \times \CN$.
    \end{propositionlist}
  \item The compositions are induced by the horizontal composition  of functors and the vertical composition of natural transformations.
  \end{propositionlist}
\end{proposition}
\begin{proof}
  It follows from Lemma \ref{lemma:balancing-and-composition} that the various compositions of 1- and 2-morphisms are well defined. It follows by a direct computation that all compositions are strictly associative and 
  strictly compatible with the units.
\end{proof}
\begin{proposition}
  \label{proposition:Tensor-extends-2-fun}
  The tensor product of module categories defines a 2-functor
  \begin{equation}
    \label{eq:modbal-functor}
    \widehat{(-)}: \ModbC\rightarrow \Catlin.
  \end{equation}
\end{proposition}
\begin{proof}
  On objects, $\widehat{(-)}$ is defined by $\widehat{\MC \times \CN}=\MC \Box \CN$ and $\widehat{\Cat{A}}=\Cat{A}$. On 1-morphisms, $\widehat{(-)}$ is defined as follows. 
  For a  bimodule functor $\mathsf{F}: \MC \times \CN \rightarrow \MpC \times \CpN$, and for a balanced functor $\mathsf{G}:\MC \times \CN \rightarrow \Cat{A} $, the functors 
  $\widehat{\mathsf{F}}: \MC \Box \CN \rightarrow \MpC \Box \CpN$ and $\widehat{\mathsf{G}}: \MC \Box \CN \rightarrow \Cat{A}$ are already defined in  Definition \ref{definition:tensor-prod} and 
  in Proposition \ref{proposition:tensor-functors}, \refitem{item:first}, respectively. For a functor $\mathsf{H}: \Cat{A} \rightarrow \Cat{B}$, we define $\widehat{\mathsf{H}}=\mathsf{H}$.
  On 2-morphisms, $\widehat{(-)}$ is already defined for bimodule natural transformations and balanced natural transformations. For a natural transformation $\eta: \mathsf{H} \rightarrow \mathsf{H}'$
  between functors $\mathsf{H},\mathsf{H}': \Cat{A} \rightarrow \Cat{B}$, we define $\widehat{ \eta}=\eta$.

  The coherence structures of $\widehat{(-)}$ are the following.
  \begin{enumerate}
  \item For all bimodule categories $\MC \times \CN$, the coherence isomorphism $\widehat{1_{\Cat{M} \times \Cat{N}}} \rightarrow 1_{\Cat{M} \Box \Cat{N}}$ is defined by $\kappa_{\Cat{M},\Cat{N}}(1_{\Cat{M}\Box \Cat{N}})$, 
    as in Proposition \ref{proposition:tensor-functors}, \refitem{item:fith}.
  \item For composable bimodule functors $\mathsf{F},\mathsf{G}$, there is a  natural isomorphism $\phi_{\mathsf{G},\mathsf{F}}: \widehat{ \mathsf{G}} \widehat{\mathsf{F}} \rightarrow \widehat{\mathsf{G}\mathsf{F}}$ that is defined by  Proposition \ref{proposition:tensor-functors}, \refitem{item:third}.
  \item  For composable functors $\mathsf{H}:\Cat{A} \rightarrow \Cat{B}$ and $\mathsf{K}:\Cat{B} \rightarrow \Cat{C}$, we define $\phi_{\mathsf{K},\mathsf{H}}=\id_{\mathsf{K}\mathsf{H}}$.
  \item For a balanced functor $\mathsf{F}: \MC \times \CN \rightarrow \Cat{A}$ and a bimodule functor $\mathsf{G}: \MpC \times \CpN \rightarrow \MC \times \CN$, it follows from Lemma \ref{lemma:action-Psi-bal}, that 
    there exists a unique balanced natural isomorphism $\phi_{\mathsf{F},\mathsf{G}}: \widehat{\mathsf{F}} \widehat{\mathsf{G}} \rightarrow \widehat{\mathsf{F}\mathsf{G}}$, such that the following diagram commutes
    \begin{equation}
      \label{eq:65}
      \begin{tikzcd}
        \mathsf{F}\mathsf{G} \ar{r}{\varphi(\mathsf{F})} \ar{d}{\varphi(\mathsf{F}\mathsf{G})}  &\widehat{\mathsf{F}}\mathsf{B}\mathsf{G} \ar{d}{\widehat{\mathsf{F}} \varphi_{\mathsf{G}}} \\
        \widehat{\mathsf{F}\mathsf{G}}\mathsf{B} & \widehat{\mathsf{F}} \widehat{\mathsf{G}} \mathsf{B}. \ar{l}{\phi_{\mathsf{F},\mathsf{G}}\mathsf{B}}
      \end{tikzcd}
    \end{equation}
  \item For a balanced functor $\mathsf{F}: \MC \times \CN \rightarrow \Cat{A}$ and a functor $\mathsf{H}: \Cat{A} \rightarrow \Cat{B}$ it follows from Lemma \ref{lemma:action-Psi-bal}, that 
    there exists a unique balanced natural isomorphism $\phi_{\mathsf{H},\mathsf{F}}: {\mathsf{H}} \widehat{\mathsf{F}} \rightarrow \widehat{\mathsf{H}\mathsf{F}}$, such that the following diagram commutes
    \begin{equation}
      \label{eq:66}
      \begin{tikzcd}
        \mathsf{H}\mathsf{F} \ar{r}{\id} \ar{d}{\varphi(\mathsf{H}\mathsf{F})}  &\mathsf{H}\mathsf{F} \ar{d}{\mathsf{H} \varphi(\mathsf{F})} \\
        \widehat{\mathsf{H}\mathsf{F}} \mathsf{B} & \mathsf{H} \widehat{\mathsf{F}} \mathsf{B}. \ar{l}{\phi_{\mathsf{F},\mathsf{G}}\mathsf{B}}
      \end{tikzcd}
    \end{equation}
  \end{enumerate}
  The proof that
  for three composable 1-morphisms, the diagram (\ref{2-fun-ax-3-morphism}) commutes is analogous to the proof of Proposition \ref{proposition:tensor-functors}  \refitem{item:forth}, while the compatibility 
  of $\widehat{(-)}$ with the unit from axiom (\ref{eq:2-fun-ax-unit}) follows analogously to the proof of  Proposition \ref{proposition:tensor-functors} \refitem{item:fith}.
\end{proof}

\subsection{The tensor product of bimodule categories}
\label{sec:tens-prod-bimod}

Next we  show that the tensor product of module categories naturally extends to a tensor product of bimodule categories. 
Furthermore we consider the corresponding extension of the tensor product as a 2-functor. 

\begin{proposition}
  \label{proposition:tensor-again-module}
  Let  $\DMC$ and $\CNE$ be bimodule categories. The tensor product $\DMC \Box \CNE$ has a canonical structure of a $(\Cat{D},\Cat{E})$-bimodule category, such that 
  \begin{equation}
    \label{eq:28}
    \mathsf{B}: \DMC \times \CNE \rightarrow \DMC \Box \CNE
  \end{equation}
  is a balanced bimodule functor and for all bimodule categories $\DAE$  the adjoint equivalence from Definition \ref{definition:tensor-prod} restricts  to an adjoint equivalence 
  \begin{equation}
    \label{eq:58}
    \begin{split}
      \Phi &: \Funl{\Cat{D},\Cat{E}}{\DMC \Box \CNE, \DAE} \rightarrow \Funball{\Cat{D},\Cat{E}}{\DMC \times \CNE, \DAE}, \\
      \Psi &:  \Funball{\Cat{D},\Cat{E}}{\DMC \times \CNE, \DAE}\rightarrow \Funl{\Cat{D},\Cat{E}}{\DMC \Box \CNE, \DAE},
    \end{split}
  \end{equation}
where  $ \Funball{\Cat{D},\Cat{E}}{\DMC \times \CNE, \DAE}$ is the category of balanced bimodule functors from  Definition \ref{definition:balanced-module}.
\end{proposition}
\begin{proof}
  To define the left $\Cat{D}$-module structure on $\Cat{M}\Box \Cat{N}$, note that for all  $d \in \Cat{D}$, the functors $L_{d}: \Cat{M} \times \Cat{N} \rightarrow \Cat{M} \times \Cat{N}$ 
  provided by the action of $d \in \Cat{D}$ are $(\Cat{C},\Cat{C})$-bimodule  functors and the module constraint for the left action of $\Cat{D}$ consists of $(\Cat{C}, \Cat{C})$-bimodule  natural isomorphisms
  $\mu_{d,d'}: L_{d} \circ L_{d'} \rightarrow L_{d \otimes d'}$ for all $d,d' \in \Cat{D}$. Hence we can apply the 2-functor $\widehat{(-)}$ from Proposition \ref{proposition:Tensor-extends-2-fun} and 
  we obtain for all $d \in \Cat{D}$ functors $\widehat{L_{d}}: \Cat{M} \Box \Cat{N} \rightarrow \Cat{M} \Box\Cat{N}$ and natural isomorphisms $\widehat{\mu_{d,d'}}:\widehat{ L_{d}} \circ\widehat{ L_{d'}} \rightarrow \widehat{L_{d \otimes d'}}$
  for all $d,d' \in \Cat{D}$. The module constraint (\ref{eq:diagramm}) for these natural isomorphisms is obtained by applying the 2-functor $\widehat{(-)}$ to the corresponding module constraint for $\CM$.

  The right $\Cat{E}$-module structure on $\CMD \Box \DNE$ is defined analogously by considering the  $(\Cat{C},\Cat{C})$-bimodule functors $R_{e}: \Cat{M} \times \Cat{N} \rightarrow \Cat{M} \times \Cat{N}$ for all $e \in \Cat{E}$.

It follows that  $\Cat{M}  \Box \Cat{N}$ is a $(\Cat{D}, \Cat{E})$-bimodule category since the bimodule constraints follow directly by  applying the 2-functor $\widehat{(-)}$ to the corresponding 
  diagrams for $\DMC \times \CNE$. 

Next we show that $\mathsf{B}: \DMC \times \CNE \rightarrow \DMC \Box \CNE$ is a balanced bimodule functor. By definition of the left $\Cat{D}$-module structure on $\DM \Box \Cat{N}$, we obtain balanced natural 
isomorphisms $\varphi_{d}: \mathsf{B} L_{d} \rightarrow \widehat{L_{d}} \mathsf{B}$ for all $d \in \Cat{D}$, that are compatible with the compositions $\widehat{L_{d}} \circ \widehat{L_{d'}}$ according to
 Proposition \ref{proposition:tensor-functors}\refitem{item:third}. This shows that $\mathsf{B}$ is a left $\Cat{D}$-module functor, and by the analogous argument, a right $\Cat{E}$-module functor. The compatibility between these two module functor structures follows from Proposition \ref{proposition:tensor-functors}, since $L_{d}R_{e}= R_{e}L_{d}$ as functors $\Cat{M} \times \Cat{N} \rightarrow \Cat{M} \times \Cat{N}$ for objects $d \in \Cat{D}$ and $e \in \Cat{E}$. 
Hence $\mathsf{B}$ is a bimodule functor, and since the module constraints are balanced natural isomorphisms, it is also a balanced bimodule functor according to Lemma \ref{lemma:characterize-bal-module}.

  In the next step we show that the functor $\Psi$ from Definition \ref{definition:tensor-prod} restricts to a functor
  \begin{equation*}
    \Psi:  \Funball{\Cat{D},\Cat{E}}{\DMC \times \CNE, \DAE}\rightarrow \Funl{\Cat{D},\Cat{E}}{\DMC \Box \CNE, \DAE}.
  \end{equation*}
  Let $\mathsf{G} \in \Funball{\Cat{D},\Cat{E}}{\DMC \times \CNE, \DAE}$ be a balanced bimodule functor. The left $\Cat{D}$-module functor structure  on $\mathsf{G}$ is given by 
  $\Cat{C}$-balanced natural isomorphisms $\phi^{\mathsf{G}}_{d}:\mathsf{G} \circ L_{d}^{\Cat{M} \times \Cat{N}} \rightarrow L^{\Cat{A}}_{d} \circ \mathsf{G}$ for all $d \in \Cat{D}$ according to Lemma \ref{lemma:characterize-bal-module}. 
  Hence we can apply the 2-functor $\widehat{(-)}$ and obtain natural isomorphisms  $\widehat{\phi^{\mathsf{G}}_{d}}: \widehat{\mathsf{G}} \widehat{L^{\Cat{M} \times \Cat{N}}_{d}} \rightarrow L_{d}^{\Cat{A}} \widehat{\mathsf{G}}$.
  Furthermore, applying  $\widehat{(-)}$  to the module constraint diagram for $\mathsf{G}$ yields the module constraint diagram for $\widehat{\mathsf{G}}$. Hence we deduce that $\widehat{\mathsf{G}}$ is a left $\Cat{D}$-module functor. 
  The proof that   $\widehat{\mathsf{G}}$ is a right $\Cat{E}$-module functor is analogous.  The compatibility between left and right module actions of $\widehat{\mathsf{G}}$ follows by applying the functor $\widehat{(-)}$ to 
  the corresponding compatibility diagram of $\mathsf{G}$. Hence $\widehat{\mathsf{G}}$ is a bimodule functor. 

  If $\eta: \mathsf{G} \rightarrow \mathsf{F}$ is a balanced bimodule natural transformation between balanced bimodule functors $\mathsf{F}$ and $\mathsf{G}$, it follows again by applying the 2-functor $\widehat{(-)}$ that $\widehat{\eta}: 
  \widehat{\mathsf{G}} \rightarrow \widehat{\mathsf{F}}$ is a bimodule natural transformation. 

  It remains to show that for all balanced bimodule functors  $\mathsf{F} \in \Funball{\Cat{D},\Cat{E}}{\DMC \times \CNE, \DAE}$, the natural isomorphism $\varphi(\mathsf{G}): \mathsf{G} \rightarrow \widehat{\mathsf{G}}\mathsf{B}$ is a balanced bimodule natural 
  isomorphism and for all bimodule functors $\mathsf{G}:  \Cat{M} \Box \Cat{N} \rightarrow \Cat{A}$, the natural isomorphism $\kappa(\mathsf{G}): \mathsf{G} \rightarrow \widehat{\mathsf{G}\mathsf{B}}$ is a bimodule natural isomorphism. 
  The first statement follows directly from the definition of the bimodule structure of $\widehat{\mathsf{G}}$. 
  For the second statement, we show that the lower rectangle in the diagram
  \begin{equation}
    \label{eq:61}
    \begin{tikzcd}[column sep=large]
      \mathsf{G}\widehat{L_{d}}\mathsf{B}  \ar{r}{\phi_{d}^{\mathsf{G}}\mathsf{B}}\ar{d}{1 \varphi_{L_{d}}^{-1}} \ar[bend right=50]{dddd}{1} & L_{d}\mathsf{G}\mathsf{B} \ar{ddd}{1 \varphi{\mathsf{G}\mathsf{B}}} \ar[bend left=50]{dddd}{1}\\
      \mathsf{G}\mathsf{B}L_{d} \ar{d}{\varphi(\mathsf{G}\mathsf{B})1} & \\
      \widehat{\mathsf{G}\mathsf{B}}\mathsf{B} L_{d} \ar{d}{1\varphi_{L_{d}}}& \\
      \widehat{\mathsf{G}\mathsf{B}}\widehat{L_{d}}\mathsf{B} \ar{r}{\widehat{\phi_{d}^{\mathsf{G}}}1}  \ar{d}{\kappa^{-1}(\mathsf{G})1} & L_{d} \widehat{\mathsf{G}\mathsf{B}} \mathsf{B} \ar{d}[left]{1\kappa^{-1}(\mathsf{G})1}\\
      \mathsf{G} \widehat{L_{d}}\mathsf{B} \ar{r}{\phi^{\mathsf{G}}_{d}1} & L_{d}\mathsf{G}\mathsf{B} 
    \end{tikzcd}
  \end{equation}
  commutes. Because  $\Phi$ is fully faithful, $\kappa(\mathsf{G})$ is then a bimodule natural isomorphism. 
  The big diagram in the middle commutes by definition of $\widehat{\phi^{\mathsf{G}}_{d}}$. The diagram on the right commutes since $\kappa$ and $\varphi$ satisfy the snake identity. The diagram on the left 
  commutes also by the snake identity for $\kappa$ and $\varphi$ after applying once the interchange law for functors and natural transformations. 
\end{proof}

We further generalize the results of the previous section. First we unify balanced bimodule functors and bimodule functors in one 2-category. 
The next statement follows directly from the obvious version of  Lemma \ref{lemma:balancing-and-composition} 
for balanced bimodule functors.

\begin{proposition}
  \label{proposition:2-cat-extend-bimod-bimod}
  For every finite tensor category $\Cat{C}$ and every pair of finite tensor categories $(\Cat{D},\Cat{E})$, the following data define a 2-category $\BimodbC(\Cat{D},\Cat{E})$. 
  \begin{propositionlist}
  \item  The objects of $\BimodbC(\Cat{D},\Cat{E})$ are $(\Cat{D},\Cat{E})$-bimodule
 categories $\DMC \times \CNE$ and  $(\Cat{D},\Cat{E})$-bimodule categories  $\DAE$.
  \item The following defines the categories of 1- and 2-morphisms between the objects:
    \begin{propositionlist}
    \item The category $\BimodbC(\DMC \times \CNE, \DMpC \times \CNpE)$  
for two  $(\Cat{D},\Cat{E})$-bimodule categories  $\DMC \times \CNE$ and  $\DMpC \times \CNpE$
   is   the category $\Funlmul{\Cat{D},\Cat{C}}{\Cat{C},\Cat{E}}{\DMC \times \CNE, \DMpC \times \CNpE}$ of $(\Cat{D} \times \rev{C} \times \Cat{C} \times \Cat{E})$-module functors 
      and $(\Cat{D} \times \rev{C} \times \Cat{C} \times \Cat{E})$-module natural transformations between them.
    \item The category $\BimodbC(\DMC \times \CNE,\DAE)$ for  bimodule categories $\DMC \times \CNE$  and $\DAE$  
      is the category $\Funball{\Cat{D},\Cat{E}}{\DMC \times \CNE, \DAE}$ of balanced bimodule functors and balanced bimodule natural transformations between them. 
    \item The category  $ \BimodbC(\DAE,\DBE)$  for two bimodule categories $\DAE$, and $\DBE$
      is the category $\BimCat(\DAE,\DBE)$ of bimodule functors and bimodule natural transformations between them.
    \item There is just the zero morphism from a bimodule category $\DAE$ to a bimodule category  $\DMC \times \CNE$.
    \end{propositionlist}
  \item The compositions are induced by the horizontal composition  of functors and the vertical composition of natural transformations.
  \end{propositionlist}
\end{proposition}
If we restrict to the case where $\Cat{D}=\Cat{E}=\Vect$, we recover the 2-category from Proposition \ref{proposition:2-cat-extend-bimod}, i.e. $\BimodbC(\Vect, \Vect)= \ModbC$.
\begin{proposition}
  \label{proposition:Box-2-functor-bimod}
  The tensor product of bimodule categories defines a 2-functor
  \begin{equation}
    \label{eq:14}
    \widehat{(-)}: \BimodbC(\Cat{D},\Cat{E}) \rightarrow \BimCat(\Cat{D},\Cat{E}).
  \end{equation}
  In particular, it induces a 2-functor 
  \begin{equation}
    \label{eq:Box-prod-erst}
    \Box: \BimCat(\Cat{D},\Cat{E}) \times \BimCat(\Cat{C},\Cat{D}) \rr \BimCat(\Cat{C},\Cat{E}).
  \end{equation}
\end{proposition}
\begin{proof}
  Proposition \ref{proposition:tensor-again-module} shows that the functors $\Phi$, $\Psi$ and the natural transformations $\varphi$ and $\kappa$ that appear in the definition 
  of the tensor product, are compatible with  the $(\Cat{C},\Cat{E})$ bimodule structure of 
  a bimodule category $\CMD \times \DNE$. It is straightforward to see that 
  the analogue of Proposition \ref{proposition:tensor-functors} and Proposition \ref{proposition:Tensor-extends-2-fun}  hold for bimodule categories. In particular, all coherence structures of the 2-functor $\widehat{(-)}$ from 
  Proposition \ref{proposition:Tensor-extends-2-fun} are bimodule natural isomorphisms. 
\end{proof}

\subsection{Multi-module categories}
\label{sec:tric-multi-module}
In the following  we consider also multiple tensor 
products of the form $  (\HKE \Box \END) \Box \DMC$. This requires an extension of 
the notion of balanced functors to so-called multi-balanced functors from 
$ (\KE \times \END) \times \DM$ to a linear category $\Cat{A}$. An example is the  functor 
\begin{equation}
  \label{eq:34}
  \begin{tikzcd}
    (\HKE \times \END) \times \DMC \ar{r}{\mathsf{B}_{\Cat{K},\Cat{N}} \times 1} & (\HKE \Box \END) \times \DMC 
    \ar{r}{\mathsf{B}_{\Cat{K} \Box \Cat{N}, \Cat{M}}}&   (\HKE \Box \END) \Box \DMC.
  \end{tikzcd}
\end{equation}
We then group these multi-balanced functors into a suitable bicategory, such that (\ref{eq:34})
is a composition in this bicategory.
Note, however, that the functor $\mathsf{B}_{\Cat{K},\Cat{N}} \times 1$ is balanced with respect to the first two categories, 
but it is a bimodule functor (the identity) with respect to the third. 
Therefore we need to extend the notion of  multi-balanced functors even 
further to so-called multi-balanced module functors, in order to guarantee that the functor $\mathsf{B}_{\Cat{K},\Cat{N}} \times 1$ is in this bicategory. 
The multi-balanced module functors will play an essential role in the proof 
that  bimodule categories forms a tricategory. 

In order to define  the associator in this tricategory, we will be careful and  distinguish the two categories $(\Cat{M} \times \Cat{N}) \times \Cat{K}$ and $\Cat{M} \times (\Cat{N} \times \Cat{K})$ for three categories $\Cat{M}$, $\Cat{N}$ and $\Cat{K}$. The relation between these categories will then finally lead to the associator in the tricategory of bimodule categories. Hence we  say that 
a \emph{bracketing} $b$ of a string $\underline{X}=(X_{1}, \ldots, X_{n})$ of letters $X_{i}$ is a choice of parenthesis  that uniquely specifies a sequence of pairings like e.g. $(X_{1}(X_{2}X_{3}))X_{4}$.

For a functor $\mathsf{F}: (\Cat{M} \times \Cat{N}) \times \Cat{K} \rightarrow \Cat{A}$, we denote the functor on objects just by $\mathsf{F}(m \times n \times k)$, if the bracketing is clear from the context. 
Recall from Remark \ref{remark:clash-biadditve}, that for a module category $\CM$ and a finite  linear category $\Cat{N}$, we consider the category  $\CM \times \Cat{N}$  
again as  module category with $\Cat{C}$-module action $\act \times \id_{\Cat{N}}$. In the following it is always understood that 
the Cartesian product of module  categories is equipped with this module action.  We call two bimodule categories $\Cat{M}$ and $\Cat{N}$ composable, if the category that acts from the left on $\Cat{N}$ coincides 
with the category that acts from the right on $\Cat{M}$. 
\begin{definition}[{{\cite[Def 3.4]{Green}}}] 
  \label{definition:multi-balanced}
  \begin{definitionlist}
  \item A multi-module category  $(\underline{\Cat{M}},b)$ from $\Cat{C}$ to $\Cat{D}$ is a finite  string of composable bimodule categories $\Cat{M}^{j}$ for $j \in \{1, \ldots, n\}$ with  $n \in \N$, 
    where $\Cat{M}^{n}$ is a $\Cat{C}$-right module category and $\Cat{M}^{1}$ is a $\Cat{D}$-left module category, 
    together with a bracketing $b$ of  $\underline{\Cat{M}}$.
    We denote by $ev(\underline{\Cat{M}},b)$ the Cartesian product of the categories $\Cat{M}_{j}$, in the  order that corresponds to the bracketing $b$.
  \item \label{item:multi-balanced}
    A multi-balanced functor $\mathsf{F}: (\underline{\Cat{M}},b) \rightarrow \Cat{A}$, from a $(\Vect,\Vect)$ multi-module category $(\underline{\Cat{M},b})$ to a linear category $\Cat{A}$ is a functor $\mathsf{F}: ev(\underline{\Cat{M}},b) \rightarrow \Cat{A}$, 
  that is balanced in each argument, i.e. it is equipped with  with natural isomorphisms 
    \begin{equation*}
      \begin{split}
   b^{\mathsf{F}}_{m_{1}, \ldots, m_{i},c, m_{i+1}, \ldots, m_{n}}: & \mathsf{F}(m_{1} \times \ldots \times m_{i}\ract c \times m_{i+1}\times \ldots m_{n}) \\
\rightarrow & \mathsf{F}(m_{1} \times \ldots \times m_{i}\times c \act m_{i+1}\times \ldots m_{n}),
      \end{split}
      \end{equation*}
    for each string of objects $\underline{m} \in ev(\underline{\Cat{M}},b)$, each $i \in J$ and $c \in \Cat{C}$, such that the natural isomorphisms $b^{\mathsf{F}}$ satisfy the diagram (\ref{eq:balanced-functor}) in each entry $i \in J$.
    In the sequel we will abbreviate  $b^{\mathsf{F}}_{m_{1}, \ldots, m_{i},c, m_{i+1}, \ldots, m_{n}}$  by $b_{\ldots, m_{i},c,m_{i+1}, \ldots}$ whenever it is unambiguous.

    Additionally, these isomorphisms are required to be compatible with the bimodule category structures, i.e. the diagram
    \begin{equation*}
      \begin{tikzcd}[column sep=large, font=\small] 
        \mathsf{F}(\ldots  m_{i-1}\ract c \times m_{i}\ract d \times m_{i+1}  \ldots) \ar{r}{b_{\ldots m_{i},d, m_{i+1} \ldots}} \ar{d}{b_{ \ldots m_{i-1},c,m_{i \ract d}\ldots }}& \mathsf{F}(\ldots m_{i-1} \ract c \times m_{i} \times d \act m_{i+1} \ldots ) 
        \ar{dd}{b_{\ldots m_{i-1},c,m_{i}\ldots}}\\
        \mathsf{F}(\ldots  m_{i-1} \times c\act (m_{i}\ract d) \times m_{i+1} \ldots) \ar{d}{\gamma^{-1}_{c,m_{i},d}} & \\
        \mathsf{F}(\ldots  m_{i-1} \times (c \act m_{i} )\ract d \times m_{i+1}  \ldots ) \ar{r}{b_{\ldots,c,d\act m_{i},m_{i+1},\ldots}} & \mathsf{F}(\ldots  m_{i-1}\times c \act m_{i} \times d \act m_{i+1} \ldots),
      \end{tikzcd}
    \end{equation*}
    commutes for  each possible entry $i \in J$ and for all possible objects. Here the argument of the functor $\mathsf{F}$ is abbreviated and only the relevant part of the string $\underline{m}$ is shown.
 \item 
    A multi-balanced natural transformation  $\eta: \mathsf{F} \rightarrow \mathsf{G}$ between multi-balanced functors $\mathsf{F},\mathsf{G}: (\underline{\Cat{M}},b) \rightarrow \Cat{A}$ is a natural transformation $\eta$ 
    that is balanced in each entry, i.e. it satisfies diagram (\ref{eq:balanced-nat}) for all entries of a string of objects $\underline{m}$ in $\underline{\Cat{M}}$.
  \item 
    For every multi-module category  $(\underline{\Cat{M}},b)$, there is a corresponding string of finite tensor categories $S(\underline{\Cat{M}'},b)$,
    that is given by the finite tensor categories acting on the bimodule categories in $(\underline{\Cat{M}},b)$ such that 
    for the string $(\HKE,\ENF,\ldots,\DMC)$, the corresponding string of finite tensor categories is $S(\HKE,\ENF, \ldots, \DMC)=(\Cat{H},\Cat{E},\Cat{F},\ldots, \Cat{D},\Cat{C})$.
    Note that by definition  $S(\underline{\Cat{M}},b)= S(\underline{\Cat{M}},b')$ is independent of the bracketing $b$ and just called $S(\underline{\Cat{M}})$ in the sequel. 
   \end{definitionlist}
\end{definition}
It is clear, that for each linear category $\Cat{A}$, the multi-balanced functors and multi-balanced natural transformations from $(\underline{\Cat{M}},b)$ to $\Cat{A}$ 
form a category denoted  $\Funbal((\underline{\Cat{M}},b), \Cat{A})$.

Next we consider multi-module functors. 
\begin{definition}
  \label{definition:muli-module-functor}
  \begin{definitionlist}
  \item A multi-module functor $\mathsf{F}: (\underline{\Cat{M}},b) \rightarrow (\underline{\Cat{M}'},b')$ between multi-module categories $ (\underline{\Cat{M}},b)$ and $(\underline{\Cat{M}'},b')$  with $S(\underline{\Cat{M}})=S(\underline{\Cat{M}'})$
    is a functor $\mathsf{F}: ev(\underline{\Cat{M}},b) \rightarrow ev(\underline{\Cat{M}'},b')$ together with
    a family of natural isomorphisms 
    \begin{equation*}
      \begin{split}
          \Phi^{\mathsf{F},l,i}_{m_{1}, \ldots, m_{i},d, m_{i+1}, \ldots, m_{n}}: &\mathsf{F}(m_{1} \times \ldots \times m_{i}\ract d \times m_{i+1}\times \ldots m_{n}) \\
\rightarrow & \mathsf{F}(m_{1} \times \ldots \times m_{i}\times  m_{i+1}\times \ldots m_{n}) \ract^{i} d,
      \end{split}
       \end{equation*}
    for each $\underline{m} \in \underline{\Cat{M}}$ and each $i \in J$, where $\ract^{i} : \underline{\Cat{M}} \times \Cat{D} \rightarrow \underline{\Cat{M}}$ denotes the action of $\Cat{D}$ on 
    $\CMD^{i}$. 
    Similarly we require that there exists a family of natural isomorphisms 
    \begin{equation*}
      \Phi^{\mathsf{F},r,i}_{m_{1}, \ldots,c, m_{i}, \ldots, m_{n}}: \mathsf{F}(m_{1} \times \ldots \times c\act  m_{i} \times \ldots m_{n}) \simeq  c \act^{i} \mathsf{F}(m_{1} \times \ldots \times m_{i}\times \ldots m_{n}),
    \end{equation*}
    where $\act^{i}: \Cat{C} \times \underline{\Cat{M}} \rightarrow \underline{\Cat{M}}$ is induced by the left action of $\Cat{C}$ on $\CMD^{i}$.
    The isomorphisms $ \Phi^{\mathsf{F},l,i}$ and $ \Phi^{\mathsf{F},r,i}$ are required to satisfy the bimodule constraint (\ref{eq:bimodule-functor-charact}) for each $i \in J$.
  \item A multi-module natural transformation $\eta:\mathsf{F} \rightarrow \mathsf{G}$ between multi-module functors $\mathsf{F}$ and $\mathsf{G}$ is a natural transformation that 
    satisfies equation
    (\ref{eq:module-nat-transf}) in each entry.
  \end{definitionlist}
\end{definition}

\begin{example}
  \begin{examplelist}
  \item For two bimodule functors $\mathsf{F}: \DMC \rightarrow \DMpC$ and 
$\mathsf{G}: \CNE \rightarrow \CNpE$, the functor $ \mathsf{G} \times \mathsf{F}: \Cat{N} \times \Cat{M} \rightarrow \Cat{N}' \times \Cat{M}'$ is a multi-module functor.
  \item  For three bimodule categories a multi-module functor $\alpha: (\HKE \times \END) \times \DMC \rightarrow \HKE \times ( \END \times \DMC)$  is given by $\alpha((h \times n) \times m)= h \times ( n \times m)$
    on objects and morphisms $((h \times n) \times m)$ in $(\HKE \times \END) \times \DMC$.
  \end{examplelist}
\end{example}
Next we consider multi-balanced module functors.
\begin{definition}
  \label{definition:multi-balanced-mod}
  \begin{definitionlist}
  \item For a string $(X_{1}, \ldots X_{n})$, a reduced string $(X'_{1}, \ldots X_{m}')$ is a string that is obtained from $(X_{1}, \ldots X_{n})$ by erasing entries as follows. 
    It is required that there exists an injective map $f: \{1, \ldots,m\} \rightarrow \{1, \ldots,n \}$ with $f(1)=1$ and $f(m)=n$ and $f(j) > f(i)$ for all $j>i$ in $\{1, \ldots,m\}$ and $X'_{i}= X_{f(i)}$ for 
    all $i \in \{1, \ldots,m\}$. An entry $X_{j}$  is called erased in   $(X_{1}, \ldots X_{n})$ if  $j \in   \{1, \ldots,n \}$  is not in the image of $f$. 
  \item Let $(\underline{\Cat{M}},b)$ and $(\underline{\Cat{M}'},b')$ be multi-module categories, such that 
    the string $S(\underline{\Cat{M}'},b')$ of finite tensor categories is obtained by  reducing the string $S(\underline{\Cat{M}},b)$. A  multi-balanced module functor  $\mathsf{F}: (\underline{\Cat{M}},b) \rightarrow (\underline{\Cat{M}'},b')$
    is a functor $\mathsf{F}:ev(\underline{\Cat{M}},b) \rightarrow ev(\underline{\Cat{M}'},b')$  that is balanced in each erased entry of $S(\underline{\Cat{M}},b)$ and is a  multi-module functor in  each other entry.
We furthermore require that at entries where  $\mathsf{F}$  is balanced, it is compatible with the bimodule category structures. That means that with each entry next to it, $\mathsf{F}$ satisfies either the diagram of Definition \ref{definition:multi-balanced}\refitem{item:multi-balanced}, if $\mathsf{F}$ is also balanced at the neighboring entry, or
 the diagram  (\ref{eq:balanced-module-diag}), if $\mathsf{F}$ if the neighboring entry is not erased.
  \item A multi-balanced module natural transformation $\eta: \mathsf{F}\rightarrow \mathsf{G}$ between multi-balanced module functors $\mathsf{F}$ and $\mathsf{G}$ is a natural transformation $\eta:\mathsf{F} \rightarrow \mathsf{G}$
    that is balanced in each erased entry in the target of $\mathsf{F}$ and $\mathsf{G}$ and a bimodule natural transformation in all other entries.
  \end{definitionlist}
\end{definition}
\begin{example}
  \begin{examplelist}
  \item Every multi-module functor and every multi-balanced functor is also a multi-balanced module functor. 
  \item For three composable bimodule categories, the functor
    $ (\mathsf{B}_{\Cat{K},\Cat{N}} \times 1): (\Cat{K} \times \Cat{N}) \times \Cat{M} \rightarrow (\Cat{K} \Box \Cat{N}) \times \Cat{M}$ is a multi-balanced module functor.
  \end{examplelist}
\end{example}
It follows directly from the definitions, that if a string of multi-module categories  $(\underline{\Cat{M}''},b'')$ is reduced 
  from a string  $(\underline{\Cat{M}'},b')$ and  $(\underline{\Cat{M}'},b')$ is reduced from  $(\underline{\Cat{M}},b)$, then the composite  $\mathsf{G}\mathsf{F}$ of multi-balanced module functors 
  $\mathsf{F}:(\underline{\Cat{M}},b) \rightarrow  (\underline{\Cat{M}'},b')$ and $\mathsf{G}: (\underline{\Cat{M}'},b') \rightarrow (\underline{\Cat{M}''},b'')$ is a multi-balanced module functor. 
We can therefore  generalize Proposition \ref{proposition:2-cat-extend-bimod-bimod} and follows.

\begin{proposition}
  \label{proposition:multi-2-cat}
  For every pair of finite tensor categories $(\Cat{C},\Cat{D})$, the following data define a 2-category $\Bm(\Cat{C},\Cat{D})$. 
  \begin{propositionlist}
  \item Objects are  multi-module categories $(\underline{\Cat{M}},b)$,  $(\underline{\Cat{M}'},b')$ from $\Cat{C}$ to $\Cat{D}$.
  \item  1-morphisms 
    between objects $(\underline{\Cat{M}},b)$ and $(\underline{\Cat{M}'},b')$ are multi-balanced module functors $\mathsf{F}: (\underline{\Cat{M}},b) \rightarrow  (\underline{\Cat{M}'},b')$ if 
    the string  $(\underline{\Cat{M}'},b')$ is reduced from $(\underline{\Cat{M}},b)$. Otherwise, the set of 1-morphisms  from  $(\underline{\Cat{M}},b)$ to $(\underline{\Cat{M}'},b')$ contains just 
the zero morphism. 
  \item  2-morphisms between multi-balanced module functors  $\mathsf{F},\mathsf{G}: (\underline{\Cat{M}},b) \rightarrow  (\underline{\Cat{M}'},b')$ are
    multi-balanced module natural transformations $\eta: \mathsf{F} \rightarrow \mathsf{G}$. 
  \item The compositions are induced by the horizontal composition  of functors and the vertical composition of natural transformations.
  \end{propositionlist}
\end{proposition}
\begin{remark}
  For all finite tensor categories $\Cat{C}$ the 2-categories $\BimodbC(\Cat{E},\Cat{D})$ from Proposition \ref{proposition:2-cat-extend-bimod-bimod} are full 2-subcategories of $\Bm(\Cat{E},\Cat{D})$ 
  whose objects  are  bimodule categories $\DMC \times \CNE$ and $\DAE$. 
\end{remark}

For a multi-module category  $(\underline{\Cat{M}},b)$ we already defined the category  $ev(\underline{\Cat{M}},b)$ that is obtained from the  Cartesian product of the elements in the string. 
Now, let  $ev_{\Box}(\underline{\Cat{M}},b)$  denote the category that is obtained by the tensor product of the bimodule categories in the string $(\underline{\Cat{M}},b)$ in the order that 
corresponds to the bracketing $b$. We call $ev_{\Box}(\underline{\Cat{M}},b)$ the tensor product of the multi-module category  $(\underline{\Cat{M}},b)$.

\begin{lemma}
  Let $(\underline{\Cat{M}},b)$ be a multi-module category from $\Cat{C}$ to $\Cat{D}$.
  Then the tensor product 
  $ev_{\Box}(\underline{\Cat{M}},b)$ of  $(\underline{\Cat{M}},b)$ is a $(\Cat{D},\Cat{C})$-bimodule category and it is 
  equipped  with
  \begin{enumerate}
  \item  a multi-balanced $(\Cat{D},\Cat{C})$-bimodule functor $\mathsf{B}_{\underline{\Cat{M}}}: (\underline{\Cat{M}},b) \rightarrow ev_{\Box}(\underline{\Cat{M}},b)$,
  \item for every bimodule category $\DAC$ with a functor 
    \begin{equation}
      \label{eq:Psi-multi-bal}
      \Psi_{\underline{\Cat{M}}}: \Bm((\underline{\Cat{M}},b), \DAC) \rightarrow  \Funl{\Cat{D},\Cat{C}}{ev_{\Box}(\underline{\Cat{M}},b), \DAC},
    \end{equation}
  \item  an adjoint equivalence between 
    the functor $\Psi_{\underline{\Cat{M}}}$ and the functor 
    \begin{equation}
      \label{eq:univer-prop-Box-muli}
      \begin{split}
        \Phi_{\underline{\Cat{M}}}:  \Funl{\Cat{D},\Cat{C}}{ev_{\Box}(\underline{\Cat{M}},b), \DAC} &\rightarrow  \Bm((\underline{\Cat{M}},b), \DAC) \\
        \mathsf{G} &\mapsto \mathsf{G} \circ B_{\underline{\Cat{M}}}. 
      \end{split}
    \end{equation}
  \end{enumerate}  
\end{lemma}
\begin{proof}
  It follows by repeated use of Proposition \ref{proposition:tensor-again-module}, starting with the inner most bracketing of $(\underline{\Cat{M}},b)$, that 
$ev_{\Box}(\underline{\Cat{M}},b)$ is a $(\Cat{D}, \Cat{C})$-bimodule category. 
  The functor  $\mathsf{B}_{\underline{\Cat{M}}}: (\underline{\Cat{M}},b) \rightarrow ev_{\Box}(\underline{\Cat{M}},b)$ is defined iteratively as indicated in equation (\ref{eq:34}) for a string of three bimodule categories. 
  It is shown in  Proposition \ref{proposition:tensor-again-module} that $\mathsf{B}: \DMC \times \CNE \rightarrow \DMC \Box \CNE$ is a multi-balanced module functor and hence $\mathsf{B}_{\underline{\Cat{M}}}$ is a multi-balanced module functor 
  as it is the composition of multi-balanced module functors.
Hence the first part is proven. To show the second statement, let $\mathsf{F}: (\underline{\Cat{M}},b) \rightarrow  \DAC  $ be a multi-balanced module functor from the multi-module category  $(\underline{M},b)=(\Cat{M}_{1}, \ldots, (\Cat{M}_{i}, \Cat{M}_{i+1}), \ldots \Cat{M}_{n})$ to a bimodule category $\DAC$. 
Assume that $\Cat{M}_{i}$ and $\Cat{M}_{i+1}$ are $\Cat{E}$- left, respectively right, module categories. 
$\mathsf{F}$ is clearly a $\Cat{E}$-balanced bimodule functor and hence induces a bimodule functor 
$\mathsf{F}_{1}: (\Cat{M}_{1}, \ldots, \Cat{M}_{i-1}, \Cat{M}_{i} \Box \Cat{M}_{i+1}, \ldots, \Cat{M}_{n}) \rightarrow \DAC$, by Proposition \ref{proposition:Box-2-functor-bimod}.  It is straightforward to see that $\mathsf{F}$ is again 
a multi-balanced module functor and we continue iteratively to obtain a bimodule functor $\widehat{\mathsf{F}}: ev_{\Box}(\underline{\Cat{M}},b) \rightarrow \DAC$. Defined in the analogous way for multi-balanced module natural transformations, this yields the functor  $\Psi_{\underline{\Cat{M}}}$.   Furthermore, it is clear that by construction, $\Psi_{\underline{\Cat{M}}}$ and  $ \Phi_{\underline{\Cat{M}}}$ form an adjoint equivalence since both functors are obtained by composing the  corresponding functors 
  from the constituents of the string $ (\underline{\Cat{M}},b)$  which from adjoint equivalences by Proposition \ref{proposition:tensor-again-module}.
\end{proof}
By using this lemma and by  repeated use of the 2-functor $\widehat{(-)}$ from Proposition \ref{proposition:Box-2-functor-bimod}, we can extend the tensor product to a 2-functor as follows. 
\begin{proposition}
  \label{proposition:widetlide-2-functor}
  For all pairs of finite tensor categories $\Cat{C}$ and $\Cat{D}$, the tensor product defines a  2-functor 
  \begin{equation}
    \label{eq:widetilde-2-fun}
    \widehat{(-)}: \Bm(\Cat{C},\Cat{D}) \rightarrow \BimCat(\Cat{C},\Cat{D}).
  \end{equation}
\end{proposition}

We are going to apply this 2-functor to diagrams of (horizontally and vertically) composable 2-morphisms. Such diagrams are called pasting diagrams and are defined with more precision in \cite{Benabou}, see also \cite{KS}.

\begin{corollary}
  \label{corollary:pasting-diag-2-fun}
  For every pasting diagram $D$ in $\BimCat^{multi}$, the 2-functor $\widehat{(-)}$ yields a pasting diagram $\widehat{D}$ with the same underlying graph in which 
 all 1-morphisms $\mathsf{F}$ are replaced by $\widehat{\mathsf{F}}$ and   all 2-morphisms $\rho$ are replaced by a composite of  $\widehat{\rho}$ with  coherence morphisms of the 2-functor $\widehat{(-)}$.
If two pasting diagrams $D$, $D'$ in $\BimCat^{multi}$ with the same 1-morphisms one the outer arrows evaluate to the same 2-morphism, then also $\widehat{D}$ and $\widehat{D'}$ evaluate to the 
same 2-morphisms. 
\end{corollary}
\begin{proof}
These statements hold for general 2-functors. 
 Assume that  $\mathsf{H}: \Cat{B}\rightarrow \Cat{R}$
is a strict 2-functor between strict 2-categories. Then it is clear that 
$\mathsf{H}$ applied to a  pasting diagrams in $\Cat{B}$ yields a pasting diagram in $\Cat{R}$. 
By the  strictification result for general 2-functors, see e.g. \cite[Chapter 2]{Gurski}, 
any 2-functor   $\mathsf{H}: \Cat{B}\rightarrow \Cat{R}$  between (not necessarily strict) 
bicategories  applied to a  pasting diagrams in $\Cat{B}$ yields a pasting diagram in $\Cat{R}$. 
The last statement follows directly for strict 2-functors and hence again for general 2-functors as wells.  
\end{proof}
Note however, that if $D$ is a commutative diagram of 1-morphisms in $\BimCat^{multi}$, the corresponding diagram $\widehat{D}$ commutes in general only up to natural 2-isomorphisms, that is built from the 
coherence structure of the 2-functor  $\widehat{(-)}$.

We now consider structures in the collection of the bicategories $\Bm(\Cat{C},\Cat{D})$ for different $\Cat{C}$,$\Cat{D}$. 
These structures are the main tool in the construction of the tricategory of bimodule categories. 
\begin{proposition}
  \label{proposition:Bimcatm-almost-3-cat}
  The family of  bicategories $\Bm(\Cat{C},\Cat{D})$ for finite tensor categories $\Cat{C}$ and $\Cat{D}$ is equipped with  the following additional structures.
  \begin{propositionlist}
  \item The Cartesian product of module categories defines  2-functors 
    \begin{equation}
      \label{eq:times-prod-m}
      \tm: \Bm(\Cat{D},\Cat{E}) \times \Bm(\Cat{C},\Cat{D}) \rr \Bm(\Cat{C},\Cat{E}).
    \end{equation}
  \item The tensor  product of module categories defines  2-functors 
    \begin{equation}
      \label{eq:Box-prod-m}
      \Box: \Bm(\Cat{D},\Cat{E}) \times \Bm(\Cat{C},\Cat{D}) \rr \Bm(\Cat{C},\Cat{E}).
    \end{equation}
  \item \label{item:B-ps-nat} The universal balanced functors in the definition of a tensor product of module categories yield a pseudo-natural transformation $\mathsf{B}: \tm \rightarrow \Box $
  \item \label{item:units-m} The canonical bimodule category $\CCC$ defines the (strict) unit 2-functors $I_{\Cat{C}}:I \rr \Bm(\Cat{C},\Cat{C})$, where $I$ denotes the 
    unit 2-category. 
  \item \label{item:associator-m} For four  finite tensor categories  $\Cat{C}$, $\Cat{D}$, $\Cat{E}$, $\Cat{F}$ there is an adjoint equivalence, where we abbreviated $\Bm$ with $\Bms$
    \begin{equation}
      \label{eq:adj-equiv-tm}
      \begin{tikzcd}[column sep=huge]
        \Bms(\Cat{E}, \Cat{F}) \times \Bms(\Cat{D}, \Cat{E}) \times \Bms(\Cat{C},\Cat{D}) \ar{r}[name=A]{(\tm) \times 1}  \ar{d}{1 \times \tm}  &\Bms(\Cat{D}, \Cat{F}) \times \Bms(\Cat{C}, \Cat{D}) \ar{d}{\tm}  \\
        \Bms(\Cat{E}, \Cat{F}) \times \Bms(\Cat{C}, \Cat{E}) \ar{r}[name=B]{\tm} & \Bms(\Cat{C}, \Cat{F}),
        \arrow[shorten <= 7pt, shorten >=7pt, Rightarrow,to path=(A) -- (B)\tikztonodes]{}{\alpha}
      \end{tikzcd}
    \end{equation}
    more precisely,  $\alpha: \tm \circ (\tm \times 1) \rightarrow \tm \circ ( 1 \times \tm)$ is a pseudo-natural transformation and there exists a pseudo-natural transformation 
$\alpha^{-}: \tm \circ ( 1 \times \tm) \rightarrow  \tm \circ (\tm \times 1) $, such that $\alpha$ and  $\alpha^{-}$ form an adjoint equivalence. 
  \item \label{item:unit-2-cells-m}For finite tensor categories $\Cat{C}$, $\Cat{D}$ there   are  pseudo-natural transformations
    \begin{equation}
      \label{eq:unitsleft-m}
      \begin{tikzcd}[column sep=scriptsize]
        {}  &\Bm(\Cat{D}, \Cat{D}) \times \Bm(\Cat{C},\Cat{D})  \ar[shorten <= 7pt, shorten >=7pt,Rightarrow]{d}{\act}
        \ar{dr}{\tm} & \\ 
        \Bm(\Cat{C},\Cat{D})  \ar{rr}[below, name=B]{1}   \ar{ur}{I_{\Cat{D}} \times 1} & {}  & \Bm(\Cat{C},\Cat{D})
      \end{tikzcd}
    \end{equation}
    and 
    \begin{equation}
      \label{eq:unitsright-m}
      \begin{tikzcd}[column sep=scriptsize]
        {}   &\Bm(\Cat{C}, \Cat{D}) \times \Bm(\Cat{C},\Cat{C})  \ar{dr}{\tm}  \ar[shorten <= 7pt, shorten >=7pt,Rightarrow]{d}{\ract}& \\
        \Bm(\Cat{C},\Cat{D})  \ar{rr}[below, name=B]{1} \ar{ur}{1 \times I_{\Cat{C}} }   & {}& \Bm(\Cat{C},\Cat{D}).
      \end{tikzcd}
    \end{equation}
  \item \label{item:mu-m} For all bimodule categories $\DMC$ and $\CNE$, the balancing constraint of $\mathsf{B}$ defines  an invertible modification $\beta $   with components
    \begin{equation}
      \label{eq:modific-mu-m}
      \begin{tikzcd}
        (\Cat{M} \times \Cat{C}) \times  \Cat{N}\ar{rr}[name=A]{\alpha} \ar{d}{\ract_{\Cat{M}} \times 1}  && \Cat{M} \times (\Cat{C}  \times \Cat{N}) \ar{d}{1 \times \act_{\Cat{N}}}    \\
        \Cat{M} \times \Cat{N}\ar{dr}{\mathsf{B}} \arrow[shorten <= 50pt, shorten >=50pt, Rightarrow ]{rr}{\beta}  && \Cat{M} \times \Cat{N} \ar{dl}{\mathsf{B}}  \\
        & \Cat{M} \Box \Cat{N}. &
      \end{tikzcd}
    \end{equation}
     \item \label{item:lambda-m} For all bimodule categories $\DMC$ and $\CNE$, the following  diagrams  of pseudo-natural transformations commute     
    \begin{equation}
      \label{eq:modific-lambda}
      \begin{tikzcd}
        (\Cat{C}\times \Cat{M})\times \Cat{N} \ar{dr}[below]{\alpha} \ar{rr}[name=A]{ \act_{\Cat{M}} \times 1}  && \Cat{M} \times \Cat{N} \\
        &\Cat{C}\times (\Cat{M} \times \Cat{N}) \ar{ur}[right, xshift=3pt, yshift=-2pt]{\act_{\Cat{M} \times \Cat{N}}},
      \end{tikzcd}
    \end{equation}
    \begin{equation}
      \label{eq:modific-rho-m}
      \begin{tikzcd}
        \Cat{M} \times (\Cat{N} \times \Cat{E}) \ar{dr}[below]{\alpha} \ar{rr}[name=A]{1 \times \ract_{\Cat{N}}}  && \Cat{M} \times \Cat{N} \\
        &(\Cat{M} \times \Cat{N}) \times \Cat{E}.\ar{ur}[right, xshift=2pt, yshift=-1pt]{\ract_{\Cat{M}\times \Cat{N}}}
      \end{tikzcd}
    \end{equation}
 \item \label{item:pi-m}
    For all composable bimodule categories $\Cat{K}$,$\Cat{N}$,$\Cat{M}$, $\Cat{L}$, the following diagram  of pseudo-natural transformations commutes
    \begin{equation}
      \label{eq:pi-m}
      \begin{tikzcd}
        {} & ((\Cat{K} \times \Cat{N})\times \Cat{M})\times \Cat{L} \ar{dr}{\alpha \times 1} \ar{dl}{\alpha} & \\
        (\Cat{K} \times \Cat{N}) \times( \Cat{M} \times \Cat{L}) \ar{d}{\alpha} && (\Cat{K} \times (\Cat{N} \times \Cat{M}) \times \Cat{L} \ar{d}{\alpha}\\
        \Cat{K} \times (\Cat{N} \times (\Cat{M} \times \Cat{L})) && \Cat{K} \times ((\Cat{N} \times \Cat{M} ) \times \Cat{L}).   \ar{ll}{1 \times \alpha}
      \end{tikzcd}
    \end{equation}
 \end{propositionlist}
  The following axioms are satisfied, where we denoted $\Cat{M} \times \Cat{N}$ by $\Cat{M}\Cat{N}$ for better legibility. 
 \begin{equation}
     \label{eq:Axiom2-alg-tricat-mulit}
  \begin{tikzpicture}[scale=0.9]
       \tikzstyle{every node}=[font=\small]
        \node (a) at (0,0) {$((\Cat{M}\Cat{C})\Cat{N})\Cat{K}$};
        \node (b) at (10,0) {$(\Cat{M}(\Cat{C}\Cat{N}))\Cat{K}$};
        \node (c) at (10,-2.9) {$(\Cat{M}\Cat{N})\Cat{K}$};
        \node (d) at (5,-5.5) {$\Cat{M}(\Cat{N}\Cat{K})$};
        \node (e) at (0,-2.9)  {$(\Cat{M}\Cat{N})\Cat{K}$};
         \node (f2) at (2.3,-4.5)  {$ \Cat{M}(\Cat{N}\Cat{K})$};
         \node (h2) at (7.7,-4.5) {$\Cat{M}(\Cat{N}\Cat{K})$};
    \node (d2) at (5,-7.4) {$\Cat{M} \Box (\Cat{N}\Cat{K})$};
          \node (f) at (2.3,-1.6)  {$(\Cat{M}\Cat{C})(\Cat{N}\Cat{K})$};
        \node (g) at (5,-2.6)  {$\Cat{M}(\Cat{C}(\Cat{N}\Cat{K}))$};
        \node (h) at (7.7,-1.6) {$\Cat{M}((\Cat{C}\Cat{N})\Cat{K})$};
   \node (i) at (2,-9.85) {$((\Cat{M}\Cat{C})\Cat{N})\Cat{K}$};
        \node (j) at (8,-9.85) {$(\Cat{M}(\Cat{C}\Cat{N}))\Cat{K}$};
        \node (k) at (8,-11.7) {$(\Cat{M}\Cat{N})\Cat{K}$};
        \node (l) at (5,-14.1) {$\Cat{M} \Box (\Cat{N}\Cat{K})$};
        \node (m) at (2,-11.7)  {$(\Cat{M}\Cat{N})\Cat{K}$};
 \node (f2u) at (3.2,-12.8)  {$ \Cat{M}(\Cat{N}\Cat{K})$};
         \node (h2u) at (6.8,-12.8) {$\Cat{M}  (\Cat{N}\Cat{K})$};
 \node (z) at (5,-9.85)  {};

%oberes Bild Vertices mit 3-morphismsn
 \node (a1) at (5,-1.2) {$\circlearrowleft $};
 \node (a2) at (6.4,-4.1) {$\circlearrowleft$};
 \node (a3) at (3.6,-4.1) {$\Rightarrow  \beta $};
       \node (a4) at (1.1,-2.3) {$\circlearrowleft$};
  \node (a5) at (8.9,-2.3) {$\circlearrowleft$};

%unteres Bild Vertices mit 3-morphismsn
  \node (b2) at (5,-11.5) {$\Rightarrow \beta$};

% %oberes Bild Kanten
         \path[->,font=\scriptsize,>=angle 90]
      
        (a) edge node[auto] {$\alpha 1$}  (b)
        (b) edge node[auto]  {$(1\act)1$}(c)
        (c) edge node[below]  {$\alpha$}(h2)
        (e) edge node[below] {$\alpha$}(f2)
        (a) edge node[left] {$(\ract1)1$}(e)
        (a) edge node[auto] {$\alpha$}(f)
        (f) edge node[auto] {$\alpha$}(g)
       
        (g) edge node[auto] {$1\act$}(d)
        (h) edge node[above] {$1\alpha$}(g)
       
        (b) edge node[above] {$\alpha$}(h)
           (d) edge node[right] {$\mathsf{B}$}(d2)  
   (f) edge node[left] {$\ract (11)$}(f2)
  (h) edge node[right] {$1(\act 1)$}(h2)
 (f2) edge node[above] {$\mathsf{B}$}(d2)
 (h2) edge node[above] {$\mathsf{B}$}(d2)
%unteres Bild Kanten 
         (i) edge node[auto] {$\alpha1$}  (j)
        (j) edge node[auto]  {$(1\act)1$}(k)
        (k) edge node[auto]  {$\alpha$}(h2u)
        (m) edge node[auto] {$\alpha$}(f2u)
        (i) edge node[left] {$(\ract 1)1$}(m)
       (f2u) edge node[above] {$\mathsf{B}$}(l)
        (h2u) edge node[above] {$\mathsf{B}$}(l);

   \draw[double,double equal sign distance,-,shorten <= 14pt, shorten >=16pt] (d2) to node[below] {} (z); 
      \end{tikzpicture}
  \end{equation}
 \begin{equation}
     \label{eq:Axiom3-alg-tricat-mulit}
  \begin{tikzpicture}[scale=0.9]
       \tikzstyle{every node}=[font=\small]
%oberes Bild Vertices aussen
        \node (a) at (0,0){$\Cat{M}((\Cat{N} \Cat{D})\Cat{K})$};
        \node (b) at (10,0)  {$\Cat{M}(\Cat{N} (\Cat{D}\Cat{K}))$};
        \node (c) at (10,-2.9) {$\Cat{M}(\Cat{N}\Cat{K})$};
        \node (d) at (5,-5.5) {$(\Cat{M}\Cat{N})\Cat{K}$};
        \node (e) at (0,-2.9)   {$\Cat{M}(\Cat{N}\Cat{K})$};
         \node (f2) at (2.3,-4.5)   {$(\Cat{M}\Cat{N})\Cat{K}$};
         \node (h2) at (7.7,-4.5)  {$(\Cat{M}\Cat{N})\Cat{K}$};
    \node (d2) at (5,-7.4) {$(\Cat{M}\Cat{N})\Box \Cat{K}$};
%oberes Bild vertices innen
          \node (f) at (2.3,-1.6)   {$(\Cat{M}(\Cat{N}\Cat{D}))\Cat{K}$};
        \node (g) at (5,-2.6)   {$((\Cat{M}\Cat{N})\Cat{D})\Cat{K}$};
        \node (h) at (7.7,-1.6) {$(\Cat{M}\Cat{N})(\Cat{D}\Cat{K})$};
 
% %unteres Bild vertices aussen
%         %  \node (i) at (0,-7.5) {$((\Cat{M}\Cat{C})\Cat{N})\Cat{K}$};
%         % \node (j) at (10,-7.5) {$(\Cat{M}(\Cat{C}\Cat{N}))\Cat{K}$};
%         % \node (k) at (9.3,-11.5) {$(\Cat{M}\Cat{N})\Cat{K}$};
%         % \node (l) at (5,-13) {$\Cat{M}(\Cat{N}\Cat{K})$};
%         % \node (m) at (0.7,-11.5)  {$(\Cat{M}\Cat{N})\Cat{K}$};
%scaled unteres Bild
   \node (i) at (2,-9.85) {$\Cat{M}((\Cat{N} \Cat{D})\Cat{K})$};
        \node (j) at (8,-9.85) {$\Cat{M}(\Cat{N} (\Cat{D}\Cat{K}))$};
        \node (k) at (8,-11.7)  {$\Cat{M}(\Cat{N}\Cat{K})$};
        \node (l) at (5,-14.1)  {$(\Cat{M}\Cat{N})\Box \Cat{K}$};
        \node (m) at (2,-11.7)    {$\Cat{M}(\Cat{N}\Cat{K})$};
 \node (f2u) at (3.2,-12.8)  {$ (\Cat{M}\Cat{N})\Cat{K}$};
         \node (h2u) at (6.8,-12.8) {$(\Cat{M}\Cat{N}) \Cat{K}$};
   % \node (d2) at (5,-7.4) {$\Cat{M} \Box (\Cat{N}\Cat{K})$};         
 \node (z) at (5,-9.85)  {};

%oberes Bild Vertices mit 3-morphismsn
 \node (a1) at (5,-1.2) {$\circlearrowleft $};
 \node (a2) at (6.4,-4.1) {$\Leftarrow \beta$};
 \node (a3) at (3.6,-4.1) {$\circlearrowleft$}; 
       \node (a4) at (1.1,-2.3) {$\circlearrowleft$};
  \node (a5) at (8.9,-2.3) {$\circlearrowleft$};

%unteres Bild Vertices mit 3-morphismsn
  \node (b2) at (5,-11.5) {$\Leftarrow \beta$};

% %oberes Bild Kanten
         \path[->,font=\scriptsize,>=angle 90]
      
        (a) edge node[auto] {$1 \alpha $}  (b)
        (b) edge node[auto]  {$1(1\act)$}(c)
        (h2) edge node[below]  {$\alpha$}(c)
        (f2) edge node[below] {$\alpha$}(e)
        (a) edge node[left] {$1(\ract 1)$}(e)
        (f) edge node[auto] {$\alpha$}(a)
        (g) edge node[auto] {$\alpha$}(f)
       
        (g) edge node[auto] {$\ract 1$}(d)
        (g) edge node[above] {$\alpha 1$}(h)
       
        (h) edge node[above] {$\alpha$}(b)
           (d) edge node[right] {$\mathsf{B}$}(d2)  
   (f) edge node[left] {$(1 \ract )1$}(f2)
  (h) edge node[right] {$(11) \act $}(h2)
 (f2) edge node[above] {$\mathsf{B}$}(d2)
 (h2) edge node[above] {$\mathsf{B}$}(d2)
%unteres Bild Kanten 
        
 (i) edge node[auto] {$1\alpha$}  (j)
        (j) edge node[auto]  {$1(1\act)$}(k)
        (h2u) edge node[auto]  {$\alpha$}(k)
        (f2u) edge node[auto] {$\alpha$}(m)
        (i) edge node[left] {$1(\ract 1)$}(m)
       (f2u) edge node[above] {$\mathsf{B}$}(l)
        (h2u) edge node[above] {$\mathsf{B}$}(l);

   \draw[double,double equal sign distance,-,shorten <= 14pt, shorten >=16pt] (d2) to node[below] {} (z); 
      \end{tikzpicture}
  \end{equation}
\end{proposition}
\begin{proof}
  The first part follows from the definitions, the second part is shown in Proposition \ref{proposition:widetlide-2-functor}. For the third part 
we first show that $\mathsf{B}$ defines a pseudo-natural transformation between the 2-functors $\times, \Box: \BimCat(\Cat{D},\Cat{E}) \times \BimCat(\Cat{C}, \Cat{D}) \rightarrow \BimCat(\Cat{C}, \Cat{E})$.
By Proposition \ref{proposition:Box-2-functor-bimod}, the bimodule natural isomorphisms $\varphi_{\mathsf{F} \times \mathsf{G}}: \mathsf{B}(\mathsf{F} \times \mathsf{G}) \rightarrow (\mathsf{F} \Box \mathsf{G}) \mathsf{B}$
for bimodule functors $\mathsf{F} \times \mathsf{G}: \END \times \DMC \rightarrow \ENpD \times \DMpC$ are compatible with the composition of bimodule functors. The compatibility of the natural isomorphisms $\varphi_{\mathsf{F} \times \mathsf{G}}$ with bimodule natural transformations follows also directly from the 2-functorial properties of the tensor product. 
Since we used only the 2-functoriality of the tensor product, this argument extends 
 first to bimodule categories and then by repeated application also to multi-module categories. This shows the third part. 
 Parts \refitem{item:units-m} and \refitem{item:associator-m} are clear. The 
properties of a pseudo-natural transformation for the module action in part \refitem{item:unit-2-cells-m} follow from the compatibility conditions  between  module actions and  bimodule functors and bimodule natural transformations.

Parts \refitem{item:mu-m}- \refitem{item:pi-m} are clear from the definitions. 
The axioms in equations (\ref{eq:Axiom2-alg-tricat-mulit}) and (\ref{eq:Axiom3-alg-tricat-mulit})  follow directly from the properties of the Cartesian product of categories.  
\end{proof}

\subsection{The tricategory of bimodule categories}

We finally show that bimodule categories define an algebraic tricategory according to Definition \ref{definition:tricategory}, that is a slight modification of  \cite[Def. 3.1.2]{Gurski}.
\begin{theorem}
  \label{theorem:Bimcat-3-cat}
  Finite tensor categories, finite bimodule categories, right exact 
bimodule functors and bimodule natural transformations from  an algebraic tricategory $\BimCat$ in the sense of Definition \ref{definition:tricategory}.
The composition $\Box$ is given by the tensor product of bimodule categories, the horizontal composition $\circ$ is given by the composition of functors and the vertical composition is defined by the 
vertical composition of natural transformations. 
\end{theorem}
By unpacking  Definition \ref{definition:tricategory}, one finds that the claim of the theorem follows from the following. 
\begin{enumerate}
\item  \label{item:2-fun-Bimcat}$\BimCat(\Cat{C},\Cat{D})$ is a strict 2-category  with the composition  of functors as horizontal composition $\circ$ and  the composition  of natural transformations
  as vertical composition $\cdot$. 
\item For any three finite tensor categories $\Cat{C}$,$\Cat{D}$,$\Cat{E}$, 
the tensor product of module categories defines  a  2-functors 
  \begin{equation}
    \label{eq:Box-prod-BimCat}
    \Box: \BimCat(\Cat{D},\Cat{E}) \times \BimCat(\Cat{C},\Cat{D}) \rr \BimCat(\Cat{C},\Cat{E}).
  \end{equation}
\item \label{item:units-bimod}The  bimodule category $\CCC$ defines for each object $\Cat{C}$ a (strict) unit 2-functor $I_{\Cat{C}}:I \rr \BimCat(\Cat{C},\Cat{C})$, where $I$ denotes the 
  unit 2-category with one object $1$, one 1-morphism $1_{1}$ and one 2-morphism $1_{1_{1}}$.
\item \label{item:associator-bimod}For any four objects $\Cat{C}$, $\Cat{D}$, $\Cat{E}$, $\Cat{F}$ there is an adjoint equivalence
  $a: \Box (\Box \times 1) \Rightarrow \Box( 1 \times \Box)$, called  associator in the following. More precisely,  $a$ consists of a  pseudo-natural transformation
  \begin{equation}
    \label{eq:adj-equiv-asso}
    \begin{xy}
      \xymatrix{
        \BimCat(\Cat{E}, \Cat{F}) \times \BimCat(\Cat{D}, \Cat{E}) \times \BimCat(\Cat{C},\Cat{D}) \ar[rr]^{\Box \times 1}  \ar[d]_{1 \times \Box} \rrtwocell<\omit>{<4>a} &&\BimCat(\Cat{D}, \Cat{F}) \times \BimCat(\Cat{C}, \Cat{D}) \ar[d]^{\Box}  \\
        \BimCat(\Cat{E}, \Cat{F}) \times \BimCat(\Cat{C}, \Cat{E}) \ar[rr]_{\Box} && \BimCat(\Cat{C}, \Cat{F}),
      }
    \end{xy}
  \end{equation}
  and there is a pseudo-natural transformation  $a^{-}:\Box( 1 \times \Box)\rightarrow \Box (\Box \times 1) $, such that $a$ and $a^{-}$ form an adjoint equivalence, see Definition \ref{definition:adj-equiv}.
\item \label{item:unit-2-cells}For any two objects $\Cat{C}$, $\Cat{D}$, there are adjoint equivalences $l: \Box(I_{\Cat{D}} \times 1)\Rightarrow 1$ and $r: \Box (1\times I_{\Cat{C}}) \Rightarrow 1$, called the unit 2-morphisms, 
  \begin{equation}
    \label{eq:unitsleft-bimod}
    \begin{tikzcd}
      {}         &\BimCat(\Cat{D}, \Cat{D}) \times \BimCat(\Cat{C},\Cat{D})  \ar{dr}{\Box}  \ar[shorten <= 7pt, shorten >=7pt,Rightarrow]{d}{l} \\
      \BimCat(\Cat{C},\Cat{D})  \ar{rr}[below]{1} \ar{ur}{I_{\Cat{D}} \times 1}   & {}& \BimCat(\Cat{C},\Cat{D})
    \end{tikzcd}
  \end{equation}
  and 
  \begin{equation}
    \label{eq:unitsright-bimod}
    \begin{tikzcd}
      {}   &\BimCat(\Cat{C}, \Cat{D}) \times \BimCat(\Cat{C},\Cat{C})  \ar{dr}{\Box}  \ar[shorten <= 7pt, shorten >=7pt,Rightarrow]{d}{r}& \\
      \BimCat(\Cat{C},\Cat{D})  \ar{rr}[below, name=B]{1} \ar{ur}{1 \times I_{\Cat{C}} }   & {}& \BimCat(\Cat{C},\Cat{D}).
    \end{tikzcd}
  \end{equation}
  By definition of an adjoint equivalence, $l$ and $r$ are pseudo-natural transformations. Furthermore there are  corresponding pseudo-natural transformations $l^{-}: 1 \Rightarrow \Box(I_{\Cat{D}} \times 1)$ and
  $r^{-}:1\Rightarrow \Box (1\times I_{\Cat{C}})$.
\item \label{item:mu-bimod} For all bimodule categories $\DMC$ and $\CNE$ there is an invertible modification $\mu$ with component 3-morphisms
  \begin{equation}
    \label{eq:modific-mu}
    \begin{tikzcd}
      (\Cat{M} \Box \Cat{C}) \Box \Cat{N}\ar{rr}[name=A]{a} \ar{dr}[below,yshift=-3pt]{r_{\Cat{M}} \Box 1}  && \Cat{M} \Box (\Cat{C}  \Box \Cat{N}) \ar{dl}{1 \Box l_{\Cat{N}}}    \\
      &\Cat{M} \Box \Cat{N}. \arrow[shorten <= 5pt, shorten >=5pt, Leftarrow,to path=-- (A) \tikztonodes]{}{\mu} &   
    \end{tikzcd}
  \end{equation}
  
\item \label{item:lambda-bimod}  For all bimodule categories $\CMD$ and $\DNE$ there is an invertible modification $\lambda$ with component 3-morphisms
  \begin{equation}
    \label{eq:modific-lambda-bimod}
    \begin{tikzcd}
      (\Cat{C}\Box \Cat{M})\Box \Cat{N} \ar{dr}[below]{a} \ar{rr}[name=A]{l_{\Cat{M}} \Box 1}  && \Cat{M} \Box \Cat{N} \\
      &\Cat{C}\Box (\Cat{M} \Box \Cat{N}). \ar{ur}[below,xshift=9pt]{l_{\Cat{M} \Box \Cat{N}}}
      \arrow[shorten <= 5pt, shorten >=5pt, Leftarrow,to path=-- (A) \tikztonodes]{}{\lambda} &   
    \end{tikzcd}
  \end{equation}
\item \label{item:rho-bimod}  For all bimodule categories $\DMC$ and $\CNE$ there is an invertible modification $\rho$ with component 3-morphisms
  \begin{equation}
    \label{eq:modific-rho}
    \begin{tikzcd}
      \Cat{M} \Box (\Cat{N} \Box \Cat{E}) \ar{dr}[below]{a} \ar{rr}[name=A]{1 \Box r_{\Cat{N}}}  && \Cat{M} \Box \Cat{N} \\
      &(\Cat{M} \Box \Cat{N}) \Box \Cat{E}\ar{ur}[below,xshift=9pt]{r_{\Cat{M}\Box \Cat{N}}}
      \arrow[shorten <= 5pt, shorten >=5pt, Leftarrow,to path=-- (A) \tikztonodes]{}{\rho} &   
    \end{tikzcd}
  \end{equation}
  
\item \label{item:pi-bimod}
  For all composable bimodule categories $\Cat{K}$,$\Cat{N}$,$\Cat{M}$ and $\Cat{L}$, there is an invertible modification $\pi$ with component 3-morphisms
  \begin{equation}
    \label{eq:pi}
    \begin{tikzcd}
      {} & ((\Cat{K} \Box \Cat{N})\Box \Cat{M})\Box \Cat{L} \ar{dr}{a \Box 1} \ar{dl}{a} & \\
      (\Cat{K} \Box \Cat{N}) \Box \Cat{M} \Box \Cat{L}) \ar{d}{a} && (\Cat{K} \Box (\Cat{N} \Box \Cat{M}) \Box \Cat{L} \ar{d}{a}\\
      \Cat{K} \Box (\Cat{N} \Box (\Cat{M} \Box \Cat{L})) && \Cat{K} \Box ((\Cat{N} \Box \Cat{M} ) \Box \Cat{L})   \ar{ll}{1 \Box a}
      \ar[shorten <= 50pt, shorten >=50pt, Rightarrow]{llu}{\pi}
    \end{tikzcd}
  \end{equation}
 
\item The following three axioms are satisfied. In the first axiom, the unmarked isomorphisms are isomorphisms from the naturality of $a$.

 \begin{equation}
      \label{eq:Axiom1-alg-tricat-bimod}
      \begin{tikzpicture}
        \tikzstyle{every node}=[font=\small]
%oberes Bild Vertices aussen
        \node (a) at (0,-2) {$((\Cat{K}(\Cat{L}\Cat{M})\Cat{N})\Cat{R}$};
        \node (b) at (2.5,-1) {$(\Cat{K}((\Cat{L}\Cat{M})\Cat{N}))\Cat{R}$};
        \node (c) at (5,0) {$(\Cat{K}(\Cat{L}(\Cat{M}\Cat{N})))\Cat{R}$};
        \node (d) at (7.5,-1) {$\Cat{K}((\Cat{L}(\Cat{M}\Cat{N}))\Cat{R})$};
        \node (e) at (10,-2) {$\Cat{K}(\Cat{L}((\Cat{M}\Cat{N})\Cat{R}))$};
        \node (f) at (10,-4) {$\Cat{K}(\Cat{L}(\Cat{M}(\Cat{N}\Cat{R})))$};
        \node (g) at (6.66,-6) {$(\Cat{K}\Cat{L})(\Cat{M}(\Cat{N}\Cat{R}))$};
        \node (h) at (3.33,-6) {$((\Cat{K}\Cat{L})\Cat{M})(\Cat{N}\Cat{R})$};
        \node (i) at (0,-4) {$(((\Cat{K}\Cat{L})\Cat{M})\Cat{N})\Cat{R}$};
%oberes Bild vertices innen
        \node (k) at (3,-2.7)  {$((\Cat{K}\Cat{L})(\Cat{M}\Cat{N}))\Cat{R}$};
        \node (l) at (6,-3.3) {$(\Cat{K}\Cat{L})((\Cat{M}\Cat{N})\Cat{R})$};
              \node (z) at (5,-6){};

%unteres Bild vertices aussen
        \node (m) at (0,-10) {$((\Cat{K}(\Cat{L}\Cat{M})\Cat{N})\Cat{R}$};
        \node (n) at (2.5,-9)  {$(\Cat{K}((\Cat{L}\Cat{M})\Cat{N}))\Cat{R}$};
        \node (o) at (5,-8) {$(\Cat{K}(\Cat{L}(\Cat{M}\Cat{N})))\Cat{R}$};
        \node (p) at (7.5,-9) {$\Cat{K}((\Cat{L}(\Cat{M}\Cat{N}))\Cat{R})$};
        \node (q) at (10,-10){$\Cat{K}(\Cat{L}((\Cat{M}\Cat{N})\Cat{R}))$};
        \node (r) at (10,-12){$\Cat{K}(\Cat{L}(\Cat{M}(\Cat{N}\Cat{R})))$};
        \node (s) at (6.66,-14)  {$(\Cat{K}\Cat{L})(\Cat{M}(\Cat{N}\Cat{R}))$};
        \node (t) at (3.33,-14) {$((\Cat{K}\Cat{L})\Cat{M})(\Cat{N}\Cat{R})$};
        \node (u) at (0,-12){$(((\Cat{K}\Cat{L})\Cat{M})\Cat{N})\Cat{R}$};
%unteres Bild vertices innen

        \node (w) at (3.33,-12)  {$(\Cat{K}(\Cat{L}\Cat{M}))(\Cat{N}\Cat{R})$};
        \node (x) at (6.66,-12) {$\Cat{K}((\Cat{L}\Cat{M})(\Cat{N}\Cat{R}))$};
        \node (y) at (5,-10) {$\Cat{K}(((\Cat{L}\Cat{M})\Cat{N})\Cat{R})$};

%oberes Bild Vertices mit 3-morphismsn
 \node (a1) at (2.4,-1.8) {$\Downarrow \pi 1$};
 \node (a2) at (6.7,-1.9) {$\Downarrow \pi$};
 \node (a3) at (8.2,-3.7) {$\simeq$};
          \node (a5) at (3.7,-4.5) {$\Downarrow \pi$};

%unteres Bild Vertices mit 3-morphismsn

 \node (b1) at (5,-9) {$\simeq $};
 \node (b2) at (3.2,-10.4) {$\Downarrow \pi$};
 \node (b3) at (7.7,-10.3) {$\Downarrow 1 \pi$};
 \node (b4) at (1.65,-12) {$\simeq$};
 \node (b5) at (5.3,-13) {$\Downarrow \pi$};

%oberes Bild Kanten
         \path[->,font=\scriptsize,>=angle 90]
      
        (a) edge node[auto] {$1a$}  (b)
        (b) edge node[auto]  {$(1a)1$}(c)
        (c) edge node[auto]  {$a$}(d)
        (d) edge node[auto] {$1a$}(e)
        (e) edge node[auto] {$1(1a)$}(f)
        (g) edge node[auto] {$a$}(f)
        (h) edge node[auto] {$a$}(g)
        (i) edge node[auto] {$a$}(h)
        (i) edge node[auto] {$(a1)1$}(a)
        (i) edge node[auto] {$a1$}(k)
        (k) edge node[right] {$a1$}(c)
        (k) edge node[auto] {$a$}(l)
        (l) edge node[auto]  {$ a $}  (e)
        (l) edge node[auto] {$(11)a$}(g)
%unteres Bild Kanten 
        (m) edge node[auto] {$a 1$}  (n)
        (n) edge node[auto]  {$(1a)1$}(o)
        (o) edge node[auto]  {$a$}(p)
        (p) edge node[auto] {$1a$}(q)
        (q) edge node[auto] {$1(1a)$}(r)
        (u) edge node[auto] {$(a1)1$}(m)
        (u) edge node[auto] {$a$}(t)
        (t) edge node[auto] {$a$}(s)
        (s) edge node[auto] {$a$}(r)
        (m) edge node[auto] {$a$}(w)
        (t) edge node[auto] {$a(11)$}(w)
        (w) edge node[auto] {$1a$}(x)
        (x) edge node[auto] {$1a$}(r)
        (n) edge node[auto]   {$ a $}  (y)
        (y) edge node[below, xshift=8.3pt, yshift=3pt] {$1(a1) $}(p)
         (y) edge node[auto] {$1a $}(x);
 \draw[double,double equal sign distance,-,shorten <= 8pt, shorten >=8pt] (z) to node[below] {} (o); 
      \end{tikzpicture}
    \end{equation}
  \begin{equation}
     \label{eq:Axiom2-alg-tricat-bimod}
  \begin{tikzpicture}[scale=0.9]
       \tikzstyle{every node}=[font=\small]
%oberes Bild Vertices aussen
        \node (a) at (0,0) {$((\Cat{M}\Cat{C})\Cat{N})\Cat{K}$};
        \node (b) at (10,0) {$(\Cat{M}(\Cat{C}\Cat{N}))\Cat{K}$};
        \node (c) at (9.3,-4) {$(\Cat{M}\Cat{N})\Cat{K}$};
        \node (d) at (5,-5.5) {$\Cat{M}(\Cat{N}\Cat{K})$};
        \node (e) at (0.7,-4)  {$(\Cat{M}\Cat{N})\Cat{K}$};
      
%oberes Bild vertices innen
  \node (f) at (2.3,-1.6)  {$(\Cat{M}\Cat{C})(\Cat{N}\Cat{K})$};
        \node (g) at (5,-2.6)  {$\Cat{M}(\Cat{C}(\Cat{N}\Cat{K}))$};
        \node (h) at (7.7,-1.6) {$\Cat{M}((\Cat{C}\Cat{N})\Cat{K})$};
 
%unteres Bild vertices aussen
        %  \node (i) at (0,-7.5) {$((\Cat{M}\Cat{C})\Cat{N})\Cat{K}$};
        % \node (j) at (10,-7.5) {$(\Cat{M}(\Cat{C}\Cat{N}))\Cat{K}$};
        % \node (k) at (9.3,-11.5) {$(\Cat{M}\Cat{N})\Cat{K}$};
        % \node (l) at (5,-13) {$\Cat{M}(\Cat{N}\Cat{K})$};
        % \node (m) at (0.7,-11.5)  {$(\Cat{M}\Cat{N})\Cat{K}$};
%scaled
   \node (i) at (2.5,-7.85) {$((\Cat{M}\Cat{C})\Cat{N})\Cat{K}$};
        \node (j) at (7.5,-7.85) {$(\Cat{M}(\Cat{C}\Cat{N}))\Cat{K}$};
        \node (k) at (7.15,-9.85) {$(\Cat{M}\Cat{N})\Cat{K}$};
        \node (l) at (5,-10.6) {$\Cat{M}(\Cat{N}\Cat{K})$};
        \node (m) at (2.85,-9.85)  {$(\Cat{M}\Cat{N})\Cat{K}$};
          \node (z) at (5,-7.85)  {};

%oberes Bild Vertices mit 3-morphismsn
 \node (a1) at (5,-1.2) {$\Downarrow \pi $};
 \node (a2) at (5.9,-3.1) {$\Rightarrow 1\lambda$};
 \node (a3) at (4.1,-3.1) {$\Leftarrow \mu$};
       \node (a4) at (2,-3) {$\simeq$};
  \node (a5) at (8,-3) {$\simeq$};

%unteres Bild Vertices mit 3-morphismsn
 \node (b1) at (4.1,-8.5) {$\Rightarrow \mu 1$};
  \node (b2) at (5,-9.5) {$\simeq $};

% %oberes Bild Kanten
         \path[->,font=\scriptsize,>=angle 90]
      
        (a) edge node[auto] {$a1$}  (b)
        (b) edge node[auto]  {$(1l)1$}(c)
        (c) edge node[above]  {$a$}(d)
        (e) edge node[auto] {$a$}(d)
        (a) edge node[left] {$(r1)1$}(e)
        (a) edge node[auto] {$a$}(f)
        (f) edge node[auto] {$a$}(g)
        (f) edge node[left] {$r(11)$}(d)
        (g) edge node[auto] {$1l$}(d)
        (h) edge node[above] {$1a$}(g)
        (h) edge node[right] {$1(l1)$}(d)
        (b) edge node[above] {$a$}(h)
%
% %unteres Bild Kanten 
         (i) edge node[auto] {$a1$}  (j)
        (j) edge node[auto]  {$(1l)1$}(k)
        (k) edge node[above]  {$a$}(l)
        (m) edge node[auto] {$a$}(l)
        (i) edge node[left] {$(r1)1$}(m)
        (j) edge node[auto] {$(1l)1$}(m);

  \draw[double,double equal sign distance,-,shorten <= 14pt, shorten >=16pt] (d) to node[below] {} (z); 
      \end{tikzpicture}
  \end{equation}

 \begin{equation}
     \label{eq:Axiom3-alg-tricat-bimod}
  \begin{tikzpicture}[scale=0.9]
       \tikzstyle{every node}=[font=\small]
%oberes Bild Vertices aussen
        \node (a) at (0,0) {$\Cat{M}((\Cat{N} \Cat{D})\Cat{K})$};
        \node (b) at (10,0) {$\Cat{M}(\Cat{N} (\Cat{D}\Cat{K}))$};
        \node (c) at (9.3,-4) {$\Cat{M}(\Cat{N}\Cat{K})$};
        \node (d) at (5,-5.5) {$(\Cat{M}\Cat{N})\Cat{K}$};
        \node (e) at (0.7,-4)  {$\Cat{M}(\Cat{N}\Cat{K})$};
      
%oberes Bild vertices innen
  \node (f) at (2.3,-1.6)  {$(\Cat{M}(\Cat{N}\Cat{D}))\Cat{K}$};
        \node (g) at (5,-2.6)  {$((\Cat{M}\Cat{N})\Cat{D})\Cat{K}$};
        \node (h) at (7.7,-1.6) {$(\Cat{M}\Cat{N})(\Cat{D}\Cat{K})$};
 
%unteres Bild vertices aussen
        %  \node (i) at (0,-7.5) {$((\Cat{M})\Cat{N})\Cat{K}$};
        % \node (j) at (10,-7.5) {$(\Cat{M}(\Cat{N}))\Cat{K}$};
        % \node (k) at (9.3,-11.5) {$(\Cat{M}\Cat{N})\Cat{K}$};
        % \node (l) at (5,-13) {$\Cat{M}(\Cat{N}\Cat{K})$};
        % \node (m) at (0.7,-11.5)  {$(\Cat{M}\Cat{N})\Cat{K}$};
%scaled unteres bild vertices
   \node (i) at (2.5,-7.85){$\Cat{M}((\Cat{N} \Cat{D})\Cat{K})$};
        \node (j) at (7.5,-7.85)  {$\Cat{M}(\Cat{N} (\Cat{D}\Cat{K}))$};
        \node (k) at (7.15,-9.85) {$\Cat{M}(\Cat{N}\Cat{K})$};
        \node (l) at (5,-10.6)  {$(\Cat{M}\Cat{N})\Cat{K}$};
        \node (m) at (2.85,-9.85)   {$\Cat{M}(\Cat{N}\Cat{K})$};
  \node (z) at (5,-7.85)  {};

%oberes Bild Vertices mit 3-morphismsn
 \node (a1) at (5,-1.2) {$\Downarrow \pi $};
 \node (a2) at (5.9,-3.1) {$\Rightarrow \mu$};
 \node (a3) at (4.1,-3.1) {$\Leftarrow \rho 1$};
       \node (a4) at (2,-3) {$\simeq$};
  \node (a5) at (8,-3) {$\simeq$};

%unteres Bild Vertices mit 3-morphismsn
 \node (b1) at (4.1,-8.5) {$\Rightarrow 1 \mu $};
  \node (b2) at (5,-9.5) {$\simeq $};

% %oberes Bild Kanten
         \path[->,font=\scriptsize,>=angle 90]
      
        (a) edge node[auto] {$1a$}  (b)
        (b) edge node[auto]  {$1(1l)$}(c)
        (d) edge node[above]  {$a$}(c)
        (d) edge node[auto] {$a$}(e)
        (a) edge node[left] {$1(r1)$}(e)
        (f) edge node[auto] {$a$}(a)
        (g) edge node[auto] {$a1$}(f)
        (f) edge node[left] {$(1r)1$}(d)
        (g) edge node[auto] {$r1$}(d)
        (g) edge node[above] {$a$}(h)
        (h) edge node[right] {$(11)l$}(d)
        (h) edge node[above] {$a$}(b)
%
% %unteres Bild Kanten 
         (i) edge node[auto] {$1a$}  (j)
        (j) edge node[auto]  {$1(1l)$}(k)
        (l) edge node[above]  {$a$}(k)
        (l) edge node[auto] {$a$}(m)
        (i) edge node[left] {$1(1r)$}(m)
        (j) edge node[auto] {$1(1r)$}(m);
      
  \draw[double,double equal sign distance,-,shorten <= 14pt, shorten >=16pt] (d) to node[below] {} (z); 
      \end{tikzpicture}
  \end{equation}
\end{enumerate}
\begin{proof}
Note first that our conventions regarding an algebraic tricategory differ slightly from the conventions in \cite[Definition 3.12]{Gurski}, see Remark \ref{remark:tricategory-convent}.
 The reason for our convention will become clear 
from the construction of the pseudo-natural transformations $l$ and $r$ in the proof. 

  The remainder of this section is concerned with the proof of   Theorem \ref{theorem:Bimcat-3-cat}. The basic idea is to apply the 2-functor  $\widehat{(-)}$ from Proposition \ref{proposition:widetlide-2-functor} 
  to the structures and axioms of $\Bm$ in Proposition \ref{proposition:Bimcatm-almost-3-cat} to obtain the corresponding structures and axioms for $\BimCat$.
  \begin{enumerate}
  \item We already remarked in Section \ref{sec:Prelim}, that  $\BimCat(\Cat{C},\Cat{D})$ is a strict 2-category.
  \item  The tensor product defines a 2-functor  $\Box: \BimCat(\Cat{D},\Cat{E}) \times \BimCat(\Cat{C},\Cat{D}) \rr \BimCat(\Cat{C},\Cat{E})$  according to  Proposition \ref{proposition:Box-2-functor-bimod}.  
  \item  The unit bimodule categories $\CCC$, the identity bimodule functor $\id_{\Cat{C}}: \CCC \rr \CCC$ and the identity bimodule natural transformation $\id_{\id_{\Cat{C}}}: \id_{\Cat{C}} \rr \id_{\Cat{C}}$ define 
    the strict 2-functor $I_{\Cat{C}}: 1 \rr \BimCat(\Cat{C},\Cat{C})$ from  \refitem{item:units-bimod}, where $1$ denotes the unit bicategory.
  \item %The associator
    We now define the structures in  \refitem{item:associator-bimod}. 
    Let $\Cat{M}$,$\Cat{N}$ and $\Cat{K}$ be composable bimodule categories. 
    The  2-functor $\widehat{(-)}$ applied to $\alpha: (\Cat{M} \times \Cat{N}) \times \Cat{K} \rightarrow \Cat{M} \times (\Cat{N} \times \Cat{K})$ from Proposition 
    \ref{proposition:Bimcatm-almost-3-cat}, 
    defines a functor 
    \begin{equation}
      \label{eq:def-a}
      a= \widehat{\alpha}: (\Cat{M} \Box \Cat{N}) \Box \Cat{K} \rightarrow \Cat{M} \Box (\Cat{N} \Box \Cat{K}).
    \end{equation}
    Since $a$ is the composite of a 2-functor with the pseudo-natural transformation $\alpha$, $a$ is also a pseudo-natural transformation. 
    Analogously, the multi-module functor  $\alpha^{-}: \Cat{M} \times (\Cat{N} \times \Cat{K}) \rightarrow (\Cat{M} \times \Cat{N}) \times \Cat{K}$ defines a 2-transformation $a^{-}$, and it
    follows from Lemma \ref{lemma:functor-presevers-adj-equiv} that $a$ and $a^{-}$ form an adjoint equivalence. 
  \item
    We construct  the  adjoint equivalence of bimodule categories  $\DDD \Box \DMC \rr \DMC$. 
    Let $\DMC$ be a bimodule category. Recall that the action $\act: \Cat{D} \times \Cat{M} \rightarrow \Cat{M}$ is a balanced bimodule functor.
 \begin{lemma}
      \label{lemma:lM-psnat}
      The   bimodule functor $l_{\Cat{M}}=\widehat{\act}: \Cat{D} \Box \Cat{M} \rightarrow \Cat{M}$  induced  by the balanced bimodule functor $\act: \Cat{D} \times \DMC \rightarrow \DMC$
      defines a pseudo-natural transformation  
      \begin{equation}
  \label{eq:unitsleft-cite}
         \begin{tikzcd}
      {}         &\BimCat(\Cat{D}, \Cat{D}) \times \BimCat(\Cat{C},\Cat{D})  \ar{dr}{\Box}  \ar[shorten <= 7pt, shorten >=7pt,Rightarrow]{d}{l} \\
      \BimCat(\Cat{C},\Cat{D})  \ar{rr}[below]{1} \ar{ur}{I_{\Cat{D}} \times 1}   & {}& \BimCat(\Cat{C},\Cat{D})
    \end{tikzcd}
  \end{equation}
      \end{lemma}
 \begin{proof}
      Let $\mathsf{F}: \DMC \rightarrow \DMpC$ be a module functor. Then the module constraint $\phi^{\mathsf{F}}$ yields the diagram 
      \begin{equation}
        \label{eq:from-funcotr-lm}
        \begin{tikzcd}
          \Cat{D} \times \Cat{M} \ar{r}[name=A]{1 \times \mathsf{F}} \ar{d}{\act} & \Cat{D} \times \Cat{M}'\ar{d}{\act} \\
          \Cat{M} \ar{r}[name=B]{\mathsf{F}} & \Cat{M}'
          \arrow[shorten <= 7pt, shorten >=7pt, Rightarrow,to path=(A) -- (B)\tikztonodes]{}{\phi^{\mathsf{F}}}.
        \end{tikzcd}
      \end{equation}
      This  defines  a bimodule natural isomorphism, where we use the abbreviation $\Cat{D} \Box \mathsf{F}= 1_{\Cat{D}} \Box \mathsf{F}$, 
      \begin{equation}
        \label{eq:lF}
        l_{\mathsf{F}}: l_{\Cat{M}'} (\Cat{D}\Box \mathsf{F}) \rightarrow \mathsf{F} l_{\Cat{M}}
      \end{equation}
      between the bimodule functors $l_{\Cat{M}'} (\Cat{D}\Box \mathsf{F})$ and $\mathsf{F} l_{\Cat{M}}$.
      We show that the isomorphisms $l_{\mathsf{F}}$ are natural in $\mathsf{F}$.
      If $\mathsf{G}: \DMC \rightarrow \DMpC$ is another bimodule  functor and $\rho: \mathsf{F} \rightarrow \mathsf{G}$ is a bimodule natural transformation, 
      we have to prove  that the following   natural transformations are equal:
 \begin{equation}
        \label{eq:15}
        \begin{tikzcd}[baseline=-0.65ex,scale=0.5, column sep=large, row sep=large]
          \Cat{D} \Box \Cat{M} \ar{d}{l_{\Cat{M}}} \ar[bend left]{r}[name=A]{\Cat{D} \Box \mathsf{F}} \ar[bend right]{r}[ name=B]{\Cat{D} \Box \mathsf{G}}   & \Cat{D} \Box \Cat{N} \ar{d}{l_{N}} \\
          \Cat{M} \ar{r}[name=C]{\mathsf{G}} & \Cat{N} 
          \arrow[shorten <= 2pt, shorten >=2pt, Rightarrow,to path=(A) -- (B)\tikztonodes]{}{\Cat{D} \Box \rho}
          \arrow[shorten <= 3pt, shorten >=3pt, Rightarrow,to path=(B) -- (C)\tikztonodes]{}{l_{\mathsf{G}}},
        \end{tikzcd}
        =
        \begin{tikzcd}[baseline=-0.65ex,scale=0.5,  column sep=large, row sep=large]
          \Cat{D} \Box \Cat{M} \ar{d}{l_{\Cat{M}}}  \ar{r}[name=C]{ \Cat{D} \Box \mathsf{F}} & \Cat{D} \Box \Cat{N} \ar{d}{l_{N}} \\
          \Cat{M}  \ar[bend left]{r}[name=A]{ \mathsf{F}} \ar[bend right]{r}[ name=B]{ \mathsf{G}}  & \Cat{N} 
          \arrow[shorten <= 2pt, shorten >=2pt, Rightarrow,to path=(A) -- (B)\tikztonodes]{}{ \rho}
          \arrow[shorten <= 3pt, shorten >=3pt, Rightarrow,to path=(C) -- (A)\tikztonodes]{}{l_{\mathsf{F}}}.
        \end{tikzcd}
      \end{equation}
      Since $\rho$ is a bimodule natural transformation, one has
 \begin{equation}
        \label{eq:16}
        \begin{tikzcd}[baseline=-0.65ex,scale=0.5, column sep=large, row sep=large]
          \Cat{D} \times \Cat{M} \ar{d}{\act_{\Cat{M}}} \ar[bend left]{r}[name=A]{\Cat{D} \times \mathsf{F}} \ar[bend right]{r}[ name=B]{\Cat{D} \times \mathsf{G}}   & \Cat{D} \times \Cat{N} \ar{d}{\act_{N}} \\
          \Cat{M} \ar{r}[name=C]{\mathsf{G}} & \Cat{N} 
          \arrow[shorten <= 2pt, shorten >=2pt, Rightarrow,to path=(A) -- (B)\tikztonodes]{}{\Cat{D} \times \rho}
          \arrow[shorten <= 3pt, shorten >=3pt, Rightarrow,to path=(B) -- (C)\tikztonodes]{}{\Phi^{\mathsf{G}}},
        \end{tikzcd}
        =
        \begin{tikzcd}[baseline=-0.65ex,scale=0.5,  column sep=large, row sep=large]
          \Cat{D} \times \Cat{M} \ar{d}{\act_{\Cat{M}}}  \ar{r}[name=C]{ \Cat{D} \times \mathsf{F}} & \Cat{D} \times \Cat{N} \ar{d}{\act_{N}} \\
          \Cat{M}  \ar[bend left]{r}[name=A]{ \mathsf{F}} \ar[bend right]{r}[ name=B]{ \mathsf{G}}  & \Cat{N} 
          \arrow[shorten <= 2pt, shorten >=2pt, Rightarrow,to path=(A) -- (B)\tikztonodes]{}{ \rho}
          \arrow[shorten <= 3pt, shorten >=3pt, Rightarrow,to path=(C) -- (A)\tikztonodes]{}{\Phi^{\mathsf{F}}}.
        \end{tikzcd}
      \end{equation}
      By applying the 2-functor $\widehat{(-)}: \Bm(\Cat{C},\Cat{D})\rightarrow \Bm(\Cat{C},\Cat{D})$ to (\ref{eq:16}), one obtains (\ref{eq:15}).
      This proves that $l_{\Cat{M}}:\Cat{D} \Box \Cat{M} \rightarrow \Cat{M}$ is a pseudo-natural transformation.
    \end{proof}
    To define the bimodule functor $l_{\Cat{M}}^{-}: \Cat{M} \rightarrow \Cat{D} \Box \Cat{M}$, denote by  $\iota : \Cat{M} \rr \Cat{D} \times \Cat{M}$ 
    the canonical embedding functor that is defined by $\iota(x)=1 \times x$ for objects and morphisms $x$ in $\Cat{M}$. 
    Clearly, $\iota$ is a right $\Cat{C}$-module functor. We define:
    \begin{equation}
      \label{eq:unit-pseudo-inverse}
      l^{-}_{\Cat{M}}= \mathsf{B}_{\Cat{D},\Cat{M}} \circ \iota: \DMC \rr \DDD \Box \DMC.
    \end{equation}
    Then $l^{-}_{\Cat{M}}$ inherits a left module functor  structure from the balancing constraint of $\mathsf{B}_{\Cat{D},\Cat{M}}$ according to Proposition \ref{proposition:tensor-again-module}
    and we have the following result.

    \begin{proposition}
      \label{proposition:adjoint-unit-equival}
      The functor  $l^{-}_{\Cat{M}}$ defines a pseudo-natural transformation and together with the functor $l_{\Cat{M}}$, it forms  an adjoint equivalence of the bimodule categories  $\DDD \Box \DMC$ and $\DMC$.
    \end{proposition}
    \begin{proof}
      Let $\mathsf{F}: \DMC \rightarrow \DMpC$ be a bimodule functor. Then the diagram 
      \begin{equation}
        \label{eq:l-psnat}
        \begin{tikzcd}[column sep=large]
          \Cat{M} \ar{r}{\mathsf{F}} \ar{d}{\iota} &  \Cat{N} \ar{d}{\iota} \\
          \Cat{D} \times \Cat{M} \ar{r}[name=A]{\Cat{D} \times \mathsf{F}}\ar{d}{B} & \Cat{D} \times \Cat{M}\ar{d}{\mathsf{B}} \\
          \Cat{D} \Box \Cat{M} \ar{r}[name=B]{\Cat{D} \Box \mathsf{F}}& \Cat{D} \Box \Cat{N}
          \arrow[shorten <= 2pt, shorten >=2pt, Rightarrow,to path=(A) -- (B)\tikztonodes]{}{ \varphi_{\mathsf{F}}}
        \end{tikzcd}
      \end{equation}
      defines the bimodule natural transformations $l_{\mathsf{F}}^{-}: l_{\Cat{N}}^{-} \circ \mathsf{F} \rightarrow (\Cat{D} \Box \mathsf{F}) \circ l_{\Cat{M}}^{-}$. It follows directly from the properties of the natural isomorphisms 
      $\varphi_{\mathsf{F}}$, that $l_{\mathsf{F}}^{-}$ is natural in $\mathsf{F}$ and compatible with the composition of bimodule functors. Hence $l_{\Cat{M}}^{-}$ is a pseudo-natural transformation.

      We now show that $l_{\Cat{M}}$ and $l^{-}_{\Cat{M}}$ form an adjoint equivalence. 
      For all bimodule categories $\DMC$ there exists a  natural isomorphism $\alpha^{l}: l_{\Cat{M}} \circ l_{\Cat{M}}^{-} \rightarrow 1_{\Cat{M}}$ defined as the composite
      \begin{equation}
        \label{eq:def-alpha}
        \begin{tikzcd}
          \alpha^{l}_{\Cat{M}}:  l_{\Cat{M}} \circ l_{\Cat{M}}^{-} =l_{\Cat{M}} \circ \mathsf{B} \circ \iota \ar{r}{ \varphi_{l}^{-1} \circ \iota}  & \act \circ \iota \ar{r}{\lambda^{\Cat{M}}} & 1_{\Cat{M}},
        \end{tikzcd}
      \end{equation}
      where $\varphi_{l}^{-1}$ is the bimodule natural transformation from Definition \ref{definition:tensor-prod} and $\lambda^{\Cat{M}}$ is the natural isomorphism from Definition \ref{definition:mod-cat}
      with component morphisms $\lambda^{\Cat{M}}_{m}: \unit_{\Cat{C}} \act m\rightarrow m$. If we equip the functor $\act \circ \iota: \Cat{M} \rightarrow \Cat{M}$ with the canonical bimodule functor structure, it follows from the axioms 
      of a module category, that
      $\lambda^{\Cat{M}}$ is a bimodule natural isomorphism. Hence the natural transformations $\alpha^{l}_{\Cat{M}}$ are bimodule natural isomorphisms. 
      Next we show that they define a  modification $\alpha^{l}$. Consider a bimodule functor $\mathsf{F}: \DMC \rightarrow \DNC$. 
      We have to show that the following two diagrams are equal
      \begin{equation}
        \label{eq:cond-alpha-mod}
        \begin{tikzcd}[baseline=-0.65ex,scale=0.5,  column sep=large, row sep=large]
          \Cat{M} \ar{r}[name=B]{l_{\Cat{M}}^{-}} \ar{d}{\mathsf{F}} & \Cat{D} \Box \Cat{M} \ar{r}[name=C]{l_{\Cat{M}}} \ar{d}{\Cat{D}\Box \mathsf{F}}  & \Cat{M} \ar{d}{\mathsf{F}} \\
          \Cat{N} \ar{r}[name=D]{l_{\Cat{N}}^{-}} \ar[bend right]{rr}[name=A, below]{1} & \Cat{D} \Box \Cat{N} \ar[Rightarrow, shorten <= 1pt, shorten >=1pt, to path= -- (A) \tikztonodes]{}{\alpha} \ar{r}[name=E]{l_{\Cat{N}}} & \Cat{N}
          \ar[Rightarrow, shorten <= 7pt, shorten >=7pt, to path=(B) -- (D) \tikztonodes]{}{l_{\mathsf{F}}^{-}}
          \ar[Rightarrow, shorten <= 7pt, shorten >=7pt, to path= (C)-- (E) \tikztonodes]{}{l_{\mathsf{F}}}
        \end{tikzcd}
        =
        \begin{tikzcd}[baseline=-0.65ex,scale=0.5,  column sep=large, row sep=large]
          \Cat{M}  \ar{r}{l_{\Cat{N}}^{-}} \ar[bend right]{rr}[name=A, below]{1} \ar{d}{\mathsf{F}} & \Cat{D} \Box \Cat{M} \ar[Rightarrow, shorten <= 1pt, shorten >=1pt, to path= -- (A) \tikztonodes]{}{\alpha}  \ar{r}{l_{\Cat{M}}}   & \Cat{M} \ar{d}{\mathsf{F}} \\
          \Cat{N} \ar{rr}{1}&  & \Cat{N}.
        \end{tikzcd}
      \end{equation}
      If we insert the  corresponding definition of the arrows in these diagrams, it is easy to see that  equation (\ref{eq:cond-alpha-mod})
      is equivalent to the equation
      \begin{equation}
        \label{eq:cond-alpha-mod-equ}
        \begin{tikzcd}[baseline=-0.65ex,scale=0.5, row sep=large]
          \Cat{M} \ar{r}{\iota} \ar{d}{\mathsf{F}} &  \Cat{D} \times \Cat{M} \ar{d}{\Cat{D} \times \mathsf{F}} \ar{r}[name=C, below]{\act}    
         & \Cat{M} \ar{d}{\mathsf{F}} \\
          \Cat{N} \ar{r}{\iota} \ar[bend right]{rr}[name=A, below]{1} &  \Cat{D} \times \Cat{N}  \ar[Rightarrow, xshift=0pt, shorten <= 1pt, shorten >=1pt, to path= -- (A) \tikztonodes]{}{\alpha}  \ar{r}[name=E]{\act}     & \Cat{N}
 \ar[Rightarrow, yshift=0pt, xshift=0pt, shorten <= 7pt, shorten >=7pt, to path=(C) -- (E) \tikztonodes]{}[right]{\phi^{\mathsf{F}}}
        \end{tikzcd}
        =
        \begin{tikzcd}[baseline=-0.65ex,scale=0.5, row sep=large]
          \Cat{M} \ar{r}{\iota}\ar{d}{\mathsf{F}} \ar[bend right=40]{rr}[name=A, below]{1} & \Cat{D} \times \Cat{M} 
  \ar{r}[name=B,below]{\act}  \ar[Rightarrow, yshift=0pt, xshift=0pt, shorten <= 2pt, shorten >=5pt, to path= -- (A) \tikztonodes]{}[right]{\lambda^{\Cat{M}}}
           & \Cat{M} \ar{d}{\mathsf{F}} \\
          \Cat{N} \ar{rr}[name=C,below]{1}&& \Cat{N}, 
        \end{tikzcd}
      \end{equation}
where $\phi^{\mathsf{F}}$ is the module functor constraint of the functor $\mathsf{F}$. 
     The commutativity of this diagram corresponds directly to the identity (\ref{eq:mod-functor-triangle}) for the module functor $\mathsf{F}$.

      To define bimodule natural isomorphisms 
      \begin{equation}
        \label{eq:beta}
        (\beta^{l}_{\Cat{M}})^{-1}: l_{\Cat{M}}^{-} \circ l_{\Cat{M}} \rightarrow 1_{\Cat{D} \Box \Cat{M}},
      \end{equation}
      note that the balancing structure of $\mathsf{B}$ provides a natural balanced isomorphism $\mathsf{B} \circ \iota \circ \act \rightarrow \mathsf{B}$ for the two balanced module functors $\mathsf{B} \circ \iota \circ \act,  \mathsf{B}: \Cat{D} \times \Cat{M} \rightarrow 
      \Cat{D} \Box \Cat{M}$. By applying the 2-functor $\widehat{(-)}$ we obtain the bimodule natural isomorphism $(\beta^{l}_{\Cat{M}})^{-1}$.
      To show that these natural isomorphisms define a modification $(\beta^{l})^{-1}$, we have to prove  the equation
      
      \begin{equation}
        \label{eq:cond-alpha-mod-explicit}
        \begin{tikzcd}[baseline=-0.65ex,scale=0.5,  column sep=large, row sep=large]
          \Cat{D} \Box  \Cat{M} \ar{r}[name=B]{l_{\Cat{M}}} \ar{d}{\Cat{D} \Box \mathsf{F}} &  \Cat{M} \ar{r}[name=C]{l_{\Cat{M}}^{-}} \ar{d}{ \mathsf{F}}  & \Cat{D} \Box \Cat{M} \ar{d}{\Cat{D} \Box \mathsf{F}} \\
          \Cat{D} \Box \Cat{N} \ar{r}[name=D]{l_{\Cat{N}}} \ar[bend right]{rr}[name=A, below]{1} & \Cat{N} \ar[Rightarrow, shorten <= 1pt, shorten >=1pt, to path= -- (A) \tikztonodes]{}{\beta^{-1}} \ar{r}[name=E]{l_{\Cat{N}}^{-}} & \Cat{D} \Box \Cat{N}
          \ar[Rightarrow, shorten <= 7pt, shorten >=7pt, to path=(B) -- (D) \tikztonodes]{}{l_{\mathsf{F}}}
          \ar[Rightarrow, shorten <= 7pt, shorten >=7pt, to path= (C)-- (E) \tikztonodes]{}{l_{\mathsf{F}}^{-}}
        \end{tikzcd}
        =
        \begin{tikzcd}[baseline=-0.65ex,scale=0.5,  column sep=large, row sep=large]
          \Cat{D} \Box  \Cat{M}  \ar{r}{l_{\Cat{N}}} \ar[bend right]{rr}[name=A, below]{1} \ar{d}{\Cat{D} \Box \mathsf{F}} &\Cat{M} \ar[Rightarrow, shorten <= 1pt, shorten >=1pt, to path= -- (A) \tikztonodes]{}{\beta^{-1}} 
          \ar{r}{l_{\Cat{M}}^{-}}   & \Cat{D} \Box \Cat{M} \ar{d}{\Cat{D} \Box \mathsf{F}} \\
          \Cat{D} \Box \Cat{N} \ar{rr}{1}&  & \Cat{D} \Box \Cat{N}
        \end{tikzcd}
      \end{equation}
      for all bimodule functors $\mathsf{F}: \DMC \rightarrow \DNC$. 
      Inserting the definitions, one finds that this is equivalent to the condition that  the following two diagrams are equal.
      \begin{equation}
        \label{eq:18}
        \begin{tikzcd}
          \Cat{D} \times \Cat{M} \ar{r}[name=A]{\act} \ar{d}{\Cat{D} \times \mathsf{F}} & \Cat{M} \ar{r}{\iota}  \ar{d}{\mathsf{F}} & 
          \Cat{D} \times \Cat{M} \ar{r}[name=C]{B} \ar{d}{\Cat{D} \times \mathsf{F}} & \Cat{D} \Box \Cat{M} \ar{d}{\Cat{D} \Box \mathsf{F}}  \\
          \Cat{D} \times \Cat{N} \ar{r}[name=G]{\act} \ar[bend right]{rrr}[name=F, below]{B} & \Cat{N} \ar{r}[name=X]{\iota} & \Cat{D} \times \Cat{N}  \ar{r}[name=H]{B} & \Cat{D} \Box \Cat{N},
          \ar[Rightarrow, shorten <= 7pt, shorten >=7pt, to path=(A) -- (G) \tikztonodes]{}{\phi^{\mathsf{F}}}
          \ar[Rightarrow, shorten <= 7pt, shorten >=7pt, to path= (C)-- (H) \tikztonodes]{}{\varphi_{\Cat{D} \times \mathsf{F}}}
          \ar[Rightarrow, shorten <= 7pt, shorten >=7pt,  to path= (X)-- (F) \tikztonodes]{}{\beta}
        \end{tikzcd}
      \end{equation}
      \begin{equation}
        \label{eq:19}
        \begin{tikzcd}[row sep=large]
          \Cat{D} \times \Cat{M} \ar{d}{\Cat{D} \times \mathsf{F}} \ar[bend right=20]{rrr}[name=A]{\mathsf{B}} \ar{r}{\act}  & \Cat{M} \ar{r}[name=B]{\iota} & \Cat{D} \times \Cat{M} \ar{r}{\mathsf{B}} & \Cat{D} \Box \Cat{M} \ar{d}{\Cat{D} \Box \mathsf{F}}\\
          \Cat{D} \times \Cat{N} \ar{rrr}[name=C, below]{\mathsf{B}} &&& \Cat{D} \Box \Cat{N}.
          \ar[Rightarrow, xshift=2.5pt, shorten <= 3pt, shorten >=3pt, to path= (B)-- (A) \tikztonodes]{}[below]{\beta}
          \ar[Rightarrow, shorten <= 7pt, shorten >=7pt, to path= (A)-- (C) \tikztonodes]{}{\varphi_{\Cat{D} \times \mathsf{F}}}
        \end{tikzcd}
      \end{equation}
      We compute both sides on objects.  When evaluated on objects $d \in \Cat{D}$ and $m \in \Cat{M}$, the first diagram yields the  morphism 
      \begin{equation}
        \label{eq:20}
        \begin{tikzcd}[row sep=1 ex, column sep=large]
          (\Cat{D} \Box \mathsf{F})\mathsf{B}(\unit \times d \act m) \ar{rr}{\varphi_{\Cat{D} \times \mathsf{F}}(\unit \times d\act m)} && \mathsf{B}(\unit \times \mathsf{F}(d \act m)) \ar{r}{\phi^{\mathsf{F}}_{d,m}}   & \mathsf{B}(\unit \times d \act \mathsf{F}(m))\\
& {}& {} \ar[shorten <= 60pt]{r}[xshift=30pt]{\beta} & \mathsf{B}(d \times \mathsf{F}(m)), 
        \end{tikzcd}
      \end{equation}
      while the other diagram corresponds to 
      \begin{equation}
        \label{eq:21}
        \begin{tikzcd}%[column sep=large]
          (\Cat{D} \Box \mathsf{F})\mathsf{B}(\unit \times d \act m) \ar{r}{b} &(\Cat{D} \Box \mathsf{F})\mathsf{B}(d \times m) \ar{rr}{\varphi_{\Cat{D} \times \mathsf{F}}(d \times m)} && \mathsf{B}(d \times \mathsf{F}(m)).
        \end{tikzcd}
      \end{equation}
      These two morphisms are equal since $\varphi_{\Cat{D} \times \mathsf{F}}$ is a balanced natural isomorphism.

      It remains to prove that the natural isomorphisms $\alpha^{l}$ and $\beta^{l}$ define an adjoint equivalence according to Definition \ref{definition:adj-equiv}.
      We have to show that the composites
      \begin{equation}
        \label{eq:22}
        \begin{tikzcd}
          l_{\Cat{M}} \ar{r}{l_{\Cat{M}} \beta^{l}} & l_{\Cat{M}} l_{\Cat{M}}^{-} l_{\Cat{M}} \ar{r}{\alpha^{l} l_{\Cat{M}}}  &l_{\Cat{M}}, &  l_{\Cat{M}}^{-} \ar{r}{\beta^{l} l_{\Cat{M}}^{-}} & l_{\Cat{M}}^{-} l_{\Cat{M}} l_{\Cat{M}}^{-} \ar{r} {l_{\Cat{M}}^{-} \alpha^{l}}  &l_{\Cat{M}}^{-}
        \end{tikzcd}
      \end{equation}
      are the respective identities.
      In the first case this is equivalent to  the commutativity of the diagram 
      \begin{equation}
        \label{eq:24}
        \begin{tikzcd}
          l_{\Cat{M}} l_{\Cat{M}}^{-} l_{\Cat{M}}= l_{\Cat{M}}\mathsf{B} \iota l_{\Cat{M}} \ar{r}{\varphi^{l}} \ar[bend right]{rr}{\alpha^{l} l_{\Cat{M}}} &   \act \iota l_{\Cat{M}}  \ar{r}{\lambda^{\Cat{M}}}  &1 l_{\Cat{M}}=l_{\Cat{M}}.
        \end{tikzcd}
      \end{equation}
      By  definition of $\alpha^{l}$, this is equivalent to the commutativity of the diagram
      \begin{equation}
        \label{eq:25}
        \begin{tikzcd}
          l_{\Cat{M}}\mathsf{B} \ar{r}{b} \ar[bend right]{rrr}{\varphi^{l}} &   l_{\Cat{M}} \mathsf{B} \iota \act  \ar{r}{\varphi^{l}}  & \act \iota \act \ar{r}{\lambda^{\Cat{M}}} & \act
        \end{tikzcd}
      \end{equation}
      Evaluated on objects, this diagram takes the form
      \begin{equation}
        \label{eq:25-eval}
        \begin{tikzcd}
          l_{\Cat{M}}\mathsf{B}(d \times m) \ar{r}{b} \ar[bend right]{rrrr}{\varphi_{l}(d \times m)} &   l_{\Cat{M}} \mathsf{B}(\unit \times  d\act m)  \ar{rr}{\varphi^{l} (\unit \times d \act m)}  && \unit \act (d \act m) \ar{r}[xshift=-1pt]{\lambda^{\Cat{M}}_{d \act m}} & d \act m.
        \end{tikzcd}
      \end{equation}
      This last diagram commutes since $\varphi$ is a balanced natural isomorphism.

      In the second case, the requirement  that morphism (\ref{eq:23}) is the identity is equivalent to the condition that   
      \begin{equation}
        \label{eq:23}
        \begin{tikzcd}  
          \mathsf{B} \iota \ar{r}{(\beta^{l})^{-1}}  & \mathsf{B} \iota \act \iota \ar{r}{\mathsf{B} \iota \lambda^{\Cat{M}}} & \mathsf{B} \iota
        \end{tikzcd}
      \end{equation}
      is the identity natural transformation on $\mathsf{B} \iota$.
      Evaluated on objects, this yields 
      \begin{equation}
        \label{eq:17}
        \begin{tikzcd}
          \mathsf{B}(\unit \times m) \ar{r}{\beta} & \mathsf{B}(\unit \times \unit \act m) \ar{r}{\mathsf{B} \iota \lambda^{\Cat{M}}} & \mathsf{B}(\unit \times m),
        \end{tikzcd}            
      \end{equation}
      which is the identity on the object $\mathsf{B}(\unit \times m)$, by equation (\ref{eq:balanced-unit}).
    \end{proof}
    The bimodule functors $r_{\Cat{M}}: \DMC \Box \CCC \rr \DMC$ and $r_{\Cat{M}}^{-}: \DMC \rr \DMC \Box \CCC$ are defined analogously using the right action of $\Cat{C}$ on $\DMC$ 
    and the proof that they define an adjoint equivalence is similar.

  \item %\paragraph{The modification $\mu$}

    The modification $\mu$,
    from \ref{theorem:Bimcat-3-cat}, \refitem{item:mu-bimod} is defined by applying the functor $\widehat{(-)}$ to the diagram (\ref{eq:modific-mu-m}) in $\Bm$.
    It follows directly that $\mu$ is a modification,  since it is the composite of a 2-functor with a modification.
  \item %\paragraph{The modifications $\lambda$ and $\rho$}
    The modification $\lambda$ is obtained by applying $\widehat{(-)}$ to the diagrams (\ref{eq:modific-lambda}).
  \item The modification
    $\rho$ is obtained analogously by  applying $\widehat{(-)}$ to the diagram (\ref{eq:modific-rho-m}).
  \item %\paragraph{The modification $\pi$}
    Applying $\widehat{(-)}$ to the diagram (\ref{eq:pi-m}) defines  the modification $\pi$.

  \item %\paragraph{The tricategory axioms}

    To complete the proof that $\BimCat$ is a tricategory, it remains to verify the three axioms in Definition \ref{definition:tricategory}.
    All the structures of $\BimCat$ are defined in terms of structures in $\BimCat^{multi}$ and every axiom for $\BimCat$ is a pasting diagram  that is 
    obtained from a pasting diagram in $\BimCat^{multi}$ according to Corollary \ref{corollary:pasting-diag-2-fun}. Hence  Corollary \ref{corollary:pasting-diag-2-fun} reduces   the proof of the axioms to the commutativity of  
     the corresponding pasting diagrams in $\BimCat^{multi}$. 
    The first axiom in  Definition \ref{definition:tricategory} is the so-called Stasheff 5-polytope, the higher analogue of the pentagon axiom for monoidal categories. 
    This axiom is trivial in $\BimCat^{multi}$, since the associator $\alpha$ in $\BimCat^{multi}$ already satisfies the pentagon axiom and hence the corresponding modification $\pi$ is the identity. 
    The remaining axioms follow by applying the 2-functor $\widehat{(-)}$ to the diagram (\ref{eq:Axiom2-alg-tricat-mulit}) and to  diagram (\ref{eq:Axiom3-alg-tricat-mulit}).
  \end{enumerate}
\end{proof}

\section{Tricategories with duals and the example of $\BimCat$}
\label{sec:bimod-categ-as}

In this section we develop the notions of a tricategory with duals and of a pivotal tricategory.  We 
provide  useful techniques for  computations involving the inner homs for bimodule categories and use these to construct duals for bimodule categories in the tricategory $\BimCat$. 
 This is achieved by  constructing for each  bimodule category $\DMC$  2-morphisms in the tricategory $\BimCat$ (recall the dual categories $\CMDld$ and $\CMDrd$ from Section  \ref{sec:Prelim})
\begin{equation}
  \label{eq:duality-for-bimodules}
  \coev{\Cat{M}}:\DDD \rr \DMC \Box \CMDld \quad \text{and}\quad \ev{\Cat{M}}:\CMDld \Box \DMC \rr \CCC, 
\end{equation}
such that the snake identities (\ref{eq:duality-snake-right}) and (\ref{eq:snake2-right}) are satisfied up to a 3-isomorphism in $\BimCat$. 

\paragraph{The example of $\Vect$-bimodule categories}

To illustrate the idea behind the construction of the dualities for bimodule categories, we first consider the  example of  finite semisimple  categories regarded as  $\Vect$-bimodule categories.  
Let $\Cat{M}$ be a finite semisimple category. Choose a finite set  $\{m_{i} \}_{i \in I}$ of  representatives of the simple objects of $\Cat{M}$.
\begin{enumerate}
\item The object ${R}^{\Cat{M}}= \oplus_{i \in I} m_{i} \boxtimes m_{i} \in \Cat{M} \boxtimes \op{M}$  represents the $\Hom_{}$-functor, i.e. there is a natural isomorphism 
\begin{equation}
    \label{eq:representing-object}
  \Hom_{\Cat{M} \boxtimes \op{M}}( m \boxtimes m',{R}^{\Cat{M}})\simeq \oplus_{i}\Hom_{\Cat{M}}(m, m_{i}) \otimes \Hom_{\op{M}}(m',m_{i})  \rightarrow \Hom_{\Cat{M}}(m, m'),
  \end{equation}
using the semisimplicity of $\Cat{M}$. Clearly the object ${R}^{\Cat{M}}$ defines a $\Vect$-bimodule functor 
$$\coevvect{\Cat{M}}: \Vect \ni V \mapsto V \act {R}^{\Cat{M}}  \in \Cat{M} \boxtimes \op{M}.$$
\item The $\Hom_{}$-functor of $\Cat{M}$ defines a $\Vect$-bimodule functor
\begin{equation}
  \label{eq:evvect}
  \evvect{\Cat{M}}: \op{M} \boxtimes \Cat{M} \ni m' \boxtimes m \mapsto \Hom_{\Cat{M}}(m',m) \in \Vect.
\end{equation}
\item 
By composing we obtain the $\Vect$-bimodule functor 
\begin{equation*}
 \Phi_{\Cat{M}}^{\Vect}= r_{\Cat{M}} \circ (\Cat{M} \boxtimes \evvect{\Cat{M}})  \circ  (\coevvect{\Cat{M}} \boxtimes \Cat{M}) \circ l_{\Cat{M}}^{-}: \Cat{M} \rightarrow \Cat{M}.  
\end{equation*}
\end{enumerate}
\begin{proposition}
There exists a canonical $\Vect$-bimodule natural isomorphism 
\begin{equation}
    \label{eq:Vect-triangulator}
    T_{\Cat{M}}^{\Vect}: \Phi_{\Cat{M}}^{\Vect} \rightarrow  1_{\Cat{M}}.
  \end{equation}
\end{proposition}
\begin{proof}
Inserting the definitions, it follows immediately, that $\Phi_{\Cat{M}}^{\Vect}(m)= \oplus_{i} m_{i} \otimes \Hom_{}(m_{i},m)$ on objects $m \in \Cat{M}$.
For all  $n \in \Cat{M}$, the defining property of ${R}$ thus yields a natural isomorphism 
\begin{equation}
  \label{eq:compute-R-vect}
  \begin{split}
      \Hom_{}(n, \Phi^{\Vect}_{\Cat{M}}(m)) &= \Hom(n,\oplus_{i} m_{i}  \otimes \Hom_{}(m_{i},m)) = \oplus_{i}\Hom_{}(n, m_{i}) \otimes \Hom_{}(m_{i},m) \\
 &\simeq \Hom_{}(n,m).
 \end{split}
\end{equation}
By the Yoneda-lemma, this defines an natural isomorphism $    T_{\Cat{M}}^{\Vect}: \Phi_{\Cat{M}}^{\Vect} \rightarrow  1_{\Cat{M}}$.
\end{proof}
In the next subsections we will show that analogous statements hold in the case of general finite bimodule category. We will thereby use the inner hom instead of the $\Hom_{}$-spaces in $\Vect$. Therefore we require module versions of 
the Yoneda lemma and of representations of module functors. 

\subsection{Tricategories with duals and pivotal structure}

The definition of duals in tricategories that we present in this subsection is  inspired by the  duals in higher categories in \cite{Lurie}  and makes use of the notion of duals in bicategories, see Definition 
\ref{definition:duality-bicat}.
To this end we briefly recall the 
 delooping procedure to obtain a $(n-1)$-category $\Cath{X}$ from a $n$-category $\Cat{X}$. We will use the following statement only for $n \leq 3$, but it is expected to hold for all reasonable models of higher categories. 
 \begin{lemma}
   Let $\Cat{X}$ be a $n$-category. The following defines a $(n-1)$-category $\Cath{X}$. 
The objects and $i$-morphisms for $i=1, \ldots (n-2)$ of $\Cath{X}$ are the same as in $\Cat{X}$. The $(n-1)$-morphisms of $\Cath{X}$ are the  isomorphism classes of $(n-1)$-morphisms in $\Cat{X}$. 
The compositions and coherence structures in $\Cath{X}$ are induced from the compositions and coherence structures in $\Cat{X}$.
 \end{lemma}
For a tricategory $\Cat{T}$, the delooping is  thus  a bicategory $\Cath{T}$.
\begin{definition}
  \label{definition:tricat-duals}
Let $\Cat{T}$ be a tricategory. 
\begin{definitionlist}
  \item  We say that $\Cat{T}$ has $*$-duals, if for objects $b, c \in \Cat{T}$, the bicategory $\Cat{T}(b,c)$ has both left and right duals according to Definition \ref{definition:duality-bicat}. 
\item  $\Cat{T}$ has $\#$-duals if the  bicategory $\Cath{T}$ is a bicategory with  left and right duals. 
\item $\Cat{T}$ is called a tricategory with duals if it has $*$-duals and $\#$-duals.
\end{definitionlist}
\end{definition}
This definition means that  for every 2-morphism $\varphi: F \Rightarrow G$ in  a tricategory with $*$-duals, there exists a 2-morphism $\varphi^{*}: G \Rightarrow F$ and 
duality 3-morphisms $\ev{\varphi}: \varphi^{*} \circ \varphi \Rrightarrow 1_{F} $ and  $\coev{\varphi}: 1_{G} \Rrightarrow \varphi \circ \varphi^{*}$, that satisfy the snake identities (\ref{eq:duality-snake-right}) and 
(\ref{eq:snake2-right}). 

The duality on $\Cath{T}$ for a tricategory with duals is denoted $\#$, hence for every 1-morphism $M: a \rightarrow b$ in a tricategory $\Cat{T}$ with $\#$-duals, there exists a 1-morphism $M^{\#}: b \rightarrow a$ in $\Cat{T}$ together with 
2-morphisms $\ev{M}: M^{\#} \Box M \Rightarrow 1_{a} $ and  $\coev{M}: 1_{b} \Rightarrow M \Box M^{\#}$, such that the snake identity holds in $\Cath{T}$. 

The following is shown in  \cite[Rem. 3.4.22]{Lurie}, \cite[Lemma 2.4.4]{DSS}. 
\begin{proposition}
\label{proposition:serre-ex}
 Let $\Cat{T}$ be a tricategory with $*$-duals such that the  bicategory $\Cath{T}$ has right duals. Then the right duals in $\Cath{T}$ are also left duals and thus $\Cat{T}$ is a tricategory with duals. 
In particular, left and right $\#$-duals in a tricategory with duals are equivalent. 
\end{proposition}

Next we turn to pivotal structure on tricategories. 
The notions of pivotal structures for bicategories and pivotal 2-functors is given in Definitions \ref{definition:pivotal-str-bicat} and \ref{definition:pivotal-2-fun}.

\begin{definition}
\label{definition:pivotal-tricat}
Let $\Cat{T}$ be a tricategory with $*$-duals.
  \begin{definitionlist}
    \item \label{item:star-fun} A pivotal structure for $\Cat{T}$ consists of a pivotal structure in  the bicategory $\Cat{T}(b,c)$
such that 
for all 1-morphisms $M: c \rightarrow d$, the 2-functors 
    \begin{equation}
      \label{eq:pivotal-2-fun-from-1morph}
      M\Box -: \Cat{T}(b,c) \rightarrow \Cat{T}(b,d) \quad \text{and}\quad -\Box M: \Cat{T}(d,e) \rightarrow \Cat{T}(c,e),
    \end{equation}
    are pivotal 2-functors for all objects $c,d,e$. A tricategory with $*$-duals together
 with a pivotal structure is called a pivotal 
tricategory.
 \item \label{item:hash} $\Cat{T}$ is a pivotal tricategory with duals, if it is a pivotal tricategory and 
    the bicategory $\Cath{T}$ is a bicategory with right duals. 
 \end{definitionlist}
\end{definition}

Concretely, in a pivotal tricategory for every pair of objects $b, c \in \Cat{T}$,
$\Cat{T}(c,b)$ is a  pivotal bicategory with duality $*$ and pivotal structure $a$.  The pivotal structure 
defines invertible 3-morphisms $a_{\varphi}:  \varphi \Rrightarrow \varphi^{**} $ for all 2-morphisms $\varphi$.

\begin{remark}
  In \cite{Schaum} it is shown that the notion of a pivotal tricategory with duals (this structure is 
called tricategory with weak duals  in \cite{Schaum})
is stable under triequivalences 
of tricategories, i.e. if $\Cat{T} \simeq \Cat{G}$ as tricategories and $\Cat{T}$ is a pivotal tricategory with 
duals, then $\Cat{G}$ is canonically a pivotal tricategory.
Using this result it is possible to strictify a pivotal tricategory with duals to a Gray category with duals in 
the sense of \cite{GrayDuals}, see \cite[Thm. 7.21]{Schaum}. 
\end{remark}

\subsection{The calculus with the inner hom and the dual categories}

In this section we provide the technical tools that will be used to construct $\#$-duals in the tricategory 
$\BimCat$. Therefore we first consider the dual bimodule categories and functors between them in more detail. 
Next we discuss various compatibilities between the inner homs, the dual categories and the tensor product of bimodule categories. Most importantly we prove a Rieffel-induction type formula that allows to compute the inner homs of 
a tensor product $\DMC \Box \CNE$ in terms of the inner homs of $\Cat{M}$ and $\Cat{N}$. Finally we discuss module versions of the Yoneda-lemma and of the notion of representations of functors.

\paragraph{Dual categories}

Recall from Section \ref{sec:Prelim}, that for every bimodule category $\DMC$, there are bimodule categories 
$\CMDrd$ and $\CMDld$. 
These so-called dual categories are compatible with the tensor product as follows. 
\begin{lemma}
  \label{lemma:dual-and-tensor}
Let $\DMC$ and $\CNE$ be bimodule categories.  There are canonical equivalences of bimodule categories 
\begin{lemmalist}
 \item \label{item:double-dual}  \begin{equation}
    \label{eq:left-of-right-id}
    \leftidx{^{\scriptscriptstyle{\#}}}{(\CMDrd)}{} \simeq  \DMC, \quad (\CMDld)^{\scriptscriptstyle{\#}}\simeq \DMC. 
  \end{equation}
\item
  \begin{equation}
    \label{eq:dual-tensor-comp}
 (\DMC \Box \CNE)^{\scriptscriptstyle{\#}} \simeq \ENCrd \Box \CMDrd,  \quad     \leftidx{^{\scriptscriptstyle{\#}}}{(\DMC \Box \CNE)}{} \simeq  \ENCld \Box \CMDld.
  \end{equation}
\end{lemmalist}
\end{lemma}
\begin{proof}
  The equivalences for the first part are obtained directly from the identifications $(\op{M})^{\opp}\simeq \Cat{M}$ and $\leftidx{^*}{(c^{*})}{} \simeq (\leftidx{^*}{c}{})^{*} \simeq c$, 
 $\leftidx{^*}{(d^{*})}{} \simeq (\leftidx{^*}{d}{})^{*} \simeq d$,  for $c \in \Cat{C}$ and $d \in \Cat{D}$. 
For the second part we define the functors $\tau: \Nrd \boxtimes \Mrd \rightarrow (\Cat{M} \boxtimes \Cat{N})^{\scriptscriptstyle{\#}} \rightarrow (\Cat{M} \Box \Cat{N})^{\scriptscriptstyle{\#}}$ by 
$\tau(n \boxtimes m)= B(m \boxtimes n)$ for $n \boxtimes m \in \Nrd \boxtimes \Mrd$. It is straightforward to see that $\tau$ is a balanced bimodule functor and moreover, that 
$ (\Cat{M} \Box \Cat{N})^{\scriptscriptstyle{\#}}$ together with $\tau$ is a tensor product of $\Nrd$ and $\Mrd$. Thus by universality of the tensor product, $\tau$ induces an equivalence 
of bimodule categories $(\DMC \Box \CNE)^{\scriptscriptstyle{\#}} \simeq \ENCrd \Box \CMDrd$. The analogous argument applies to the left duals.
\end{proof}

The dual bimodule categories  extend to functors as follows. 
For each bimodule functor $\mathsf{F}: \DMC \rightarrow \DNC$, there are corresponding bimodule functors $\Funrd{F}:\CMDrd \rightarrow \CNDrd$, and 
$\Funld{F}: \CMDld \rightarrow \CNDld$ that are just the obvious functors $\op{\mathsf{F}}: \op{M} \rightarrow \op{N}$ as linear functors, with bimodule structures induced from the bimodule structures of $\mathsf{F}$.  
 Furthermore, each bimodule natural transformation $\eta: \mathsf{F} \rightarrow \mathsf{G}$ 
between two such bimodule functors defines canonical bimodule natural transformations $\eta^{\#}: \Funrd{F}  \rightarrow \Funrd{G}$ and 
$^{\#}\eta: {}^{\#}\mathsf{F} \rightarrow  {}^{\#}\mathsf{G}$.  In total, we obtain  2-functors $(-)^{\#}, ^{\#}(-): \BimCat(\Cat{D}, \Cat{C}) \rightarrow \BimCat(\Cat{C}, \Cat{D})$.

Next we consider the duals of the unit bimodule categories. 
 The following statement follows directly from the definitions. 
\begin{lemma}
  \label{lemma:bimodule-functors-unit}
  \begin{lemmalist}
    \item The right dual functor is an equivalence of bimodule categories $(-)^{*}: \DDDrd \rightarrow \DDD$ with quasi-inverse $\leftidx{^*}{(-)}{}: \DDD \rightarrow \DDDrd$. 
\item The left dual functor is an equivalence of bimodule categories $\leftidx{^*}{(-)}{}: \DDDld \rightarrow \DDD$ with quasi-inverse $(-)^{*}: \DDD \rightarrow \DDDld$. 
  \end{lemmalist}
\end{lemma}
Let  $\mathsf{F}: \DDD \rightarrow \DMD$ and $\mathsf{G}: \DMD \rightarrow \DDD$ be bimodule functors.
The equivalences from Lemma \ref{lemma:bimodule-functors-unit} induce bimodule functors
\begin{equation}
  \label{eq:induce-dual-unit}
  \begin{split}
    \Funltd{\mathsf{F}} &: \DDD \simeq \DDDld \stackrel{\Funld{F}}{\longrightarrow} \DMDld , \quad   \Funrtd{\mathsf{F}} : \DDD \simeq \DDDrd \stackrel{\Funrd{F}}{\longrightarrow} \DMDrd \\
   \Funltd{\mathsf{G}} &: \DMDld \stackrel{\Funld{G}}{\longrightarrow} \DDDld \simeq \DDD, \quad  \Funrtd{\mathsf{G}} : \DMDrd \stackrel{\Funrd{G}}{\longrightarrow} \DDDrd \simeq \DDD,
  \end{split}
\end{equation}
Hence on objects $d \in \Cat{D}$ and $m \in \Cat{M}$, these functors take the following values:
\begin{equation}
  \label{eq:values-semi-hash}
  \Funltd{\mathsf{F}}(d)= \mathsf{F}(x^{*}), \quad \Funrtd{\mathsf{F}}(d)= \mathsf{F}(\leftidx{^*}{d}{}), \quad \Funltd(\mathsf{G})(m)= \leftidx{^*}{\mathsf{G}(m)}{}, \quad \Funrtd{\mathsf{G}}(m)= \mathsf{G}(m)^{*}.
\end{equation}

It follows from Lemma \ref{lemma:dual-and-tensor} that 
$({}^{\widetilde{\#}}\mathsf{F})^{\widetilde{\#}} \simeq \mathsf{F} \simeq {}^{\widetilde{\#}}(\Funrtd{\mathsf{F}})$ and similarly for $\mathsf{G}$. 
These constructions will be needed for the coevaluation functor for bimodule categories.

\paragraph{Representation of module functors}
Our next goal is to obtain a module version of the Yoneda-Lemma and a notion of representation of module functors. 
First we extend the notion of balanced functors such that it includes the $\Hom_{}$-functor of 
a bimodule category. 

\begin{definition}
\label{definition:multi-bal-to-Vect}
  A functor $\mathsf{F}: \DM \boxtimes \ND \rightarrow \Cat{A}$ between the product of  two module categories and a linear category $\Cat{A}$ is called $\Cat{D}$-balanced, if it is equipped with natural isomorphisms 
  \begin{equation}
    \label{eq:outer-bal}
    \mathsf{F}(d\act m \boxtimes n) \simeq \mathsf{F}(m \boxtimes n \ract d^{**}), 
  \end{equation}
for all $d \in \Cat{D}$, $m \in \Cat{M}$ and $n \in \Cat{N}$, that satisfy the usual pentagon axiom  with respect to the tensor product of $\Cat{D}$ and the triangle axiom with respect to the unit of $\Cat{D}$.
Balanced natural transformation between two balanced functors of this type are defined as natural transformation, such that the obvious  analogue of the diagram (\ref{eq:balanced-nat}) commutes. 

We say that a functor $\mathsf{F}: \DMC \boxtimes \CND \rightarrow \Cat{A}$ is multi-balanced if it is $\Cat{C}$-balanced, $\Cat{D}$-balanced in the given sense and furthermore both structures are compatible as 
in Definition \ref{definition:multi-balanced}. 
\end{definition}
It turns out  that the $\Hom_{}$-functor is balanced in this sense: 
\begin{example}
\label{example:hom-mb}
  The functor $\Hom_{\Cat{M}}:  \CMDld \boxtimes \DMC \rightarrow \Vect$ is multi-balanced. It is clearly $\Cat{C}$-balanced, and the $\Cat{D}$-balancing structure is obtained from the isomorphisms
  \begin{equation*}
    \Hom_{\Cat{M}}(c \actld \widetilde{m}, m) = \Hom_{\Cat{M}}(\widetilde{m} \ract c^{*},m) \simeq \Hom_{\Cat{M}}(\widetilde{m},m \ract c^{**}).
  \end{equation*}
It is easy to see that both balancing structures are compatible. 
If we compose the $\Hom_{}$-functor with the inner hom functor $\idm{-,-}: \DMC \boxtimes \CMDld \rightarrow \DDD$, we see  
that $\Hom_{\Cat{D}}(m, d \act \widetilde{m})$ and $\Hom_{\Cat{M}}( \idm{m,\widetilde{m}},d)$ define  multi-balanced functors $\DMClld \boxtimes \CMDld \boxtimes \DDD \rightarrow \Vect$. 

Moreover, the 
inner hom defines a multi-balanced natural isomorphism $\Hom_{\Cat{D}}(m, d \act \widetilde{m}) \simeq \Hom_{\Cat{M}}( \idm{m,\widetilde{m}},d)$.
\end{example}

Now we turn to the module version of the  Yoneda-Lemma.

\begin{lemma}
  \label{lemma:module-yoneda}
  \begin{lemmalist}
    \item \label{item:Yoneda-bal-mod} Let $\mathsf{F}, \mathsf{G}: \DMC \rightarrow \DNC$ be bimodule functors. The set of bimodule natural transformations from $\mathsf{F}$ to $\mathsf{G}$ is in bijection 
with the set of multi-balanced natural transformations 
\begin{equation}
  \label{eq:bal-Vect-yoneda}
  \Hom_{\Cat{N}}(-, \mathsf{F}(-)) \rightarrow \Hom_{\Cat{N}}(-, \mathsf{G}(-)),
\end{equation}
between multi-balanced functors  $\CNDld \boxtimes \DMC \rightarrow \Vect$.
\item \label{item:Yoneda-inner-hom}
Let $\mathsf{F}, \mathsf{G}: \DMC \rightarrow \DNC$ be bimodule functors. The set of bimodule natural transformations from $\mathsf{F}$ to $\mathsf{G}$ is in bijection 
with the set of balanced bimodule transformations 
\begin{equation}
  \label{eq:bal-inner-hom-yoned}
  \idn{\mathsf{F}(-),-} \rightarrow \idn{\mathsf{G}(-),-} 
\end{equation}
between balanced bimodule functors  $ \DMC \boxtimes \CNDld \rightarrow \DDD$.
  \end{lemmalist}
There are the analogue statements  switching the entries of the $\Hom_{}$-functors and replacing covariant with contravariant. 
\end{lemma}
\begin{proof}Let $\Phi_{n,m}: \Hom_{\Cat{N}}(n, \mathsf{F}(m)) \rightarrow \Hom_{\Cat{N}}(n, \mathsf{G}(m))$ be a multi-balanced natural transformation. As in the usual Yoneda-lemma we 
define the corresponding natural transformation $\eta: \mathsf{F}\rightarrow \mathsf{G}$ by $\eta_{m}=\Phi_{\mathsf{F}(m),m}(\id_{\mathsf{F}(m)}): \mathsf{F}(m) \rightarrow \mathsf{G}(m)$. 
The following diagram commutes since $\Phi$ is $\Cat{D}$-balanced:
\begin{equation}
  \label{eq:comm-diag-mod-yon}
  \begin{tikzcd}[column sep=large]
    \begin{array}{r}
     \Hom_{}(d^{*} \act d \act \mathsf{F}(m), \mathsf{F}(m))  \simeq  \\
 \Hom_{}(d \act \mathsf{F}(m), d \act \mathsf{F}(m)) 
    \end{array}
    \ar{r}{\Phi} \ar{d}{\simeq} &  
    \begin{array}{r}
     \Hom_{}(d^{*} \act d \act \mathsf{F}(m), \mathsf{G}(m)) \simeq\\
 \Hom_{}(d \act \mathsf{F}(m), d \act \mathsf{G}(m))
    \end{array}
  \ar{d}{\simeq} \\
\Hom_{}(\mathsf{F}(d\act m), \mathsf{F}(d\act m)) \ar{r}{\Phi}& \Hom_{}(\mathsf{F}(d\act m), \mathsf{G}(d \act m))
  \end{tikzcd}
\end{equation}
If we consider the identity in the upper left $\Hom_{}$-space, it gets mapped to $d \act \eta_{m}$ by the upper arrow and to $\eta_{d \act m}$ in the lower left $\Hom_{}$-space by the lower arrow. 
Hence it follows that $\eta$ is $\Cat{D}$-balanced. The proof that $\eta$ is $\Cat{C}$-balanced is analogous. 
The converse and the contravariant versions of the statements follow analogously. 

For the second part, assume that we are given a balanced bimodule natural transformation $\Phi:   \idn{\mathsf{F}(-),-} \rightarrow \idn{\mathsf{G}(-),-}$. It is straightforward to see that 
$\Phi$ induces a multi-balanced module natural transformation $\Hom_{\Cat{D}}(\idm{\mathsf{F}(m),n},d)\simeq \Hom_{\Cat{D}}(\idn{\mathsf{G}(n),m},d)$ of multi-balanced functors  $\DMClld \boxtimes \CMDld \boxtimes \DDD \rightarrow \Vect$.
 Restricting to $d = \unit_{\Cat{D}}$, this yields a balanced natural transformation $\Hom_{}(\mathsf{F}(m),n) \rightarrow \Hom_{}(\mathsf{G}(m),n)$, 
which defines a bimodule natural transformation $\mathsf{F} \rightarrow \mathsf{G}$ by the contravariant version of the first part. 
The remaining statements follow analogously. 
 \end{proof}
Next we consider the module version of representable functors. 

\begin{proposition}
\label{proposition:D-represenations-functors}
  \begin{propositionlist}
    \item   Every right exact module functor $\mathsf{F}: \DM \rightarrow \LDD$ is (left) $\Cat{D}$-representable by $\mathsf{F}^{r}(\unit_{\Cat{D}})$, i.e. there exists a module natural isomorphism 
      \begin{equation}
        \label{eq:D-repres}
        \mathsf{F}(x) \simeq \idm{x,\mathsf{F}^{r}(\unit_{\Cat{D}})}.
      \end{equation}
This isomorphism extends to an equivalence of bimodule functors 
\begin{equation}
  \label{eq:equiva-short}
  \idm{-, \Funltd{(\mathsf{F}^{r})}(-)}  \simeq \mathsf{F}(-) \otimes (-),
\end{equation}
from  $\DM \boxtimes \RDD $ to $\DDD$. 
\item \label{item:repres-bimod-func}Let $\mathsf{F}: \DMD \rightarrow \DDD$ be a right exact bimodule functor. 
Then the equivalence (\ref{eq:equiva-short}) is an equivalence of balanced bimodule functors $\DMD \boxtimes \DDD \rightarrow \DDD$. 
Furthermore, $\mathsf{F}$ is right $\Cat{D}$-representable by $\Funrtd{(\mathsf{F}^{r})}$, i.e. there is a bimodule natural isomorphism
\begin{equation}
  \label{eq:short-D-bimod}
  \imd{\Funrtd{(\mathsf{F}^{r})}(-), -} \simeq (-) \otimes \mathsf{F}(-)
\end{equation}
between balanced bimodule functors $\DDD \boxtimes \DMD \rightarrow \DDD$.
  \end{propositionlist}
\end{proposition}
\begin{proof}
Since $\mathsf{F}$ is right exact,  it admits a right adjoint, see Proposition \ref{proposition:adj-ex}. For the   first part we first compute for all $m \in \Cat{M}$, $x, d \in \Cat{D}$
  \begin{equation}
    \label{eq:inner-hom-rep}
    \begin{split}
        \Hom_{\Cat{D}}(\idm{m, (\Funltd{\mathsf{F}^{r}})(d)},x) &\simeq \Hom_{\Cat{M}}(m, x \act \mathsf{F}^{r}(d^{*})) \simeq \Hom_{\Cat{M}}(m, \mathsf{F}^{r}(x \otimes d^{*})) \\ 
&\simeq\Hom_{\Cat{D}}(\mathsf{F}(m), x \otimes d^{*}) \simeq \Hom_{} (\mathsf{F}(m) \otimes d,x).
    \end{split}
    \end{equation}
 All isomorphisms are multi-balanced natural isomorphisms, hence the statement follows from the module Yoneda-Lemma \ref{lemma:module-yoneda}\refitem{item:Yoneda-bal-mod}.
If $\mathsf{F}$ is a bimodule functor as in the second statement, it follows that $\mathsf{F}^{r}: \DDD \rightarrow \DMD$ is also a bimodule functor. 
Then   equation (\ref{eq:inner-hom-rep}) consists of multi-balanced natural isomorphisms and thus (\ref{eq:equiva-short}) is a bimodule natural isomorphism. We additionally have a natural multi-balanced isomorphism
\begin{equation}
  \label{eq:right-D-rep}
 \Hom_{\Cat{D}}(\imd{ \Funrtd{(\mathsf{F}^{r})}(d),m },x) \simeq \Hom_{\Cat{M}}(d \otimes \mathsf{F}(m),x),
\end{equation}
that is constructed in a similar way to (\ref{eq:inner-hom-rep}). 
This yields  by the module Yoneda-Lemma,  
the claimed second bimodule natural isomorphism.
\end{proof}
The next statement follows using the same techniques as in the previous proposition. 
\begin{lemma}
  \label{lemma:adjoint-inner-hom}
Let $\mathsf{F}: \DMC \rightarrow \DNC$  be a right exact bimodule functor. Then there exist balanced bimodule natural isomorphisms 
\begin{equation}
  \label{eq:bal-iso-adj}
  \idn{\mathsf{F}(m),n} \simeq \idm{m, \mathsf{F}^{r}(n)} \quad \text{and} \quad \inc{n, \mathsf{F}(m)} \simeq \imc{\mathsf{F}^{r}(n),m}
\end{equation}
for $m \in \Cat{M}$ and $n \in \Cat{N}$. The analogous statement holds for left exact bimodule functors and 
the left adjoint functor. 
\end{lemma}

\paragraph{Compatibilities inner hom, dual categories and tensor product}

The inner hom and the dual categories are compatible in the following sense. Here and in the sequel we use the equivalences from Lemma \ref{lemma:dual-and-tensor} without mentioning. Recall the operation $\widetilde{\#}$ from 
equation (\ref{eq:values-semi-hash}).
\begin{lemma}
\label{lemma:inner-hom-and-duals}
\begin{lemmalist}
  \item \label{item:canon-bal-ldd}There is a canonical balanced bimodule natural isomorphism
between the  balanced bimodule functors 
\begin{equation}
  \label{eq:bal-d-val-inn}
  \idm{-,-}, \imldd{-,-}: \DMC \boxtimes \CMDld \rightarrow \DDD.
\end{equation} 
as well as between the balanced bimodule functors 
  \begin{equation}
    \label{eq:C-val-inn-equ}
 \imc{-,-},    \icmrd{-,-}: \CMDrd \boxtimes \DMC \rightarrow \CCC.
  \end{equation}
\item  There are canonical balanced bimodule natural isomorphism
between the  balanced bimodule functors 
\begin{equation}
  \label{eq:double-dual-inn-hom}
  \Funrrtd{(\idm{-,-})}, \imrdd{-,-}: \DMCrrd \boxtimes \CMDrd \rightarrow \DDD.
\end{equation}
and 
\begin{equation}
  \label{eq:double-left-inn-hom}
  \Funlltd{(\imc{-,-})}, \icmld{-,-}: \CMDld \boxtimes \DMClld. 
\end{equation}
\end{lemmalist}
\end{lemma}
\begin{proof}
Using the definitions of the inner homs we compute
  \begin{equation*}
    \begin{split}
      \Hom_{\Cat{D}}(\imldd{m,\widetilde{m}},d) &\simeq \Hom_{^{\#\!}\Cat{M}_{\Cat{D}}}(\widetilde{m}, m \ractld d) =\Hom_{\Cat{M}}(d^{*} \act m, \widetilde{m}) \\
& \simeq \Hom_{\Cat{M}}(m, d \act \widetilde{m}) = \Hom_{\Cat{D}}(\idm{m,\widetilde{m}},d). 
    \end{split}
  \end{equation*}
 Lemma \ref{lemma:module-yoneda} thus yields 
a balanced bimodule natural isomorphism   $\idm{m,\widetilde{m}} \simeq \imldd{m,\widetilde{m}}$.
 The argument for the $\Cat{C}$-valued inner hom follows directly from the first statement and Lemma \ref{lemma:dual-and-tensor}\refitem{item:double-dual}.
 For the second part, let $d \in \Cat{D}$, $m, \widetilde{m} \in \Cat{M}$. By the definition of the inner hom and the duality in $\Cat{D}$, we have the following chain of natural isomorphisms. 
  \begin{equation}
    \label{eq:teck-chain}
    \begin{split}
      \Hom_{\Cat{D}}(\imrdd{\widetilde{m},m}, d) \simeq & \Hom_{\Cat{M}^{\sss{\#}}}(m, \widetilde{m} \ractrd d) \simeq \Hom_{\Cat{M}}( \leftidx{^*}{d}{} \act \widetilde{m},m) \\
&\simeq \Hom_{\Cat{M}}(\widetilde{m}, \leftidx{^{**}}{d}{} \act m) \simeq \Hom_{\Cat{D}}(\idm{\widetilde{m},m}, \leftidx{^{**}}{d}{} )\\
& \simeq \Hom_{\Cat{D}}((\idm{\widetilde{m},m})^{**},d).
    \end{split}
  \end{equation}
All isomorphisms are multi-balanced bimodule natural isomorphisms and induce the required balanced bimodule natural isomorphism. The last statement follows again from the previous one and Lemma  \ref{lemma:dual-and-tensor}\refitem{item:double-dual}.
\end{proof}

Finally we discuss the compatibility of the inner hom and the tensor product of module categories. 

\begin{proposition}
\label{proposition:Rieffel}
  \begin{propositionlist}
  \item  The functor
 \begin{equation}
    \label{eq:functor-Verschachtelt-inner-hom}
    \begin{split}
      \Lambda: \DMC \boxtimes \CNE  \boxtimes \ENCld \boxtimes \CMDld & \rr \DDD \\
      m \boxtimes n\boxtimes \widetilde{n} \boxtimes \widetilde{m} & \rr \idm{ m  \ract \icn{n, \widetilde{n} }, \widetilde{m}}
    \end{split}
  \end{equation}
  is a  multi-balanced  $(\Cat{D},\Cat{D})$-bimodule functor. 
\item \label{item:Rieffel-ind}  The following diagram of multi-balanced module functors commutes up to a canonical multi-balanced module natural isomorphism
  \begin{equation}
\label{eq:Lambda-defines-inner-hom}
    \begin{tikzcd}
    \DMC \boxtimes \CNE  \boxtimes \ENCld \boxtimes \CMDld  \ar{d}{\mathsf{B} \boxtimes \Funld{B} } \ar{drr}[name=A]{\Lambda}&&\\
       \DMC \Box \CNE  \boxtimes \ENCld \Box \CMDld  \ar{rr}[below]{\idmn{-,-}} 
&  & \DDD.
    \end{tikzcd} 
  \end{equation}
\item Analogously, there exists a multi-balanced module natural isomorphism 
 \begin{equation}
\label{eq:right-inner-hom-box}
    \imne{\widetilde{m}\Box \widetilde{n}, m \Box n } \simeq \ine{\widetilde{n}, \inc{\widetilde{m},m} \act n}.
  \end{equation}
  \end{propositionlist}
\end{proposition}
\begin{proof}
  It follows directly from the properties of the inner hom functors, see Proposition \ref{proposition:properties-inner-hom},  that the functor $\Lambda$ has the structure of a multi-balanced bimodule functor. 
For the second statement, note that the functor $\mathsf{B}: \Cat{M} \boxtimes \Cat{N} \rightarrow \Cat{M} \Box \Cat{N}$ and the functor $\idmn{-,-}$ are multi-balanced, hence 
the composite of the functors on the lower arrows in (\ref{eq:Lambda-defines-inner-hom})
defines as well a multi-balanced bimodule functor.
To construct the multi-balanced natural isomorphism between these two functors, we proceed in two steps. First we proof that it suffices to show the result for $\Cat{D}= \CMstar=\Funl{\Cat{C}}{\MC,\MC}$.
Let $\mathsf{F}: \Cat{D} \rightarrow \CMstar$ be the tensor functor given by the action of $\Cat{D}$ on $\Cat{M}$, i.e. $d \act m = \mathsf{F}(m)$. This functor is exact, since the functor 
$\act: \Cat{D} \times \Cat{M} \rightarrow \Cat{M}$ is biexact. Hence the adjoint functors to $\mathsf{F}$ exist. The equation
\begin{equation*}
  \Hom_{\Cat{D}}(\idm{m,\widetilde{m}},d)= \Hom_{\Cat{M}}(m, d \act \widetilde{m}) = \Hom_{\Cat{C}_{\Cat{M}}^{*}}(\icdualm{m,\widetilde{m}},\mathsf{F}(d))
\end{equation*}
shows that $\idm{m,\widetilde{m}}= \mathsf{F}^{l}(\icdualm{m,\widetilde{m}})$. 

Assume now that the statement is proven for $\CMstar$, then 
\begin{equation}
  \label{eq:Cmdual-then-D}
  \begin{split}
      \idmn{m \Box n, \widetilde{m}\Box \widetilde{n}}&= \mathsf{F}^{l}(\icdualmn{m \Box n, \widetilde{m}\Box n}) \simeq \mathsf{F}^{l}(\icdualm{m \ract \icn{n,\widetilde{n}}, \widetilde{m}})\\
  &= \idm{m \ract \icn{n,\widetilde{n}},\widetilde{m}}
 \end{split}
\end{equation}
shows that the statement follows for $\Cat{D}$. 

Next we prove the assertion for $\Cat{D}= \CMstar$ in the case that $\MC=\AModC$ and $ \CN = \ModCB$ 
for algebras $A,B \in \Cat{C}$. Recall from Example \ref{example:inner-homs-algebra} that in this case $\CMstar=\AModCA$. By Theorem \ref{theorem:existence-tensor}, the tensor product $\MC \Box \CN$ is given by the 
category of $(A,B)$-bimodules in $\Cat{C}$ with universal balancing functor given by the tensor product in $\Cat{C}$. Let $m, \widetilde{m} \in \AModC$ and $n, \widetilde{n} \in \ModCB$ and $x \in \AModCA$.  
\begin{equation}
  \label{eq:inner-hom-mn-left}
  \begin{split}
      \Hom_{_{A\!} \mathsf{Mod}(\Cat{C})_{B}}(\icdualmn{m \otimes n, \widetilde{m} \otimes \widetilde{n}}, x) &=\Hom_{_{A\!} \mathsf{Mod}(\Cat{C})_{B}}(m \otimes n , (x \tensor{A} \widetilde{m})\otimes \widetilde{ n}) \\
&=   \Hom_{_{A\!} \mathsf{Mod}(\Cat{C})} (m \otimes ( n \tensor{B} \leftidx{^*}{\widetilde{n}}{}), x \tensor{A} \widetilde{m}) \\
&=   \Hom_{_{A\!} \mathsf{Mod}(\Cat{C})_{A}} (A, x \tensor{A} \widetilde{m} \otimes (m \otimes (n \tensor{B} \leftidx{^*}{\widetilde{n}}{}))^{*})\\
&=  \Hom_{_{A\!} \mathsf{Mod}(\Cat{C})_{A}} (\leftidx{^*}{(\widetilde{m} \otimes (m \otimes (n \tensor{B} \leftidx{^*}{n}{}))^{*})  }{},x)
  \end{split}
\end{equation}
 The description of the inner hom objects from Example \ref{example:inner-homs-algebra} shows that on the other side 
 \begin{equation}
   \label{eq:inner-hom-MN-alg}
   \icdualm{m \ract \icn{n,\widetilde{n}}, \widetilde{m}}= \icdualm{m \ract (n \tensor{B} \leftidx{^*}{n}{}), \widetilde{m}}= \leftidx{^*}{(\widetilde{m} \otimes (m \otimes (n \tensor{B} \leftidx{^*}{n}{}))^{*} ) }{}
 \end{equation}
This completes the second part in the case that $\MC= \AModC$. Since every module category is equivalent to one of this type, see Theorem \ref{theorem:all-alg}, the statement holds in general.   The third statement follows directly 
by applying the second part  to the $\Cat{E}$-valued inner hom of $(\DMC \Box \CNE)^{\scriptscriptstyle{\#}}$ and using Lemma \ref{lemma:inner-hom-and-duals}.
\end{proof}
As we remarked in Section \ref{sec:Prelim}, the inner hom can be regarded as a categorification of an algebra valued  inner product except one compatibility with the $*$-involution. 
In view of this analogy,  the functor (\ref{eq:Lambda-defines-inner-hom}) can be regarded as a categorification of the Rieffel induction formula (\ref{eq:Rieffel-ind-inn}).

\subsection{$\#$-duals for bimodule categories using inner homs}

We finally show using the inner hom functors, that the dual categories are indeed $\#$-duals in the tricategory $\BimCat$.

\paragraph{The evaluation functors}
Recall from Proposition \ref{proposition:properties-inner-hom}, that the inner hom functors for a bimodule category $\DMC$ are  right exact balanced bimodule functors
\begin{equation}
  \label{eq:inner-homs}
  \idm{-,-}: \DMC \times \CMDld \rightarrow \DDD, \quad \imc{-,-}: \CMDrd \times \DMC \rightarrow \CCC.
\end{equation}
Using the universal property of the tensor product of bimodule categories we obtain the following functors. 
\begin{definition}
 Let $\DMC$ be a bimodule category. The inner hom functors induce  bimodule functors
  \begin{equation}
    \label{eq:eval-bimod}
    \ev{_{\Cat{D\!\!}}\Cat{M}}: \DMC \Box \CMDld \rightarrow \DDD, \quad \ev{\Cat{M}_{\Cat{C}}}: \CMDrd \Box \DMC \rightarrow \CCC,
  \end{equation}
that are  called the left, respectively right evaluation functors. 
\end{definition}
From the definition of the tensor product it follows that the evaluation functors are right exact.

By applying the universal property of the tensor product to Lemma \ref{lemma:inner-hom-and-duals} and  Proposition \ref{proposition:Rieffel}, 
we obtain the following compatibilities of the evaluation functors with   the dual categories and the tensor product. 
\begin{corollary}
   \label{corollary:kappa-und-Box}
 Let  $\DMC$ and $\CNE$ be two  bimodule categories.  
 \begin{corollarylist}
\item  \label{item:ev-duals-comp}There are bimodule natural isomorphisms between the bimodule functors 
\begin{equation}
  \label{eq:eval-duals}
  \ev{_{\Cat{D\!\!}}\Cat{M}}, \ev{^{\scriptscriptstyle{\#}\!}\Cat{M}_{\Cat{D}}}: \DMC \Box \CMDld \rightarrow \DDD \quad \text{and} \quad \ev{\Cat{M}_{\Cat{C}}},
\ev{_{\Cat{C}\!\!}\Cat{M}^{\scriptscriptstyle{\#}}}: \CMDrd \Box \DMC \rightarrow \CCC.
\end{equation}
   \item The bimodule functor $\ev{_{\Cat{D\!\!}}\Cat{M} \Box \Cat{N}}: \DMC\Box \CNE \Box ^{\scriptscriptstyle{\#}}(\DMC \Box \CNE) \rightarrow \DDD$ is equivalent as a bimodule functor to the composite
     \begin{equation}
       \label{eq:evl-and-box}
       \begin{tikzcd}
       \Cat{M} \Box \Cat{N} \Box  ^{\scriptscriptstyle{\#}}(\Cat{M} \Box \Cat{N}) \simeq  \Cat{M} \Box \Cat{N} \Box \Nld \Box \Mld 
\ar{r}{1 \Box \ev{_{\Cat{C\!\!}}\Cat{N}} \Box 1} & \Cat{M} \Box \Cat{C} \Box \Mld \ar{r}{r_{\Cat{M}} \Box 1} & \Cat{M} \Box \Mld \ar{r}{\ev{_{\Cat{D\!\!}}\Cat{M}}} &\Cat{D}. 
       \end{tikzcd}
     \end{equation}
\item The bimodule functor $\ev{\Cat{M} \Box \NE}: (\Cat{M} \Box \Cat{N})^{\scriptscriptstyle{\#}} \Box \Cat{M} \Box \Cat{N} \rightarrow \Cat{E}$ is equivalent as a bimodule functor to the composite
  \begin{equation}
    \label{eq:evr-box}
    \begin{tikzcd}
     (\Cat{M} \Box \Cat{N})^{\scriptscriptstyle{\#}} \Box \Cat{M} \Box \Cat{N}  \simeq  \Nrd \Box \Mrd \Box \Cat{M} \Box \Cat{N}   \ar{r}{1 \Box \ev{\Cat{M}_{\Cat{C}}} \Box 1} & 
 \Nrd \Box \Cat{C} \Box  \Cat{N}  \ar{r}{1 \Box l_{\Cat{N}}} &  \Nrd         \Box \Cat{N}         \ar{r}{\ev{\Cat{N}_{\Cat{E}}}} & \Cat{E}. 
    \end{tikzcd}
  \end{equation}
 \end{corollarylist}
In these formulas we suppressed the associativity functors of the tricategory $\BimCat$.
\end{corollary}

Furthermore we have the following version of the Yoneda-lemma and the representations of module functors by applying the tensor product to Lemma \ref{lemma:module-yoneda}\refitem{item:Yoneda-inner-hom}
and to Proposition \ref{proposition:D-represenations-functors}\refitem{item:repres-bimod-func}.
\begin{lemma}
  \label{lemma:ev-representablility}
 Let $\DMC$, $\DNC$ and $\DKD$ be  bimodule categories. 
  \begin{lemmalist}
    \item \label{item:ev-yoneda} For bimodule functors $\mathsf{F}, \mathsf{G}: \DMC \rightarrow \DNC$ the set of bimodule natural transformations from $\mathsf{F}$ to $\mathsf{G}$ is in bijection 
with the set of balanced bimodule transformations 
\begin{equation}
  \label{eq:bal-inner-hom-eval-yoned}
  \ev{_{\Cat{D}\!\!}\Cat{N}}\circ (\mathsf{F} \Box 1) \rightarrow \ev{_{\Cat{D}\!\!}\Cat{N}} \circ (\mathsf{G} \Box 1)  \quad \text{between functors} \quad \DMC \Box \CNDld \rightarrow \DDD,
\end{equation}
\item \label{item:repres-ev} Let $\mathsf{F}: \DKD \rightarrow \DDD$ be a right exact bimodule functor. Then $\mathsf{F}$ is equivalent as a bimodule functor to the composites 
  \begin{equation}
    \label{eq:comp-fundd}
    \begin{tikzcd}
      \DKD \simeq \DKD \Box \DDD \ar{r}{\Funrtd{(\mathsf{F}^{r})}} & \DKD \Box \DKDld \ar{r}{\ev{\Cat{K}_{\Cat{D}}}} & \DDD,
    \end{tikzcd}
  \end{equation}
as well as to 
\begin{equation}
  \label{eq:comp-fundd-sec}
  \begin{tikzcd}
    \DKD \simeq \DDD \Box \DKD \ar{r}{\Funltd{(\mathsf{F}^{r})}} & \DKDrd \Box \DDD \ar{r}{\ev{_{\Cat{D}\!}\Cat{K}}} & \DDD.
  \end{tikzcd}
\end{equation}
  \end{lemmalist}
\end{lemma}

\paragraph{The coevaluation functors}

In analogy to the case of module categories over $\Vect$ in the introduction to this section, we  
define the coevaluation functor using a representing object for the inner hom functors. Recall that the evaluation functors are right exact, hence there exists a right adjoint. 
\begin{definition}
  Let $\DMC$ be an  bimodule category. The left coevaluation functor is defined by
  \begin{equation}
    \label{eq:coev-left}
    \coev{_{\Cat{D}\!\!}\Cat{M}}= \Funrtd{( \ev{_{\Cat{D\!\!}}\Cat{M}}^{r})}: \DDD  \longrightarrow (\DMC \Box \CMDld)^{\scriptscriptstyle{\#}} \simeq \DMC \Box \CMDrd, 
  \end{equation}
while the right coevaluation functor is defined by
 \begin{equation}
    \label{eq:coev-right}
    \coev{\Cat{M}_{\Cat{C}}}= \Funltd{( \ev{\Cat{M}_{\Cat{C}}}^{r}  )}: \CCC  \longrightarrow  ^{\scriptscriptstyle{\#} \!\!}(\CMDrd \Box \DMC) \simeq  \CMDld \Box \DMC.
  \end{equation}
In these formulas we used the equivalences from Lemma \ref{lemma:dual-and-tensor}. 
\end{definition}

First we clarify the compatibility of the coevaluation functors and the dual bimodule categories.
\begin{lemma}
  \label{lemma:comp-dual-coev}
For a bimodule category  $\DMC$ the bimodule functors 
$$\coev{_{\Cat{D}\!}\Cat{M}},  \coev{\Cat{M}^{\sss{\#}}_{\Cat{D}}}: \Cat{D} \rightarrow \DM \Box \MDrd, \quad \text{and}\quad \coev{\Cat{M}_{\Cat{C}}}, \coev{_{\Cat{C} \!}^{\sss{\#}\!}\Cat{M}}:\Cat{C} \rightarrow \CMld \Box \MC   $$ 
are equivalent as bimodule functors. 
\end{lemma}
\begin{proof} 
To construct the first natural isomorphism  it is sufficient to show that 
$\Funrrtd{(\ev{_{\Cat{D}\!}\Cat{M}}^{r})} \simeq (\ev{\Cat{M}^{\sss{\#}}_{\Cat{D}}})^{r}$ as bimodule functors. It is straightforward to see that taking the right adjoint and the operation 
$\Funrrtd{(-)}$ commute up to a natural bimodule isomorphism. 
Hence it suffices to show that $\Funrrtd{(\ev{_{\Cat{D}\!}\Cat{M}})} \simeq \ev{\Cat{M}^{\sss{\#}}_{\Cat{D}}}$. This follows from Lemma \ref{lemma:inner-hom-and-duals} \refitem{item:canon-bal-ldd}. The remaining isomorphism for the $\Cat{C}$-valued inner hom follows directly from the first statement using
the equivalence $(\CMld)^{\sss{\#}} \simeq \MC$. 
\end{proof}

\begin{lemma}
  \label{lemma:representing-inner-hom}
  \begin{lemmalist}
    \item The coevaluation functors represent the inner hom functors, i.e. there are natural bimodule isomorphism 
\begin{equation}
  \label{eq:inner-hom-repres}
  \idmmld{-, \coev{_{\Cat{D}\!\!}\Cat{M}}(\unit)} \simeq \ev{_{\Cat{D\!\!}}\Cat{M}}(-) \simeq \idmmrd{\coev{_{\Cat{D}\!\!}\Cat{M}}(\unit), -},
\end{equation}
between bimodule functors $\DMC \Box \CMDld \rightarrow \DDD$.
Furthermore there are natural bimodule isomorphisms 
\begin{equation}
  \label{eq:inner-homC-rep}
  \icmrdm{-, \coev{\Cat{M}_{\Cat{C}}}(\unit)} \simeq \ev{\Cat{M}_{\Cat{C}}} \simeq \imrdmc{\coev{\Cat{M}_{\Cat{C}}}(\unit), -}
\end{equation}
between bimodule functors $\DMC \Box \CMDld \rightarrow \DDD$.
\item \label{item:bimodule-coev-ev} The composite  bimodule functor
  \begin{equation}
    \label{eq:composite-coev-ev}
    \begin{tikzcd}[column sep=large]
          \DMC \Box \CMDld \simeq \Cat{D} \Box \Cat{M} \Box \Mld  \ar{r}{\coev{_{\Cat{D}\!\!}\Cat{M}}\Box 1} & ( \Cat{M} \Box \Mld)^{\scriptscriptstyle{\#}} \Box \Cat{M} \Box \Mld \ar{r}[yshift=1pt]{\ev{\Cat{M} \Box ^{\scriptscriptstyle{\#}\!}\Cat{M}_{\Cat{D}}}} & \Cat{D}
    \end{tikzcd}
  \end{equation}
is equivalent to the evaluation functor $\ev{_{\Cat{D\!\!}}\Cat{M}}$ as bimodule functor.
  \end{lemmalist}
The composite bimodule functor 
\begin{equation}
  \label{eq:comp-C-coev}
  \begin{tikzcd}[column sep=large]
    \CMDrd \Box \DMC \simeq  \Mrd \Box \Cat{M} \Box \Cat{C} \ar{r}[yshift=1pt]{1 \Box \coev{\Cat{M}_{\Cat{C}}}} & \Mrd \Box \Cat{M} \Box ^{\scriptscriptstyle{\#}\!\!}(\Mrd \Box \Cat{M})  
\ar{r}{\ev{_{\Cat{C}\!\!}^{\scriptscriptstyle{\#}\!\!}\Cat{M} \Box \Cat{M}}} &\Cat{C}
  \end{tikzcd} 
\end{equation}
is equivalent to $\ev{\Cat{M}_{\Cat{C}}}$ as a  bimodule functor. 
\end{lemma}
\begin{proof}
  The first part  follows  directly from the  definition of the coevaluation functor  and from  Proposition \ref{proposition:D-represenations-functors}. The second part
 follows from applying Lemma \ref{lemma:ev-representablility}\refitem{item:repres-ev}.
\end{proof}

\paragraph{The triangulators}
We finally show that the evaluation and coevaluation functors satisfy the snake identities up to a bimodule natural isomorphism. This is shown in \cite[Proposition 4.2.1]{DSS} using 
the description of the tensor product as functor category.  Our method of proof is a  generalization of the semisimple case considered in  \cite{Schaum}. 
\begin{proposition}
\label{proposition:triang}
Let $\DMC$ be a  bimodule category.  There exists  a  bimodule natural isomorphism between the composite
 \begin{equation}
  \label{eq:definition-Phi-erst}
  \begin{tikzcd}[column sep=large]
  \Phi_{\Cat{M}}:  \Cat{M}\simeq \Cat{D} \Box \Cat{M} \ar{r}{\coev{_{\Cat{D}\!\!}\Cat{M}} \Box 1} & 
\Cat{M} \Box \Mrd \Box \Cat{M} \ar{r}{1 \Box\ev{\Cat{M}_{\Cat{C}}}} &\Cat{M} \Box \Cat{C}\simeq \Cat{M},
  \end{tikzcd}
\end{equation}
and the identity functor on $\DMC$. Similarly, the bimodule functor 
\begin{equation}
  \label{eq:bimod-other}
  \begin{tikzcd}[column sep=large]
    \Cat{M} \simeq  \Cat{M} \Box \Cat{C} \ar{r}[yshift=1pt]{1 \Box \coev{\Cat{M}_{\Cat{C}}}} & \Cat{M} \Box \Mld \Box \Cat{M} \ar{r}{\ev{\Cat{D}\!\!\Cat{M}} \Box 1} & \Cat{D} \Box \Cat{M} \simeq \Cat{M},
  \end{tikzcd}
\end{equation}
is equivalent to the identity bimodule functor $\DMC$. These natural isomorphisms are  called left and right triangulators, respectively. 
\end{proposition} 
\begin{proof}
  We show that there exists a bimodule natural isomorphisms between $\ev{_{\Cat{D\!\!}}\Cat{M}} \circ (\Phi_{\Cat{M}} \Box 1) $ and $\ev{_{\Cat{D\!\!}}\Cat{M}}$ as 
bimodule functors from $\DMC \Box \CMDld$ to $\DDD$. Then the statement will follow by Lemma \ref{lemma:ev-representablility} \refitem{item:ev-yoneda}.
According to Corollary \ref{corollary:kappa-und-Box}, the functor $\ev{\Cat{M} \Box ^{\scriptscriptstyle{\#}\!}\Cat{M}_{\Cat{D}} }: \Cat{M} \Box \Mrd \Box \Cat{M} \Box \Mrd \rightarrow \Cat{D}$ is isomorphic 
as bimodule functor to $ \ev{^{\scriptscriptstyle{\#}\!}\Cat{M}_{\Cat{D}}} \circ(1 \Box \ev{\Cat{M}_{\Cat{C}}} \Box 1)$, and hence according to equation (\ref{eq:eval-duals}) to $\ev{_{\Cat{D\!\!}}\Cat{M}} \circ (1 \Box \ev{\Cat{M}_{\Cat{C}}} \Box 1)$.
In these formulas we suppressed the unit bimodule functors for simplicity. 
This implies that  Lemma \ref{lemma:representing-inner-hom}\refitem{item:bimodule-coev-ev} gives a bimodule natural isomorphism  $\ev{_{\Cat{D\!\!}}\Cat{M}} \circ (\Phi_{\Cat{M}} \Box 1) \rightarrow \ev{_{\Cat{D\!\!}}\Cat{M}}$.
The second natural bimodule isomorphism is constructed analogously. 
\end{proof}
With the help of the evaluation and coevaluation functors we obtain $\#$-duals for $\BimCat$.
\begin{theorem}
\label{theorem:hash-duals-bimcat}
   Bimodule categories over finite tensor categories form a tricategory with  $\#$-duals.
\end{theorem}
\begin{proof}
 By Proposition \ref{proposition:triang}, the dual bimodule categories 
with the evaluation and coevaluation functors  satisfy the first snake identity  (\ref{eq:duality-snake-right})
 for the  $\#$-duals up to the triangulator. The remaining triangulator for $\CMDrd$ is constructed as follows. 
By Lemma \ref{lemma:comp-dual-coev} and Corollary \ref{corollary:kappa-und-Box}\refitem{item:ev-duals-comp} the 
functor  
\begin{equation}
  \label{eq:rem-snake}
  \begin{tikzcd}[column sep=large]
    \CMDrd \simeq \Mrd \Box \Cat{D} \ar{r}{1 \Box \coev{_{\Cat{D}\!}\Cat{M}}} & \Mrd \Box \Cat{M} \Box \Mrd \ar{r}{\ev{\Cat{M}_{\Cat{C}}} \Box 1} & \Cat{C} \Box \Mrd \simeq \CMDrd,
  \end{tikzcd}
\end{equation}
is equivalent as bimodule functor to the composite
\begin{equation}
  \label{eq:rem-snake-sec}
  \begin{tikzcd}[column sep=large]
    \CMDrd \simeq \Mrd \Box \Cat{D} \ar{r}{1 \Box \coev{\Cat{M}_{\Cat{D}}^{\sss{\#}}}} & \Mrd \Box \Cat{M} \Box \Mrd \ar{r}{\ev{_{\Cat{C}\!}\Cat{M}^{\sss{\#}}} \Box 1} & \Cat{C} \Box \Mrd \simeq \CMDrd.
  \end{tikzcd}
\end{equation}
 Equation (\ref{eq:bimod-other}) in  Proposition \ref{proposition:triang}  applied to $\CMDrd$  shows that  this bimodule functor is equivalent to the identity. 
\end{proof}

\subsection{Separable bimodule categories and  the Serre equivalence}
First we show that separable bimodule categories form a tricategory with duals. This implies that the left and right dual of a bimodule category are equivalent by Proposition \ref{proposition:serre-ex}.  We characterize 
these so called Serre equivalences using the inner homs. This might be important for applications to TFTs, since the Serre equivalences encode the homotopy action corresponding to the framing change in a framed TFT \cite{Lurie, DSS}.

\paragraph{The tricategory of separable bimodule categories}

 It follows almost directly from results in \cite{DSS} that separable bimodule categories form a tricategory with duals. 
We just need to proof that separable bimodule categories are biexact and hence the 
adjoints of all module functors exists.  
\begin{definition}
  \label{definition:separable}
Let $\Cat{C}$, $\Cat{D}$ be finite tensor categories. 
  \begin{definitionlist}
    \item An algebra $A$ in  $\Cat{C}$ is called separable, if the multiplication $m: A \otimes A \rightarrow A$ splits as a map of $(A,A)$-bimodules, i.e. 
if there exists a bimodule morphism $s: A \rightarrow A \otimes A$ such that $m \circ s= 1_{A}$. 
\item A module category $\DM$ over  $\Cat{D}$ is called separable if there exists a separable algebra $A \in \Cat{D}$ such that $\DM \simeq \ModDA$ as module categories. 
\item A bimodule category $\DMC$ is called separable if it is separable as $\Cat{D}$-left and also as $\Cat{C}$-right module category. 
   \end{definitionlist}
\end{definition}

\begin{lemma}
  \label{lemma:sep-ex}
If $\DM$ is a separable module category then it is also an exact module category. 
\end{lemma}
\begin{proof}
  First we choose a separable algebra $A$ and an equivalence $ \DM \simeq \ModDA$ as module categories. To show that $\ModDA$ is exact we need to show that for all projective $P \in \Cat{D}$ and all $m \in \ModDA$, the object $P \otimes m $
is projective in $\ModDA$. Clearly $ P \otimes m$ is projective in $\Cat{D}$. Now consider 
the following morphism  in $\ModDA$.
\begin{equation}
  \label{eq:morph-sep-ex}
  P \otimes m \simeq P \otimes m \tensor{A} A \stackrel{1 \tensor{A} s}{\rightarrow} P \otimes m \tensor{A}(A \otimes A) \simeq P \otimes m \otimes A.
\end{equation}
This establishes $P \otimes m$ as retract of $P \otimes m \otimes A$.
Since the object $P \otimes m \otimes A$ is projective in $\ModDA$, it follows that also $P \otimes m$ is projective in $\ModDA$.
\end{proof}
\begin{proposition}
  \label{proposition:sep-and-tenosr}
Let $\DMC$ and $\CNE$ be separable bimodule categories. Then the $(\Cat{D},\Cat{E})$-bimodule category $\DMC \Box \CNE$ is separable. 
\end{proposition}
\begin{proof}
  This is shown in \cite[Thm. 3.5.5]{DSS} in the case that the tensor categories $\Cat{D}$, $\Cat{C}$ and $\Cat{E}$ are semisimple. The proof does not require this assumption and thus the result follows 
in the general case in exactly the same way. 
\end{proof}
\begin{theorem}
\label{theorem:sep-tric-duals}
  \begin{theoremlist}
  \item The following defines a tricategory $\BimCat^{sep}$. Objects are finite tensor categories, 1-morphisms are separable bimodule categories, 2-morphisms are bimodule functors, 3-morphisms are bimodule natural transformations. 
The tricategory structures are induced from $\BimCat$.
\item \label{item:duals-functors-sep}$\BimCat^{sep}$ is a tricategory with duals, where the duals of a 1-morphism $\DMC$ are the bimodule categories $\CMDrd$ and $\CMDld$. 
The duals of bimodule functors are the left and right adjoint functors.   
  \end{theoremlist}
  \end{theorem}
  \begin{proof}
    We have to show that the tricategory structures of $\BimCat$ are well defined on $\BimCat^{sep}$. The unit bimodule category $\CCC$ is a separable bimodule category, since we can choose  the tensor unit as separable algebra 
for the left and right module structures. According to Proposition \ref{proposition:sep-and-tenosr}, the tensor product is well defined for $\BimCat^{sep}$. Hence Theorem \ref{theorem:Bimcat-3-cat} implies that $\BimCat^{sep}$
is a tricategory. For the second part note that if $A$ is a separable algebra in a finite tensor category $\Cat{D}$, then also $A^{**}$ and $\leftidx{^{**}}{A}{}$ are canonically separable algebras.  
According to \cite[Cor. 3.4.14]{DSS}, if $\DM$ is equivalent to $\ModDA$, then $\DMrd \simeq \AddModD$ and $\DMld \simeq \AModD$ and analogous for right module categories. It thus follows that $\CMDrd$ and $\CMDld$ are separable 
bimodule categories if $\DMC$ is. Thus it follows from Theorem \ref{theorem:hash-duals-bimcat} that $\BimCat^{sep}$ is a tricategory with $\#$-duals. 
From  Lemma \ref{lemma:sep-ex} and it follows that a separable bimodule category is biexact and hence exact according to Lemma  \ref{lemma:biex-then-ex}. 
Thus every module functor between separable bimodule categories is exact and has a left and right adjoint. 
  \end{proof}

\paragraph{The Serre equivalence}

It follows from Theorem \ref{theorem:hash-duals-bimcat} and the general Proposition  \ref{proposition:serre-ex} that for a separable bimodule category $\DMC$, there is an equivalence $\mathsf{S}_{\Cat{M}}: \CMDrd \rightarrow \CMDld$ of bimodule categories. We apply the calculus of the inner hom to construct this equivalence more explicitly. 

First we  consider the duals of the inner hom and their properties. By applying the operation $\widetilde{\#}$ from Equation (\ref{eq:induce-dual-unit}) to the inner hom functor $\idm{-,-}$, we obtain  balanced 
bimodule functors 

\begin{equation}
  \label{eq:dual-inn-hom}
  \begin{split}
      \Funrtd{(\idm{-,-})}:& \DMC \boxtimes \CMDrd \ni m \boxtimes \widetilde{m} \mapsto \idmstar{\widetilde{m},m} \in \DDD, \\
 \Funltd{(\imc{-,-})}: & \CMDld \boxtimes \DMC \ni \widetilde{m} \boxtimes m \mapsto \imcstar{m,\widetilde{m}} \in \CCC.
  \end{split}
\end{equation} 
The following is the analogue of Proposition \ref{proposition:D-represenations-functors}.
\begin{proposition}
\label{proposition:ext-D-left-reps}
\begin{propositionlist}
  \item Let $\mathsf{F}: \DM \rightarrow \LDD$ be an left exact module functor. There exists a module natural isomorphism
\begin{equation}
  \label{eq:module-dual-rep}
 \idm{\mathsf{F}^{l}(\unit_{\Cat{D}}),x} \simeq  \leftidx{^*}{\mathsf{F}(x)}{},
\end{equation}
for all $x \in \DM$. 
This extends to an equivalence of bimodule functors 
\begin{equation}
\label{eq:equiv-Funl-inn}
  \idm{\mathsf{F}^{l}(-),-} \simeq (-) \otimes \Funltd{\mathsf{F}}: \LDD \boxtimes \MDld \rightarrow \DDD.
\end{equation}
\item 
Let $\mathsf{F}: \DMD \rightarrow \DDD$ be a left exact bimodule functor. The equivalence from equation (\ref{eq:equiv-Funl-inn}) is an equivalence of bimodule functors $\DDD \boxtimes \DMDld \rightarrow \DDD$. 
Furthermore, there is an
 equivalence of balanced bimodule functors 
 \begin{equation}
   \imd{-, \mathsf{F}^{l}(-)} \simeq \Funrtd{\mathsf{F}}(-)\otimes (-): \DMDrd \boxtimes \DDD \rightarrow \DDD. 
 \end{equation}
\end{propositionlist}
\end{proposition}
\begin{proof}
  We compute for $d, x \in \Cat{D}$ and $\widetilde{m} \in \MDrd$
  \begin{equation}
    \label{eq:comp-dual-rep}
    \begin{split}
    \Hom_{\Cat{D}}(\idm{\mathsf{F}^{l}(d), \widetilde{m}},x) \simeq \Hom_{\Cat{M}}(\mathsf{F}^{l}(d), x \act \widetilde{m}) \simeq \Hom_{\Cat{D}}(d \otimes \Funltd{\mathsf{F}}(\widetilde{m}), x).
    \end{split}
  \end{equation}
All isomorphisms are balanced natural isomorphisms between functors and hence the chain of isomorphisms induces the claimed module natural isomorphism by the module Yoneda lemma \ref{lemma:module-yoneda}. 
The second part is shown analogously. 
\end{proof}
According to Theorem \ref{theorem:sep-tric-duals}\refitem{item:duals-functors-sep}, the  left adjoint of the evaluation functor exists for separable bimodule categories. 
\begin{definition}
  \label{definition:Serre-fun}
Let $\DMC$ be an separable bimodule category. The $\Cat{D}$-Serre bimodule functor  $\mathsf{S}_{ _{\Cat{D}\!}\Cat{M}}: \CMDrd \rightarrow \CMDld$ is defined as the composite
\begin{equation}
  \label{eq:Serre-bimodule}
  \begin{tikzcd}
  \mathsf{S}_{ _{\Cat{D}\!}\Cat{M}}:  \CMDrd  \simeq \Mrd \Box \Cat{D} \ar{r}{1 \Box \ev{_{\Cat{D\!\!}}\Cat{M}}^{l}}  &  \Mrd \Box \Cat{M} \Box \Mld \ar{r}{\ev{\Cat{M}_{\Cat{C}}} \Box 1} &
   \Cat{C}  \Box \Mld \simeq \CMDld. 
   \end{tikzcd}
\end{equation}
The $\Cat{C}$-Serre bimodule functor $\mathsf{S}_{\Cat{M}_{\Cat{C}}}: \CMDld \rightarrow \CMDrd$ is the composite 
\begin{equation}
  \label{eq:serre-c}
  \begin{tikzcd}
    \mathsf{S}_{\Cat{M}_{\Cat{C}}}: \CMDld \simeq \Cat{C} \Box \Mld \ar{r}{ \ev{\Cat{M}_{\Cat{C}}}^{l} \Box 1 } & \Mrd \Box \Cat{M} \Box \Mld \ar{r}{1 \Box \ev{ _{\Cat{D}\!}\Cat{M}}} & \Mrd \Box \Cat{D} \simeq \CMDrd .
  \end{tikzcd}
\end{equation}
\end{definition}

\begin{proposition}
\begin{propositionlist}
  \item The Serre bimodule functors from Definition \ref{definition:Serre-fun} are equivalences of bimodule categories that are defined up to unique bimodule natural isomorphisms by the following properties. 
\item  For $m \boxtimes \widetilde{m} \in \DMC \boxtimes \CMDrd$ there is a canonical balanced bimodule natural isomorphism
  \begin{equation}
    \label{eq:Serre-prop}
    \idmstar{\widetilde{m},m} \simeq \idm{m, \mathsf{S}_{ _{\Cat{D}\!}\Cat{M}}(\widetilde{m})},
  \end{equation}
between balanced bimodule functors $\DMC \boxtimes \CMDrd \rightarrow \DDD$.
\item  There exists a balanced bimodule natural isomorphism 
\begin{equation}
  \label{eq:Serre-C-prop}
  \begin{tikzcd}
    \imcstar{m, \widetilde{m}} \simeq \imc{ \mathsf{S}_{\Cat{M}_{\Cat{C}}}(\widetilde{m}),m}
  \end{tikzcd}
\end{equation}
between balanced bimodule functors $\CMDld \boxtimes \DMC \rightarrow \CCC$. 
\end{propositionlist}
 \end{proposition}
\begin{proof}
To show that the Serre functors are equivalences, we argue that the quasi-inverse of the $\Cat{D}$-Serre functor is given by the composite 
\begin{equation}
  \label{eq:quasi-inv-DSerr}
  \begin{tikzcd}
    \CMDld \simeq \Mld \Box \Cat{D} \ar{r}{1 \Box \coev{ _{\Cat{D}\!}\Cat{M}}} & \Mld \Box \Cat{M} \Box \Mrd \ar{r}{\coev{\Cat{M}_{\Cat{C}}}^{l}\Box 1} & \Cat{C} \Box \Mrd \simeq \CMDrd.
  \end{tikzcd}
\end{equation}
To see that this functor is quasi-inverse to $\mathsf{S}_{ _{\Cat{D}\!}\Cat{M}}$, one uses first the coherence structure 
in the tricategory $\BimCat$, then twice the triangulator from Proposition \ref{proposition:triang}. The quasi-inverse of $\mathsf{S}_{\Cat{M}_{\Cat{C}}}$ is 
constructed similarly. 
The proof of the remaining statements is  analogous to the proof of Proposition \ref{proposition:triang}:  To show equation (\ref{eq:Serre-prop}), first note that the composite 
 \begin{equation}
    \label{eq:composite-coev-ev-serre}
    \begin{tikzcd}
          \DMC \Box \CMDrd \simeq \Cat{M} \Box \Mld \Box \Cat{D} \ar{r}{ 1 \Box \ev{_{\Cat{D}\!\!}\Cat{M}}^{l}} & ( \Cat{M} \Box \Mld)^{\scriptscriptstyle{\#}} \Box \Cat{M} \Box \Mld \ar{r}{\ev{\Cat{M} \Box ^{\scriptscriptstyle{\#}\!}\Cat{M}_{\Cat{D}}}} & \Cat{D}
    \end{tikzcd}
  \end{equation}
is equivalent to the bimodule functor $\Funrtd{(\ev{_{\Cat{D}\!\!}\Cat{M}})}$ by Proposition \ref{proposition:ext-D-left-reps}.  
This induces an equivalence of bimodule functors $\ev{_{\Cat{D}\!}\Cat{M}} \circ ( 1 \Box S_{\Cat{M}}) \simeq \Funrtd{(\ev{_{\Cat{D}\!\!}\Cat{M}})}$. Thus the statement follows by  using the universality of the tensor product. The proof of (\ref{eq:Serre-C-prop}) is analogous. By the module Yoneda-lemma, these properties characterize the Serre functors up to a unique equivalence of bimodule functors. 
 \end{proof}  
\section{Inner-product bimodule categories as pivotal tricategory}
\label{sec:inner-prod}
For the rest of this article we consider bimodule categories over pivotal finite tensor categories. 
We  develop the notion of inner-product bimodule category over pivotal finite tensor categories and show that it is compatible with the tensor product and the duality operations $\#$.  In this way we obtain a tricategory
$\BimCat^{\theta}$ of inner-product bimodule categories with objects pivotal finite tensor categories, 1-morphisms inner-product bimodule categories, 2- and 3-morphisms as in  $\BimCat$.
Furthermore the structure of inner-product bimodule categories induces pivotal structures on the categories of functors $\BimCat^{\theta}$. This implies that  inner-product bimodule categories
 over pivotal finite tensor categories form a pivotal tricategory and thus exhibit structures that we expect from defects for oriented TFTs.

\subsection{Inner-product bimodule categories and the tensor product}

We define inner-product bimodule categories over pivotal finite tensor categories, investigate the interaction with the tensor product and show that this structure is essentially unique for indecomposable module categories if it exists. 
 Furthermore we relate this structure to the Serre bimodule functors.

First we consider the dual bimodule categories for bimodule categories over pivotal tensor categories $\Cat{C}$ and $\Cat{D}$. Recall that a  pivotal structure on a tensor category is a monoidal natural isomorphism $a: \id \rightarrow (-)^{**}$, see Definition \ref{definition:pivotal-str}.  
For a  bimodule category $\DMC$, the pivotal structures of $\Cat{C}$ and $\Cat{D}$ define canonical 
equivalences of bimodule categories $A_{\Cat{M}}: \CMDld \rightarrow \CMDrd$ and $A_{\Cat{M}}^{-}: \CMDrd \rightarrow \CMDld$ as follows. 
As linear functors, $A_{\Cat{M}}$ and $A_{\Cat{M}}^{-}$  are the identities on $\op{M}$. The module structures are constructed in the obvious way from the pivotal structures. 
Similarly, there are equivalences of bimodule categories $\DMCrrd \simeq \DMC$ and $\DMClld \simeq \DMC$. Precomposing these equivalences $A_{\Cat{M}}$ and $A_{\Cat{M}}^{-}$ with the bimodule functors in equation (\ref{eq:dual-inn-hom}),
 we obtain the following. 
\begin{lemma}
  Let $\Cat{C}$, $\Cat{D}$ be pivotal finite tensor categories and $\DMC$ a bimodule category. The functors 
  \begin{equation}
    \label{eq:inner-hom-star}
    \begin{split}
       \DMC \boxtimes \CMDld \ni m\boxtimes \widetilde{m} &\mapsto \idmstar{\widetilde{m},m} \in \Cat{D}, \\
 \CMDrd \boxtimes \DMC \ni \widetilde{m} \boxtimes m& \mapsto \imcstar{m,\widetilde{m}} \in \Cat{C},
    \end{split}
     \end{equation}
are balanced bimodule functors with respect to the pivotal structures of $\Cat{C}$ and $\Cat{D}$. 
\end{lemma}

\begin{definition}
  \label{definition:inner-prod-bimodule}
Let $\Cat{C}$, $\Cat{D}$ be pivotal finite tensor categories. 
  \begin{definitionlist}
\item An inner-product module category over  $\Cat{D}$ is a  module category $\DM$ together with a bimodule natural isomorphism 
  \begin{equation}
    \label{eq:inner-prod-module}
    I_{m,\widetilde{m}}^{\Cat{M}}: \idm{m, \widetilde{m}} \simeq \idmstar{\widetilde{m},m}, 
  \end{equation}
of bimodule functors $\DM \boxtimes \MDld \rightarrow \DDD$. 
 \item A  $\Cat{D}$-inner-product bimodule category is a bimodule category $\DMC$ together with a $\Cat{C}$-balanced $\Cat{D}$-bimodule natural isomorphism 
\begin{equation}
    \label{eq:inner-prod-module-bimodule-left}
   I^{_{\Cat{D}\!\!}\Cat{M}}_{m,\widetilde{m}}:\idm{m, \widetilde{m}} \simeq \idmstar{\widetilde{m},m} 
  \end{equation}
of $\Cat{C}$-balanced $\Cat{D}$-bimodule functors $\DMC \boxtimes \CMDld \rightarrow \DDD$. 
\item A  $\Cat{C}$-inner-product bimodule category is a bimodule category $\DMC$ together with a $\Cat{D}$-balanced $\Cat{C}$-bimodule natural isomorphism 
\begin{equation}
    \label{eq:inner-prod-module-bimodule-right}
   I^{\Cat{M}_{\Cat{C}}}_{\widetilde{m},m}:\imc{ \widetilde{m},m} \simeq \imcstar{m,\widetilde{m}}, 
 \end{equation}
of $\Cat{D}$-balanced $\Cat{C}$-bimodule functors $ \CMDrd \boxtimes \DMC \rightarrow \CCC$. 
\item  An inner-product  $(\Cat{D}, \Cat{C})$-bimodule category is a  bimodule category $\DMC$ together with the structures of a 
$\Cat{D}$- and $\Cat{C}$-inner-product bimodule category. 
  \end{definitionlist}
\end{definition}

\begin{remark}
\label{remark:inner-prod-ex}
Let $\DMC$ be a separable bimodule category. 
  \begin{remarklist}
    \item  By the module Yoneda-Lemma \ref{lemma:module-yoneda} it follows that the structure of a $\Cat{D}$-inner-product bimodule category on $\DMC$ is the same as a bimodule natural isomorphism 
from $A_{\Cat{M}} \circ \mathsf{S}_{ _{\Cat{D}\!}\Cat{M}}: \CMDrd \rightarrow \CMDrd$ to the identity functor on $\CMDrd$. Similarly, the structure of a $\Cat{C}$-inner-product bimodule category on $\DMC$
is the same as a bimodule natural isomorphism from the bimodule functor $A_{\Cat{M}} \circ S_{\Cat{M}_{\Cat{C}}}$ to the identity on $\CMDrd$. 
\item It is clear that for an inner-product module category $\DM$ the inner hom functor is exact. In the proof of \cite[Prop. 3.16]{FinTen} it is shown that a module category is exact if and only if the inner hom functor is exact. 
Thus an inner-product module category is necessarily exact.  
  \end{remarklist}
 \end{remark}
It follows directly from the definition, that by passing to the Grothendieck group, an inner-product bimodule category defines an inner-product bimodule 
over the Grothendieck ring of the tensor categories in the sense of Definition \ref{definition:inner-prod-starald}.

\begin{example}
  \label{example:inner-product}
  \begin{examplelist}
    \item \label{item:can-inn-prod} Let $\Cat{C}$ be a pivotal finite tensor category with pivotal structure $a: \id_{\Cat{C}} \rightarrow (-)^{**}$.
 Then the bimodule category $\CCC$ has the structure of an inner-product bimodule category induced by the pivotal structure. 
Indeed, for $\widetilde{x},x \in \Cat{C}$  the left inner-product  structure is defined by
$$\icgenstar{\widetilde{x},x}=   x \otimes \widetilde{x}^{*} \stackrel{1\otimes \leftidx{^*}{a_{\widetilde{x}}}{}}{\simeq} x \otimes  \leftidx{^*}{\widetilde{x}}{}=  \icgen{x,\widetilde{x}}, $$
 while  the right inner-product  structure is induced by
$$\icstarl{x, \widetilde{x}} = \leftidx{^*}{\widetilde{x}}{} \otimes x \stackrel{ a_{\leftidx{^*}{\widetilde{x}}{}} \otimes 1}{\simeq} \widetilde{x}^{*} \otimes x= \igenc{\widetilde{x},x}.$$
Moreover, if $a$ and $b$ are two pivotal structures for $\Cat{C}$, it is easy to see 
that $\CaCCb$ has a structure of and inner-product bimodule category if and only if $a=b$. 
\item \label{item:dual-inn-prod} Let $\DMC$ be an inner-product bimodule category. Then the dual categories $\CMDld$ and $\CMDrd$ have a natural structure of  inner-product bimodule categories: 
It is shown 
in Lemma \ref{lemma:inner-hom-and-duals} that the $\Cat{C}$-valued inner hom of $\CMDrd$ and the 
$\Cat{D}$-valued inner hom of $\CMDld$ are given by inner homs of $\DMC$.
Using the pivotal structure of $\Cat{D}$, every bimodule functor $\mathsf{F}: \DMD \rightarrow \DDD$ is equivalent to the composite
\begin{equation}
  \label{eq:comp-equiv-bimod}
  \DMD \simeq \DMDrrd \stackrel{\Funrrtd{\mathsf{F}}}{\simeq} \DDD.
\end{equation}
It thus follows from  Lemma \ref{lemma:inner-hom-and-duals} that the $\Cat{D}$-valued inner hom of $\CMDrd$ is equivalent as balanced bimodule functor to the composite
\begin{equation}
  \label{eq:D-val-rd}
  \DMCrrd \boxtimes \CMDrd \simeq \DMC \boxtimes \CMDld \stackrel{\idm{-,-}}{\longrightarrow} \DDD. 
\end{equation}
Analogously, the $\Cat{C}$-valued inner hom of $\CMDld$ can be expressed by the $\Cat{C}$-valued inner hom of $\DMC$. 

Thus it follows that the inner-product bimodule category structure of $\DMC$ induces the structures of inner-product bimodule categories on $\CMDld$ and $\CMDrd$.
 \end{examplelist} 
\end{example} 
The following shows that the structure of an inner-product (bi)-module category is essentially unique if it exists. 
\begin{proposition}
  \label{proposition:uniquen-inn-prod}
Let $\DM$ be an indecomposable exact  module  category over a pivotal finite tensor category $\Cat{D}$ with and inner-product module structure $I$.
 For any other balanced  natural isomorphism $I':  \idm{m, \widetilde{m}} \simeq \idmstar{\widetilde{m},m}$, there exists 
a scalar $\lambda \in \Bbbk^{\times}$ such that $I'= \lambda \cdot I$. 
\end{proposition}
\begin{proof}
  Assume $I'$ provides another inner-product module structure on $\DM$. Then the 
natural isomorphism 
\begin{equation}
  \label{eq:comp-I-Ip}
   \idm{m, \widetilde{m}} \stackrel{I'}{\simeq} \idmstar{\widetilde{m},m} \stackrel{I^{-1}}{\simeq} \idm{m, \widetilde{m}},
\end{equation}
is balanced and defines by Lemma \ref{lemma:module-yoneda} a module natural isomorphism  $1_{\Cat{M}} \rightarrow 1_{\Cat{M}}$, which is 
an endomorphism of the tensor unit  in the category $\DMstar= \Funl{\Cat{D}}{\DM,\DM}$. According to \cite[Lemma 3.24]{FinTen}, this category is a finite tensor category. In 
particular the unit object is absolutely simple. Hence the module natural isomorphism  $1_{\Cat{M}} \rightarrow 1_{\Cat{M}}$ is a multiple of the identity and thus the statement follows. 
\end{proof}
For indecomposable  bimodule categories, the structure of an inner-product bimodule category is thus unique up to simultaneous scaling of both balanced natural isomorphisms with independent scalars. 

Next we consider the compatibility of inner-product bimodule categories with the tensor product. 
\begin{proposition}
  \label{proposition:inner-prod-bimod-tensor}
Let $\Cat{C}$, $\Cat{D}$, $\Cat{E}$ be pivotal finite tensor categories and $\DMC$, $\CNE$ inner-product bimodule categories. There is an induced structure of an inner-product bimodule category on $\DMC \Box \CNE$.
\end{proposition}
\begin{proof}
  In order to show that $\DMC \Box \CNE$ has the structure of an $\Cat{D}$-inner-product bimodule category, we construct for $m,\widetilde{m} \in \Cat{M}$ and $n,\widetilde{n} \in \Cat{N}$ 
a multi-balanced $\Cat{D}$-bimodule natural isomorphism 
$$(\idmn{B(\widetilde{m} \boxtimes \widetilde{n}, B(m \boxtimes n)})^{*} \simeq \idmn{B(m \boxtimes n), B( \widetilde{m} \boxtimes \widetilde{n})}$$ of multi-balanced bimodule functors 
$\DMC \boxtimes \CNE \boxtimes \ENCld \boxtimes \CMDld \rightarrow \DDD$. 
Note that both functors are $\Cat{C}$-balanced in two arguments and $\Cat{E}$-balanced in the middle argument. By universality of the tensor product, such a multi-balanced isomorphism induces 
the structure of a $\Cat{D}$-inner-product bimodule category on $\DMC \Box \CNE$.
The claimed isomorphism is defined as the composite
\begin{equation}
  \label{eq:composite-inner-mn}
  \begin{split}
      \idstar{B(\widetilde{m} \boxtimes \widetilde{n}), B(m \boxtimes n)} &\simeq \idmstar{\widetilde{m} \ract \icn{\widetilde{n},n}, m} \\
&\stackrel{I^{_{\Cat{C}\!\!}\Cat{M}}}{\simeq }  \idm{m, \widetilde{m }\ract \icn{\widetilde{n},n}} \\
&\simeq \idm{m \ract \icnstarl{\widetilde{n},n}, \widetilde{m}} \\
&\stackrel{I^{_{\Cat{C}\!\!}\Cat{N}}}{\simeq } \idm{m \ract \icn{n,\widetilde{n}}, \widetilde{m}} \\
&\simeq  \idmn{B(m \boxtimes n), B( \widetilde{m} \boxtimes \widetilde{n})}.
  \end{split}
\end{equation}
By definition, the isomorphisms in step two and four are multi-balanced bimodule natural isomorphism. The first and the last  isomorphism are multi -balanced bimodule isomorphism obtained from 
Proposition \ref{proposition:Rieffel} \refitem{item:Rieffel-ind}. 
The remaining isomorphism in the third step is induced by the duality of $\Cat{C}$ and clearly is a   multi-balanced bimodule isomorphism. 

The proof that $\DMC \Box \CNE$ is an $\Cat{E}$-inner-product bimodule category is analogous. 
\end{proof}

\subsection{The pivotal tricategory of  inner-product bimodule categories}
We finally prove that inner-product bimodule categories yield a pivotal tricategory with duals. To this end we first show that there is a natural way to identify a bimodule functor 
between inner-product bimodule categories with its double left adjoint bimodule functor. These natural bimodule isomorphisms are then shown to constitute a pivotal structure on 
 the following 2-categories. 
\begin{definition}
  Let $\Cat{C}$ and $\Cat{D}$ be pivotal finite tensor categories. The 2-categories of  $(\Cat{C},\Cat{D})$ inner-product bimodule categories together with bimodule functors and 
bimodule natural transformations  are denoted $\BimCattet(\Cat{C},\Cat{D})$.
\end{definition}
A  pivotal structure on  the 2-categories $\BimCattet(\Cat{C},\Cat{D})$ is constructed  by generalizing   \cite[Thm. 4.5]{Moduletr} in the semisimple case. 
\begin{theorem}
  \label{theorem:module-functors-ambidextrous}
  Let  $\DMC, \DNC, \DYC \in \BimCattet$. For all  module functors $\mathsf{F}:\DMC \rr \DNC$,
  the $\Cat{D}$-inner product structures on $\Cat{M}$ and $\Cat{N}$ induce a 
bimodule  natural  isomorphism $a_{\mathsf{F}}: \mathsf{F} \rr \mathsf{F}^{ll}$ 
 from $\mathsf{F}$ to the double left adjoint module functor of $\mathsf{F}$. 
  \begin{theoremlist}
  \item \label{item:natural-aF} The natural isomorphisms $a_{\mathsf{F}}$ are natural with respect to bimodule natural transformations, i.e.~for any bimodule functor $\mathsf{G}: \DMC \rr \DNC$  and 
    any bimodule natural transformation $\rho: \mathsf{F} \rr \mathsf{G}$, the following diagram commutes
    \begin{equation*}
      \begin{xy}
        \xymatrix{ 
        \mathsf{F} \ar[r]^{a_{\mathsf{F}}} \ar[d]_{\rho}& \mathsf{F}^{ll} \ar[d]^{\rho^{ll}} \\
        \mathsf{G}   \ar[r]^{a_{\mathsf{G}}}   &   \mathsf{G}^{ll}. 
      }
      \end{xy}
    \end{equation*}
  \item  \label{item:comp-with-comp-aF}For all bimodule functors $\mathsf{F}:\DMC \rightarrow \DNC$ 
and $\mathsf{K}: \DNC \rightarrow \DYC$,
    \begin{equation}
      \label{eq:compat-of-adj-comp}
      a_{\mathsf{K}\mathsf{F}}=a_{\mathsf{K}}\circ a_{\mathsf{F}}: \mathsf{K}\circ \mathsf{F} \rr (\mathsf{K} \circ \mathsf{F})^{ll}.
    \end{equation}
\item \label{item:comp-a-unit} For the identity bimodule functor $1_{M}: \DMC \rightarrow \DMC$, the 
natural isomorphism is given by $a_{1_{\Cat{M}}}=\id_{1_{M}}$.
 \item \label{item:together-piv-2-cat} The 2-categories  $\BimCattet(\Cat{C},\Cat{D})$ equipped with these natural isomorphisms are pivotal 2-categories. 
 \end{theoremlist}
\end{theorem}
\begin{proof}
  For a bimodule functor $\mathsf{F}: \DMC \rightarrow \DNC$, consider the following composite of balanced bimodule natural isomorphism of functors $\DMC \boxtimes \CMDld \rightarrow \DDD$: 
  \begin{equation}
    \label{eq:def-pivotal}
   \idn{\mathsf{F}(m),n}  \stackrel{I^{_{\Cat{D}\!\!}\Cat{N}}}{\simeq}  \idnstar{n,\mathsf{F}(m)} \simeq    \idmstar{\mathsf{F}^{l}(n),m}   \stackrel{(I^{_{\Cat{D}\!\!}\Cat{M}})^{-1}}{\simeq} \idm{m,\mathsf{F}^{l}(n)} 
\simeq \idn{\mathsf{F}^{ll}(m),n}.
  \end{equation}
Here the first and third natural  isomorphisms are obtained from Lemma \ref{lemma:adjoint-inner-hom}. 
By Lemma \ref{lemma:module-yoneda}\refitem{item:Yoneda-inner-hom}, the composite induces a bimodule natural 
isomorphism  $a_{\mathsf{F}}: \mathsf{F} \rr \mathsf{F}^{ll}$.  Let $\rho: \mathsf{F} \rr \mathsf{G}$ be a bimodule natural isomorphism between bimodule functors from $\DMC$ to $\DNC$ in $\BimCattet(\Cat{C},\Cat{D})$. Then consider the following diagram:
 \begin{equation}
    \begin{xy}
      \xymatrix{
       \igen{\mathsf{F}m,n}  \ar[d]^{I^{\Cat{N}}}            \ar@/_4pc/[dddd]_{ \igen{a_{\mathsf{F}},1}}                             \ar[rr]^{\igen{\rho m,n}}       &             &  \igen{\mathsf{G}m,n}  \ar[d]_{I^{\Cat{N}}}   \ar@/^4pc/[dddd]^{ \igen{a_{\mathsf{G}},1}}  \\
  \igenstar{n,\mathsf{F}m}\ar[d]^{\simeq}   \ar[rr]^{\igenstar{n,\rho m}}    &                  &  \igenstar{n,\mathsf{G}m}  \ar[d]_{\simeq}  \\
    \igenstar{\mathsf{F}^{l}n,m}  \ar[d]^{(I^{\Cat{M}})^{-1}}           \ar[rr]^{\igenstar{\rho^{l}m,n}}                                &  &    \igenstar{\mathsf{G}^{l}n,m}   \ar[d]_{(I^{\Cat{M}})^{-1}}        \\
   \igen{m,\mathsf{F}^{l}n}      \ar[d]^{\simeq}  \ar[rr]^{\igen{n,\rho^{l}m}}                                &    &    \igen{m,\mathsf{G}^{l}n}        \ar[d]_{\simeq}     \\
        \igen{\mathsf{F}^{ll}m,n}         \ar[rr]^{\igen{\rho^{ll}m,n}}          &    &    \igen{\mathsf{G}^{ll} m,n}   .
       }
    \end{xy}
  \end{equation}
It is straightforward to see that all small diagrams in this diagram commute. Thus the outer diagram commutes as well which shows part \refitem{item:natural-aF}.
  For the second part consider additionally a bimodule functor $\mathsf{K}: \DNC \rightarrow \DYC$ in $\BimCattet(\Cat{C},\Cat{D})$. 
  It is enough to prove that the following diagram commutes:
  \begin{equation} 
    \label{eq:adjoint-functor-comp}
    \begin{xy}
      \xymatrix{
        \igen{\mathsf{K}\mathsf{F}m,y}  \ar[d]^{I^{\Cat{Y}}}            \ar[ddrrr]^{\igen{a_{\mathsf{K}\mathsf{F}},1}} \ar@/_6pc/[dddd]^{ \igen{a_{\mathsf{K}}\mathsf{F} ,1} }  &                  &                 &               \\
        \igenstar{y, \mathsf{K}\mathsf{F}m}        \ar[d]^{\simeq}&                    &                         \\
        \igenstar{\mathsf{K}^{l} y, \mathsf{F}m}  \ar[d]^{(I^{\Cat{N}})^{-1}}  \ar[r]^{\simeq} &      \igenstar{\mathsf{F}^{l} \mathsf{K}^{l} y,m}  \ar[r]^{(I^{\Cat{M}})^{-1}}                             &    \igen{m,\mathsf{F}^{l} \mathsf{K}^{l} y}      \ar[r]^{\simeq}      & \igen{\mathsf{K}^{ll} \mathsf{F}^{ll} m,y}           \\
        \igen{\mathsf{F} m, \mathsf{K}^{l}y}    \ar[d]^{\simeq}       &                          &                       & \\
        \igen{\mathsf{K}^{ll} \mathsf{F}m ,y }    \ar[uurrr]_{\igen{\mathsf{K}^{ll} a_{\mathsf{F}},1}}.     &                           &                            &    \\  %@/_4pc/
        &                            &               &
      }
    \end{xy}
  \end{equation}
  The upper triangle and the lower subdiagram commute due to the definition of $a_{\mathsf{K}\mathsf{F}}$ and $a_{\mathsf{K}}$, respectively. The remaining diagram commutes due to the naturality of the adjunctions. 
The  part \refitem{item:together-piv-2-cat} is just the collection of the previous statements according to Definition \ref{definition:pivotal-str-bicat} of a pivotal 2-category.
\end{proof}

\begin{remark}
If we consider $\LCC$ as an inner-product module category according to Example \ref{example:inner-product}\refitem{item:can-inn-prod}, the induced pivotal structure 
on $\Funl{\Cat{C}}{\LCC,\LCC} \simeq \Cat{C}$ coincides with the pivotal structure of $\Cat{C}$.

By replacing the $\Cat{D}$-valued with the  $\Cat{C}$-valued inner product, the analogue of equation (\ref{eq:def-pivotal}) defines another pivotal structure on $\BimCattet(\Cat{C},\Cat{D})$ that will be in general   different than the one considered, in particular since we can scale each of the two isomorphisms $  I^{_{\Cat{D}\!\!}\Cat{M}}$ and $I^{\Cat{M}_{\Cat{C}}}$ independently.  
\end{remark}

The following statement follows directly from Theorem \ref{theorem:module-functors-ambidextrous} and the explicit description of the induced pivotal structure in equation (\ref{eq:def-pivotal}).
\begin{corollary}
Let $\DM$ be an inner-product module category over a pivotal finite tensor category.  By Theorem \ref{theorem:module-functors-ambidextrous}, the (multi-)tensor category $\DMstar$ acquires a pivotal structure. 
With respect to this pivotal structure, the natural isomorphism $I^{\Cat{M}}$ from $\DM$ is $(\Cat{D}, \DMstar)$-balanced. 
 \end{corollary}

Next we consider the compatibility of these pivotal structures with the tensor product of inner-product bimodule categories.

\begin{lemma}
  \label{lemma:left-adj-tensor}
Let $\mathsf{F}: \DMC \rightarrow \DMpC$ and $\mathsf{G}: \CNE \rightarrow \CNpE$ be bimodule functors. The left inner hom
induces a bimodule natural isomorphism $\xi_{\mathsf{F},\mathsf{G}}: (\mathsf{F} \Box \mathsf{G})^{l} \rightarrow \mathsf{F}^{l} \Box \mathsf{G}^{l}$.
\end{lemma}
\begin{proof}
  Consider the following composite of balanced bimodule natural isomorphisms
  \begin{equation}
    \label{eq:composite-xi}
    \begin{split}
      \idmn{(\mathsf{F} \Box \mathsf{G})^{l} (m' \Box n'), m \Box n} &\simeq \idmpnp{m' \Box n', \mathsf{F}(m)\Box \mathsf{G}(n)}\\
& \simeq \idmp{m'\ract \idnp{n', \mathsf{G}(n)}, \mathsf{F}(m)} \\
&\simeq \idm{\mathsf{F}^{l}(m') \ract \idn{\mathsf{G}^{l}(n'),n}, m}\\
&\simeq \idmn{(\mathsf{F}^{l} \Box \mathsf{G}^{l})(m' \Box n'),m \Box n}.
    \end{split}
  \end{equation}
By Lemma \ref{lemma:module-yoneda} this induces the claimed bimodule natural isomorphism. 
\end{proof}

\begin{proposition}
\label{proposition:piv-2-fun-bimod}
Let   $\mathsf{F}: \DMC \rightarrow \DMpC$ and $\mathsf{G}: \CNE \rightarrow \CNpE$ be bimodule functors 
between inner-product bimodule categories. The following diagram of bimodule natural isomorphisms commutes
  \begin{equation}
    \label{eq:Box-piv-fun}
    \begin{tikzcd}
      \mathsf{F} \Box \mathsf{G} \ar{r}{a_{\mathsf{F} \Box \mathsf{G}}}  \ar{d}[left]{a_{\mathsf{F} } \Box a_{\mathsf{G}}}& (\mathsf{F} \Box \mathsf{G})^{ll} \\
\mathsf{F}^{ll} \Box \mathsf{G}^{ll} & (\mathsf{F}^{l} \Box \mathsf{G}^{l})^{l} \ar{l}{\xi_{\mathsf{F}^{l},\mathsf{G}^{l}}} \ar{u}[right]{(\xi_{\mathsf{F},\mathsf{G}})^{l}}.
    \end{tikzcd}
  \end{equation} 
\end{proposition}
\begin{proof}
  It is sufficient to 
show that the corresponding 
diagram inside the inner homs commutes: 
 \begin{equation}
    \label{eq:Box-piv-fun-inner-hom}
    \begin{tikzcd}
    \idmpnp{  \mathsf{F} \Box \mathsf{G}(m\Box n),m' \Box n'} \ar{r}{a_{\mathsf{F} \Box \mathsf{G}}}  \ar{d}[left]{a_{\mathsf{F} } \Box a_{\mathsf{G}}}&\idmpnp{ (\mathsf{F} \Box \mathsf{G})^{ll}(m\Box n),m' \Box n'} \\
\idmpnp{ \mathsf{F}^{ll} \Box \mathsf{G}^{ll}(m\Box n),m' \Box n'} & \idmpnp{(\mathsf{F}^{l} \Box \mathsf{G}^{l})^{l}(m\Box n),m' \Box n'} \ar{l}{\xi_{\mathsf{F}^{l},\mathsf{G}^{l}}} \ar{u}[right]{(\xi_{\mathsf{F},\mathsf{G}})^{l}}.
    \end{tikzcd}
  \end{equation} 
This follows from a lengthy but straightforward computation using the 
definitions of the pivotal structure in Theorem \ref{theorem:module-functors-ambidextrous} and Lemma \ref{lemma:left-adj-tensor} as well as Proposition \ref{proposition:Rieffel} and the Definition of the inner-product structure on $\Cat{M} \Box \Cat{N}$ by Proposition \ref{proposition:inner-prod-bimod-tensor}. 
\end{proof}
Combining the results of the previous subsections together we obtain the following.
\begin{theorem}
\label{theorem:piv-tri-bimtet}
Finite pivotal tensor categories, inner-product bimodule categories, bimodule functors and bimodule natural transformations form a tricategory $\BimCattet$.
 With the pivotal structure induced by Theorem \ref{theorem:module-functors-ambidextrous}, the tricategory $\BimCattet$ is a pivotal tricategory duals according to Definition \ref{definition:pivotal-tricat}.
\end{theorem}
\begin{proof}
It is shown in Example \ref{example:inner-product}\refitem{item:can-inn-prod}, that  for a pivotal tensor category $\Cat{C}$ the 
unit bimodule category $\CCC$ is in $\BimCattet$. Furthermore, the 
tensor product of two inner-product bimodule categories is again an inner-product bimodule category due to Proposition \ref{proposition:inner-prod-bimod-tensor}. Hence $\BimCattet$ 
forms a tricategory with  the tricategory structure induced from $\BimCat$.

 Inner-product module categories are exact, see Remark \ref{remark:inner-prod-ex}, hence the adjoints of bimodule functors exist and provide   $*$-duals for $\BimCattet$. The pivotal structure on the bicategories $\BimCattet(\Cat{C},\Cat{D})$ is 
defined in Theorem \ref{theorem:module-functors-ambidextrous}. It follows directly from Proposition \ref{proposition:piv-2-fun-bimod} by restricting to the case $\mathsf{F}=\id$, that 
the 2-functors 
\begin{equation*}
  \EKD \Box -: \BimCattet(\Cat{C},\Cat{D}) \rightarrow \BimCattet(\Cat{C},\Cat{E})
\end{equation*}
are pivotal 2-functors for all finite pivotal tensor categories $\Cat{C}$, $\Cat{D}$, $\Cat{E}$ and all 
$\EKD \in \BimCattet$. Analogously it follows 
that $- \Box \EKD$ are pivotal 2-functors. Hence $\BimCattet$ is a pivotal tricategory. The $\#$-duals in $\BimCattet$ are 
defined in Theorem \ref{theorem:hash-duals-bimcat}. The duals of inner-product bimodule categories are again in
 $\BimCattet$ according to Example \ref{example:inner-product}  \refitem{item:dual-inn-prod}.
\end{proof}
\begin{remark}
  One might also consider a generalization of inner-product bimodule categories to bimodule categories with two tensor category-valued inner products that are not necessarily given by the inner homs. 
If the inner products have properties analogous to the inner homs, i.e. right exact balanced bimodule functors with balanced bimodule natural isomorphisms as in (\ref{eq:inner-prod-module-bimodule-left}), (\ref{eq:inner-prod-module-bimodule-right}), then the Rieffel-formula (\ref{eq:functor-Verschachtelt-inner-hom}) defines inner products on the same type on the tensor product of two such generalized inner-product bimodule categories. 
This leads to a tricategory of generalized inner-product bimodule categories. However, this is not investigated further in this article. 
\end{remark}
For finite tensor categories, the 3-groupoid that is induced from $\BimCat$ is called Brauer-Picard 3-groupoid. In the semisimple case this has been investigated in \cite{ENOfuhom}.
As a more refined 3-groupoid for  the theory of pivotal finite tensor categories we propose the following. 
Recall the notion of invertible bimodule category from \cite[Sec.4]{ENOfuhom}.
\begin{definition}
  \label{definition:Star-BrPic} 
The $*$-Brauer-Picard 3-groupoid $\underline{\underline{\mathsf{BrPic}^{*}}}$ 
of pivotal finite tensor categories is the following 3-groupoid. Objects are finite pivotal tensor categories, 1-morphisms invertible 
inner-product bimodule categories, 2-morphisms equivalences of bimodule categories and  3-morphisms bimodule natural isomorphisms. 
\end{definition}

This defines the notion of  $*$-Morita equivalence for pivotal finite tensor categories: $\Cat{C}$ and $\Cat{D}$
are called $*$-Morita equivalent if there exists an invertible inner-product bimodule category $\DMC$.
This notion of equivalence allows to distinguish different pivotal structures:
\begin{corollary}
   If $\DMC$ is a 1-morphism in  $\underline{\underline{\mathsf{BrPic}^{*}}}$, then the pivotal structure of $\Cat{C}$ is uniquely determined by the 
pivotal structure of $\Cat{D}$ and the $\Cat{D}$-valued inner-product module structure of $\DM$.
\end{corollary}
\begin{proof}
  Note that Formula (\ref{eq:def-pivotal}) defines a pivotal structure for $\Cat{C} \simeq \Funl{\Cat{D}}{\DM,\DM}$.  
Now the balancing of the isomorphism  $I^{_{\Cat{D}\!\!}\Cat{M}}$ in the inner-product module structure of $\DM$ demands that this pivotal structure agrees with the pivotal structure of $\Cat{C}$. 
\end{proof}

\subsection{Inner-product module categories from Frobenius algebras}

We show that special symmetric  Frobenius algebras in pivotal finite tensor categories provide examples of inner-product module categories.

First recall the definition of a special symmetric normalized Frobenius algebra. In the following we use the graphical calculus for tensor categories, where objects are presented by strings, tensor product is 
presented by juxtaposition and  diagrams are read from up to down.
\begin{definition}
Let $\Cat{C}$ be a pivotal tensor category.
 \begin{definitionlist}
   \item A Frobenius algebra $A \in \Cat{C}$ is a algebra $A$ that is also a coalgebra with 
\begin{equation}
\text{multiplication} \; 
\ifx\du\undefined
  \newlength{\du}
\fi
\setlength{\du}{10\unitlength}
\begin{tikzpicture}[baseline]
\pgftransformxscale{1.000000}
\pgftransformyscale{-1.000000}
\definecolor{dialinecolor}{rgb}{0.000000, 0.000000, 0.000000}
\pgfsetstrokecolor{dialinecolor}
\definecolor{dialinecolor}{rgb}{1.000000, 1.000000, 1.000000}
\pgfsetfillcolor{dialinecolor}
\pgfsetlinewidth{0.060000\du}
\pgfsetdash{}{0pt}
\pgfsetdash{}{0pt}
\pgfsetbuttcap
{
\definecolor{dialinecolor}{rgb}{0.000000, 0.000000, 0.000000}
\pgfsetfillcolor{dialinecolor}
% was here!!!
\definecolor{dialinecolor}{rgb}{0.000000, 0.000000, 0.000000}
\pgfsetstrokecolor{dialinecolor}
\draw (0.002589\du,-0.084342\du)--(0.003125\du,1.552188\du);
}
\pgfsetlinewidth{0.060000\du}
\pgfsetdash{}{0pt}
\pgfsetdash{}{0pt}
\pgfsetmiterjoin
\pgfsetbuttcap
{
\definecolor{dialinecolor}{rgb}{0.000000, 0.000000, 0.000000}
\pgfsetfillcolor{dialinecolor}
% was here!!!
\definecolor{dialinecolor}{rgb}{0.000000, 0.000000, 0.000000}
\pgfsetstrokecolor{dialinecolor}
\pgfpathmoveto{\pgfpoint{-0.876256\du}{-1.502252\du}}
\pgfpathcurveto{\pgfpoint{-0.876256\du}{0.397748\du}}{\pgfpoint{0.923744\du}{0.397748\du}}{\pgfpoint{0.923744\du}{-1.502252\du}}
\pgfusepath{stroke}
}
\end{tikzpicture}
, \quad \text{unit} \; 
\ifx\du\undefined
  \newlength{\du}
\fi
\setlength{\du}{10\unitlength}
\begin{tikzpicture}[baseline]
\pgftransformxscale{1.000000}
\pgftransformyscale{-1.000000}
\definecolor{dialinecolor}{rgb}{0.000000, 0.000000, 0.000000}
\pgfsetstrokecolor{dialinecolor}
\definecolor{dialinecolor}{rgb}{1.000000, 1.000000, 1.000000}
\pgfsetfillcolor{dialinecolor}
\pgfsetlinewidth{0.060000\du}
\pgfsetdash{}{0pt}
\pgfsetdash{}{0pt}
\pgfsetbuttcap
{
\definecolor{dialinecolor}{rgb}{0.000000, 0.000000, 0.000000}
\pgfsetfillcolor{dialinecolor}
% was here!!!
\definecolor{dialinecolor}{rgb}{0.000000, 0.000000, 0.000000}
\pgfsetstrokecolor{dialinecolor}
\draw (0.002735\du,-0.975563\du)--(0.002735\du,1.182760\du);
}
\definecolor{dialinecolor}{rgb}{0.000000, 0.000000, 0.000000}
\pgfsetfillcolor{dialinecolor}
\pgfpathellipse{\pgfpoint{0.002735\du}{-1.075560\du}}{\pgfpoint{0.100000\du}{0\du}}{\pgfpoint{0\du}{0.100000\du}}
\pgfusepath{fill}
\pgfsetlinewidth{0.060000\du}
\pgfsetdash{}{0pt}
\pgfsetdash{}{0pt}
\definecolor{dialinecolor}{rgb}{0.000000, 0.000000, 0.000000}
\pgfsetstrokecolor{dialinecolor}
\pgfpathellipse{\pgfpoint{0.002735\du}{-1.075560\du}}{\pgfpoint{0.100000\du}{0\du}}{\pgfpoint{0\du}{0.100000\du}}
\pgfusepath{stroke}
\definecolor{dialinecolor}{rgb}{0.000000, 0.000000, 0.000000}
\pgfsetfillcolor{dialinecolor}
\pgfpathellipse{\pgfpoint{-0.006504\du}{-1.097285\du}}{\pgfpoint{0.155571\du}{0\du}}{\pgfpoint{0\du}{0.155571\du}}
\pgfusepath{fill}
\pgfsetlinewidth{0.060000\du}
\pgfsetdash{}{0pt}
\pgfsetdash{}{0pt}
\definecolor{dialinecolor}{rgb}{0.000000, 0.000000, 0.000000}
\pgfsetstrokecolor{dialinecolor}
\pgfpathellipse{\pgfpoint{-0.006504\du}{-1.097285\du}}{\pgfpoint{0.155571\du}{0\du}}{\pgfpoint{0\du}{0.155571\du}}
\pgfusepath{stroke}
\end{tikzpicture}
, \quad \text{comultiplication} \; 
\ifx\du\undefined
  \newlength{\du}
\fi
\setlength{\du}{10\unitlength}
\begin{tikzpicture}[baseline]
\pgftransformxscale{1.000000}
\pgftransformyscale{-1.000000}
\definecolor{dialinecolor}{rgb}{0.000000, 0.000000, 0.000000}
\pgfsetstrokecolor{dialinecolor}
\definecolor{dialinecolor}{rgb}{1.000000, 1.000000, 1.000000}
\pgfsetfillcolor{dialinecolor}
\pgfsetlinewidth{0.060000\du}
\pgfsetdash{}{0pt}
\pgfsetdash{}{0pt}
\pgfsetbuttcap
{
\definecolor{dialinecolor}{rgb}{0.000000, 0.000000, 0.000000}
\pgfsetfillcolor{dialinecolor}
% was here!!!
\definecolor{dialinecolor}{rgb}{0.000000, 0.000000, 0.000000}
\pgfsetstrokecolor{dialinecolor}
\draw (0.002589\du,-1.525074\du)--(0.003125\du,0.111458\du);
}
\pgfsetlinewidth{0.060000\du}
\pgfsetdash{}{0pt}
\pgfsetdash{}{0pt}
\pgfsetmiterjoin
\pgfsetbuttcap
{
\definecolor{dialinecolor}{rgb}{0.000000, 0.000000, 0.000000}
\pgfsetfillcolor{dialinecolor}
% was here!!!
\definecolor{dialinecolor}{rgb}{0.000000, 0.000000, 0.000000}
\pgfsetstrokecolor{dialinecolor}
\pgfpathmoveto{\pgfpoint{-0.900000\du}{1.506066\du}}
\pgfpathcurveto{\pgfpoint{-0.900000\du}{-0.393934\du}}{\pgfpoint{0.900000\du}{-0.393934\du}}{\pgfpoint{0.900000\du}{1.506066\du}}
\pgfusepath{stroke}
}
\end{tikzpicture}
\;, \quad  \text{counit morphism} \;
\ifx\du\undefined
  \newlength{\du}
\fi
\setlength{\du}{10\unitlength}
\begin{tikzpicture}[baseline]
\pgftransformxscale{1.000000}
\pgftransformyscale{-1.000000}
\definecolor{dialinecolor}{rgb}{0.000000, 0.000000, 0.000000}
\pgfsetstrokecolor{dialinecolor}
\definecolor{dialinecolor}{rgb}{1.000000, 1.000000, 1.000000}
\pgfsetfillcolor{dialinecolor}
\pgfsetlinewidth{0.060000\du}
\pgfsetdash{}{0pt}
\pgfsetdash{}{0pt}
\pgfsetbuttcap
{
\definecolor{dialinecolor}{rgb}{0.000000, 0.000000, 0.000000}
\pgfsetfillcolor{dialinecolor}
% was here!!!
\definecolor{dialinecolor}{rgb}{0.000000, 0.000000, 0.000000}
\pgfsetstrokecolor{dialinecolor}
\draw (-0.009375\du,-1.171870\du)--(-0.009375\du,0.859271\du);
}
\definecolor{dialinecolor}{rgb}{1.000000, 1.000000, 1.000000}
\pgfsetfillcolor{dialinecolor}
\pgfpathellipse{\pgfpoint{-0.015424\du}{1.004540\du}}{\pgfpoint{0.155571\du}{0\du}}{\pgfpoint{0\du}{0.155571\du}}
\pgfusepath{fill}
\pgfsetlinewidth{0.060000\du}
\pgfsetdash{}{0pt}
\pgfsetdash{}{0pt}
\definecolor{dialinecolor}{rgb}{0.000000, 0.000000, 0.000000}
\pgfsetstrokecolor{dialinecolor}
\pgfpathellipse{\pgfpoint{-0.015424\du}{1.004540\du}}{\pgfpoint{0.155571\du}{0\du}}{\pgfpoint{0\du}{0.155571\du}}
\pgfusepath{stroke}
\end{tikzpicture}
\;,
  \end{equation}
 such that 
\begin{equation}
  \label{eq:Frob-equ-graphical}
\ifx\du\undefined
  \newlength{\du}
\fi
\setlength{\du}{10\unitlength}
\begin{tikzpicture}[baseline]
\pgftransformxscale{1.000000}
\pgftransformyscale{-1.000000}
\definecolor{dialinecolor}{rgb}{0.000000, 0.000000, 0.000000}
\pgfsetstrokecolor{dialinecolor}
\definecolor{dialinecolor}{rgb}{1.000000, 1.000000, 1.000000}
\pgfsetfillcolor{dialinecolor}
\pgfsetlinewidth{0.060000\du}
\pgfsetdash{}{0pt}
\pgfsetdash{}{0pt}
\pgfsetbuttcap
{
\definecolor{dialinecolor}{rgb}{0.000000, 0.000000, 0.000000}
\pgfsetfillcolor{dialinecolor}
\definecolor{dialinecolor}{rgb}{0.000000, 0.000000, 0.000000}
\pgfsetstrokecolor{dialinecolor}
\draw (3.000000\du,0.803293\du)--(3.000000\du,-0.796707\du);
}
\pgfsetlinewidth{0.060000\du}
\pgfsetdash{}{0pt}
\pgfsetdash{}{0pt}
\pgfsetmiterjoin
\pgfsetbuttcap
{
\definecolor{dialinecolor}{rgb}{0.000000, 0.000000, 0.000000}
\pgfsetfillcolor{dialinecolor}
% was here!!!
\definecolor{dialinecolor}{rgb}{0.000000, 0.000000, 0.000000}
\pgfsetstrokecolor{dialinecolor}
\pgfpathmoveto{\pgfpoint{2.100000\du}{-2.196707\du}}
\pgfpathcurveto{\pgfpoint{2.100000\du}{-0.296707\du}}{\pgfpoint{3.900000\du}{-0.296707\du}}{\pgfpoint{3.900000\du}{-2.196707\du}}
\pgfusepath{stroke}
}
\pgfsetlinewidth{0.060000\du}
\pgfsetdash{}{0pt}
\pgfsetdash{}{0pt}
\pgfsetmiterjoin
\pgfsetbuttcap
{
\definecolor{dialinecolor}{rgb}{0.000000, 0.000000, 0.000000}
\pgfsetfillcolor{dialinecolor}
% was here!!!
\definecolor{dialinecolor}{rgb}{0.000000, 0.000000, 0.000000}
\pgfsetstrokecolor{dialinecolor}
\pgfpathmoveto{\pgfpoint{2.100000\du}{2.203293\du}}
\pgfpathcurveto{\pgfpoint{2.100000\du}{0.303293\du}}{\pgfpoint{3.900000\du}{0.303293\du}}{\pgfpoint{3.900000\du}{2.203293\du}}
\pgfusepath{stroke}
}
\end{tikzpicture}
=
\ifx\du\undefined
  \newlength{\du}
\fi
\setlength{\du}{10\unitlength}
\begin{tikzpicture}[baseline]
\pgftransformxscale{1.000000}
\pgftransformyscale{-1.000000}
\definecolor{dialinecolor}{rgb}{0.000000, 0.000000, 0.000000}
\pgfsetstrokecolor{dialinecolor}
\definecolor{dialinecolor}{rgb}{1.000000, 1.000000, 1.000000}
\pgfsetfillcolor{dialinecolor}
\pgfsetlinewidth{0.060000\du}
\pgfsetdash{}{0pt}
\pgfsetdash{}{0pt}
\pgfsetbuttcap
{
\definecolor{dialinecolor}{rgb}{0.000000, 0.000000, 0.000000}
\pgfsetfillcolor{dialinecolor}
% was here!!!
\definecolor{dialinecolor}{rgb}{0.000000, 0.000000, 0.000000}
\pgfsetstrokecolor{dialinecolor}
\draw (-7.600000\du,-2.200000\du)--(-7.600000\du,-1.400000\du);
}
\pgfsetlinewidth{0.060000\du}
\pgfsetdash{}{0pt}
\pgfsetdash{}{0pt}
\pgfsetmiterjoin
\pgfsetbuttcap
{
\definecolor{dialinecolor}{rgb}{0.000000, 0.000000, 0.000000}
\pgfsetfillcolor{dialinecolor}
% was here!!!
\definecolor{dialinecolor}{rgb}{0.000000, 0.000000, 0.000000}
\pgfsetstrokecolor{dialinecolor}
\pgfpathmoveto{\pgfpoint{-10.300000\du}{-0.000000\du}}
\pgfpathcurveto{\pgfpoint{-10.300000\du}{1.900000\du}}{\pgfpoint{-8.500000\du}{1.900000\du}}{\pgfpoint{-8.500000\du}{-0.000000\du}}
\pgfusepath{stroke}
}
\pgfsetlinewidth{0.060000\du}
\pgfsetdash{}{0pt}
\pgfsetdash{}{0pt}
\pgfsetmiterjoin
\pgfsetbuttcap
{
\definecolor{dialinecolor}{rgb}{0.000000, 0.000000, 0.000000}
\pgfsetfillcolor{dialinecolor}
% was here!!!
\definecolor{dialinecolor}{rgb}{0.000000, 0.000000, 0.000000}
\pgfsetstrokecolor{dialinecolor}
\pgfpathmoveto{\pgfpoint{-8.500000\du}{-0.000000\du}}
\pgfpathcurveto{\pgfpoint{-8.500000\du}{-1.900000\du}}{\pgfpoint{-6.700000\du}{-1.900000\du}}{\pgfpoint{-6.700000\du}{-0.000000\du}}
\pgfusepath{stroke}
}
\pgfsetlinewidth{0.060000\du}
\pgfsetdash{}{0pt}
\pgfsetdash{}{0pt}
\pgfsetbuttcap
{
\definecolor{dialinecolor}{rgb}{0.000000, 0.000000, 0.000000}
\pgfsetfillcolor{dialinecolor}
% was here!!!
\definecolor{dialinecolor}{rgb}{0.000000, 0.000000, 0.000000}
\pgfsetstrokecolor{dialinecolor}
\draw (-9.400000\du,1.400000\du)--(-9.400000\du,2.200000\du);
}
\pgfsetlinewidth{0.060000\du}
\pgfsetdash{}{0pt}
\pgfsetdash{}{0pt}
\pgfsetbuttcap
{
\definecolor{dialinecolor}{rgb}{0.000000, 0.000000, 0.000000}
\pgfsetfillcolor{dialinecolor}
% was here!!!
\definecolor{dialinecolor}{rgb}{0.000000, 0.000000, 0.000000}
\pgfsetstrokecolor{dialinecolor}
\draw (-6.700000\du,0.000000\du)--(-6.700000\du,2.200000\du);
}
\pgfsetlinewidth{0.060000\du}
\pgfsetdash{}{0pt}
\pgfsetdash{}{0pt}
\pgfsetbuttcap
{
\definecolor{dialinecolor}{rgb}{0.000000, 0.000000, 0.000000}
\pgfsetfillcolor{dialinecolor}
% was here!!!
\definecolor{dialinecolor}{rgb}{0.000000, 0.000000, 0.000000}
\pgfsetstrokecolor{dialinecolor}
\draw (-10.300000\du,-2.200000\du)--(-10.300000\du,0.000000\du);
}
\end{tikzpicture}
=
\ifx\du\undefined
  \newlength{\du}
\fi
\setlength{\du}{10\unitlength}
\begin{tikzpicture}[baseline]
\pgftransformxscale{1.000000}
\pgftransformyscale{-1.000000}
\definecolor{dialinecolor}{rgb}{0.000000, 0.000000, 0.000000}
\pgfsetstrokecolor{dialinecolor}
\definecolor{dialinecolor}{rgb}{1.000000, 1.000000, 1.000000}
\pgfsetfillcolor{dialinecolor}
\pgfsetlinewidth{0.060000\du}
\pgfsetdash{}{0pt}
\pgfsetdash{}{0pt}
\pgfsetbuttcap
{
\definecolor{dialinecolor}{rgb}{0.000000, 0.000000, 0.000000}
\pgfsetfillcolor{dialinecolor}
% was here!!!
\definecolor{dialinecolor}{rgb}{0.000000, 0.000000, 0.000000}
\pgfsetstrokecolor{dialinecolor}
\draw (-2.330672\du,-2.199480\du)--(-2.330672\du,-1.399480\du);
}
\pgfsetlinewidth{0.060000\du}
\pgfsetdash{}{0pt}
\pgfsetdash{}{0pt}
\pgfsetmiterjoin
\pgfsetbuttcap
{
\definecolor{dialinecolor}{rgb}{0.000000, 0.000000, 0.000000}
\pgfsetfillcolor{dialinecolor}
% was here!!!
\definecolor{dialinecolor}{rgb}{0.000000, 0.000000, 0.000000}
\pgfsetstrokecolor{dialinecolor}
\pgfpathmoveto{\pgfpoint{-1.430672\du}{0.000520\du}}
\pgfpathcurveto{\pgfpoint{-1.430672\du}{1.900520\du}}{\pgfpoint{0.369328\du}{1.900520\du}}{\pgfpoint{0.369328\du}{0.000520\du}}
\pgfusepath{stroke}
}
\pgfsetlinewidth{0.060000\du}
\pgfsetdash{}{0pt}
\pgfsetdash{}{0pt}
\pgfsetmiterjoin
\pgfsetbuttcap
{
\definecolor{dialinecolor}{rgb}{0.000000, 0.000000, 0.000000}
\pgfsetfillcolor{dialinecolor}
% was here!!!
\definecolor{dialinecolor}{rgb}{0.000000, 0.000000, 0.000000}
\pgfsetstrokecolor{dialinecolor}
\pgfpathmoveto{\pgfpoint{-3.230672\du}{0.000520\du}}
\pgfpathcurveto{\pgfpoint{-3.230672\du}{-1.899480\du}}{\pgfpoint{-1.430672\du}{-1.899480\du}}{\pgfpoint{-1.430672\du}{0.000520\du}}
\pgfusepath{stroke}
}
\pgfsetlinewidth{0.060000\du}
\pgfsetdash{}{0pt}
\pgfsetdash{}{0pt}
\pgfsetbuttcap
{
\definecolor{dialinecolor}{rgb}{0.000000, 0.000000, 0.000000}
\pgfsetfillcolor{dialinecolor}
% was here!!!
\definecolor{dialinecolor}{rgb}{0.000000, 0.000000, 0.000000}
\pgfsetstrokecolor{dialinecolor}
\draw (-0.533164\du,1.409918\du)--(-0.533164\du,2.209918\du);
}
\pgfsetlinewidth{0.060000\du}
\pgfsetdash{}{0pt}
\pgfsetdash{}{0pt}
\pgfsetbuttcap
{
\definecolor{dialinecolor}{rgb}{0.000000, 0.000000, 0.000000}
\pgfsetfillcolor{dialinecolor}
% was here!!!
\definecolor{dialinecolor}{rgb}{0.000000, 0.000000, 0.000000}
\pgfsetstrokecolor{dialinecolor}
\draw (0.369328\du,-2.199480\du)--(0.369328\du,0.000520\du);
}
\pgfsetlinewidth{0.060000\du}
\pgfsetdash{}{0pt}
\pgfsetdash{}{0pt}
\pgfsetbuttcap
{
\definecolor{dialinecolor}{rgb}{0.000000, 0.000000, 0.000000}
\pgfsetfillcolor{dialinecolor}
% was here!!!
\definecolor{dialinecolor}{rgb}{0.000000, 0.000000, 0.000000}
\pgfsetstrokecolor{dialinecolor}
\draw (-3.230672\du,0.000520\du)--(-3.230672\du,2.200520\du);
}
\end{tikzpicture}
\;. 
\end{equation}
 \item A Frobenius algebra $A \in \Cat{C}$ is called special  if there exist $\beta_{1},\beta_{A} \in \Bbbk \setminus\{0\}$ such that
 \begin{equation}
      \label{special-diagram}
\ifx\du\undefined
  \newlength{\du}
\fi
\setlength{\du}{10\unitlength}
\begin{tikzpicture}[baseline]
\pgftransformxscale{1.000000}
\pgftransformyscale{-1.000000}
\definecolor{dialinecolor}{rgb}{0.000000, 0.000000, 0.000000}
\pgfsetstrokecolor{dialinecolor}
\definecolor{dialinecolor}{rgb}{1.000000, 1.000000, 1.000000}
\pgfsetfillcolor{dialinecolor}
\pgfsetlinewidth{0.060000\du}
\pgfsetdash{}{0pt}
\pgfsetdash{}{0pt}
\pgfsetbuttcap
{
\definecolor{dialinecolor}{rgb}{0.000000, 0.000000, 0.000000}
\pgfsetfillcolor{dialinecolor}
% was here!!!
\definecolor{dialinecolor}{rgb}{0.000000, 0.000000, 0.000000}
\pgfsetstrokecolor{dialinecolor}
\draw (0.002735\du,-1.047378\du)--(0.002735\du,1.182760\du);
}
\definecolor{dialinecolor}{rgb}{0.000000, 0.000000, 0.000000}
\pgfsetfillcolor{dialinecolor}
\pgfpathellipse{\pgfpoint{-0.007800\du}{-1.149612\du}}{\pgfpoint{0.155571\du}{0\du}}{\pgfpoint{0\du}{0.155571\du}}
\pgfusepath{fill}
\pgfsetlinewidth{0.060000\du}
\pgfsetdash{}{0pt}
\pgfsetdash{}{0pt}
\definecolor{dialinecolor}{rgb}{0.000000, 0.000000, 0.000000}
\pgfsetstrokecolor{dialinecolor}
\pgfpathellipse{\pgfpoint{-0.007800\du}{-1.149612\du}}{\pgfpoint{0.155571\du}{0\du}}{\pgfpoint{0\du}{0.155571\du}}
\pgfusepath{stroke}
\definecolor{dialinecolor}{rgb}{1.000000, 1.000000, 1.000000}
\pgfsetfillcolor{dialinecolor}
\pgfpathellipse{\pgfpoint{0.005831\du}{1.188151\du}}{\pgfpoint{0.155571\du}{0\du}}{\pgfpoint{0\du}{0.155571\du}}
\pgfusepath{fill}
\pgfsetlinewidth{0.060000\du}
\pgfsetdash{}{0pt}
\pgfsetdash{}{0pt}
\definecolor{dialinecolor}{rgb}{0.000000, 0.000000, 0.000000}
\pgfsetstrokecolor{dialinecolor}
\pgfpathellipse{\pgfpoint{0.005831\du}{1.188151\du}}{\pgfpoint{0.155571\du}{0\du}}{\pgfpoint{0\du}{0.155571\du}}
\pgfusepath{stroke}
\end{tikzpicture}
= \beta_{1}, \quad \text{and} \qquad
\ifx\du\undefined
  \newlength{\du} 
\fi
\setlength{\du}{10\unitlength}
\begin{tikzpicture}[baseline]
\pgftransformxscale{1.000000}
\pgftransformyscale{-1.000000}
\definecolor{dialinecolor}{rgb}{0.000000, 0.000000, 0.000000}
\pgfsetstrokecolor{dialinecolor}
\definecolor{dialinecolor}{rgb}{1.000000, 1.000000, 1.000000}
\pgfsetfillcolor{dialinecolor}
\pgfsetlinewidth{0.060000\du}
\pgfsetdash{}{0pt}
\pgfsetdash{}{0pt}
\pgfsetbuttcap
{
\definecolor{dialinecolor}{rgb}{0.000000, 0.000000, 0.000000}
\pgfsetfillcolor{dialinecolor}
% was here!!!
\definecolor{dialinecolor}{rgb}{0.000000, 0.000000, 0.000000}
\pgfsetstrokecolor{dialinecolor}
\draw (0.093750\du,-3.506250\du)--(0.105000\du,-1.554378\du);
}
\pgfsetlinewidth{0.060000\du}
\pgfsetdash{}{0pt}
\pgfsetdash{}{0pt}
\pgfsetbuttcap
{
\definecolor{dialinecolor}{rgb}{0.000000, 0.000000, 0.000000}
\pgfsetfillcolor{dialinecolor}
% was here!!!
\definecolor{dialinecolor}{rgb}{0.000000, 0.000000, 0.000000}
\pgfsetstrokecolor{dialinecolor}
\draw (0.105000\du,1.593750\du)--(0.093750\du,3.568750\du);
}
\definecolor{dialinecolor}{rgb}{1.000000, 1.000000, 1.000000}
\pgfsetfillcolor{dialinecolor}
\pgfpathellipse{\pgfpoint{0.105000\du}{0.019686\du}}{\pgfpoint{1.375000\du}{0\du}}{\pgfpoint{0\du}{1.574064\du}}
\pgfusepath{fill}
\pgfsetlinewidth{0.060000\du}
\pgfsetdash{}{0pt}
\pgfsetdash{}{0pt}
\definecolor{dialinecolor}{rgb}{0.000000, 0.000000, 0.000000}
\pgfsetstrokecolor{dialinecolor}
\pgfpathellipse{\pgfpoint{0.105000\du}{0.019686\du}}{\pgfpoint{1.375000\du}{0\du}}{\pgfpoint{0\du}{1.574064\du}}
\pgfusepath{stroke}
\end{tikzpicture}
=
\ifx\du\undefined
  \newlength{\du}
\fi
\setlength{\du}{10\unitlength}
\begin{tikzpicture}[baseline]
\pgftransformxscale{1.000000}
\pgftransformyscale{-1.000000}
\definecolor{dialinecolor}{rgb}{0.000000, 0.000000, 0.000000}
\pgfsetstrokecolor{dialinecolor}
\definecolor{dialinecolor}{rgb}{1.000000, 1.000000, 1.000000}
\pgfsetfillcolor{dialinecolor}
\pgfsetlinewidth{0.060000\du}
\pgfsetdash{}{0pt}
\pgfsetdash{}{0pt} 
\pgfsetbuttcap
{
\definecolor{dialinecolor}{rgb}{0.000000, 0.000000, 0.000000}
\pgfsetfillcolor{dialinecolor}
% was here!!!
\definecolor{dialinecolor}{rgb}{0.000000, 0.000000, 0.000000}
\pgfsetstrokecolor{dialinecolor}
\draw (0.031250\du,-3.481250\du)--(0.018750\du,3.581250\du);
}
% setfont left to latex
\definecolor{dialinecolor}{rgb}{0.000000, 0.000000, 0.000000}
\pgfsetstrokecolor{dialinecolor}
\node[anchor=west] at (-2.415000\du,0.070000\du){$\beta_A \cdot$};
\end{tikzpicture}
\;.
    \end{equation}
A special Frobenius algebra  $A$ is called normalized if $\beta_{A}=1$.
\item A Frobenius algebra $A$ is called symmetric, if 
 \begin{equation}
    \label{symmetric-then-inverse}
\ifx\du\undefined
  \newlength{\du}
\fi
\setlength{\du}{10\unitlength}
\begin{tikzpicture}[baseline,yscale=-1]
\pgftransformxscale{1.000000}
\pgftransformyscale{-1.000000}
\definecolor{dialinecolor}{rgb}{0.000000, 0.000000, 0.000000}
\pgfsetstrokecolor{dialinecolor}
\definecolor{dialinecolor}{rgb}{1.000000, 1.000000, 1.000000}
\pgfsetfillcolor{dialinecolor}
\pgfsetlinewidth{0.060000\du}
\pgfsetdash{}{0pt}
\pgfsetdash{}{0pt}
\pgfsetbuttcap
{
\definecolor{dialinecolor}{rgb}{0.000000, 0.000000, 0.000000}
\pgfsetfillcolor{dialinecolor}
% was here!!!
\definecolor{dialinecolor}{rgb}{0.000000, 0.000000, 0.000000}
\pgfsetstrokecolor{dialinecolor}
\draw (1.000000\du,0.729289\du)--(1.000000\du,1.629290\du);
}
\definecolor{dialinecolor}{rgb}{0.000000, 0.000000, 0.000000}
\pgfsetfillcolor{dialinecolor}
\pgfpathellipse{\pgfpoint{1.000000\du}{1.729290\du}}{\pgfpoint{0.100000\du}{0\du}}{\pgfpoint{0\du}{0.100000\du}}
\pgfusepath{fill}
\pgfsetlinewidth{0.060000\du}
\pgfsetdash{}{0pt}
\pgfsetdash{}{0pt}
\definecolor{dialinecolor}{rgb}{0.000000, 0.000000, 0.000000}
\pgfsetstrokecolor{dialinecolor}
\pgfpathellipse{\pgfpoint{1.000000\du}{1.729290\du}}{\pgfpoint{0.100000\du}{0\du}}{\pgfpoint{0\du}{0.100000\du}}
\pgfusepath{stroke}
\pgfsetlinewidth{0.060000\du}
\pgfsetdash{}{0pt}
\pgfsetdash{}{0pt}
\pgfsetbuttcap
{
\definecolor{dialinecolor}{rgb}{0.000000, 0.000000, 0.000000}
\pgfsetfillcolor{dialinecolor}
% was here!!!
\definecolor{dialinecolor}{rgb}{0.000000, 0.000000, 0.000000}
\pgfsetstrokecolor{dialinecolor}
\draw (-2.100000\du,-0.570711\du)--(-2.100000\du,2.129290\du);
}
\pgfsetlinewidth{0.060000\du}
\pgfsetdash{}{0pt}
\pgfsetdash{}{0pt}
\pgfsetmiterjoin
\pgfsetbuttcap
{
\definecolor{dialinecolor}{rgb}{0.000000, 0.000000, 0.000000}
\pgfsetfillcolor{dialinecolor}
% was here!!!
\definecolor{dialinecolor}{rgb}{0.000000, 0.000000, 0.000000}
\pgfsetstrokecolor{dialinecolor}
\pgfpathmoveto{\pgfpoint{0.100000\du}{-0.670711\du}}
\pgfpathcurveto{\pgfpoint{0.100000\du}{1.229290\du}}{\pgfpoint{1.900000\du}{1.229290\du}}{\pgfpoint{1.900000\du}{-0.670711\du}}
\pgfusepath{stroke}
}
\pgfsetlinewidth{0.060000\du}
\pgfsetdash{}{0pt}
\pgfsetdash{}{0pt}
\pgfsetmiterjoin
\pgfsetbuttcap
{
\definecolor{dialinecolor}{rgb}{0.000000, 0.000000, 0.000000}
\pgfsetfillcolor{dialinecolor}
% was here!!!
\definecolor{dialinecolor}{rgb}{0.000000, 0.000000, 0.000000}
\pgfsetstrokecolor{dialinecolor}
\pgfpathmoveto{\pgfpoint{-2.100000\du}{-0.570711\du}}
\pgfpathcurveto{\pgfpoint{-2.061250\du}{-2.537590\du}}{\pgfpoint{0.056250\du}{-2.655080\du}}{\pgfpoint{0.100000\du}{-0.670711\du}}
\pgfusepath{stroke}
}
\pgfsetlinewidth{0.060000\du}
\pgfsetdash{}{0pt}
\pgfsetdash{}{0pt}
\pgfsetbuttcap
{
\definecolor{dialinecolor}{rgb}{0.000000, 0.000000, 0.000000}
\pgfsetfillcolor{dialinecolor}
% was here!!!
\definecolor{dialinecolor}{rgb}{0.000000, 0.000000, 0.000000}
\pgfsetstrokecolor{dialinecolor}
\draw (1.900000\du,-2.070710\du)--(1.898020\du,-0.667412\du);
}
% setfont left to latex
\definecolor{dialinecolor}{rgb}{0.000000, 0.000000, 0.000000}
\pgfsetstrokecolor{dialinecolor}
\node[anchor=west] at (-3.730787\du,0.157804\du){$A^{*}$};
\end{tikzpicture}
\quad = \quad 
\ifx\du\undefined
  \newlength{\du}
\fi
\setlength{\du}{10\unitlength}
\begin{tikzpicture}[baseline,yscale=-1]
\pgftransformxscale{1.000000}
\pgftransformyscale{-1.000000}
\definecolor{dialinecolor}{rgb}{0.000000, 0.000000, 0.000000}
\pgfsetstrokecolor{dialinecolor}
\definecolor{dialinecolor}{rgb}{1.000000, 1.000000, 1.000000}
\pgfsetfillcolor{dialinecolor}
\pgfsetlinewidth{0.060000\du}
\pgfsetdash{}{0pt}
\pgfsetdash{}{0pt}
\pgfsetbuttcap
{
\definecolor{dialinecolor}{rgb}{0.000000, 0.000000, 0.000000}
\pgfsetfillcolor{dialinecolor}
% was here!!!
\definecolor{dialinecolor}{rgb}{0.000000, 0.000000, 0.000000}
\pgfsetstrokecolor{dialinecolor}
\draw (1.000000\du,0.729289\du)--(1.000000\du,1.629290\du);
}
\definecolor{dialinecolor}{rgb}{0.000000, 0.000000, 0.000000}
\pgfsetfillcolor{dialinecolor}
\pgfpathellipse{\pgfpoint{1.000000\du}{1.729290\du}}{\pgfpoint{0.100000\du}{0\du}}{\pgfpoint{0\du}{0.100000\du}}
\pgfusepath{fill}
\pgfsetlinewidth{0.060000\du}
\pgfsetdash{}{0pt}
\pgfsetdash{}{0pt}
\definecolor{dialinecolor}{rgb}{0.000000, 0.000000, 0.000000}
\pgfsetstrokecolor{dialinecolor}
\pgfpathellipse{\pgfpoint{1.000000\du}{1.729290\du}}{\pgfpoint{0.100000\du}{0\du}}{\pgfpoint{0\du}{0.100000\du}}
\pgfusepath{stroke}
\pgfsetlinewidth{0.060000\du}
\pgfsetdash{}{0pt}
\pgfsetdash{}{0pt}
\pgfsetbuttcap
{
\definecolor{dialinecolor}{rgb}{0.000000, 0.000000, 0.000000}
\pgfsetfillcolor{dialinecolor}
% was here!!!
\definecolor{dialinecolor}{rgb}{0.000000, 0.000000, 0.000000}
\pgfsetstrokecolor{dialinecolor}
\draw (4.100000\du,-0.800000\du)--(4.100000\du,1.900000\du);
}
\pgfsetlinewidth{0.060000\du}
\pgfsetdash{}{0pt}
\pgfsetdash{}{0pt}
\pgfsetmiterjoin
\pgfsetbuttcap
{
\definecolor{dialinecolor}{rgb}{0.000000, 0.000000, 0.000000}
\pgfsetfillcolor{dialinecolor}
% was here!!!
\definecolor{dialinecolor}{rgb}{0.000000, 0.000000, 0.000000}
\pgfsetstrokecolor{dialinecolor}
\pgfpathmoveto{\pgfpoint{0.100000\du}{-0.700000\du}}
\pgfpathcurveto{\pgfpoint{0.100000\du}{1.200000\du}}{\pgfpoint{1.900000\du}{1.200000\du}}{\pgfpoint{1.900000\du}{-0.700000\du}}
\pgfusepath{stroke}
}
\pgfsetlinewidth{0.060000\du}
\pgfsetdash{}{0pt}
\pgfsetdash{}{0pt}
\pgfsetmiterjoin
\pgfsetbuttcap
{
\definecolor{dialinecolor}{rgb}{0.000000, 0.000000, 0.000000}
\pgfsetfillcolor{dialinecolor}
% was here!!!
\definecolor{dialinecolor}{rgb}{0.000000, 0.000000, 0.000000}
\pgfsetstrokecolor{dialinecolor}
\pgfpathmoveto{\pgfpoint{1.900000\du}{-0.700000\du}}
\pgfpathcurveto{\pgfpoint{1.938750\du}{-2.666880\du}}{\pgfpoint{4.056250\du}{-2.784370\du}}{\pgfpoint{4.100000\du}{-0.800000\du}}
\pgfusepath{stroke}
}
\pgfsetlinewidth{0.060000\du}
\pgfsetdash{}{0pt}
\pgfsetdash{}{0pt}
\pgfsetbuttcap
{
\definecolor{dialinecolor}{rgb}{0.000000, 0.000000, 0.000000}
\pgfsetfillcolor{dialinecolor}
% was here!!!
\definecolor{dialinecolor}{rgb}{0.000000, 0.000000, 0.000000}
\pgfsetstrokecolor{dialinecolor}
\draw (0.104419\du,-2.182320\du)--(0.102436\du,-0.679024\du);
}
% setfont left to latex
\definecolor{dialinecolor}{rgb}{0.000000, 0.000000, 0.000000}
\pgfsetstrokecolor{dialinecolor}
\node[anchor=west] at (4.094341\du,1.480845\du){$A^{*}$};
% setfont left to latex
\definecolor{dialinecolor}{rgb}{0.000000, 0.000000, 0.000000}
\pgfsetstrokecolor{dialinecolor}
\node[anchor=west] at (3.968254\du,-1.178456\du){$^{*}A$};
\pgfsetlinewidth{0.060000\du}
\pgfsetdash{}{0pt}
\pgfsetdash{}{0pt}
\pgfsetbuttcap
\pgfsetmiterjoin
\pgfsetlinewidth{0.060000\du}
\pgfsetbuttcap
\pgfsetmiterjoin
\pgfsetdash{}{0pt}
\definecolor{dialinecolor}{rgb}{1.000000, 1.000000, 1.000000}
\pgfsetfillcolor{dialinecolor}
\fill (3.757541\du,-0.539893\du)--(3.757541\du,0.360107\du)--(4.628509\du,0.360107\du)--(4.628509\du,-0.539893\du)--cycle;
\definecolor{dialinecolor}{rgb}{0.000000, 0.000000, 0.000000}
\pgfsetstrokecolor{dialinecolor}
\draw (3.757541\du,-0.539893\du)--(3.757541\du,0.360107\du)--(4.628509\du,0.360107\du)--(4.628509\du,-0.539893\du)--cycle;
\pgfsetbuttcap
\pgfsetmiterjoin
\pgfsetdash{}{0pt}
\definecolor{dialinecolor}{rgb}{0.000000, 0.000000, 0.000000}
\pgfsetstrokecolor{dialinecolor}
\draw (3.757541\du,-0.539893\du)--(3.757541\du,0.360107\du)--(4.628509\du,0.360107\du)--(4.628509\du,-0.539893\du)--cycle;
% setfont left to latex
\definecolor{dialinecolor}{rgb}{0.000000, 0.000000, 0.000000}
\pgfsetstrokecolor{dialinecolor}
\node[anchor=west] at (4.620470\du,0.085762\du){$^{*}a_A$};
\end{tikzpicture}
\;.
  \end{equation}
  \end{definitionlist}
\end{definition}
A special Frobenius algebra can always be normalized by an appropriate scaling of $\Delta$ and $\epsilon$. 
If $A$ is a  special symmetric normalized Frobenius algebra, there exists a projector onto $ m \tensor{A} \leftidx{^*}{\widetilde{m}}{}$ for all $m,\widetilde{m} \in \ModCA$, that has the following graphical 
description, where we use the obvious picture to present the module multiplication $m \otimes A \rightarrow m$: 
\begin{equation}\label{equation:projector1}
  P_{m,\widetilde{m}}=
\ifx\du\undefined
  \newlength{\du}
\fi
\setlength{\du}{10\unitlength}
\begin{tikzpicture}[baseline]
\pgftransformxscale{1.000000}
\pgftransformyscale{-1.000000}
\definecolor{dialinecolor}{rgb}{0.000000, 0.000000, 0.000000}
\pgfsetstrokecolor{dialinecolor}
\definecolor{dialinecolor}{rgb}{1.000000, 1.000000, 1.000000}
\pgfsetfillcolor{dialinecolor}
\pgfsetlinewidth{0.060000\du}
\pgfsetdash{}{0pt}
\pgfsetdash{}{0pt}
\pgfsetbuttcap
{
\definecolor{dialinecolor}{rgb}{0.000000, 0.000000, 0.000000}
\pgfsetfillcolor{dialinecolor}
% was here!!!
\definecolor{dialinecolor}{rgb}{0.000000, 0.000000, 0.000000}
\pgfsetstrokecolor{dialinecolor}
\draw (2.980610\du,-3.293930\du)--(2.980610\du,3.306070\du);
}
\pgfsetlinewidth{0.060000\du}
\pgfsetdash{}{0pt}
\pgfsetdash{}{0pt}
\pgfsetbuttcap
{
\definecolor{dialinecolor}{rgb}{0.000000, 0.000000, 0.000000}
\pgfsetfillcolor{dialinecolor}
% was here!!!
\definecolor{dialinecolor}{rgb}{0.000000, 0.000000, 0.000000}
\pgfsetstrokecolor{dialinecolor}
\draw (-2.219390\du,-3.293930\du)--(-2.219390\du,3.306070\du);
}
% setfont left to latex
\definecolor{dialinecolor}{rgb}{0.000000, 0.000000, 0.000000}
\pgfsetstrokecolor{dialinecolor}
\node[anchor=west] at (-3.796167\du,-0.020097\du){$m$};
\pgfsetlinewidth{0.060000\du}
\pgfsetdash{}{0pt}
\pgfsetdash{}{0pt}
\pgfsetbuttcap
{
\definecolor{dialinecolor}{rgb}{0.000000, 0.000000, 0.000000}
\pgfsetfillcolor{dialinecolor}
% was here!!!
\definecolor{dialinecolor}{rgb}{0.000000, 0.000000, 0.000000}
\pgfsetstrokecolor{dialinecolor}
\draw (0.387500\du,-1.968750\du)--(0.389363\du,-0.995170\du);
}
\definecolor{dialinecolor}{rgb}{0.000000, 0.000000, 0.000000}
\pgfsetfillcolor{dialinecolor}
\pgfpathellipse{\pgfpoint{0.395241\du}{-2.131709\du}}{\pgfpoint{0.171450\du}{0\du}}{\pgfpoint{0\du}{0.155571\du}}
\pgfusepath{fill}
\pgfsetlinewidth{0.060000\du}
\pgfsetdash{}{0pt}
\pgfsetdash{}{0pt}
\definecolor{dialinecolor}{rgb}{0.000000, 0.000000, 0.000000}
\pgfsetstrokecolor{dialinecolor}
\pgfpathellipse{\pgfpoint{0.395241\du}{-2.131709\du}}{\pgfpoint{0.171450\du}{0\du}}{\pgfpoint{0\du}{0.155571\du}}
\pgfusepath{stroke}
% setfont left to latex
\definecolor{dialinecolor}{rgb}{0.000000, 0.000000, 0.000000}
\pgfsetstrokecolor{dialinecolor}
\node[anchor=west] at (2.980610\du,0.006070\du){$\leftidx{^*}{\widetilde{m}}{}$};
\pgfsetlinewidth{0.060000\du}
\pgfsetdash{}{0pt}
\pgfsetdash{}{0pt}
\pgfsetmiterjoin
\pgfsetbuttcap
{
\definecolor{dialinecolor}{rgb}{0.000000, 0.000000, 0.000000}
\pgfsetfillcolor{dialinecolor}
% was here!!!
\definecolor{dialinecolor}{rgb}{0.000000, 0.000000, 0.000000}
\pgfsetstrokecolor{dialinecolor}
\pgfpathmoveto{\pgfpoint{-0.517500\du}{0.430003\du}}
\pgfpathcurveto{\pgfpoint{-0.517500\du}{-1.470001\du}}{\pgfpoint{1.282500\du}{-1.470001\du}}{\pgfpoint{1.282500\du}{0.430003\du}}
\pgfusepath{stroke}
}
\pgfsetlinewidth{0.060000\du}
\pgfsetdash{}{0pt}
\pgfsetdash{}{0pt}
\pgfsetmiterjoin
\pgfsetbuttcap
{
\definecolor{dialinecolor}{rgb}{0.000000, 0.000000, 0.000000}
\pgfsetfillcolor{dialinecolor}
% was here!!!
\definecolor{dialinecolor}{rgb}{0.000000, 0.000000, 0.000000}
\pgfsetstrokecolor{dialinecolor}
\pgfpathmoveto{\pgfpoint{2.982741\du}{2.530791\du}}
\pgfpathcurveto{\pgfpoint{2.169689\du}{1.725061\du}}{\pgfpoint{1.270241\du}{1.518291\du}}{\pgfpoint{1.270241\du}{0.237041\du}}
\pgfusepath{stroke}
}
\pgfsetlinewidth{0.060000\du}
\pgfsetdash{}{0pt}
\pgfsetdash{}{0pt}
\pgfsetmiterjoin
\pgfsetbuttcap
{
\definecolor{dialinecolor}{rgb}{0.000000, 0.000000, 0.000000}
\pgfsetfillcolor{dialinecolor}
% was here!!!
\definecolor{dialinecolor}{rgb}{0.000000, 0.000000, 0.000000}
\pgfsetstrokecolor{dialinecolor}
\pgfpathmoveto{\pgfpoint{-2.198509\du}{2.505791\du}}
\pgfpathcurveto{\pgfpoint{-1.348509\du}{1.724541\du}}{\pgfpoint{-0.505311\du}{1.402741\du}}{\pgfpoint{-0.498509\du}{0.168291\du}}
\pgfusepath{stroke}
}
\end{tikzpicture}
\;.
\end{equation}
\begin{lemma}
  \label{lemma:techh-frob}
Let $A \in \Cat{C}$ be a special symmetric normalized Frobenius algebra in a pivotal finite tensor category. Then for every $m \in \ModCA$, the following two morphisms in $\Cat{C}$ agree. 
\begin{equation}
  \label{eq:left-right-prof}
\ifx\du\undefined
  \newlength{\du}
\fi
\setlength{\du}{10\unitlength}
\begin{tikzpicture}[baseline]
\pgftransformxscale{1.000000}
\pgftransformyscale{-1.000000}
\definecolor{dialinecolor}{rgb}{0.000000, 0.000000, 0.000000}
\pgfsetstrokecolor{dialinecolor}
\definecolor{dialinecolor}{rgb}{1.000000, 1.000000, 1.000000}
\pgfsetfillcolor{dialinecolor}
% setfont left to latex
\definecolor{dialinecolor}{rgb}{0.000000, 0.000000, 0.000000}
\pgfsetstrokecolor{dialinecolor}
\node[anchor=west] at (3.339426\du,1.652423\du){$A$};
\pgfsetlinewidth{0.060000\du}
\pgfsetdash{}{0pt}
\pgfsetdash{}{0pt}
\pgfsetbuttcap
{
\definecolor{dialinecolor}{rgb}{0.000000, 0.000000, 0.000000}
\pgfsetfillcolor{dialinecolor}
% was here!!!
\definecolor{dialinecolor}{rgb}{0.000000, 0.000000, 0.000000}
\pgfsetstrokecolor{dialinecolor}
\draw (2.532335\du,-2.403694\du)--(2.529802\du,-1.798860\du);
}
\definecolor{dialinecolor}{rgb}{0.000000, 0.000000, 0.000000}
\pgfsetfillcolor{dialinecolor}
\pgfpathellipse{\pgfpoint{2.532874\du}{-2.465055\du}}{\pgfpoint{0.100000\du}{0\du}}{\pgfpoint{0\du}{0.100000\du}}
\pgfusepath{fill}
\pgfsetlinewidth{0.060000\du}
\pgfsetdash{}{0pt}
\pgfsetdash{}{0pt}
\definecolor{dialinecolor}{rgb}{0.000000, 0.000000, 0.000000}
\pgfsetstrokecolor{dialinecolor}
\pgfpathellipse{\pgfpoint{2.532874\du}{-2.465055\du}}{\pgfpoint{0.100000\du}{0\du}}{\pgfpoint{0\du}{0.100000\du}}
\pgfusepath{stroke}
\pgfsetlinewidth{0.060000\du}
\pgfsetdash{}{0pt}
\pgfsetdash{}{0pt}
\pgfsetmiterjoin
\pgfsetbuttcap
{
\definecolor{dialinecolor}{rgb}{0.000000, 0.000000, 0.000000}
\pgfsetfillcolor{dialinecolor}
% was here!!!
\definecolor{dialinecolor}{rgb}{0.000000, 0.000000, 0.000000}
\pgfsetstrokecolor{dialinecolor}
\pgfpathmoveto{\pgfpoint{0.733860\du}{1.858116\du}}
\pgfpathcurveto{\pgfpoint{0.733860\du}{3.858116\du}}{\pgfpoint{-1.856448\du}{3.836825\du}}{\pgfpoint{-1.856448\du}{1.836825\du}}
\pgfusepath{stroke}
}
\pgfsetlinewidth{0.060000\du}
\pgfsetdash{}{0pt}
\pgfsetdash{}{0pt}
\pgfsetmiterjoin
\pgfsetbuttcap
{
\definecolor{dialinecolor}{rgb}{0.000000, 0.000000, 0.000000}
\pgfsetfillcolor{dialinecolor}
% was here!!!
\definecolor{dialinecolor}{rgb}{0.000000, 0.000000, 0.000000}
\pgfsetstrokecolor{dialinecolor}
\pgfpathmoveto{\pgfpoint{1.495875\du}{-1.064814\du}}
\pgfpathcurveto{\pgfpoint{1.447861\du}{-0.036607\du}}{\pgfpoint{0.985721\du}{0.160508\du}}{\pgfpoint{0.750740\du}{0.441339\du}}
\pgfusepath{stroke}
}
\pgfsetlinewidth{0.060000\du}
\pgfsetdash{}{0pt}
\pgfsetdash{}{0pt}
\pgfsetmiterjoin
\pgfsetbuttcap
{
\definecolor{dialinecolor}{rgb}{0.000000, 0.000000, 0.000000}
\pgfsetfillcolor{dialinecolor}
% was here!!!
\definecolor{dialinecolor}{rgb}{0.000000, 0.000000, 0.000000}
\pgfsetstrokecolor{dialinecolor}
\pgfpathmoveto{\pgfpoint{4.958164\du}{-1.951213\du}}
\pgfpathcurveto{\pgfpoint{4.958164\du}{-3.951219\du}}{\pgfpoint{0.726743\du}{-3.743975\du}}{\pgfpoint{0.726743\du}{-1.743979\du}}
\pgfusepath{stroke}
}
\pgfsetlinewidth{0.060000\du}
\pgfsetdash{}{0pt}
\pgfsetdash{}{0pt}
\pgfsetbuttcap
{
\definecolor{dialinecolor}{rgb}{0.000000, 0.000000, 0.000000}
\pgfsetfillcolor{dialinecolor}
% was here!!!
\definecolor{dialinecolor}{rgb}{0.000000, 0.000000, 0.000000}
\pgfsetstrokecolor{dialinecolor}
\draw (3.520762\du,-1.069759\du)--(3.518928\du,3.358539\du);
}
\pgfsetlinewidth{0.060000\du}
\pgfsetdash{}{0pt}
\pgfsetdash{}{0pt}
\pgfsetmiterjoin
\pgfsetbuttcap
{
\definecolor{dialinecolor}{rgb}{0.000000, 0.000000, 0.000000}
\pgfsetfillcolor{dialinecolor}
% was here!!!
\definecolor{dialinecolor}{rgb}{0.000000, 0.000000, 0.000000}
\pgfsetstrokecolor{dialinecolor}
\pgfpathmoveto{\pgfpoint{1.488206\du}{-1.030710\du}}
\pgfpathcurveto{\pgfpoint{1.493216\du}{-2.038101\du}}{\pgfpoint{3.512375\du}{-2.051121\du}}{\pgfpoint{3.512375\du}{-1.051124\du}}
\pgfusepath{stroke}
}
\pgfsetlinewidth{0.060000\du}
\pgfsetdash{}{0pt}
\pgfsetdash{}{0pt}
\pgfsetbuttcap
{
\definecolor{dialinecolor}{rgb}{0.000000, 0.000000, 0.000000}
\pgfsetfillcolor{dialinecolor}
% was here!!!
\definecolor{dialinecolor}{rgb}{0.000000, 0.000000, 0.000000}
\pgfsetstrokecolor{dialinecolor}
\draw (0.726743\du,-1.763257\du)--(0.746021\du,1.870561\du);
}
% setfont left to latex
\definecolor{dialinecolor}{rgb}{0.000000, 0.000000, 0.000000}
\pgfsetstrokecolor{dialinecolor}
\node[anchor=west] at (-4.037603\du,-1.682924\du){$m^{*}$};
% setfont left to latex
\definecolor{dialinecolor}{rgb}{0.000000, 0.000000, 0.000000}
\pgfsetstrokecolor{dialinecolor}
\node[anchor=west] at (4.869850\du,1.697518\du){$^{*}m$};
% setfont left to latex
\definecolor{dialinecolor}{rgb}{0.000000, 0.000000, 0.000000}
\pgfsetstrokecolor{dialinecolor}
\node[anchor=west] at (5.362912\du,0.360642\du){$^{*}a_m$};
\pgfsetlinewidth{0.060000\du}
\pgfsetdash{}{0pt}
\pgfsetdash{}{0pt}
\pgfsetbuttcap
{
\definecolor{dialinecolor}{rgb}{0.000000, 0.000000, 0.000000}
\pgfsetfillcolor{dialinecolor}
% was here!!!
\definecolor{dialinecolor}{rgb}{0.000000, 0.000000, 0.000000}
\pgfsetstrokecolor{dialinecolor}
\draw (-1.862704\du,-3.352740\du)--(-1.855043\du,1.892469\du);
}
\pgfsetlinewidth{0.060000\du}
\pgfsetdash{}{0pt}
\pgfsetdash{}{0pt}
\pgfsetbuttcap
{
\definecolor{dialinecolor}{rgb}{0.000000, 0.000000, 0.000000}
\pgfsetfillcolor{dialinecolor}
% was here!!!
\definecolor{dialinecolor}{rgb}{0.000000, 0.000000, 0.000000}
\pgfsetstrokecolor{dialinecolor}
\draw (4.953698\du,-1.979868\du)--(4.963251\du,3.351220\du);
}
\pgfsetlinewidth{0.060000\du}
\pgfsetdash{}{0pt}
\pgfsetdash{}{0pt}
\pgfsetbuttcap
\pgfsetmiterjoin
\pgfsetlinewidth{0.060000\du}
\pgfsetbuttcap
\pgfsetmiterjoin
\pgfsetdash{}{0pt}
\definecolor{dialinecolor}{rgb}{1.000000, 1.000000, 1.000000}
\pgfsetfillcolor{dialinecolor}
\fill (4.514851\du,-0.177542\du)--(4.514851\du,0.722458\du)--(5.385818\du,0.722458\du)--(5.385818\du,-0.177542\du)--cycle;
\definecolor{dialinecolor}{rgb}{0.000000, 0.000000, 0.000000}
\pgfsetstrokecolor{dialinecolor}
\draw (4.514851\du,-0.177542\du)--(4.514851\du,0.722458\du)--(5.385818\du,0.722458\du)--(5.385818\du,-0.177542\du)--cycle;
\pgfsetbuttcap
\pgfsetmiterjoin
\pgfsetdash{}{0pt}
\definecolor{dialinecolor}{rgb}{0.000000, 0.000000, 0.000000}
\pgfsetstrokecolor{dialinecolor}
\draw (4.514851\du,-0.177542\du)--(4.514851\du,0.722458\du)--(5.385818\du,0.722458\du)--(5.385818\du,-0.177542\du)--cycle;
% setfont left to latex
\definecolor{dialinecolor}{rgb}{0.000000, 0.000000, 0.000000}
\pgfsetstrokecolor{dialinecolor}
\node[anchor=west] at (4.898031\du,-1.222011\du){$m^{*}$};
\end{tikzpicture}
=
\ifx\du\undefined
  \newlength{\du}
\fi
\setlength{\du}{10\unitlength}
\begin{tikzpicture}[baseline]
\pgftransformxscale{1.000000}
\pgftransformyscale{-1.000000}
\definecolor{dialinecolor}{rgb}{0.000000, 0.000000, 0.000000}
\pgfsetstrokecolor{dialinecolor}
\definecolor{dialinecolor}{rgb}{1.000000, 1.000000, 1.000000}
\pgfsetfillcolor{dialinecolor}
% setfont left to latex
\definecolor{dialinecolor}{rgb}{0.000000, 0.000000, 0.000000}
\pgfsetstrokecolor{dialinecolor}
\node[anchor=west] at (-4.444195\du,2.105485\du){$A$};
\pgfsetlinewidth{0.060000\du}
\pgfsetdash{}{0pt}
\pgfsetdash{}{0pt}
\pgfsetbuttcap
{
\definecolor{dialinecolor}{rgb}{0.000000, 0.000000, 0.000000}
\pgfsetfillcolor{dialinecolor}
% was here!!!
\definecolor{dialinecolor}{rgb}{0.000000, 0.000000, 0.000000}
\pgfsetstrokecolor{dialinecolor}
\draw (-0.134770\du,-3.086895\du)--(-0.114361\du,-2.378725\du);
}
\definecolor{dialinecolor}{rgb}{0.000000, 0.000000, 0.000000}
\pgfsetfillcolor{dialinecolor}
\pgfpathellipse{\pgfpoint{-0.134770\du}{-3.186895\du}}{\pgfpoint{0.100000\du}{0\du}}{\pgfpoint{0\du}{0.100000\du}}
\pgfusepath{fill}
\pgfsetlinewidth{0.060000\du}
\pgfsetdash{}{0pt}
\pgfsetdash{}{0pt}
\definecolor{dialinecolor}{rgb}{0.000000, 0.000000, 0.000000}
\pgfsetstrokecolor{dialinecolor}
\pgfpathellipse{\pgfpoint{-0.134770\du}{-3.186895\du}}{\pgfpoint{0.100000\du}{0\du}}{\pgfpoint{0\du}{0.100000\du}}
\pgfusepath{stroke}
\pgfsetlinewidth{0.060000\du}
\pgfsetdash{}{0pt}
\pgfsetdash{}{0pt}
\pgfsetmiterjoin
\pgfsetbuttcap
{
\definecolor{dialinecolor}{rgb}{0.000000, 0.000000, 0.000000}
\pgfsetfillcolor{dialinecolor}
% was here!!!
\definecolor{dialinecolor}{rgb}{0.000000, 0.000000, 0.000000}
\pgfsetstrokecolor{dialinecolor}
\pgfpathmoveto{\pgfpoint{3.207360\du}{1.379075\du}}
\pgfpathcurveto{\pgfpoint{3.207360\du}{3.379075\du}}{\pgfpoint{0.722350\du}{3.371925\du}}{\pgfpoint{0.722350\du}{1.371925\du}}
\pgfusepath{stroke}
}
\pgfsetlinewidth{0.060000\du}
\pgfsetdash{}{0pt}
\pgfsetdash{}{0pt}
\pgfsetmiterjoin
\pgfsetbuttcap
{
\definecolor{dialinecolor}{rgb}{0.000000, 0.000000, 0.000000}
\pgfsetfillcolor{dialinecolor}
% was here!!!
\definecolor{dialinecolor}{rgb}{0.000000, 0.000000, 0.000000}
\pgfsetstrokecolor{dialinecolor}
\pgfpathmoveto{\pgfpoint{-1.137382\du}{-1.631915\du}}
\pgfpathcurveto{\pgfpoint{-1.185400\du}{-0.603707\du}}{\pgfpoint{-3.315570\du}{0.446919\du}}{\pgfpoint{-3.267370\du}{1.786705\du}}
\pgfusepath{stroke}
}
\pgfsetlinewidth{0.060000\du}
\pgfsetdash{}{0pt}
\pgfsetdash{}{0pt}
\pgfsetmiterjoin
\pgfsetbuttcap
{
\definecolor{dialinecolor}{rgb}{0.000000, 0.000000, 0.000000}
\pgfsetfillcolor{dialinecolor}
% was here!!!
\definecolor{dialinecolor}{rgb}{0.000000, 0.000000, 0.000000}
\pgfsetstrokecolor{dialinecolor}
\pgfpathmoveto{\pgfpoint{0.726090\du}{1.378815\du}}
\pgfpathcurveto{\pgfpoint{0.726090\du}{-0.621194\du}}{\pgfpoint{-1.252870\du}{-0.579561\du}}{\pgfpoint{-1.252870\du}{1.420435\du}}
\pgfusepath{stroke}
}
\pgfsetlinewidth{0.060000\du}
\pgfsetdash{}{0pt}
\pgfsetdash{}{0pt}
\pgfsetbuttcap
{
\definecolor{dialinecolor}{rgb}{0.000000, 0.000000, 0.000000}
\pgfsetfillcolor{dialinecolor}
% was here!!!
\definecolor{dialinecolor}{rgb}{0.000000, 0.000000, 0.000000}
\pgfsetstrokecolor{dialinecolor}
\draw (-3.261910\du,1.767495\du)--(-3.248090\du,3.377105\du);
}
\pgfsetlinewidth{0.060000\du}
\pgfsetdash{}{0pt}
\pgfsetdash{}{0pt}
\pgfsetmiterjoin
\pgfsetbuttcap
{
\definecolor{dialinecolor}{rgb}{0.000000, 0.000000, 0.000000}
\pgfsetfillcolor{dialinecolor}
% was here!!!
\definecolor{dialinecolor}{rgb}{0.000000, 0.000000, 0.000000}
\pgfsetstrokecolor{dialinecolor}
\pgfpathmoveto{\pgfpoint{0.752740\du}{1.629205\du}}
\pgfpathcurveto{\pgfpoint{1.240240\du}{1.091705\du}}{\pgfpoint{1.609850\du}{0.630056\du}}{\pgfpoint{1.417070\du}{0.013175\du}}
\pgfpathcurveto{\pgfpoint{1.224300\du}{-0.603707\du}}{\pgfpoint{0.897190\du}{-0.930280\du}}{\pgfpoint{0.886940\du}{-1.683245\du}}
\pgfusepath{stroke}
}
\pgfsetlinewidth{0.060000\du}
\pgfsetdash{}{0pt}
\pgfsetdash{}{0pt}
\pgfsetmiterjoin
\pgfsetbuttcap
{
\definecolor{dialinecolor}{rgb}{0.000000, 0.000000, 0.000000}
\pgfsetfillcolor{dialinecolor}
% was here!!!
\definecolor{dialinecolor}{rgb}{0.000000, 0.000000, 0.000000}
\pgfsetstrokecolor{dialinecolor}
\pgfpathmoveto{\pgfpoint{-1.137203\du}{-1.615775\du}}
\pgfpathcurveto{\pgfpoint{-1.132193\du}{-2.623165\du}}{\pgfpoint{0.896580\du}{-2.635055\du}}{\pgfpoint{0.896580\du}{-1.635055\du}}
\pgfusepath{stroke}
}
\pgfsetlinewidth{0.060000\du}
\pgfsetdash{}{0pt}
\pgfsetdash{}{0pt}
\pgfsetbuttcap
{
\definecolor{dialinecolor}{rgb}{0.000000, 0.000000, 0.000000}
\pgfsetfillcolor{dialinecolor}
% was here!!!
\definecolor{dialinecolor}{rgb}{0.000000, 0.000000, 0.000000}
\pgfsetstrokecolor{dialinecolor}
\draw (-1.240930\du,1.381295\du)--(-1.233590\du,3.357825\du);
}
\pgfsetlinewidth{0.060000\du}
\pgfsetdash{}{0pt}
\pgfsetdash{}{0pt}
\pgfsetbuttcap
{
\definecolor{dialinecolor}{rgb}{0.000000, 0.000000, 0.000000}
\pgfsetfillcolor{dialinecolor}
% was here!!!
\definecolor{dialinecolor}{rgb}{0.000000, 0.000000, 0.000000}
\pgfsetstrokecolor{dialinecolor}
\draw (3.190610\du,-3.172435\du)--(3.207390\du,1.424415\du);
}
% setfont left to latex
\definecolor{dialinecolor}{rgb}{0.000000, 0.000000, 0.000000}
\pgfsetstrokecolor{dialinecolor}
\node[anchor=west] at (3.083987\du,-2.478507\du){$m^{*}$};
% setfont left to latex
\definecolor{dialinecolor}{rgb}{0.000000, 0.000000, 0.000000}
\pgfsetstrokecolor{dialinecolor}
\node[anchor=west] at (3.129047\du,0.530052\du){$^{*}m$};
\pgfsetlinewidth{0.060000\du}
\pgfsetdash{}{0pt}
\pgfsetdash{}{0pt}
\pgfsetbuttcap
\pgfsetmiterjoin
\pgfsetlinewidth{0.060000\du}
\pgfsetbuttcap
\pgfsetmiterjoin
\pgfsetdash{}{0pt}
\definecolor{dialinecolor}{rgb}{1.000000, 1.000000, 1.000000}
\pgfsetfillcolor{dialinecolor}
\fill (2.816740\du,-1.507055\du)--(2.816740\du,-0.607055\du)--(3.687708\du,-0.607055\du)--(3.687708\du,-1.507055\du)--cycle;
\definecolor{dialinecolor}{rgb}{0.000000, 0.000000, 0.000000}
\pgfsetstrokecolor{dialinecolor}
\draw (2.816740\du,-1.507055\du)--(2.816740\du,-0.607055\du)--(3.687708\du,-0.607055\du)--(3.687708\du,-1.507055\du)--cycle;
\pgfsetbuttcap
\pgfsetmiterjoin
\pgfsetdash{}{0pt}
\definecolor{dialinecolor}{rgb}{0.000000, 0.000000, 0.000000}
\pgfsetstrokecolor{dialinecolor}
\draw (2.816740\du,-1.507055\du)--(2.816740\du,-0.607055\du)--(3.687708\du,-0.607055\du)--(3.687708\du,-1.507055\du)--cycle;
% setfont left to latex
\definecolor{dialinecolor}{rgb}{0.000000, 0.000000, 0.000000}
\pgfsetstrokecolor{dialinecolor}
\node[anchor=west] at (3.687708\du,-1.057055\du){$^*a_m$};
% setfont left to latex
\definecolor{dialinecolor}{rgb}{0.000000, 0.000000, 0.000000}
\pgfsetstrokecolor{dialinecolor}
\node[anchor=west] at (-1.368088\du,2.489190\du){$^{*}m$};
\end{tikzpicture}
\end{equation}
\end{lemma}
\begin{proof}
 This follows from a straightforward computation using the diagrammatic calculus: First use equation (\ref{symmetric-then-inverse}) on the left hand side, then the pivotal structure to transform the resulting right dual of a morphism in $\Cat{C}$ to the left dual. 
\end{proof}
Using this lemma, the following proposition is straightforward to show using the diagrammatic calculus. 
\begin{proposition}
  \label{proposition:crucial-Frob}
Let $\Cat{C}$ be a pivotal finite tensor category and $A \in \Cat{C}$ a special symmetric normalized Frobenius algebra.
The following diagram of morphisms in $\Cat{C}$ commutes
\begin{equation}
  \label{eq:comm-diag-proj}
  \begin{tikzcd}[column sep=large]
    \begin{array}{r}
      (\widetilde{m} \otimes \leftidx{^*}{m}{})^{*} \\
\simeq m \otimes \widetilde{m}^{*}
    \end{array}
\ar{r}{P^{*}_{\widetilde{m},m}} \ar{d}{1 \otimes \leftidx{^*}{a_{\widetilde{m}}}{}}  & 
\begin{array}{r}
      (\widetilde{m} \otimes \leftidx{^*}{m}{})^{*} \\
\simeq m \otimes \widetilde{m}^{*}
\end{array}
\ar{d}{1 \otimes \leftidx{^*}{a_{\widetilde{m}}}{}}  \\
m \otimes \leftidx{^*}{\widetilde{m}}{} 
\ar{r}{P_{m,\widetilde{m}}} & m \otimes \leftidx{^*}{ \widetilde{m}}{}.
  \end{tikzcd}
\end{equation}
\end{proposition}
This result allows to construct inner-product module categories from Frobenius algebras. 

\begin{theorem}
\label{theorem:Frob-then-inner-prod}
  Let $\Cat{C}$ be a pivotal finite tensor category and  $A \in \Cat{C}$ a  special symmetric normalized Frobenius algebra. Then the pivotal structure of $\Cat{C}$ induces the structure 
of an inner-product module category on $\CM= \ModCA$. 
\end{theorem}
\begin{proof}
According to Example \ref{example:inner-homs-algebra}, the inner hom of  $\CM= \ModCA$ is given by $\icm{m,\widetilde{m}}= m \tensor{A} \leftidx{^*}{\widetilde{m}}{}$. Thus we 
have to construct a natural isomorphism $$ \icmstar{\widetilde{m},m}=(\widetilde{m} \tensor{A} \leftidx{^*}{m}{})^{*} \rightarrow m \tensor{A} \leftidx{^*}{\widetilde{m}}{}.$$
  Since $P_{\widetilde{m},m}: \widetilde{m} \otimes \leftidx{^*}{m}{} \rightarrow \widetilde{m} \otimes \leftidx{^*}{m}{}$ is a projector onto $\widetilde{m} \tensor{A} \leftidx{^*}{m}{}$, it follows that 
also its right dual $P_{\widetilde{m},m}^{*}$ is a projector, indeed it projects onto $(\widetilde{m} \tensor{A} \leftidx{^*}{m}{})^{*}$. The commutativity of the diagram (\ref{eq:comm-diag-proj}) shows 
that the pivotal structure of $\Cat{C}$ descends to an isomorphism $$1 \tensor{A} \leftidx{^*}{a_{\widetilde{m}}}{}:(\widetilde{m} \tensor{A} \leftidx{^*}{m}{})^{*} \rightarrow m \tensor{A} \leftidx{^*}{\widetilde{m}}{}.$$
It is clear that this isomorphism is natural in $m$ and $\widetilde{m}$, so that we are left with showing that it is a bimodule natural isomorphism. 
Since it induced from the identity on  $m \in \Cat{M}$, it is clearly a module natural isomorphism with respect to $m$. In the other argument it follows from the fact that the pivotal structure is a monoidal 
natural isomorphism, that $1 \tensor{A} \leftidx{^*}{a_{\widetilde{m}}}{}$ respects the module structure. This implies that $\ModCA$ is an inner-product module category. 
\end{proof}
Examples of special symmetric Frobenius algebras in finite tensor categories are obtained from certain coends in \cite{SchwHopf}. 
\subsection{Inner-product bimodule categories in the  semisimple case}
We finally consider the case of semisimple bimodule categories over pivotal fusion categories.
For a semisimple module category $\CM$ over a pivotal fusion category $\Cat{C}$, a module trace is a $\Cat{C}$-balanced natural isomorphism 
$\Hom_{\Cat{M}}(\widetilde{m},m) \simeq \Hom_{\Cat{M}}(m,\widetilde{m})^{*}$. In \cite{Moduletr} it is shown that there is a one to one correspondence between 
module categories with module traces and special symmetric Frobenius algebras in $\Cat{C}$. 
In view of Theorem \ref{theorem:Frob-then-inner-prod} a module trace thus provides an inner-product module category. 
In the sequel we  clarify the relation of inner-product module categories and module categories with module traces more directly. 

Let $\Cat{C}$ be a pivotal tensor category with pivotal structure $a: \id \rightarrow (-)^{**}$. Then for every endomorphism $f: c\rightarrow c$ of an object $c \in \Cat{C}$ 
there exists the left and right trace $\tr^{L}(f), \tr^{R}(f) \in \End(\unit_{\Cat{C}})$, see Definition \ref{definition:pivotal-str} \refitem{item:traces}.  To emphasis  the dependence of the pivotal structure, 
we sometimes write $\tr^{L,a}$ and $\tr^{R,a}$. 
We recall the existence of conjugate pivotal structure for pivotal fusion categories. 
\begin{proposition}
  \label{proposition:conjugate-piv}
Let $\Cat{C}$ be a pivotal fusion category with pivotal structure $a: \id \rightarrow (-)^{**}$. There exists a pivotal structure $\cc{a}$ for $\Cat{C}$ that is 
uniquely characterized by the property $\tr^{R,\cc{a}}(f)= \tr^{L,a}(f)$ for all $f:c \rightarrow c$ and all $c \in \Cat{C}$. 
\end{proposition}
\begin{proof}
  This follows from the existence of a canonical monoidal natural isomorphism $\id \rightarrow (-)^{****}$ for fusion categories, see \cite{ENOfus}.  Explicitly, the conjugate pivotal structure is constructed in \cite[Sec. 4.3]{Moduletr}. 
The property  $\tr^{R,\cc{a}}(f)= \tr^{L,a}(f)$ is shown in the proof of \cite[Prop. 4.10]{Moduletr}.
\end{proof}
In particular, a pivotal fusion category is spherical if and only if the pivotal structure satisfies $a= \cc{a}$.
If $\Cat{C}$ is a pivotal fusion category, we denote by $\overline{\Cat{C}}$ the fusion category $\Cat{C}$ equipped with the conjugate pivotal structure.

Let $\Cat{C}, \Cat{D}$ be pivotal fusion categories, $\DMC$ a bimodule category. While the $\Hom_{\Cat{M}}$-functor always defines a multi-balanced functor $\Hom_{}: \CMDld \boxtimes \DMC \rightarrow \Vect$ in the sense of Definition \ref{definition:multi-bal-to-Vect}, see Example \ref{example:hom-mb}, a multi-balancing structure of the dual $\Hom_{}$-functor depends on the choice of pivotal structures for $\Cat{C}$ and $\Cat{D}$. 
We will consider the functor $\op{M} \boxtimes \Cat{M} \ni \widetilde{m} \boxtimes m \mapsto \Hom_{\Cat{M}}(m, \widetilde{m})^{*} \in \Vect$ as 
multi-balanced functor  $\CMDcld \boxtimes \DcMC \rightarrow \Vect $ and  $\CcMDld \boxtimes \DMCc \rightarrow \Vect $.

For the unit bimodule category $\CCC$ we obtain the following interpretation of the left and right traces in $\Cat{C}$. 
\begin{lemma} 
  \label{lemma:piv-str-conj-mb}
Let $\Cat{C}$ be a pivotal fusion category. 
\begin{lemmalist}
  \item The left trace $\tr^{L,a}$ defines a multi-balanced
natural isomorphism
\begin{equation}
  \label{eq:left-multi-bal}
  \tr^{L,a}: \Hom_{\Cat{C}}(\widetilde{c},c) \simeq \Hom_{\Cat{C}}(c, \widetilde{c})^{*},
\end{equation}
between multi-balanced functors $ \CcCCld \boxtimes \CCCc  \rightarrow \Vect$. 
\item The right trace $\tr^{R,a}$ defines a multi-balanced
natural isomorphism
\begin{equation} 
  \label{eq:right-multi-bal}
  \tr^{R,a}: \Hom_{\Cat{C}}(\widetilde{c},c) \simeq \Hom_{\Cat{C}}(c, \widetilde{c})^{*},
\end{equation}
between multi-balanced functors $ \CCCcld \boxtimes \CcCC  \rightarrow \Vect$. 
\end{lemmalist}
\end{lemma}
\begin{proof}
  Due to the semisimplicity of $\Cat{C}$, $\tr^{L,a}$ is an isomorphism. To show that $\tr^{L,a}$ is balanced with respect to the action of $\cc{\Cat{C}}$ amounts to the commutativity of 
the following  diagram  for $\widetilde{c},c,x \in \Cat{C}$
\begin{equation}
  \label{eq:comm-diag-bal-piv}
  \begin{tikzcd}
    \Hom_{}(\widetilde{c}, c \otimes x) \ar{r}{\tr^{L,a}} \ar{dd}{\simeq} & \Hom_{}(c \otimes x, \widetilde{c})^{*} \ar{d}{\simeq} \\
& \Hom_{}(c, \widetilde{c} \otimes x^{*}) \ar{d}{\Hom_{}(1, 1 \otimes \leftidx{^*}{\cc{a}_{x}}{})}\\
 \Hom_{}(\widetilde{c} \otimes \leftidx{^*}{x}{},c) \ar{r}{\tr^{L,a}} &\Hom_{}(c, \widetilde{c} \otimes \leftidx{^*}{x}{})^{*}.
  \end{tikzcd}
\end{equation}
This diagram commutes due to the characterization of the conjugate pivotal structure $\cc{a}$ in Proposition \ref{proposition:conjugate-piv}, as can be seen most easily using the diagrammatic calculus for $\Cat{C}$. 
The balancing with respect to $\Cat{C}$ and the second part are shown analogously. 
\end{proof}
Note that in the case of non-semisimple pivotal finite tensor categories, the pairings induced by the pivotal structure are degenerate \cite[Prop. 5.7]{DeligneSes}.

Lemma \ref{lemma:piv-str-conj-mb} allows us to characterize inner-product bimodule categories in the semisimple case. 
\begin{theorem}
\label{theorem:chara-inn-prod-fusion}
  Let $\Cat{C}$, $\Cat{D}$ be pivotal fusion categories and $\DMC$ a semisimple bimodule category. The structure of a inner-product bimodule category on $\DMC$ is equivalent to the collection of the following  
structures: 
\begin{itemize}
\item A multi-balanced natural isomorphism $\eta^{L}: \Hom_{\Cat{M}}(m, \widetilde{m}) \simeq \Hom_{\Cat{M}}(\widetilde{m},m)^{*}$ between 
multi-balanced functors $\CMDcld \boxtimes \DcMC \rightarrow \Vect $. 
\item A multi-balanced natural isomorphism $\eta^{R}: \Hom_{\Cat{M}}(m, \widetilde{m}) \simeq \Hom_{\Cat{M}}(\widetilde{m},m)^{*}$ between 
multi-balanced functors $\CcMDld \boxtimes \DMCc \rightarrow \Vect $. 
\end{itemize}
\end{theorem}
\begin{proof}
Assume that $\DMC$ has the structure of a $\Cat{D}$-inner-product bimodule category with balanced bimodule natural isomorphism $I^{\Cat{D}}: \idm{m,\widetilde{m}} \simeq \idmstar{\widetilde{m},m}$. 
Composing with  $\Hom_{}: \CMDld \boxtimes \DMC \rightarrow \Vect$, $I^{\Cat{D}}$ induces a multi-balanced natural isomorphism
$$ \Hom_{\Cat{D}}(\idm{m,\widetilde{m}},d) \simeq 
\Hom_{\Cat{D}}(\idmstar{\widetilde{m},m},d): \DMClld \boxtimes \CMDld \boxtimes \DDD \rightarrow \Vect.$$
 Here we used Lemma \ref{lemma:dual-and-tensor} to regard 
$\widetilde{m} \boxtimes m \boxtimes d$ as object in   $\DMClld \boxtimes \CMDld \boxtimes \DDD$.
Now consider the following chain of natural isomorphisms
\begin{equation}
  \label{eq:chain-tech}
  \begin{split}
      \Hom_{\Cat{M}}(m, d \act \widetilde{m}) &\simeq        \Hom_{\Cat{D}}(\idm{m,\widetilde{m}},d)        \stackrel{(I^{\Cat{D}})^{-1}}{\simeq}   \Hom_{\Cat{D}}(\idmstar{\widetilde{m},m},d) \\
& \simeq  \Hom_{\Cat{D}}(\leftidx{^*}{d}{}, \idm{\widetilde{m},m})\stackrel{\tr^{L,a}}{\simeq}   \Hom_{\Cat{D}}( \idm{\widetilde{m},m}, \leftidx{^*}{d}{})^{*}   \\
&\simeq \Hom_{\Cat{M}}(\widetilde{m}, \leftidx{^*}{d}{} \act m)^{*} \simeq \Hom_{\Cat{M}}(d \act \widetilde{m},m)^{*}. 
  \end{split}
\end{equation}
We claim that this chain consists of multi-balanced natural isomorphisms between multi-balanced functors  $\DcMClld \boxtimes \CMDld \boxtimes \DDDc \rightarrow \Vect$, 
where all functors with dual $\Hom_{}$-spaces use the pivotal structures for the balancing isomorphisms in all arguments as indicated by the indices. 
The forth natural isomorphism is induced by the left trace in $\Cat{D}$, which is a multi-balanced natural isomorphism between functors $\DcDDld \boxtimes \DDDc\rightarrow \Vect$ by
Lemma \ref{lemma:piv-str-conj-mb}. Composing with the balanced bimodule functor $\idm{-,-}$ thus gives the required multi-balanced natural isomorphism in step four. 
The remaining natural isomorphisms are multi-balanced for every choice of pivotal structure for the fusion category between neighboring arguments.  Hence the composition (\ref{eq:chain-tech}) 
is multi-balanced and thus induces 
a multi-balanced natural isomorphism $\Hom_{\Cat{M}}(m,\widetilde{m}) \simeq \Hom_{\Cat{M}}(\widetilde{m},m)^{*}$ between multi-balanced functors $\CMDcld \boxtimes \DcMC \rightarrow \Vect$. 

Conversely, such a  multi-balanced natural isomorphism  $\Hom_{\Cat{M}}(m,\widetilde{m}) \simeq \Hom_{\Cat{M}}(\widetilde{m},m)^{*}$ induces the structure of a $\Cat{D}$-inner-product bimodule category on $\DMC$ by applying this 
argument in the other direction. The equivalence of a $\Cat{C}$-inner-product bimodule category structure and  a multi-balanced natural isomorphism $\eta^{R}$ is shown analogously using the right trace in $\Cat{C}$. 
\end{proof}
The following example, that is taken from \cite[Ex. 3.13]{Moduletr}, shows that not every module category has the structure of an inner-product module category. 
\begin{example}
  Let $G$ be a finite group and $\Cat{C}= \Vect[G]$ the corresponding fusion category of $G$-graded vector spaces. The pivotal structures on $\Cat{C}$ are in bijection with group homomorphisms $\kappa: G \rightarrow \Bbbk^{\times}$. 
A subgroup $H \subset G$ yields a left $\Cat{C}$-module category $\CM = \Vect[H \backslash G]$ given by action of $G$ on the left cosets in $H \backslash G$. The module category $\CM$ has a module trace and 
thus the structure of an inner-product module category
if and only if $\kappa |_{H} =1$. 
\end{example}

For the case of spherical fusion categories there is a simpler way to obtain inner-product bimodule categories. This is of particular importance 
since spherical fusion category define a prominent 3-dimensional oriented topological field theory \cite{TurVir, BarWes}.
\begin{definition}[\cite{Moduletr}]
  A bimodule trace on an  semisimple $(\Cat{D},\Cat{C})$-bimodule category over pivotal fusion categories $\Cat{C}$ and $\Cat{D}$ is a multi-balanced natural isomorphism 
  \begin{equation}
    \label{eq:bimod-trace}
    \eta_{\widetilde{m},m}: \Hom_{\Cat{M}}(\widetilde{m},m) \rightarrow \Hom_{\Cat{M}}(m,\widetilde{m})^{*}.
  \end{equation}
\end{definition}

It follows that this structure is sufficient to define an inner-product bimodule category. 

\begin{corollary}
   Let $\DMC$ be an semisimple bimodule category over spherical fusion categories $\Cat{C}$ and $\Cat{D}$ with bimodule trace. 
Then $\DMC$ is canonically an inner-product bimodule category.
\end{corollary}
\begin{proof}
For a spherical structure the identity $a=\cc{a}$ for the spherical structure holds. Thus Theorem \ref{theorem:chara-inn-prod-fusion} shows that a bimodule trace on $\DMC$ defines the 
structure of an inner-product bimodule category on $\Cat{M}$.
\end{proof}
Note however that conversely an inner-product bimodule category over spherical fusion categories might yield two different bimodule traces by the correspondence in Theorem \ref{theorem:chara-inn-prod-fusion}. 

\subsection*{Acknowledgments}

Above all, the author thanks Catherine Meusburger for excellent support during the PhD. During many fruitful discussions with her this project  emerged and developed. 
He furthermore thanks Christopher Schommer-Pries for many useful comments and explanations, as well as Christoph Schweigert and 
J\"{u}rgen Fuchs for support and useful disussions.  The author is grateful to Catherine Meusburger and Alexei Davydov for careful proofreading and comments. 
  
The author thanks the  Mack Planck Institute for Mathematics in Bonn for funding and excellent working conditions, as well as the ESI for funding during 
two stays in February and March 2014. 
The work during the authors PhD was funded by the
research grant ME 3425/1-1 in the Emmy-Noether program of the German Research Foundation.  

The results of this work were presented at the workshop ``Modern trends in TQFT'' at the  ESI, March 2014. 
\appendix
\section{Duals in Bicategories}
In this section we summarize our conventions regarding duals in monoidal categories and in bicategories and  
state some basic results about bicategories.
In the next subsection we  recall the definition of a finite tensor category. 

\subsection{Duals in monoidal categories}
In the following definition we assume for simplicity that all monoidal categories are strict. 
\begin{definition}
  \label{definition:duality-tensor}
  Let $\Cat{C}$ be a monoidal category.
  \begin{definitionlist}

  \item A right dual of an object $x \in \Cat{C}$ is an object $x^{*} \in \Cat{C}$,  together with morphisms $\ev{x}: x^{*} \otimes x \rr 1$ and
    $\coev{x}: 1 \rr x \otimes x^{*}$ satisfying the so-called snake identities 
    \begin{equation}
      \label{eq:duality-snake-right}
      (1_{x} \otimes \ev{x})\cdot (\coev{x} \otimes 1_{x})=1_{x},  
    \end{equation}
    and 
    \begin{equation}
      \label{eq:snake2-right}
      (\ev{x} \otimes 1_{x^{*}}) \cdot (1_{x^{*}} \otimes \coev{x})=1_{x^{*}}.
    \end{equation}
    The morphisms $\ev{x}$ and $\coev{x}$ are called (right) duality morphisms of $x$.
  \item A left dual of an object $x \in \Cat{C}$ is an object $\leftidx{^{*}}{x}{} \in \Cat{C}$,  together with morphisms $\evp{x}: x \otimes \leftidx{^{*}}{x}{} \rr 1$ and
    $\coevp{x}: 1 \rr \leftidx{^{*}}{x}{} \otimes x$ satisfying the identities 
    \begin{equation}
      \label{eq:duality-snake-left}
      ( \evp{x} \otimes 1_{x})\cdot (1_{x} \otimes \coevp{x})=1_{x},  
    \end{equation}
    and 
    \begin{equation}
      \label{eq:snake2-left}
      ( 1_{\leftidx{^*}{x}{}} \otimes \evp{x} ) \cdot ( \coevp{x} \otimes 1_{\leftidx{^*}{x}{}} )=1_{\leftidx{^*}{x}{}}.
    \end{equation}
    The morphisms $\evp{x}$ and $\coevp{x}$ are called (left) duality morphisms of $x$.
  \item A monoidal category $\Cat{C}$ is said to have  right (left) duals if every object of $\Cat{C}$ has a right (left) dual object. In case every object of  $\Cat{C}$ has both a right and a left dual object, 
    $\Cat{C}$ is said to have duals and  $\Cat{C}$  is called rigid.
  \end{definitionlist}
\end{definition}

If $\Cat{C}$ has right duals, the duals are unique up to unique isomorphism and  the double dual functor is canonically a monoidal functor $(-)^{**}: \Cat{C} \rightarrow \Cat{C}$.
\begin{definition}
  \label{definition:pivotal-str}
  Let $\Cat{C}$ be a monoidal category with right duals. 
  \begin{definitionlist}
    \item A pivotal structure $a$ on $\Cat{C}$  is a monoidal natural isomorphism 
  \begin{equation}
    \label{eq:piv-str}
    a:  \id_{\Cat{C}} \rightarrow (-)^{**}.
  \end{equation}
A rigid monoidal category with pivotal structure is called a pivotal category. 
\item \label{item:traces} Let $f \in \Hom_{\Cat{C}}(c,c)$ be a morphism in a pivotal category. The right trace of $f$ is defined as 
  \begin{equation}
    \label{eq:left-trace}
    \tr^{R}(f)= \evp{x} (f \otimes \leftidx{^*}{a_{x}}{}) \coev{x} \in \End(\unit_{\Cat{C}})
  \end{equation} 
and the left trace is defined as 
\begin{equation}
  \label{eq:right-trace}
  \tr^{L}(f)=  \ev{x} (a_{\leftidx{^*}{x}{}} \otimes   f) \coevp{x} \in \End(\unit_{\Cat{C}}).
\end{equation}
The left dimension of an object $x \in \Cat{C}$ is defined as $\dim^{L}=\tr^{L}(1_{x})$ and the right dimension as $\dim^{R}=\tr^{R}(1_{x})$.
\item  \label{item:spherical} A pivotal structure $a$ on $\Cat{C}$ is called spherical  if $\tr^{L}(f)=\tr^{R}(f)$ for all $f \in \Hom_{\Cat{C}}(c,c)$ and all $c \in \Cat{C}$.
In this case $\Cat{C}$ is called a spherical category.
  \end{definitionlist} 
\end{definition}
Recall that an abelian category is called finite if every object has finite length, it has enough projectives and there are only finitely many isomorphism classes of simple objects. 

\begin{definition}[\cite{ENO:Notes}]
  \label{definition:Tensor-cat}
  \begin{definitionlist}
  \item A  tensor category $\Cat{C}$ over $\Bbbk$ is a $\Bbbk$-linear  abelian
    category with bilinear monoidal structure, finite dimensional $\Hom_{}$-spaces and right and left-duals for all objects, for which every object is of finite length and 
 in which the tensor unit $\unit_{\Cat{C}}$ satisfies $\End(\unit_{\Cat{C}})= \Bbbk$. A finite tensor category is a tensor category that is finite as abelian category. 
  \item  A tensor category $\Cat{C}$ is called a fusion category if it is  finite 
    semisimple as abelian category.
\item A pivotal (spherical) fusion category $\Cat{C}$ is a fusion category $\Cat{C}$ that is in addition a pivotal (spherical) category. 
  \end{definitionlist}
\end{definition}

\subsection{Bicategories}
We next recall the definitions of a bicategory, a 2-functor, a 2-natural transformation and a modification. 
\begin{definition}
  \label{definition:bicategories}
  A bicategory $\Cat{B}$ consists of the following data:
  \begin{definitionlist}
  \item a collection of objects  $a,b \in \Obj(\Cat{B})$, 
  \item for any two objects $a,b$ a category $\Cat{B}(a,b)$, whose objects are called 1-morphisms and denoted $F,G: a\rightarrow b$ and whose morphisms are called 2-morphisms and denoted $\eta: F \Rightarrow G$. 
The composition of 2-morphisms in $\Cat{B}(a,b)$ is called vertical composition,
  \item for any three objects $a,b,c$ a functor $\circ: \Cat{B}(b,c) \times \Cat{B}(a,b)\rightarrow \Cat{B}(a,c)$, called horizontal composition, and 
    for any object $b$ a functor $I_{b}: I \rightarrow \Cat{B}(b,b)$, where $I$ is the unit category with one object and one morphism. The image of $I_{b}$ on the object of $I$ is called $1_{b}: b \rightarrow b$ and  
    the image on the morphism is called $1_{1_{b}}:1_{b} \Rightarrow 1_{b}$,
  \item for any  three 1-morphisms $F: c \rightarrow d$, $G: b \rightarrow c$ and $H: a \rightarrow b$,  invertible 2-morphisms $\omega^{\Cat{B}}_{F,G,H}: (F \circ G)\circ H \Rightarrow F\circ (G\circ H)$,
  \item for any 1-morphism  $F:a \rightarrow b$  invertible 2-morphisms $\lambda^{\Cat{B}}_{F}: I_{b} \circ F \Rightarrow F$ and  $\rho_{F}^{\Cat{B}}: F\circ I_{a} \Rightarrow F$,  
  \end{definitionlist}
  such that the 2-morphisms $\omega^{\Cat{B}}_{H,G,F}$,$\lambda_{F}^{\Cat{B}}$ and $\rho^{\Cat{B}}_{F}$ are natural in their arguments and the following diagrams commute for all 1-morphisms where these expressions are defined
  \begin{equation}
    \label{eq:diagramm-tensor-bicat}
    \begin{tikzcd}
      {} & ((F \circ G) \circ H) \circ K   \ar{ld}{\omega^{\Cat{B}}_{F,G,H} \circ 1_{K}}
      \ar{rd}{\omega^{\Cat{B}}_{F\circ G,H,K}}& \\
      (F \circ (G \circ H)) \circ K \ar{d}{\omega^{\Cat{B}}_{F,G\circ H,K}} & & (F
      \circ G ) \circ (H \circ K) \ar{d}{\omega^{\Cat{B}}_{F,G, H\circ K}}\\
      F \circ (( G \circ H) \circ K)\ar{rr}{1_F \circ \omega^{\Cat{B}}_{G,H,K}}& & F \circ (G \circ (H \circ K)),
    \end{tikzcd}   
  \end{equation}
   \begin{equation}
    \label{eq:triangle-bicat}
    \begin{tikzcd}
      (F\circ 1_{a}) \circ  G \ar{rr}{\omega^{\Cat{B}}_{F,1,G}}\ar{dr}{\rho_{F}^{\Cat{B}}  \circ G} & & F \circ (1_{a} \circ  G) \ar{dl}{1_{F} \circ \lambda_{G}}\\
      {} & F \circ G. & 
    \end{tikzcd}
  \end{equation}
  A 2-category $\Cat{B}$ is   a strict bicategory $\Cat{B}$, i.e.  a bicategory, in which 
  all 2-morphisms  $\omega^{\Cat{B}}_{H,G,F}$, $\lambda_{F}^{\Cat{B}}$ and $\rho^{\Cat{B}}_{F}$ are  identities.
\end{definition}

The notion of equivalence of categories can be formulated in a general bicategory as follows. 
\begin{definition}
   \label{definition:adj-equiv}
  Let $\Cat{B}$ be a bicategory. 
  \begin{definitionlist}
    \item \label{item:def-eqiv-inbicat} Two objects $b,c$ in $\Cat{B}$ are called equivalent, if there exist 1-morphism $F:b \rightarrow c$ and $G: c \rightarrow b$ together with 
  invertible 2-morphisms $\eta: F \circ G  \Rightarrow 1_{c}$ and $\rho: G \circ F \Rightarrow 1_{b}$. 
\item An adjoint equivalence  $(f,g,\alpha, \beta  )$ between objects $b$ and $c$ in $\Cat{B}$ consists of 1-morphisms $f: b \rightarrow c$ and 
  $g: c\rightarrow b$, together with isomorphisms $\alpha: fg \rightarrow 1_{c}$ and $\beta :  1_{b} \rightarrow gf$, such that 
  the snake identities (\ref{eq:duality-snake-right}) and (\ref{eq:snake2-right}) hold in the monoidal categories $\Cat{B}(b,b)$ and $\Cat{B}(c,c)$.
  \end{definitionlist}
\end{definition}

\begin{definition}\label{lax2func}
  A 2-functor $\mathsf{F}:\mathcal C\rightarrow \mathcal D$ between bicategories $\mathcal C,\mathcal D$ is given by the following data
  \begin{definitionlist}
  \item A function $\mathsf{F}_0: \Obj(\mathcal C)\rightarrow \Obj(\mathcal D)$.
  \item For all objects $a,b$ of $\mathcal C$, a functor $\mathsf{F}_{a,b}: \mathcal C_{a,b}\rightarrow \mathcal D_{\mathsf{F}_0(a), \mathsf{F}_0(b)}$.
  \item For all objects $a,b,c$ of $\mathcal C$ a natural isomorphism $\Phi_{abc}\colon\circ\;(\mathsf{F}_{b,c}\times \mathsf{F}_{a,b})\to \mathsf{F}_{a,c}\;\circ$. These determine, for all 1-morphisms $H: a\rightarrow b$, $G: b\rightarrow c$, a invertible 2-morphism $\Phi_{G,H}: \mathsf{F}_{b,c}(G)\circ \mathsf{F}_{a,b}(H)\rightarrow \mathsf{F}_{a,c}(G\circ H)$. 
  \item For all objects $a$, an invertible 2-morphism $\Phi_a\colon 1_{\mathsf{F}_0(a)}\to \mathsf{F}_{a,a}(1_a)$.
  \end{definitionlist}
  The function $\mathsf{F}_0$, the functors $\mathsf{F}_{a,b}$ and the 2-morphisms $\Phi_{G,H}$ and $\Phi_a$ are required to satisfy the following consistency conditions
  \begin{definitionlist} \addtocounter{enumi}{4}
  \item
\label{item:lax-2-1}For all  1-morphisms $H: a\rightarrow b$:
    \begin{equation}
      \label{eq:2-fun-ax-unit}
      \begin{tikzcd}
        { \begin{array}{r}
            \mathsf{F}_{a,b}(H)=\mathsf{F}_{a,b}(H) \circ 1_{F_{0}(a)}  \\
                             = 1_{F_{0}(b)}  \circ \mathsf{F}_{a,b}(H)
          \end{array}}  \ar{r}{1_{\mathsf{F}_{a,b}(H)} \circ \Phi_{a}}  \ar{dr}{\id} \ar{d}{\Phi_{b}\circ 1_{\mathsf{F}_{a,b}(H)}} & \mathsf{F}_{a,b}(H) \circ \mathsf{F}_{a,a}(1_{a}) \ar{d}{\Phi_{H,1_{a}}} \\
 \mathsf{F}_{b,b}(1_{b}) \circ \mathsf{F}_{a,b}(H)  \ar{r}{\Phi_{1_{b},H}} &\mathsf{F}_{a,b}(H \circ 1_{a})= \mathsf{F}_{a,b}(1_{b} \circ H)=\mathsf{F}_{a,b}(H)
      \end{tikzcd}
    \end{equation}
  \item \label{item:lax2-2} For all 1-morphisms $H: a\rightarrow b$, $G: b\rightarrow c$, $K: c\rightarrow d$, the following diagram commutes
    \begin{align}
      \label{2-fun-ax-3-morphism}
      \xymatrix{
        \mathsf{F}_{c,d}(K)\circ \mathsf{F}_{b,c}(G)\circ \mathsf{F}_{a,b}(H)\ar[d]^{1\circ \Phi_{G,H}} \ar[r]^{\qquad\Phi_{K,G}\circ 1} & \mathsf{F}_{b,d}(K\circ G)\circ \mathsf{F}_{a,b}(H)\ar[d]^{\Phi_{K\circ G,H}}\\
        \mathsf{F}_{c,d}(K)\circ \mathsf{F}_{a,c}(G\circ H) \ar[r]^{\Phi_{K,G\circ H}} & \mathsf{F}_{a,d}(K\circ G\circ H).
      }
    \end{align}
  \end{definitionlist}
A 2-functor is said to have strict units if the 2-morphisms  $\Phi_a$ are all identities, and it is
  called strict if the 2-morphisms $\Phi_{G,F}$ and $\Phi_a$ are all identities. In this case, one has
  $$\mathsf{F}_{a,c}(G\circ H)=\mathsf{F}_{b,c}(G)\circ \mathsf{F}_{a,b}(H)\qquad 1_{\mathsf{F}_0(a)}=\mathsf{F}_{a,a}(1_a).$$
\end{definition}

The following notion of natural 2-transformation of 2-functors  adopts  the convention of \cite{GPS, Gurski} and is sometimes also referred to as `oplax 2-transformation'.

\begin{definition} \label{definition:nat2transformations}
  \begin{definitionlist}
    \item  A natural 2-transformation $\rho    : \mathsf{F}\to \mathsf{G}$ between  2-functors $\mathsf{F},\mathsf{G}: \Cat{C} \rightarrow \Cat{D}$
   is given by the following data:
  \begin{enumerate}
  \item For all objects $a$ of $\mathcal C$, a 1-morphism $\rho_{a}: \mathsf{F}_0(a)\to \mathsf{G}_0(a)$.
  \item For all objects $a,b$ of $\mathcal C$ a natural transformation $$\rho_{a,b}: (\rho_b\circ -) \mathsf{F}_{a,b} \rightarrow  (-\circ \rho_a) \mathsf{G}_{a,b} ,$$ where  $-\circ \rho_a:\mathcal D_{\mathsf{G}_0(a), \mathsf{G}_0(b)}\to\mathcal D_{\mathsf{F}_0(a), \mathsf{G}_0(b)}$ and  $\rho_b\circ -: \mathcal D_{\mathsf{F}_0(a),\mathsf{F}_0(b)}$ $\to \mathcal D_{\mathsf{F}_0(a), \mathsf{G}_0(b)}$ denote the functors given by pre- and post-composition with $\rho_a$ and $\rho_b$. These natural transformations determine for all 1-morphisms $ H : a\to b$ a 2-morphism
    $\rho_H: \rho_b\circ \mathsf{F}_{a,b}(H) \to  \mathsf{G}_{a,b}(H)\circ \rho_a$.
  \end{enumerate}
  The 1-morphisms $\rho_a$ and 2-morphisms $\rho_{H}$ are required to satisfy  the following consistency conditions:
  \begin{enumerate}
  \item \label{item:ax1}For all 1-morphisms $H: a\to b$ and $K:b\to c$ the following diagram commutes
    \begin{align}
      \xymatrix{\rho_c\circ \mathsf{F}_{b,c}(K)\circ \mathsf{F}_{a,b}(H)  \ar[d]^{1\circ \Phi_{K,H}}  \ar[r]^{\rho_K \circ 1} & \mathsf{G}_{b,c}(K)\circ \rho_b\circ \mathsf{F}_{a,b}(H)
        \ar[d]^{1 \circ \rho_H} \\
        \rho_c\circ \mathsf{F}_{a,c}(K\circ H) \ar[d]^{\rho_{K\circ H}} &  \mathsf{G}_{b,c}(K)\circ \mathsf{G}_{a,b}(H)\circ \rho_a  \ar[ld]^{\Psi_{K,H}\circ 1 }\\
        \mathsf{G}_{a,c}(K\circ H)\circ \rho_a.
      }\nonumber
    \end{align}
  \item \label{item:ax2} For all objects $a$ of $\mathcal C$ the following diagram commutes
    \begin{align}
      \xymatrix{
        1_{\mathsf{G}_0(a)}\circ \rho_a=  \rho_a=\rho_a\circ 1_{\mathsf{F}_0(a)} \ar[d]^{ 1\circ \Phi_a} \ar[rd]^{\Psi_a\circ 1}\\
        \rho_a\circ \mathsf{F}_{a,a}(1_a)  \ar[r]^{\rho_{1_a}} &  \mathsf{G}_{a,a}(1_a)\circ \rho_a .
      }\nonumber
    \end{align}
  \end{enumerate}
\item A pseudo-natural transformation $\rho: \mathsf{F}\to \mathsf{G}$ of 2-functors $\mathsf{F},\mathsf{G}:\mathcal C\to\mathcal D$ is a natural 2-transformation of 2-functors in which all  
  2-morphisms $\rho_{H}: \rho_a\circ \mathsf{F}_{a,b}(H) \to \mathsf{G}_{a,b}(H)\circ \rho_a$ are isomorphisms. 
\item 
  A pseudo-natural transformation $\rho$ is called an equivalence if all the 1-morphisms $\rho_{a}$ are equivalences in the bicategory $\Cat{D}$, see Definition \ref{definition:adj-equiv} \refitem{item:def-eqiv-inbicat}.
\item 
  A 1-identity natural 2-transformation $\rho: \mathsf{F} \rightarrow \mathsf{G}$ between 2-functors $\mathsf{F}$ and $\mathsf{G}$ such that $\mathsf{F}_{0}(a)=\mathsf{F}_{0}(a)$ for all objects $a$ of $\Cat{C}$ is a natural $\mathsf{F}$
  2-transformation $\rho$ such that all 1-morphisms $\rho_{a}$ are the identities for all objects $a$ of $\Cat{C}$. 
\item 
  A natural 2-isomorphism is a pseudo-natural transformation which is a 1-identity natural 2-transformation. 
  \end{definitionlist}
 \end{definition}

\begin{definition}
  \label{modi_gen}
  Let $\rho=(\rho_{a}, \rho_{a,b}): \mathsf{F}\to \mathsf{G}$ and  $\tau=(\tau_a,\tau_{a,b}): \mathsf{F}\to \mathsf{G}$ be natural 2-transformations between 2-functors $\mathsf{F}=(\mathsf{F}_0,\mathsf{F}_{a,b}, \Phi_{H,K}, \Phi_a), \mathsf{G}=(\mathsf{G}_0,\mathsf{G}_{a,b}, \Psi_{H,K}, \Psi_{a}):\mathcal C\to\mathcal D$. A modification $\Psi: \rho\Rightarrow\tau$ is a collection of 2-morphisms $\Psi_a: \rho_a\Rightarrow\tau_a$ for every object $a$ of $\mac \mathsf{G}$ such that for all 1-morphisms $H: a\to b$ 
  $$
  \tau_H\cdot (\Psi_a\circ 1_{\mathsf{F}_{a,b}(H)})= (1_{\mathsf{G}_{a,b}(H)}\circ \Psi_b) \cdot \rho_H
  $$
  A modification is called invertible if all 2-morphisms $\Psi_a$ are invertible.
\end{definition}

\begin{lemma}
  \label{lemma:functor-presevers-adj-equiv}
  Let $\mathsf{F}:\Cat{A}\rightarrow \Cat{B}$ be a 2-functor between bicategories $\Cat{A}$ and $\Cat{B}$. If $(f,g,\alpha,\beta )$ is an adjoint equivalence between two objects $x$ and $y$ in $\Cat{A}$, 
  then $(\mathsf{F}(f),\mathsf{F}(g),\mathsf{F}(\alpha),\mathsf{F}(\beta ))$ is an adjoint equivalence between $\mathsf{F}_{0}(x)$ and $\mathsf{F}_{0}(y)$ in $\Cat{B}$.
\end{lemma}
\begin{proof}
  The proof of this statement is a combination of the proof that monoidal functors respect duality and the fact that functors respect isomorphisms.
\end{proof}

\subsection{Duals in Bicategories}
In this subsection we discuss duals and pivotal structures in bicategories.

\begin{definition}
  \label{definition:duality-bicat}
  Let $\Cat{X}$ be a bicategory. 
  \begin{definitionlist}
  \item  A  right dual of a 1-morphism $F: c \rightarrow d $ in  $\Cat{X}$ is a 1-morphism $F^{*}: d \rightarrow c$ such that there exist 2-morphisms 
$\ev{F}: F^{*} \circ F \rightarrow 1_{c}$ and $\coev{F}: 1_{d} \rightarrow F \circ F^{*}$ that satisfy the snake identities (\ref{eq:duality-snake-right}) and (\ref{eq:snake2-right}) with the monoidal product replaced by the
horizontal composition.  The 2-morphisms $\ev{F}$ and  $\coev{F}$ are called right duality morphisms. If every 1-morphism in $\Cat{X}$ has a right dual then the bicategory $\Cat{X}$ is said to have right duals.
  \item A left dual of a 1-morphism $F:c \rightarrow d$ is a 1-morphism $\leftidx{^{*}}{F}{}: d \rightarrow c$ such that there exist 2-morphisms 
    $\evp{F}: F \circ \leftidx{^{*}}{F}{} \rightarrow 1_{d} $ and $\coevp{F}: 1_{c} \rightarrow \leftidx{^{*}}{F}{} \circ F$ that satisfy  
the snake identities (\ref{eq:duality-snake-left}) and (\ref{eq:snake2-left}). 
    The 2-morphisms $\evp{F}$ and $\coevp{F}$ are called left duality 2-morphisms.  If every 1-morphism in $\Cat{X}$ has a left dual then the bicategory $\Cat{X}$ is said to have left duals.
  \end{definitionlist}
\end{definition}

As for a monoidal category, duals in a bicategory are unique up to unique isomorphism.

\begin{lemma}
  \label{lemma:functor-respect-dual}
 Let $\mathsf{F}: \Cat{X} \rightarrow \Cat{Y}$ be a  2-functor between bicategories.
\begin{lemmalist} 
  \item For every right dual  $G^{*}:c \rightarrow b$ of a 1-morphism $G: b \rightarrow c$ in $\Cat{X}$, $\mathsf{F}(G^{*})$ is a right 
    dual of $\mathsf{F}(G)$.
  \item Let  $\Cat{X}$ and $\Cat{Y}$ be bicategories with  right duals. 
    There exists a natural 2-isomorphism $\xi^{\mathsf{F}}: (-)^{*} \circ \mathsf{F} \rightarrow \mathsf{F} \circ (-)^{*}$, that is uniquely 
    determined by 
    \begin{equation}
      \label{eq:unique-duali-xiF}
      (1_{\mathsf{F}(G)} \circ \xi^{\mathsf{F}}_{G}) \cdot \coev{\mathsf{F}(G)} =\mathsf{F}( \coev{G})
    \end{equation}
    for all 1-morphisms $G$ in $\Cat{X}$.
 \item \label{item:transport-duals}Let $\mathsf{F}: \Cat{X}\rightarrow \Cat{Y}$ be a biequivalence of bicategories $\Cat{X}$ and $\Cat{Y}$.
    If $\Cat{X}$ has (right) duals, then $\Cat{Y}$ has (right) duals as well. 
  \end{lemmalist}
\end{lemma}

\begin{definition}
  \label{definition:pivotal-str-bicat}
  Let $\Cat{X}$ be a bicategory with right duals. A pivotal structure $a$ on $\Cat{X}$  is a natural 2-isomorphism 
  \begin{equation}
    \label{eq:piv-str-bicat}
    a: \id_{\Cat{X}} \rightarrow (-)^{**}.
  \end{equation}
\end{definition}

\begin{definition}
  \label{definition:pivotal-2-fun}
  A 2-functor $\mathsf{F}: \Cat{X} \rightarrow \Cat{Y}$ between pivotal bicategories is called pivotal, if the diagram
  \begin{equation}
    \label{eq:piv-2-fun}
    \begin{tikzcd}
      \mathsf{F}(H)^{**}    & \mathsf{F}(H)   \ar{l}{a_{\mathsf{F}(H)}^{\Cat{Y}}} \ar{d}{\mathsf{F}(a_{H}^{\Cat{X}})}\\
      \mathsf{F}(H^{*})^{*} \ar{u}{(\xi_{H}^{\mathsf{F}})^{*}} \ar{r}{\xi_{H^{*}}} &  \mathsf{F}(H^{**}) 
    \end{tikzcd}
  \end{equation}
  commutes for all 1-morphisms $H: a \rightarrow b$.
\end{definition}

\section{Tricategories}
 
The following definition is a slight modification from  \cite[Def. 3.1.2]{Gurski}.
\begin{definition}
  \label{definition:tricategory}
 A tricategory $\Cat{T}$  consists of the following data 
\begin{definitionlist}
\item A set of objects $a,b \in \Obj(\Cat{T})$.
\item  \label{item:2-fun} For any two objects $a,b$ a  bicategory  $\Cat{T}(a,b)$  of 1- and 2-morphisms with  horizontal composition $\circ$ and   vertical composition $\cdot$. 
\item For any three objects $a,b,c$,  2-functors 
  \begin{equation}
    \label{eq:Box-prod}
    \Box: \Cat{T}(b,c) \times \Cat{T}(a,b) \rr \Cat{T}(a,c),
  \end{equation}
called $\Box$-product  of 1-morphisms.
\item \label{item:units} For any object $a$  a 2-functor $I_{a}:I \rr \Cat{T}(a,a)$, where $I$ denotes the 
  unit 2-category with one object $1$, one 1-morphism $1_{1}$ and one 2-morphism $1_{1_{1}}$. The image of the functor $I_{a}$ on the object of $I$ is the 1-morphism also denoted $I_{a}:a \rightarrow a$.  
\item \label{item:associator}For any four objects $a,b,c,d$,  an adjoint equivalence
  $a: \Box (\Box \times 1) \Rightarrow \Box( 1 \times \Box)$, called  associator. More precisely,  $a$ consists of a  pseudo-natural transformation
  \begin{equation}
    \label{eq:adj-equiv-asso-bimod}
    \begin{xy}
      \xymatrix{
        \Cat{T}(c,d) \times \Cat{T}(b,c) \times \Cat{T}(a,b) \ar[rr]^{\Box \times 1}  \ar[d]_{1 \times \Box} \rrtwocell<\omit>{<4>a} &&\Cat{T}(b,d) \times \Cat{T}(a,b) \ar[d]^{\Box}  \\
        \Cat{T}(c,d) \times \Cat{T}(a,c) \ar[rr]_{\Box} && \Cat{T}(a,d),
      }
    \end{xy}
  \end{equation}
  and, a pseudo-natural transformation  $a^{-}:\Box( 1 \times \Box)\rightarrow \Box (\Box \times 1) $, such that $a$ and $a^{-}$ form an adjoint equivalence, see Definition \ref{definition:adj-equiv}.
\item \label{item:unit-2-cells-bimod}For any two objects $a,b$, there are adjoint equivalences $l: \Box(I_{b} \times 1)\Rightarrow 1$ and $r: \Box (1\times I_{a}) \Rightarrow 1$, called the unit 2-morphisms, 
  \begin{equation}
    \label{eq:unitsleft}
    \begin{tikzcd}
      {}         &\Cat{T}(b, b) \times \Cat{T}(a,b)  \ar{dr}{\Box}  \ar[shorten <= 7pt, shorten >=7pt,Rightarrow]{d}{l} \\
      \Cat{T}(a,b)  \ar{rr}{1} \ar{ur}{I_{b} \times 1}   & {}& \Cat{T}(a,b)
    \end{tikzcd}
  \end{equation}
  and 
  \begin{equation}
    \label{eq:unitsright}
    \begin{tikzcd}
      {}   &\Cat{T}(a,b) \times \Cat{T}(a,a)  \ar{dr}{\Box}  \ar[shorten <= 7pt, shorten >=7pt,Rightarrow]{d}{r}& \\
      \Cat{T}(a,b)  \ar{rr}[below, name=B]{1} \ar{ur}{1 \times I_{a} }   & {}& \Cat{T}(a,b).
    \end{tikzcd}
  \end{equation}
  By definition of an adjoint equivalence, $l$ and $r$ are pseudo-natural transformations. Furthermore there are  corresponding pseudo-natural transformations $l^{-}: 1 \Rightarrow \Box(I_{b} \times 1)$ and
  $r^{-}:1\Rightarrow \Box (1\times I_{a})$.
\item \label{item:mu} For all objects $a,b,c$, an invertible modification $\mu$
  \begin{equation}
    \label{eq:3}
    \begin{tikzcd}[baseline=-0.65ex, column sep=large, row sep=large]
      \Cat{T}^{2}  \ar{dr}[below, left, name=A]{1}\ar{r}{1 \times I \times 1} & \Cat{T}^{3} \ar{r}{\Box \times 1} \ar{d}{1 \times \Box} 
\arrow[shorten <= 5pt, shorten >=5pt, Rightarrow,to path=-- (A) \tikztonodes]{}{1 \times l}
& \Cat{T}^{2} \ar{d}{\Box} \ar[shorten <= 10pt, shorten >=10pt, Rightarrow]{dl}{a} \\
        {} & \Cat{T}^{2} \ar{r}{\Box} &\Cat{T} 
    \end{tikzcd}
\stackrel{\mu}{\Rrightarrow}
    \begin{tikzcd}[baseline=-0.65ex, column sep=large, row sep=large]
       \Cat{T}^{2}  \ar{dr}[below, left]{1}\ar{r}{1 \times I \times 1} \ar[bend right]{rr}[below, name=B]{1} & \Cat{T}^{3} \ar{r}{\Box \times 1}
 \arrow[shorten <= 2pt, shorten >=2pt, Rightarrow,to path=-- (B) \tikztonodes]{}{ r \times 1}
 & \Cat{T}^{2} \ar{d}{\Box} \\
        {} & \Cat{T}^{2} \ar{r}{\Box} &\Cat{T},
    \end{tikzcd}
  \end{equation}
where we used for example the abbreviation  $\Cat{T}^{3}=  \Cat{T}(b,c) \times  \Cat{T}(b,b) \times\Cat{T}(a,b)$.

\item \label{item:lambda-general}  For all objects $a,b,c$, an invertible modification $\lambda$
\begin{equation}
    \label{eq:lambda-gen}
    \begin{tikzcd}[baseline=-0.65ex, column sep=large, row sep=large]
      {} & \Cat{T}^{3} \ar{dr}{\Box \times 1} \ar[shorten <= 8pt, shorten >=8pt, Rightarrow]{d}{l\times 1} &  {}   \\
   \Cat{T}^{2} \ar{ur}{I \times 1 \times 1} \ar{rr}[below]{1}  \ar{d}{\Box} & {}& \Cat{T}^{2} \ar{d}{\Box}  \\
\Cat{T} \ar{rr}{1}  && \Cat{T}
    \end{tikzcd}
\stackrel{\lambda}{\Rrightarrow}
    \begin{tikzcd}[baseline=-0.65ex, column sep=large, row sep=large]
        {} & \Cat{T}^{3} \ar{d}{1 \times \Box} \ar{dr}[name=A]{\Box \times 1} &     \\
   \Cat{T}^{2}  \ar{d}{\Box}\ar{ur}{I \times 1 \times 1}  & \Cat{T}^{2} \ar{dr}[name=B, below]{\Box} \ar[shorten <= 8pt, shorten >=8pt, Rightarrow]{d}{l}
 & \Cat{T}^{2} \ar{d}{\Box} \\
\Cat{T} \ar{rr}[name=C, below]{1} \ar{ur}{I \times 1} &{}& \Cat{T}
 \arrow[shorten <= 15pt, shorten >=15pt, Rightarrow,to path=(A)-- (B) \tikztonodes]{}{a}
    \end{tikzcd}
  \end{equation}
\item \label{item:rho-general}  For all objects $a,b,c$, an invertible modification $\rho$

  \begin{equation}
    \label{eq:rho-gen}
 \begin{tikzcd}[baseline=-0.65ex, column sep=large, row sep=large] 
     {} & \Cat{T}^{3} \ar{r}{\Box \times 1}  \ar{d}{1 \times \Box}   & \Cat{T }^{2} \ar{d}{\Box}  \ar[shorten <= 10pt, shorten >=10pt, Rightarrow]{dl}{a} \\
    \Cat{T}^{2}    \ar{r}[name=C, xshift=-4pt]{1}    \ar{ur}{1 \times 1 \times I}  &  \Cat{T}^{2} \ar{r}{\Box}    &\Cat{T}
\latearrow{shorten <= 14pt, shorten >=8pt, Rightarrow, to path=-- (C) \tikztonodes}{1-2}{}{1 \times r}
    \end{tikzcd}
\stackrel{\rho}{\Rrightarrow}
    \begin{tikzcd}[baseline=-0.65ex, column sep=large, row sep=large]
     {} & \Cat{T}^{3} \ar{r}{\Box \times 1}  & \Cat{T }^{2} \ar{d}{\Box}  \\
    \Cat{T}^{2} \ar{r}{\Box}\ar{ur}{1 \times 1 \times I} & \Cat{T} \ar{r}[name=D]{1} \ar{ur}{  1 \times I} &\Cat{T} 
\latearrow{shorten <= 10pt, shorten >=9pt, Rightarrow,to path=-- (D) \tikztonodes}{1-3}{}{r}
    \end{tikzcd}
 \end{equation}
\item \label{item:pi-gen}
  For all objects $a,b,c,d,e$, an invertible modification $\pi$ 
  \begin{equation}
    \label{eq:pi-alg-tricat}
    \begin{tikzcd}[baseline=-0.65ex]
      {} &\Cat{T}^{4} \ar{dl}[left, above,  xshift=-8pt]{1 \times 1 \times \Box} \ar{rr}{\Box \times 1 \times 1}  \ar{dr}{1 \times \Box \times 1}&& \Cat{T}^{3}  \ar{dr}{\Box \times 1} 
\ar[shorten <= 8pt, shorten >=8pt, Rightarrow]{dl}{a \times 1}  & \\
 \Cat{T}^{3} \ar{dr}[below, left, xshift=-3pt]{1 \times \Box} & & \Cat{T}^{3}\ar{rr}{\Box \times 1} \ar{dl}{1 \times \Box}  \ar[shorten <= 25pt, shorten >=25pt, Rightarrow]{ll}{1 \times a}
 && \Cat{T}^{2} \ar{dl}{\Box} \ar[shorten <= 47pt, shorten >=47pt, Rightarrow]{dlll}{a} \\
& \Cat{T}^{2} \ar{rr}[below]{\Box} & {} 
\ar[shorten >=5pt,shorten <=5pt, doublearrow]{d}{\pi} 
  \ar[shorten >=5pt,shorten <=5pt, thirdline]{d}{}
 &  \Cat{T} & \\
     {} &\Cat{T}^{4} \ar{dl}[left, above,  xshift=-8pt]{1 \times 1 \times \Box} \ar{rr}{\Box \times 1 \times 1} &{}& \Cat{T}^{3}  \ar{dl}{1 \times \Box} \ar{dr}{\Box \times 1} 
 \ar[shorten <= 50pt, shorten >=50pt, Rightarrow]{dlll}{\id} & \\
 \Cat{T}^{3} \ar{dr}[below, left, xshift=-3pt]{1 \times \Box } \ar{rr}{\Box \times 1} & & \Cat{T}^{2} \ar{dr}{\Box}
\ar[shorten <= 8pt, shorten >=8pt, Rightarrow]{dl}{a}  && \Cat{T}^{2} \ar{dl}{\Box}  \ar[shorten <= 25pt, shorten >=25pt, Rightarrow]{ll}{a} \\
& \Cat{T}^{2} \ar{rr}[below]{\Box} & &  \Cat{T} &
    \end{tikzcd}
  \end{equation}
\end{definitionlist}
 This data is required to satisfy the following three axioms. 
 In the first axiom, the unmarked isomorphisms are isomorphisms induced by the naturality of  the associator $a$.
  \begin{enumerate}
  \item  \begin{equation}
      \label{eq:Axiom1-alg-tricat}
      \begin{tikzpicture}
        \tikzstyle{every node}=[font=\small]
%oberes Bild Vertices aussen
        \node (a) at (0,-2) {$((K(JH)G)F$};
        \node (b) at (2.5,-1) {$(K((JH)G))F$};
        \node (c) at (5,0) {$(K(J(HG)))F$};
        \node (d) at (7.5,-1) {$K((J(HG))F)$};
        \node (e) at (10,-2) {$K(J((HG)F))$};
        \node (f) at (10,-4) {$K(J(H(GF)))$};
        \node (g) at (6.66,-6) {$(KJ)(H(GF))$};
        \node (h) at (3.33,-6) {$((KJ)H)(GF)$};
        \node (i) at (0,-4) {$(((KJ)H)G)F$};
%oberes Bild vertices innen
        \node (k) at (3,-2.7)  {$((KJ)(HG))F$};
        \node (l) at (6,-3.3) {$(KJ)((HG)F)$};
              \node (z) at (5,-6){};

%unteres Bild vertices aussen
        \node (m) at (0,-10) {$((K(JH)G)F$};
        \node (n) at (2.5,-9)  {$(K((JH)G))F$};
        \node (o) at (5,-8) {$(K(J(HG)))F$};
        \node (p) at (7.5,-9) {$K((J(HG))F)$};
        \node (q) at (10,-10){$K(J((HG)F))$};
        \node (r) at (10,-12){$K(J(H(GF)))$};
        \node (s) at (6.66,-14)  {$(KJ)(H(GF))$};
        \node (t) at (3.33,-14) {$((KJ)H)(GF)$};
        \node (u) at (0,-12){$(((KJ)H)G)F$};
%unteres Bild vertices innen

        \node (w) at (3.33,-12)  {$(K(JH))(GF)$};
        \node (x) at (6.66,-12) {$K((JH)(GF))$};
        \node (y) at (5,-10) {$K(((JH)G)F)$};

%oberes Bild Vertices mit 3-morphismsn
 \node (a1) at (2.4,-1.8) {$\Downarrow \pi 1$};
 \node (a2) at (6.7,-1.9) {$\Downarrow \pi$};
 \node (a3) at (8.2,-3.7) {$\simeq$};
          \node (a5) at (3.7,-4.5) {$\Downarrow \pi$};

%unteres Bild Vertices mit 3-morphismsn

 \node (b1) at (5,-9) {$\simeq $};
 \node (b2) at (3.2,-10.4) {$\Downarrow \pi$};
 \node (b3) at (7.7,-10.3) {$\Downarrow 1 \pi$};
 \node (b4) at (1.65,-12) {$\simeq$};
 \node (b5) at (5.3,-13) {$\Downarrow \pi$};

%oberes Bild Kanten
         \path[->,font=\scriptsize,>=angle 90]
      
        (a) edge node[auto] {$1a$}  (b)
        (b) edge node[auto]  {$(1a)1$}(c)
        (c) edge node[auto]  {$a$}(d)
        (d) edge node[auto] {$1a$}(e)
        (e) edge node[auto] {$1(1a)$}(f)
        (g) edge node[auto] {$a$}(f)
        (h) edge node[auto] {$a$}(g)
        (i) edge node[auto] {$a$}(h)
        (i) edge node[auto] {$(a1)1$}(a)
        (i) edge node[auto] {$a1$}(k)
        (k) edge node[right] {$a1$}(c)
        (k) edge node[auto] {$a$}(l)
        (l) edge node[auto]  {$ a $}  (e)
        (l) edge node[auto] {$(11)a$}(g)
%unteres Bild Kanten 
        (m) edge node[auto] {$a 1$}  (n)
        (n) edge node[auto]  {$(1a)1$}(o)
        (o) edge node[auto]  {$a$}(p)
        (p) edge node[auto] {$1a$}(q)
        (q) edge node[auto] {$1(1a)$}(r)
        (u) edge node[auto] {$(a1)1$}(m)
        (u) edge node[auto] {$a$}(t)
        (t) edge node[auto] {$a$}(s)
        (s) edge node[auto] {$a$}(r)
        (m) edge node[auto] {$a$}(w)
        (t) edge node[auto] {$a(11)$}(w)
        (w) edge node[auto] {$1a$}(x)
        (x) edge node[auto] {$1a$}(r)
        (n) edge node[auto]   {$ a $}  (y)
        (y) edge node[below, xshift=8.3pt, yshift=3pt] {$1(a1) $}(p)
         (y) edge node[auto] {$1a $}(x);
 \draw[double,double equal sign distance,-,shorten <= 8pt, shorten >=8pt] (z) to node[below] {} (o); 
 % \path[equality/.style]
 %   (z) edge node[auto] {} (o);
        %\draw[double,double equal sign distance,-implies,shorten <= 40pt, shorten >=40pt] (g) to node[auto] {$b$} (jk);  [shorten >=7pt,shorten <=7pt]
      \end{tikzpicture}
    \end{equation}
\item 
 \begin{equation}
     \label{eq:Axiom2-alg-tricat}
  \begin{tikzpicture}[scale=0.9]
       \tikzstyle{every node}=[font=\small]
%oberes Bild Vertices aussen
        \node (a) at (0,0) {$((HI_{b})G)F$};
        \node (b) at (10,0) {$(H(I_{b}G))F$};
        \node (c) at (9.3,-4) {$(HG)F$};
        \node (d) at (5,-5.5) {$H(GF)$};
        \node (e) at (0.7,-4)  {$(HG)F$};
      
%oberes Bild vertices innen
  \node (f) at (2.3,-1.6)  {$(HI_{b})(GF)$};
        \node (g) at (5,-2.6)  {$H(I_{b}(GF))$};
        \node (h) at (7.7,-1.6) {$H((I_{b}G)F)$};
 
%unteres Bild vertices aussen
        %  \node (i) at (0,-7.5) {$((HI_{b})G)F$};
        % \node (j) at (10,-7.5) {$(H(I_{b}G))F$};
        % \node (k) at (9.3,-11.5) {$(HG)F$};
        % \node (l) at (5,-13) {$H(GF)$};
        % \node (m) at (0.7,-11.5)  {$(HG)F$};
%scaled
   \node (i) at (2.5,-7.85) {$((HI_{b})G)F$};
        \node (j) at (7.5,-7.85) {$(H(I_{b}G))F$};
        \node (k) at (7.15,-9.85) {$(HG)F$};
        \node (l) at (5,-10.6) {$H(GF)$};
        \node (m) at (2.85,-9.85)  {$(HG)F$};
          \node (z) at (5,-7.85)  {};

%oberes Bild Vertices mit 3-morphismsn
 \node (a1) at (5,-1.2) {$\Downarrow \pi $};
 \node (a2) at (5.9,-3.1) {$\Rightarrow 1\lambda$};
 \node (a3) at (4.1,-3.1) {$\Leftarrow \mu$};
       \node (a4) at (2,-3) {$\simeq$};
  \node (a5) at (8,-3) {$\simeq$};

%unteres Bild Vertices mit 3-morphismsn
 \node (b1) at (4.1,-8.5) {$\Rightarrow \mu 1$};
  \node (b2) at (5,-9.5) {$\simeq $};

% %oberes Bild Kanten
         \path[->,font=\scriptsize,>=angle 90]
      
        (a) edge node[auto] {$a1$}  (b)
        (b) edge node[auto]  {$(1l)1$}(c)
        (c) edge node[above]  {$a$}(d)
        (e) edge node[auto] {$a$}(d)
        (a) edge node[left] {$(r1)1$}(e)
        (a) edge node[auto] {$a$}(f)
        (f) edge node[auto] {$a$}(g)
        (f) edge node[left] {$r(11)$}(d)
        (g) edge node[auto] {$1l$}(d)
        (h) edge node[above] {$1a$}(g)
        (h) edge node[right] {$1(l1)$}(d)
        (b) edge node[above] {$a$}(h)
%
% %unteres Bild Kanten 
         (i) edge node[auto] {$a1$}  (j)
        (j) edge node[auto]  {$(1l)1$}(k)
        (k) edge node[above]  {$a$}(l)
        (m) edge node[auto] {$a$}(l)
        (i) edge node[left] {$(r1)1$}(m)
        (j) edge node[auto] {$(1l)1$}(m);

  \draw[double,double equal sign distance,-,shorten <= 14pt, shorten >=16pt] (d) to node[below] {} (z); 
      \end{tikzpicture}
  \end{equation}
\item 
 \begin{equation}
     \label{eq:Axiom3-alg-tricat}
  \begin{tikzpicture}[scale=0.9]
       \tikzstyle{every node}=[font=\small]
%oberes Bild Vertices aussen
        \node (a) at (0,0) {$H((G I_{c})F)$};
        \node (b) at (10,0) {$H(G (I_{c}F))$};
        \node (c) at (9.3,-4) {$H(GF)$};
        \node (d) at (5,-5.5) {$(HG)F$};
        \node (e) at (0.7,-4)  {$H(GF)$};
      
%oberes Bild vertices innen
  \node (f) at (2.3,-1.6)  {$(H(GI_{c}))F$};
        \node (g) at (5,-2.6)  {$((HG)I_{c})F$};
        \node (h) at (7.7,-1.6) {$(HG)(I_{c}F)$};
 
%unteres Bild vertices aussen
        %  \node (i) at (0,-7.5) {$((H)G)F$};
        % \node (j) at (10,-7.5) {$(H(G))F$};
        % \node (k) at (9.3,-11.5) {$(HG)F$};
        % \node (l) at (5,-13) {$H(GF)$};
        % \node (m) at (0.7,-11.5)  {$(HG)F$};
%scaled unteres bild vertices
   \node (i) at (2.5,-7.85){$H((G I_{c})F)$};
        \node (j) at (7.5,-7.85)  {$H(G (I_{c}F))$};
        \node (k) at (7.15,-9.85) {$H(GF)$};
        \node (l) at (5,-10.6)  {$(HG)F$};
        \node (m) at (2.85,-9.85)   {$H(GF)$};
  \node (z) at (5,-7.85)  {};

%oberes Bild Vertices mit 3-morphismsn
 \node (a1) at (5,-1.2) {$\Downarrow \pi $};
 \node (a2) at (5.9,-3.1) {$\Rightarrow \mu$};
 \node (a3) at (4.1,-3.1) {$\Leftarrow \rho 1$};
       \node (a4) at (2,-3) {$\simeq$};
  \node (a5) at (8,-3) {$\simeq$};

%unteres Bild Vertices mit 3-morphismsn
 \node (b1) at (4.1,-8.5) {$\Rightarrow 1 \mu $};
  \node (b2) at (5,-9.5) {$\simeq $};

% %oberes Bild Kanten
         \path[->,font=\scriptsize,>=angle 90]
      
        (a) edge node[auto] {$1a$}  (b)
        (b) edge node[auto]  {$1(1l)$}(c)
        (d) edge node[above]  {$a$}(c)
        (d) edge node[auto] {$a$}(e)
        (a) edge node[left] {$1(r1)$}(e)
        (f) edge node[auto] {$a$}(a)
        (g) edge node[auto] {$a1$}(f)
        (f) edge node[left] {$(1r)1$}(d)
        (g) edge node[auto] {$r1$}(d)
        (g) edge node[above] {$a$}(h)
        (h) edge node[right] {$(11)l$}(d)
        (h) edge node[above] {$a$}(b)
%
% %unteres Bild Kanten 
         (i) edge node[auto] {$1a$}  (j)
        (j) edge node[auto]  {$1(1l)$}(k)
        (l) edge node[above]  {$a$}(k)
        (l) edge node[auto] {$a$}(m)
        (i) edge node[left] {$1(1r)$}(m)
        (j) edge node[auto] {$1(1r)$}(m);
      
  \draw[double,double equal sign distance,-,shorten <= 14pt, shorten >=16pt] (d) to node[below] {} (z); 
      \end{tikzpicture}
  \end{equation}

  \end{enumerate}

\end{definition}
\begin{remark}
\label{remark:tricategory-convent}
  Our definition of a tricategory differs from \cite{Gurski} in that we replaced the arrow of the 
right unit $r^{-}$ in the definition of the pseudo natural transformation $\mu$ in \cite{Gurski} with its adjoint $r$. Consequently the axioms (\ref{eq:Axiom2-alg-tricat}) and (\ref{eq:Axiom3-alg-tricat}) have a different shape. It is straightforward to see that the two definitions are equivalent. 
\end{remark}

\bibliographystyle{hplain}	% (uses file "plain.bst")  
\bibliography{bimod}		

\begin{thebibliography}{10}

\bibitem{GrayDuals}
J.~W. Barrett, C.~Meusburger, and G.~Schaumann.
\newblock Gray categories with duals and their diagrams, November 2012,
  arXiv:1211.0529.

\bibitem{BarWes}
J.~W. Barrett and B.~W. Westbury.
\newblock {Invariants of piecewise-linear 3-manifolds.}
\newblock {\em Trans. Am. Math. Soc.}, 348(10):3997--4022, 1996.

\bibitem{Benabou}
J.~B{\'e}nabou.
\newblock Introduction to bicategories.
\newblock In {\em Reports of the {M}idwest {C}ategory {S}eminar}, pages 1--77.
  Springer, Berlin, 1967.

\bibitem{BurWald}
H.~Bursztyn and S.~Waldmann.
\newblock {Completely positive inner products and strong Morita equivalence.}
\newblock {\em Pac. J. Math.}, 222(2):201--236, 2005.

\bibitem{DavyRun}
A.~Davydov, L.~Kong, and I.~Runkel.
\newblock Field theories with defects and the centre functor.
\newblock In {\em Mathematical foundations of quantum field theory and
  perturbative string theory}, volume~83 of {\em Proc. Sympos. Pure Math.},
  pages 71--128. Amer. Math. Soc., Providence, RI, 2011.

\bibitem{DelCat}
P.~Deligne.
\newblock {Cat\'egories tannakiennes. (Tannaka categories).}
\newblock {The Grothendieck Festschrift, Collect. Artic. in Honor of the 60th
  Birthday of A. Grothendieck. Vol. II, Prog. Math. 87, 111-195 }, 1990.

\bibitem{DeligneSes}
P.~Deligne.
\newblock La cat\'egorie des repr\'esentations du groupe sym\'etrique {$S_t$},
  lorsque {$t$} n'est pas un entier naturel.
\newblock In {\em Algebraic groups and homogeneous spaces}, Tata Inst. Fund.
  Res. Stud. Math., pages 209--273. Tata Inst. Fund. Res., Mumbai, 2007.

\bibitem{DSS}
C.L. Douglas, C.~Schommer-Pries, and N.~Snyder.
\newblock Dualizable tensor categories, 2013, arXiv:1312.7188.

\bibitem{DSStri}
C.L. Douglas, C.~Schommer-Pries, and N.~Snyder.
\newblock The 3-category of tensor categories, 2014, Draft.

\bibitem{DSSbal}
C.L. Douglas, C.~Schommer-Pries, and N.~Snyder.
\newblock The balanced tensor product of module categories, 2014, Draft.

\bibitem{ENO:Notes}
P.~Etingof, S.~Gelaki, D.~Nikshych, and V.~Ostrik.
\newblock Tensor categories.
\newblock {\url{http://www-math.mit.edu/~etingof/tenscat.pdf}}, 2009.
\newblock Lecture notes, MIT open course.

\bibitem{ENOfus}
P.~Etingof, D.~Nikshych, and V.~Ostrik.
\newblock On fusion categories.
\newblock {\em Ann. of Math. (2)}, 162(2):581--642, 2005.

\bibitem{ENOfuhom}
P.~Etingof, D.~Nikshych, and V.~Ostrik.
\newblock Fusion categories and homotopy theory.
\newblock {\em Quantum Topol.}, 1(3):209--273, 2010.
\newblock With an appendix by Ehud Meir.

\bibitem{FinTen}
P.~Etingof and V.~Ostrik.
\newblock Finite tensor categories.
\newblock {\em Mosc. Math. J.}, 4(3):627--654, 782--783, 2004.

\bibitem{DualDef}
J.~Fr{\"o}hlich, J.~Fuchs, I.~Runkel, and C.~Schweigert.
\newblock Duality and defects in rational conformal field theory.
\newblock {\em Nuclear Phys. B}, 763(3):354--430, 2007.

\bibitem{FuSchwVal}
J.~Fuchs, C.~Schweigert, and A.~Valentino.
\newblock Bicategories for boundary conditions and for surface defects in 3-d
  tft, 2012, arXiv:1203.4568.

\bibitem{Runkel}
M.~R. Gaberdiel and I.~Runkel.
\newblock {Logarithmic bulk and boundary conformal field theory and the full
  centre construction}.
\newblock 2012, arXiv:1201.6273.

\bibitem{GPS}
R.~Gordon, A.~J. Power, and R.~Street.
\newblock Coherence for tricategories.
\newblock {\em Mem. Amer. Math. Soc.}, 117(558):vi+81, 1995.

\bibitem{Green}
J.~Greenough.
\newblock {Monoidal 2-structure of bimodule categories.}
\newblock {\em J. Algebra}, 324(8):1818--1859, 2010.

\bibitem{Gurski}
M.~N. Gurski.
\newblock {\em An algebraic theory of tricategories}.
\newblock PhD thesis, University of Chicago, 2006.

\bibitem{Kap}
I.~Kaplansky.
\newblock Modules over operator algebras.
\newblock {\em Amer. J. Math.}, 75:839--858, 1953.

\bibitem{KapSau}
A.~Kapustin and N.~Saulina.
\newblock Surface operators in 3d topological field theory and 2d rational
  conformal field theory.
\newblock In {\em Mathematical foundations of quantum field theory and
  perturbative string theory}, volume~83 of {\em Proc. Sympos. Pure Math.},
  pages 175--198. Amer. Math. Soc., Providence, RI, 2011.

\bibitem{KS}
G.~M. Kelly and R.~Street.
\newblock Review of the elements of {$2$}-categories.
\newblock In {\em Category {S}eminar ({P}roc. {S}em., {S}ydney, 1972/1973)},
  pages 75--103. Lecture Notes in Math., Vol. 420. Springer, Berlin, 1974.

\bibitem{Lyu}
T.~Kerler and V.~V. Lyubashenko.
\newblock {\em Non-semisimple topological quantum field theories for
  3-manifolds with corners}, volume 1765 of {\em Lecture Notes in Mathematics}.
\newblock Springer-Verlag, Berlin, 2001.

\bibitem{Kong}
A.~Kitaev and L.~Kong.
\newblock Models for {G}apped {B}oundaries and {D}omain {W}alls.
\newblock {\em Comm. Math. Phys.}, 313(2):351--373, 2012.

\bibitem{Lurie}
J.~Lurie.
\newblock On the classification of topological field theories.
\newblock In {\em Current developments in mathematics, 2008}, pages 129--280.
  Int. Press, Somerville, MA, 2009.

\bibitem{Ostrik}
V.~Ostrik.
\newblock {Module categories, weak Hopf algebras and modular invariants.}
\newblock {\em Transform. Groups}, 8(2):177--206, 2003.

\bibitem{Rief}
M.~A. Rieffel.
\newblock {Morita equivalence for C$^*$-algebras and W$^*$-algebras.}
\newblock {\em J. pure appl. Algebra}, 5:51--96, 1974.

\bibitem{Schaum}
G.~Schaumann.
\newblock {\em Duals in tricategories and in the tricategory of bimodule
  categories}.
\newblock PhD thesis, Friedrich-Alexander-Universit\"{a}t
  Erlangen-N\"{u}rnberg,
  \url{http://opus4.kobv.de/opus4-fau/frontdoor/index/index/docId/3732},
  September 2013.

\bibitem{Moduletr}
G.~Schaumann.
\newblock Traces on module categories over fusion categories.
\newblock {\em J. Algebra}, 379:382--425, 2013.

\bibitem{SchwHopf}
C.~Schweigert and J.~Fuchs.
\newblock {Hopf algebras and Frobenius algebras in finite tensor categories.}
\newblock In {\em {Highlights in Lie algebraic methods. Outgrowth of a two-week
  summer school on structures in Lie theory, crystals, derived functors,
  Harish-Chandra modules, invariants and quivers at Jacobs University, Bremen,
  Germany August 9--22, 2009}}, pages 189--203. Basel: Birkh\"auser, 2012.

\bibitem{TurVir}
V.~G. Turaev and O.~Ya. Viro.
\newblock State sum invariants of {$3$}-manifolds and quantum {$6j$}-symbols.
\newblock {\em Topology}, 31(4):865--902, 1992.

\end{thebibliography}
\end{document}